\documentclass[leqno,11pt]{amsart}

\usepackage{amsmath,amstext,amssymb,amsopn,amsthm,mathrsfs}
\usepackage{bbm}
\usepackage{verbatim}
\usepackage{enumerate}  
\usepackage{enumitem}
\usepackage{lipsum}
\usepackage{color}

\newcommand\blfootnote[1]{%
  \begingroup
  \renewcommand\thefootnote{}\footnote{#1}%
  \addtocounter{footnote}{-1}%
  \endgroup
}

\newcommand*\RR{\mathbb{R}}

\newcommand*\ZZ{\mathbb{Z}}

\newcommand*\al{\alpha}
\newcommand*\be{\beta}

\newcommand*\te{\theta}
\newcommand*\va{\varphi}

\newcommand*\dd{{\rm d} }

\newcommand*\Apol{\mathbf{A}}

\allowdisplaybreaks

\theoremstyle{plain}
\newtheorem{thm}{Theorem}[section]
\newtheorem{bigthm}{Theorem}
\newtheorem{biglem}[bigthm]{Lemma}
\newtheorem{thm2}{Theorem} 
\newtheorem{mthm}[thm2]{Theorem} 
\newtheorem{lm}[thm]{Lemma}
\newtheorem{lem}[thm]{Lemma}
\newtheorem{prop}[thm]{Proposition}
\newtheorem{cor}[thm]{Corollary}

\newtheorem{rem}[thm]{Remark}

\numberwithin{equation}{section}

\theoremstyle{plain}

\DeclareMathOperator{\dist}{dist}

\DeclareMathOperator*{\vol}{vol}
\DeclareMathOperator*{\diam}{diam}

\def\a{\alpha}
\def\b{\beta}
\def\ab{\alpha,\beta}
\def\r{\varrho}

\author[A.\ Nowak]{Adam Nowak} 
\address[Adam Nowak]{Institute of Mathematics,
Polish Academy of Sciences,
\'Sniadeckich 8,
00-656 Warsaw, Poland}
\email{anowak@impan.pl}

\author[P.\ Plewa]{Pawe\l{} Plewa} 
\address[Pawe\l{} Plewa]{
Faculty of Mathematics,
Wroc{\l}aw University of Science and Technology,
	Wyb.\ Wys\-pia{\'n}\-skie\-go  27,
	50-370 Wroc{\l}aw, Poland
}
\email{pawel.plewa@pwr.edu.pl}

\author[T.Z.\ Szarek]{Tomasz Z.\ Szarek} 
\address[Tomasz Z. Szarek]{
	Department of Mathematics, University of Georgia,
Athens, GA 30602, USA
and
Mathematical Institute,
University of Wroc{\l}aw,
Pl.\ Grunwaldzki 2,
50-384 Wroc{\l}aw,
Poland}
\email{tzs10705@uga.edu}

\begin{document}

\title[Jacobi heat kernel]{A unified proof \\ of sharp bounds for the Jacobi heat kernel \\ 
 with trace and estimates of multiplicative constants \medskip}

\begin{abstract}
We give a unified and optimized proof of the sharp bounds for the Jacobi heat kernel,
which were obtained gradually in several papers in recent years.
We lay particular emphasis on tracing and estimating all constants appearing throughout the entire reasoning.
This allows us to quantitatively control the multiplicative constants in the Jacobi
heat kernel bounds in terms of the parameters involved.
Consequently, analogous control extends to a number of interrelated heat kernels.
In particular, we obtain quantitative control in terms of the associated dimension for the spherical heat kernel
and for all other heat kernels on compact rank one symmetric spaces.
\end{abstract}

\thanks{The second-named author was supported by the Foundation for Polish Science (START 057.2023). 
	The third-named author was supported by
	the Simons Foundation grant SFI-MPS-TSM-00013714 and by
	the National Science Centre, Poland, grant Sonata Bis 2022/46/E/ST1/00036.}

\maketitle

\thispagestyle{empty}


\newpage
\,\vfill
{
\tableofcontents
\vfill
}
\newpage

\section{Introduction} \label{sec:intro}

Let $\a,\b > -1$. The \emph{\textbf{Jacobi heat kernel}} associated with the parameters $\a$ and $\b$ is given by\,\footnote{
\,The numbers $n(n+\a+\b+1)$ occurring in the exponential factors are the eigenvalues of the Jacobi Laplacian.
} 
\begin{equation*}
G_t^{\ab}(x,y) = \sum_{n=0}^{\infty} e^{-t n(n+\alpha+\beta+1)} \; 
\frac{P_n^{\ab}(x) P_n^{\ab}(y)}{h_n^{\ab}}, \qquad x,y \in [-1,1], \quad t>0.
\end{equation*}
Here, $P_n^{\ab}$ are the classical \emph{Jacobi polynomials} (cf.\ \cite{Sz}), and
$h_n^{\ab} := \|P_n^{\ab}\|^2_{L^2(d\r_{\ab})}$ with the Jacobi orthogonality measure on $[-1,1]$
$$
\dd\r_{\ab}(x) = (1-x)^{\alpha} (1+x)^{\beta}\, \dd x.
$$
The normalizing constants have the explicit form
$$
h_n^{\ab} = \frac{2^{\alpha+\beta+1}\Gamma(n+\alpha+1)\Gamma(n+\beta+1)}
{(2n+\alpha+\beta+1)\Gamma(n+\alpha+\beta+1)\Gamma(n+1)},
$$
where for $n=0$ and $\alpha+\beta=-1$ the product $(2n+\alpha+\beta+1)\Gamma(n+\alpha+\beta+1)$
is interpreted as $1$.
Note that $h_0^{\ab} = \r_{\ab}([-1,1])$.

The Jacobi heat kernel is an important object in analysis, probability, physics, and other areas.
The kernel $G_t^{\ab}(x,y)$ provides solution to the heat equation based on the \emph{Jacobi Laplacian}
$$
J^{\ab} = - \big( 1-x^2 \big) \frac{\dd^2}{\dd x^2} - \big[ \b-\a - (\a+\b+2)x\big] \frac{\dd}{\dd x}.
$$
From a probabilistic perspective, $G_t^{\ab}(x,y)$ is the transition probability density for the \emph{Jacobi diffusion process},
which has attracted considerable attention due to its utility in stochastic modeling
in physics, neuroscience, genetics, environmental science, and economics, among others; see the Internet.
Investigations related to the behavior of the Jacobi heat kernel have deep theoretical and practical motivations.
In particular, it subordinates in a sense a number of other more elementary heat kernels that are important and of interest.
One of the most prominent examples here is the \emph{spherical heat kernel}.

\bigskip

For $\theta,\varphi \in [0,\pi]$ and $t > 0$ define
$$
Z^{\a,\b}(t;\theta,\varphi) := \big[ t + \theta\varphi\big]^{-\a-1/2} \big[ t + (\pi-\theta)(\pi-\varphi)\big]^{-\b-1/2}\,
	\frac{1}{\sqrt{t}} \exp\bigg( - \frac{(\theta-\varphi)^2}{4t} \bigg).
$$
The following sharp estimates were proved in a series of papers by Nowak, Sj\"ogren and Szarek.
\begin{mthm}[{\cite{NSS,NSS2,NSS3}}] \label{thm:jhk}
Let $\a,\b > -1$ and $T > 0$ be fixed. There exists a constant $C=C(\a,\b,T)>1$ such that
$$
C^{-1}\cdot Z^{\a,\b}(t;\theta,\varphi) \le G_t^{\a,\b}(\cos\theta,\cos\varphi) \le C \cdot Z^{\a,\b}(t;\theta,\varphi)
$$
for $\theta,\varphi \in [0,\pi]$ and $0 < t \le T$.
\end{mthm}
A direct consequence of the above, see e.g., \cite[p.\,233]{NoSj} or Section \ref{sec:lmtime} below, are sharp large time estimates
$$
C^{-1} \le G_t^{\a,\b}(x,y) \le C, \qquad x,y \in [-1,1], \quad t \ge T,
$$
for any $\a,\b > -1$ and $T >0$, with some constant $C=C(\a,\b,T)>1$ independent of $x,y,t$.

\bigskip

A major deficiency of the above estimates is the complete lack of information about the sizes of the
constants. The main objective of this paper is to address this gap, which is crucial for various applications.
Our secondary objective is to optimize and unify the entire existing proof of the sharp bounds,
which is scattered in \cite{NSS,NSS2,NSS3}.

One should emphasize that the oscillatory series defining the Jacobi heat kernel cannot be computed explicitly.
However, it is known that in the four simple cases $\a,\b = \pm 1/2$ the kernel can be written by means of
non-oscillating series as follows.\,\footnote{
\,The argument is based on simple initial-value problems for the classical heat equation on an interval with
Neumann/Dirichlet boundary conditions. Moreover, the formulas for $\mathbb{G}_t^{-1/2,-1/2}(\theta,\varphi)$ and
$\mathbb{G}_t^{1/2,1/2}(\theta,\varphi)$ are, in principle, well-known Jacobi type identities, see \cite[Sec.\,1.8.4]{DM}.
}
Let
$$
\vartheta_t(z) =
	\sum_{n \in \ZZ} \frac{1}{\sqrt{4\pi t}}\,\exp\bigg(-\frac{(z+2\pi n)^2}{4t}\bigg), \qquad t > 0,
$$
be the periodized Gauss-Weierstrass kernel.\,\footnote{
	\,The function $\vartheta_t$ is essentially one of the Jacobi theta functions, which has deep
	connections with elliptic functions and number theory. In the notation of
	the Bateman Manuscript Project \cite[Vol.\,2, Sec.\,13.19]{Bat}, one has
	$\vartheta_t(z) = \frac{1}{2\pi}\theta_3(\frac{z}{2\pi}| \frac{it}{\pi})$ 
	(see e.g., \cite[Sec.\,1.7.5]{DM} for the relevant identity).
}
Further, denote
$$
\mathbb{G}_t^{\ab}(\theta,\varphi) := 2^{\a+\b+1} e^{-t(\frac{\a+\b+1}2)^2}
\Big( \sin\frac{\theta}2\sin\frac{\varphi}2\Big)^{\a+1/2}
\Big( \cos\frac{\theta}2\cos\frac{\varphi}2\Big)^{\b+1/2} G_t^{\ab}(\cos\theta,\cos\varphi);
$$
this is the heat kernel associated with Jacobi function expansions, see e.g., \cite[Sec.\,2]{NoSj}.
Then
\begin{align*}
\mathbb{G}_t^{-1/2,-1/2}(\theta,\varphi) & = \vartheta_t(\theta-\varphi) + \vartheta_t(\theta+\varphi), \\
\mathbb{G}_t^{1/2,1/2}(\theta,\varphi) & = 
\vartheta_t(\theta-\varphi) - \vartheta_t(\theta+\varphi),\\
\mathbb{G}_t^{-1/2,1/2}(\theta,\varphi) & = \frac{1}{2} 
\bigg[ \vartheta_{t/4}\Big(\frac{\theta}2 - \frac{\varphi}2\Big) + 
\vartheta_{t/4}\Big(\frac{\theta}2 + \frac{\varphi}2\Big) -
\vartheta_{t/4}\Big(\frac{\theta}2 - \frac{\varphi}2 + \pi\Big) -
\vartheta_{t/4}\Big(\frac{\theta}2 + \frac{\varphi}2 + \pi\Big)\bigg],\\
\mathbb{G}_t^{1/2,-1/2}(\theta,\varphi) & = \frac{1}{2} 
\bigg[ \vartheta_{t/4}\Big(\frac{\theta}2 - \frac{\varphi}2\Big) - 
\vartheta_{t/4}\Big(\frac{\theta}2 + \frac{\varphi}2\Big) -
\vartheta_{t/4}\Big(\frac{\theta}2 - \frac{\varphi}2 + \pi\Big) +
\vartheta_{t/4}\Big(\frac{\theta}2 + \frac{\varphi}2 + \pi\Big)\bigg].
\end{align*}
Accordingly, one obtains analogous formulas for $G_t^{\pm 1/2, \pm 1/2}(\cos\theta,\cos\varphi)$,
with a limiting interpretation in certain cases. 
These formulas allow one to describe the behavior in a direct,
though technically sophisticated, way. 
Nevertheless, our approach in this paper is different and oriented toward general $\a$ and $\b$.

\subsubsection*{\textbf{\emph{Historical background}}}
Special instances of the Jacobi heat kernel have existed (at least implicitly) in the literature since the 19th century,
and the question of describing its behavior has been a long-standing open problem.
Strict positivity of $G_t^{\a,\b}(x,y)$ was shown by Karlin and McGregor \cite{KM} in 1960.
Slightly earlier, Bochner \cite{B} proved that the ultraspherical heat kernel $G_t^{\a,\a}(x,y)$ is non-negative when $\a \ge -1/2$.
Some later results on the positivity, from the 1970s, can be found in Gasper \cite{Ga} and Bochner \cite{B2}.

Apart from the positivity, the behavior of the Jacobi heat kernel has not been effectively investigated until quite recently.
Qualitatively sharp bounds for $G_t^{\a,\b}(x,y)$ were obtained by Coulhon, Kerkyacharian and Petrushev \cite{CKP}, using the theory of Dirichlet forms and other abstract tools. Independently, and with a completely different analytic approach,
qualitatively sharp estimates were proved by Nowak and Sj\"ogren \cite{NoSj}, under the restriction $\a,\b \ge -1/2$.

The difference between (genuinely) sharp and qualitatively sharp estimates is that, in the latter case, the constants 
appearing in the exponential factors in the lower and upper bounds (which are $1/4$ in Theorem \ref{thm:jhk}) differ
from each other and from the optimal one. This exponential gap absorbs various polynomial factors and therefore
significantly affects the precision of the bounds.

Sharp estimates for $G_t^{\a,\b}(x,y)$, see Theorem \ref{thm:jhk}, were derived gradually in \cite{NSS,NSS2,NSS3}.
More precisely, in \cite{NSS}, the result was obtained when $\a=\b=d/2-1$, $d \ge 1$, and with $y=1$ the endpoint.
Then in \cite{NSS2} all parameters $\a,\b \ge -1/2$ were admitted and the constraint on $y$ removed.
Finally, the most difficult case when $\a$ or $\b$ is below $-1/2$ was recently dealt with in \cite{NSS3}.

\subsubsection*{\textbf{\emph{Description of present results}}}
We give a unified and optimized proof of the sharp bounds stated in Theorem \ref{thm:jhk}.
Comparing to the known reasoning \cite{NSS,NSS2,NSS3}, in the unified proof
we avoid many intermediate estimates that deal with the most delicate situation when the arguments of the heat kernel
lie at opposite endpoints of $[-1,1]$. Moreover, we bypass multiple iterations of certain intermediate steps what
leads to more optimal multiplicative constants.

Following the main objective of this paper, we make all the relevant constants explicit in the reasoning, so that
multiplicative constants in the Jacobi heat kernel bounds can be directly traced back through the whole proof
and, consequently, explicitly controlled in terms of the parameters and the time threshold.
This requires overcoming numerous conceptual and technical difficulties.

The unified proof has a deliberate block structure, consisting of Steps A to F, presented in Section~\ref{sec:uproof}.
These Steps form autonomous modules, which can be independently analyzed and possibly refined separately in the future.
The whole proof combines many facts/tools/results related to the Jacobi framework and also involves other auxiliary technical results.
Perhaps the most crucial device underlying the presented analysis is a product formula for Jacobi polynomials due to
Dijksma and Koornwinder \cite{DK}.
It is worth noting that aspects of the theory of special functions and combinatorics are also involved in the reasoning.

\medskip

Let $\a,\b > -1$. For $\theta,\varphi \in [0,\pi]$, $t > 0$, and a constant $\kappa > 0$ define
$$
\Xi^{\a,\b}_{\kappa}(t; \theta,\varphi) := 
 \big[ \mathbb{D}_{\a}t \vee \kappa\theta\varphi\big]^{-\a-1/2}
	\big[ \mathbb{D}_{\b}t \vee \kappa(\pi-\theta)(\pi-\varphi) \big]^{-\b-1/2}
	\frac{1}{\sqrt{t}} \exp\bigg( - \frac{(\theta-\varphi)^2}{4t} \bigg),
$$
where $\mathbb{D}_{\a} = \Gamma(\a+3/2)^{1/(\a+1/2)}$ with the limiting understanding when $\a=-1/2$
(see Section \ref{ssec:int}), and $\vee$ denotes the maximum of two quantities.
The constant $\mathbb{D}_{\a}$ depends approximately linearly on $\a$, see Remark \ref{rem:D}.
With $\kappa > 0$ and $\a,\b > -1$ fixed, the quantity $\Xi^{\a,\b}_{\kappa}(t;\theta,\varphi)$ is explicitly comparable
with $Z^{\a,\b}(t;\theta,\varphi)$, uniformly in $\theta,\varphi \in [0,\pi]$ and $t > 0$.
It turns out that $\Xi^{\a,\b}_{\kappa}(t;\theta,\varphi)$ is better suited
to our approach for describing the behavior of the Jacobi heat kernel.
In particular, it allows us to formulate more accurate estimates than the analogous ones based on $Z^{\a,\b}(t;\theta,\varphi)$.

We now state a principal result that
follows from the unified proof. See Lemma \ref{lem:F} in Section \ref{sec:uproof} and Proposition
\ref{prop:c1} together with \eqref{sc111} and Propositions \ref{prop:c2}, \ref{prop:c3} and \ref{prop:c4} in Section \ref{ssec:Jacobi}.
\begin{mthm} \label{thm:main}
Let $\a,\b > -1$ and $T > 0$ be fixed. There exist positive constants $C_{\a,\b,T}$ and $c_{\a,\b,T}$ such that
$$
c_{\a,\b,T} \cdot \Xi^{\a,\b}_{1/2}(t;\theta,\varphi) \le G_t^{\a,\b}(\cos\theta,\cos\varphi) \le
C_{\a,\b,T} \cdot \Xi^{\a,\b}_{2/\pi^2}(t;\theta,\varphi)
$$
for $\theta,\varphi \in [0,\pi]$ and $0 < t \le T$.
Moreover, the multiplicative constants satisfy the following bounds, given the indicated time thresholds.
\begin{itemize}
\item[(i)] For $\a,\b \ge -1/2$ such that $\a+\b+1 \in \mathbb{N}$,
\begin{align*}
c_{\a,\b,T} & \ge \frac{1}{11\cdot (18/5)^{\a+\b+1}} \qquad \textrm{for} \quad T > 0, \\
C_{\a,\b,T} & \le \begin{cases}
										(11/3) \cdot (17/5)^{\a+\b+1} & \quad \textrm{if} \quad T = 2/(\a+\b+3/2) \\
										16 \cdot (8/5)^{\a+\b+1} & \quad \textrm{if} \quad T= 1/(\a+\b+3/2)^2
									\end{cases}.
\end{align*}
\item[(ii)] For $\a,\b \ge -1/2$ such that $\a+\b+1/2 \in \mathbb{N}$,
\begin{align*}
c_{\a,\b,T} & \ge \frac{1}{14\cdot 31^{\a+\b+1}} \qquad \textrm{if} \quad T=4/(\a+\b+5/4), \\
C_{\a,\b,T} & \le \begin{cases}
										11 \cdot 23^{\a+\b+1} & \quad \textrm{if} \quad T = 4/(\a+\b+5/4) \\
										45 \cdot 5^{\a+\b+1} & \quad \textrm{if} \quad T= 1/(\a+\b+5/4)^2
									\end{cases}.
\end{align*}
\item[(iii)] For $\a,\b \ge -1/2$ such that $\a+\b > -1/2$ and $2\a+2\b+1 \notin \mathbb{N}$,
\begin{align*}
c_{\a,\b,T} & \ge \frac{1}{318\cdot 31^{\a+\b+1}} \qquad \textrm{if} \quad T=4/(\a+\b+7/4), \\
C_{\a,\b,T} & \le \begin{cases}
										28 \cdot 23^{\a+\b+1} & \quad \textrm{if} \quad T = 4/(\a+\b+7/4) \\
										124 \cdot 5^{\a+\b+1} & \quad \textrm{if} \quad T= 1/(\a+\b+7/4)^2
									\end{cases}.
\end{align*}
\item[(iv)] For $\a,\b \ge -1/2$ such that $-1 < \a+\b < -1/2$,
\begin{align*}
c_{\a,\b,T} & \ge \frac{1}{8\cdot 10^{6}} \qquad \textrm{if} \quad T=16/11, \\
C_{\a,\b,T} & \le \begin{cases}
										20107 & \quad \textrm{if} \quad T = 16/11 \\
										1189 & \quad \textrm{if} \quad T= 16/121
									\end{cases}.
\end{align*}
\end{itemize}
\end{mthm}

Items (i)--(iv) of Theorem \ref{thm:main} cover all parameters $\a,\b \ge -1/2$.
From the unified proof, analogous estimates can be obtained for the remaining values of the parameters. This requires more effort, but
is merely a matter of lengthy technical elaboration. Since the case when one of the parameters is below $-1/2$ is not
crucial for applications, we leave the details to interested readers.

We believe that the ratio of the upper and lower multiplicative constants in
Theorem \ref{thm:main} must grow no slower than exponentially with $\a$ and $\b$.
This is due to a certain incompatibility between the bounds' structure and the actual behavior of the kernel.
Estimates for the ratio in the case $\a,\b \ge -1/2$ can be found in Propositions
\ref{prop:c1}, \ref{prop:c2}, \ref{prop:c3} and \ref{prop:c4} in Section \ref{ssec:Jacobi}.
Those bounds grow exponentially in $\a$ and $\b$.
On the other hand, no worse than only exponential loss was intended by us and, in fact, it required quite a lot of effort
to obtain it comparing to e.g.\ a loss of order of magnitude given by the gamma function.
The optimal choice of the structure and shape of the expression describing the behavior of
the kernel is, no doubt, a tricky question, beyond the scope of this paper.

\medskip

Theorem \ref{thm:main} deals essentially with small times. But this also covers medium and large times in a well-known direct way,
see Section \ref{sec:lmtime}. Indeed, if $G_t^{\a,\b}(x,y)$ is bounded below and above by two positive constants at some
$t=t_0 > 0$, independently of $x,y \in [-1,1]$, then it remains bounded by the same constants for all $t \ge t_0$.

On the other hand, it is known that, see e.g., \cite[Thm.\,A]{NoSj},
$$
G_t^{\a,\b}(x,y) \longrightarrow \frac{1}{h_0^{\a,\b}} = \frac{\Gamma(\a+\b+2)}{2^{\a+\b+1}\Gamma(\a+1)\Gamma(\b+1)},
	\qquad t \to \infty,
$$
and the convergence is uniform in $x,y \in [-1,1]$.
Thus, for large $t$, even sharp estimates of the kernel are unsatisfactory, since
they neither detect nor quantify the convergence.
Therefore, we present an additional result, see Theorem \ref{thm:large}, which makes the convergence quantitative in
terms of suitable estimates. In the statement below $\wedge$ denotes the minimum of two quantities.
\begin{mthm} \label{thm:mainL}
Let $\a,\b > -1$. Then for $x,y \in [-1,1]$ and $t \ge 2\log 2$
$$
\bigg| G_t^{\a,\b}(x,y) - \frac{1}{h_0^{\a,\b}}\bigg| \le \frac{\Gamma(\a+\b+2)}{2^{\a+\b}\Gamma(\a\wedge\b+2)\Gamma(\a\vee\b+1)}
	e^{-(t-1)(\a+\b+2)} \qquad \textrm{when} \quad \a \vee\b \ge -1/2
$$
and
$$
\bigg| G_t^{\a,\b}(x,y) - \frac{1}{h_0^{\a,\b}}\bigg| \le \frac{27\,\Gamma(\a+\b+2)}{5\cdot 2^{\a+\b+2}\Gamma(\a+2)\Gamma(\b+2)}
	e^{-(t-1)(\a+\b+2)} \qquad \textrm{when} \quad \a \vee\b < -1/2.
$$
\end{mthm}

\subsubsection*{\textbf{\emph{Applications}}}

The Jacobi heat kernel is intimately related to a number of other heat kernels that are of interest and importance.
The list includes, above all, the spherical heat kernel and, more generally, heat kernels on all
compact rank one symmetric spaces. Further, it encloses heat kernels corresponding to Fourier--Bessel and Fourier--Dini expansions, 
as well as heat kernels associated with multi-dimensional balls, simplices, cones, conical surfaces and other solids and surfaces
of revolution; see Section~\ref{ssec:other} for more details.

In this connection, our results for the Jacobi heat kernel can be used to deduce similar results for the other heat kernels just mentioned. 
As pivotal examples, we obtain sharp bounds with explicit multiplicative constants for heat kernels,
viz.\ Brownian motion\,\footnote{\,The heat kernels we consider correspond to time-scaled Brownian motions $B_{2t}$.
}
transition probability densities, on all compact rank one symmetric spaces.
The results stated below focus on small times, since this is the real essence of the matter.
Fully explicit, sharp medium and large time bounds can then be derived in a straightforward manner,
as discussed in Section \ref{sec:app}.

Sharp global estimates for the spherical heat kernel were obtained only recently in \cite{NSS}, and the result was new
even for the ordinary sphere of dimension $2$. We can now significantly enhance that result. Let $S^d$ be the Euclidean
unit sphere of dimension $d \ge 2$, equipped with the standard, non-normalized surface area measure.
The corresponding heat kernel depends on its arguments only through their spherical geodesic distance,
say $\phi$, and thus it can be viewed as a one-dimensional function $[0,\pi] \ni \phi \mapsto K_t^d(\phi)$.
See the beginning of Section \ref{sec:oddsph} and Section \ref{ssec:spherical}.

We first distinguish the case of the ordinary $2$-dimensional sphere, see \eqref{kr88} in Section \ref{ssec:spherical},
since it seems to be fundamental from the perspective of real-world applications.
\begin{mthm} \label{thm:2dsphere}
Let $0 < T \le 2\pi^2$. Then
$$
\frac{2\pi^{-1/2} e^{-\pi/4}}{\sqrt{\frac{1}2t \vee (\pi-\phi)}} \, \frac{1}{4\pi t} \exp\bigg(-\frac{\phi^2}{4t}\bigg)
\le K_t^2(\phi) \le
\frac{\sqrt{2}\,\pi^{3/2} e^{T/4}}{\sqrt{\frac{\pi^2}8t \vee (\pi-\phi)}} \, \frac{1}{4\pi t} \exp\bigg(-\frac{\phi^2}{4t}\bigg)
$$
for $\phi \in [0,\pi]$ and $0 < t \le T$.
\end{mthm}

Next, we give explicit sharp bounds for heat kernels on spheres of higher dimensions, see Section~\ref{ssec:spherical}.
\begin{mthm} \label{thm:sphere}
Let $d \ge 3$ and $\phi \in [0,\pi]$. Then
\begin{align*}
  K_t^d(\phi) & \ge \frac{(11/10)\cdot(3/4)^{d-1}}{\big[ \mathbb{D}_{d/2-1}t \,\vee\, \pi(\pi-\phi)/2 \big]^{(d-1)/2}}
	\, \frac{1}{(4\pi t)^{d/2}} \exp\bigg(-\frac{\phi^2}{4t}\bigg), \\
 K_t^d(\phi) & \le
 \frac{(8/5)\cdot(24/5)^{d-1}}{\big[ \mathbb{D}_{d/2-1}t \,\vee\, 2(\pi-\phi)/\pi \big]^{(d-1)/2}}
	\, \frac{1}{(4\pi t)^{d/2}} \exp\bigg(-\frac{\phi^2}{4t}\bigg)
\end{align*}
for $0 < t \le 2/(d-1/2)$, where $\mathbb{D}_{d/2-1} = \Gamma(d/2+1/2)^{2/(d-1)}$. Moreover,
$$
K_t^d(\phi) \le
 \frac{7\cdot(9/4)^{d-1}}{\big[ \mathbb{D}_{d/2-1}t \,\vee\, 2(\pi-\phi)/\pi \big]^{(d-1)/2}}
	\, \frac{1}{(4\pi t)^{d/2}} \exp\bigg(-\frac{\phi^2}{4t}\bigg)
$$
for $0 < t \le 1/(d-1/2)^2$.
\end{mthm}
A simple scaling argument provides analogous results for spheres of arbitrary radii.
In the above statements, we omit the one-dimensional sphere, i.e.\ the torus, since that case is classical;
see Section~\ref{ssec:spherical}.

Let $\mathbb{M}$ be a compact rank one symmetric space.
Such spaces are completely classified, they are Euclidean spheres, projective spaces over
reals, complex numbers and quaternions, and the projective plane over octonions; see Section \ref{ssec:CROSS}.
Each $\mathbb{M}$ is a Riemannian manifold. Thus, we consider $\mathbb{M}$ equipped with the Riemannian geodesic distance
and the Riemannian volume measure. For simplicity, we assume that the diameter of $\mathbb{M}$ is $\pi$;
this does not affect generality, by a simple scaling.
The heat kernel corresponding to $\mathbb{M}$ depends on its arguments only through their geodesic distance, say $\phi$,
and can therefore be interpreted as a one-dimensional function $[0,\pi] \ni \phi \mapsto K_t^{\mathbb{M}}(\phi)$.
See Section \ref{ssec:CROSS}.
Sharp estimates for $K_t^{\mathbb{M}}(\phi)$ with unspecified multiplicative constants were established recently in \cite{NSS2}.

In the following statements, the lowest possible dimensions are excluded, since in those cases the spaces in question are identified
with spheres of certain dimensions. For the results presented in Theorems \ref{thm:rproj}--\ref{thm:cayley},
see Section \ref{ssec:CROSS}.

\begin{mthm} \label{thm:rproj}
Let $\mathbb{M}$ be the real projective space of dimension $d \ge 2$ and diameter $\pi$. Let $\phi \in [0,\pi]$.
\begin{itemize}
\item[(i)] Assume $d$ is odd. Then
\begin{align*}
K_t^{\mathbb{M}}(\phi) & \le \frac{8}5 \cdot \Big(\frac{17}5 \Big)^{(d-1)/2}
	\,\frac{1}{(4\pi t)^{d/2}} \exp\bigg(-\frac{\phi^2}{4t}\bigg), \qquad 0 < t \le \frac{4}d, \\
K_t^{\mathbb{M}}(\phi) & \le 7 \cdot \Big(\frac{\pi}2\Big)^{(d-1)/2}
	\,\frac{1}{(4\pi t)^{d/2}} \exp\bigg(-\frac{\phi^2}{4t}\bigg), \qquad 0 < t \le \frac{4}{d^2}.
\end{align*}
Moreover, for $d=3$ and any $0 < T \le 2\pi^2$
$$
K_t^{\mathbb{M}}(\phi) \le \pi e^{T/4} \,\frac{1}{(4\pi t)^{3/2}} \exp\bigg(-\frac{\phi^2}{4t}\bigg), \qquad 0 < t \le T.
$$
\item[(ii)] Assume $d$ is even. If $d \ge 4$, then
\begin{align*}
K_t^{\mathbb{M}}(\phi) & \le \frac{16}5 \cdot 23^{(d-1)/2}
	\,\frac{1}{(4\pi t)^{d/2}} \exp\bigg(-\frac{\phi^2}{4t}\bigg), \qquad 0 < t \le \frac{8}{d-1/2}, \\
K_t^{\mathbb{M}}(\phi) & \le 14 \cdot 5^{(d-1)/2}
	\,\frac{1}{(4\pi t)^{d/2}} \exp\bigg(-\frac{\phi^2}{4t}\bigg), \qquad 0 < t \le \frac{4}{(d-1/2)^2}.
\end{align*}
If $d=2$, then
\begin{equation*}
K_t^{\mathbb{M}}(\phi) \le 4 \pi e^{T/16}
	\,\frac{1}{4\pi t} \exp\bigg(-\frac{\phi^2}{4t}\bigg), \qquad 0 < t \le T,
\end{equation*}
for any $0 < T \le 8 \pi^2$.
\end{itemize}
\end{mthm}
We remark that under the assumptions of Theorem \ref{thm:rproj} the lower bound
$$
K_t^{\mathbb{M}}(\phi) \ge \frac{1}{(4\pi t)^{d/2}} \exp\bigg(-\frac{\phi^2}{4t}\bigg), \qquad t > 0, 
$$
holds by a general theory, see Section \ref{ssec:CROSS} for the details.

\begin{mthm} \label{thm:cproj}
Let $\mathbb{M}$ be the complex projective space of real dimension $d \ge 4$ and diameter $\pi$. Let $\phi \in [0,\pi]$. Then
\begin{align*}
K_t^{\mathbb{M}}(\phi) & \ge \frac{3/5}{\sqrt{\pi t/4 \,\vee\, \pi(\pi-\phi)/2}}
\,\frac{1}{(4\pi t)^{d/2}} \exp\bigg(-\frac{\phi^2}{4t}\bigg), \qquad t > 0, \\
K_t^{\mathbb{M}}(\phi) & \le \frac{(16/5)\cdot (17/5)^{d/2}}{\sqrt{\pi t/4 \,\vee\, 2(\pi-\phi)/\pi}}
	\,\frac{1}{(4\pi t)^{d/2}} \exp\bigg(-\frac{\phi^2}{4t}\bigg), \qquad 0 < t \le \frac{4}{d+1}, \\
K_t^{\mathbb{M}}(\phi) & \le  \frac{14\cdot (\pi/2)^{d/2}}{\sqrt{\pi t/4 \,\vee\, 2(\pi-\phi)/\pi}}
	\,\frac{1}{(4\pi t)^{d/2}} \exp\bigg(-\frac{\phi^2}{4t}\bigg), \qquad 0 < t \le \frac{4}{(d+1)^2}.
\end{align*}
\end{mthm}

\begin{mthm} \label{thm:quproj}
Let $\mathbb{M}$ be the quaternionic projective space of real dimension $d \ge 8$ and diameter $\pi$. Let $\phi \in [0,\pi]$. Then
\begin{align*}
K_t^{\mathbb{M}}(\phi) & \ge \frac{1/2}{\big[{ \mathbb{D}_1\,t \,\vee\, \pi(\pi-\phi)/2}\big]^{3/2}}
\,\frac{1}{(4\pi t)^{d/2}} \exp\bigg(-\frac{\phi^2}{4t}\bigg), \qquad t > 0, \\
K_t^{\mathbb{M}}(\phi) & \le \frac{(23/5) \cdot (17/5)^{d/2+1}}{\big[{ \mathbb{D}_1\,t \,\vee\, 2(\pi-\phi)/\pi}\big]^{3/2}}
	\,\frac{1}{(4\pi t)^{d/2}} \exp\bigg(-\frac{\phi^2}{4t}\bigg), \qquad 0 < t \le \frac{4}{d+3}, \\
K_t^{\mathbb{M}}(\phi) & \le \frac{20 \cdot (\pi/2)^{d/2+1}}{\big[{ \mathbb{D}_1\,t \,\vee\, 2(\pi-\phi)/\pi}\big]^{3/2}}
	\,\frac{1}{(4\pi t)^{d/2}} \exp\bigg(-\frac{\phi^2}{4t}\bigg), \qquad 0 < t \le \frac{4}{(d+3)^2},
\end{align*}
where $\mathbb{D}_1 = (3\sqrt{\pi}/4)^{2/3}$.
\end{mthm}

Recall that the Cayley projective plane over octonions has real dimension $16$.
\begin{mthm} \label{thm:cayley}
Let $\mathbb{M}$ be the Cayley projective plane of diameter $\pi$. Let $\phi \in [0,\pi]$. Then
\begin{align*}
K_t^{\mathbb{M}}(\phi) & \ge \frac{1/4}{\big[{ \mathbb{D}_3\,t \,\vee\, \pi(\pi-\phi)/2}\big]^{7/2}}
\,\frac{1}{(4\pi t)^{8}} \exp\bigg(-\frac{\phi^2}{4t}\bigg), \qquad t > 0, \\
K_t^{\mathbb{M}}(\phi) & \le \frac{19 \cdot (17/5)^{11}}{\big[{ \mathbb{D}_3\,t \,\vee\, 2(\pi-\phi)/\pi}\big]^{7/2}}
	\,\frac{1}{(4\pi t)^{8}} \exp\bigg(-\frac{\phi^2}{4t}\bigg), \qquad 0 < t \le \frac{4}{23}, \\
K_t^{\mathbb{M}}(\phi) & \le \frac{80 \cdot (\pi/2)^{11}}{\big[{ \mathbb{D}_3\,t \,\vee\, 2(\pi-\phi)/\pi}\big]^{7/2}}
	\,\frac{1}{(4\pi t)^{8}} \exp\bigg(-\frac{\phi^2}{4t}\bigg), \qquad 0 < t \le \frac{4}{529},
\end{align*}
where $\mathbb{D}_3 = (105\sqrt{\pi}/16)^{2/7}$.
\end{mthm}

By the unified proof and arguments parallel to those leading to Theorems \ref{thm:2dsphere}--\ref{thm:cayley},
see Sections \ref{ssec:spherical} and \ref{ssec:CROSS},
one can also prove sharp bounds with explicit multiplicative constants for the (negative) derivatives of the heat kernels on
all compact rank one symmetric spaces, viewed as
functions of geodesic distances. This is based on the differentiation rule from Lemma \ref{lem:diff}.
Such bounds with unspecified multiplicative constants were obtained earlier in \cite[Cor.\,2]{NSS} and \cite[Cor.\,4.2]{NSS2}.

\medskip

We note that, excluding trivial cases,
sharp heat kernel estimates with explicit numerical values of multiplicative constants appear to be very rare in the literature.
In fact, we are not able to point out any reasonable concrete examples.

\subsubsection*{\textbf{\emph{Structure of the paper}}}
Section \ref{sec:prel} consists of technical preliminaries. First, it provides a description of the general notation used in the paper,
and then it gathers various facts, formulas and relevant results related to the Jacobi setting, as well as other tools
and auxiliary results.
Section~\ref{sec:oddsph} contains estimates related to odd-dimensional spheres, which constitute the foundations
of the unified proof of the Jacobi heat kernel bounds.
Section \ref{sec:uproof} presents the core of the unified proof, which is divided into Steps A--F concluded respectively with
Lemmas \ref{lem:A}--\ref{lem:F}. It also contains some refinements, stated in Lemmas \ref{lem:G} and \ref{lem:H},
for certain special cases that arise in applications.
Section~\ref{sec:lmtime} deals with medium and large time estimates.
Finally, in Section~\ref{sec:app}, previous results are applied to find explicit multiplicative constants in the Jacobi
heat kernel bounds and in heat kernel estimates on compact rank one symmetric spaces.
Also, comments pertaining to other interrelated heat kernels are given.

A tabulation of constants and selected auxiliary functions used throughout the paper can be found in Appendix.

\section{Technical preliminaries} \label{sec:prel}

We shall use various tools from \cite{NoSj,NSS0,NSS,NSS2,NSS3}.
We shall enhance some of them so that multiplicative constants they provide are explicit.
We will also need new tools and auxiliary results.
All that is gathered in this section.

\subsection{General notation} \label{ssec:not} \,

\medskip

The set of natural numbers $\{0,1,2,\ldots\}$ is denoted by $\mathbb{N}$.
By $\vee$ and $\wedge$ we indicate operations of taking maximum and minimum, respectively, of two quantities.
The priority of executing $\vee$ and $\wedge$ is always lower than multiplication/division but higher than addition/subtraction.

For non-negative quantities $X,Y$ and positive quantities $c,C$
$$
X \simeq \left\{ C \atop c \right\} Y \qquad \equiv \qquad   c Y \le X \le C Y.
$$
Multiplicative constants in lower bounds will always be denoted by lower case letters, and those in upper bounds by capital letters.
Positive numerical constants (independent of any parameters) will be denoted by Gothic letters with exception of
classical constants like $\pi$, $e$ and $\gamma$.
Approximations of numerical expressions and constants are given (with a few exceptions) up to six decimal places after the dot,
with rounding to the closest rational number of this form.

In notation of the Jacobi heat kernel and related objects where both the parameters $\ab$ appear, we abbreviate ``$\a,\!\a$''
(the Jacobi-ultraspherical case when the two parameters coincide) to ``$\a$'',
e.g.\ $G_t^{\lambda,\lambda}(x,y) \equiv G_t^{\lambda}(x,y)$, $h_n^{\lambda,\lambda} \equiv h_n^{\lambda}$, etc.

\subsection{Facts and formulas related to the Jacobi heat kernel} \label{ssec:jac} \,

\medskip

Let $\a,\b > -1$.
The kernel $G_t^{\ab}(x,y)$ is a strictly positive smooth function of $(t,x,y) \in (0,\infty)\times [-1,1]^2$.
For each $y \in [-1,1]$ fixed, in the variables $(t,x)$ it satisfies the heat equation based on the Jacobi operator $J^{\ab}$,
\begin{equation} \label{he_G}
\Big( \frac{\dd}{\dd t} + J^{\ab}_x \Big) G_t^{\ab}(x,y) = 0, \qquad x \in [-1,1], \quad t > 0.
\end{equation}
Further, we have the following.
\begin{lem}[{\cite[Lem.\,3.3]{NSS2}}] \label{lem:diff}
Let $\a,\b > -1$. Then
\begin{align*}
\frac{\dd}{\dd x}\, G_t^{\a,\b}(x,1) =
2 (\a+1) e^{-t(\a+\b+2)} G_t^{\a+1,\b+1}(x,1),
\qquad x \in [-1,1], \quad t > 0.
\end{align*}
In particular, the function $x \mapsto G_t^{\a,\b}(x,1)$ is strictly increasing on $[-1,1]$.
\end{lem}

Following \cite[Sec.\,4]{NSS3}, for $-3/2 < \lambda \le -1$ we consider an auxiliary function
$$
H_t^{\lambda}(x) = \frac{\Gamma(\lambda+3/2)}{\sqrt{\pi}\Gamma(\lambda+2)} + \frac{1}{2^{2\lambda+1}\Gamma(\lambda+2)}
	\sum_{n=1}^{\infty} e^{-t 2n(2n+2\lambda+1)} \frac{(4n+2\lambda+1)\Gamma(2n+2\lambda+1)}{\Gamma(2n + \lambda+1)} P_{2n}^{\lambda}(x),
$$
which is well-defined for $x \in [-1,1]$ and $t > 0$. This function can be regarded as an extension for $-3/2 < \lambda \le -1$
of (a constant times) the even part in $x$ of the Jacobi-ultraspherical kernel $G_t^{\lambda}(x,1)$.
Note that $H_t^{\lambda}(x)$ is a smooth function of
$(t,x) \in (0,\infty) \times [-1,1]$ which satisfies the heat equation based on the Jacobi-ultraspherical operator $J^{\lambda}$,
\begin{equation} \label{he_H}
\Big( \frac{\dd}{\dd t} + J_x^{\lambda} \Big) H_t^{\lambda}(x) = 0, \qquad x \in [-1,1], \quad t > 0.
\end{equation}
Furthermore, we have the following.
\begin{lem}[{\cite[Lem.\,4.2]{NSS3}}] \label{lem:diffH}
Let $-3/2 < \lambda \le -1$. Then
$$
\frac{\dd}{\dd x}\, H_t^{\lambda}(x) = 2e^{-t(2\lambda+2)}\Big[ G_t^{\lambda+1}(x,1)\Big]_{\emph{odd}},
	\qquad x \in [-1,1], \quad t > 0,
$$
where the subscript `odd' indicates the odd part of the function in $x$. Similarly,
$$
\frac{\dd^2}{\dd x^2}\, H_t^{\lambda}(x) = 4(\lambda+2)e^{-t(4\lambda+6)}
	\Big[ G_t^{\lambda+2}(x,1)\Big]_{\emph{even}}, \qquad x \in [-1,1], \quad t > 0,
$$
where the subscript `even' indicates the even part of the function in $x$.
\end{lem}

For $-1 < \a \neq -1/2$ denote
$$
\Pi_{\a}(u) = \frac{\Gamma(\a+1)}{\sqrt{\pi}\,\Gamma(\a+1/2)} \int_0^u \big(1-w^2\big)^{\a-1/2}\, \dd w.
$$
Observe that $\Pi_{\a}$ is an odd function in $-1 < u < 1$, which is increasing if $\a > -1/2$ and decreasing if $\a < -1/2$.
For $\a > -1/2$ the measure $\dd \Pi_{\a}$ is a probability measure in the interval $[-1,1]$. The weak limit of $\dd \Pi_{\a}$
as $\a \searrow -1/2$ is
$$
\dd \Pi_{-1/2} := \frac{\delta_{-1} + \delta_{1}}{2},
$$
where $\delta_{\pm 1}$ is a unit point mass at $\pm 1$. For $-1 < \a < -1/2$ the distribution derivative $\dd \Pi_{\a}$
is a local measure in $(-1,1)$ and its density is negative, even, and non-integrable in $(-1,1)$.
 
The next result is a \emph{reduction formula}. It was obtained first under the restriction $\a,\b \ge -1/2$ in \cite{NoSj},
and more recently in its full generality in \cite{NSS3}.
\begin{lem}[{\cite[Thm.\,4.3]{NSS3}}] \label{lem:red_ext}
Let $t > 0$ and $\theta, \varphi \in [0,\pi]$. Then the following identities hold, with all integrations taken over $[-1,1]^2$.
\begin{itemize}
\item[(i)] If $\a,\b \ge -1/2$, then
\begin{align*}
& \frac{h_0^{\a,\b}}{h_0^{\a+\b+1/2}}  \,G_t^{\a,\b}(\cos\theta,\cos\varphi) \\
& \quad = \iint G_{t/4}^{\a+\b+1/2}\Big( u \sin\frac{\theta}2\sin\frac{\varphi}2
		+ v \cos\frac{\theta}2\cos\frac{\varphi}2, 1 \Big) \, \dd\Pi_{\a}(u) \, \dd\Pi_{\b}(v).
\end{align*}
\item[(ii)] If $-1 < \b < -1/2 \le \a$, then
\begin{align*}
& \frac{h_0^{\a,\b}}{h_0^{\a+\b+1/2}}\, G_t^{\a,\b}(\cos\theta,\cos\varphi) \\
& \quad = \iint \Big\{ -2(\a+\b+3/2) e^{-t(\a+\b+3/2)/2} \cos\frac{\theta}2\cos\frac{\varphi}2 \\
& \qquad \qquad \qquad \times	G_{t/4}^{\a+\b+3/2}\Big(u \sin\frac{\theta}2\sin\frac{\varphi}2 +
	v\cos\frac{\theta}2\cos\frac{\varphi}2,1\Big)
	\, \dd\Pi_{\a}(u) \, \Pi_{\b}(v)\, \dd v \\
& \qquad \qquad + G_{t/4}^{\a+\b+1/2}\Big(u \sin\frac{\theta}2\sin\frac{\varphi}2 + v\cos\frac{\theta}2\cos\frac{\varphi}2,1\Big)
	\, \dd \Pi_{\a}(u)\, \dd \Pi_{-1/2}(v) \Big\}.
\end{align*}
\item[(iii)] If $-1 < \a < -1/2 \le \b$, then
\begin{align*}
& \frac{h_0^{\a,\b}}{h_0^{\a+\b+1/2}}\, G_t^{\a,\b}(\cos\theta,\cos\varphi) \\
& \quad = \iint \Big\{ -2(\a+\b+3/2) e^{-t(\a+\b+3/2)/2} \sin\frac{\theta}2\sin\frac{\varphi}2 \\
& \qquad \qquad \qquad \times	G_{t/4}^{\a+\b+3/2}\Big(u \sin\frac{\theta}2\sin\frac{\varphi}2 +
	v\cos\frac{\theta}2\cos\frac{\varphi}2,1\Big)
	\, \Pi_{\a}(u)\, \dd u \, \dd \Pi_{\b}(v) \\
& \qquad \qquad + G_{t/4}^{\a+\b+1/2}\Big(u \sin\frac{\theta}2\sin\frac{\varphi}2 + v\cos\frac{\theta}2\cos\frac{\varphi}2,1\Big)
	\, \dd\Pi_{-1/2}(u)\, \dd\Pi_{\b}(v) \Big\}.
\end{align*}
\item[(iv)] If $-1 < \a,\b < -1/2$ and $\a+\b > -3/2$, then
\begin{align*}
& \frac{h_0^{\a,\b}}{h_0^{\a+\b+1/2}}\, G_t^{\a,\b}(\cos\theta,\cos\varphi) \\
& \quad = \iint \Big\{ (\a+\b+3/2)(\a+\b+5/2) e^{-t(\a+\b+2)} \sin\theta\sin\varphi \\
& \qquad \qquad \qquad \times	G_{t/4}^{\a+\b+5/2}\Big(u \sin\frac{\theta}2\sin\frac{\varphi}2 +
	v\cos\frac{\theta}2\cos\frac{\varphi}2,1\Big)
 \, \Pi_{\a}(u)\, \dd u \, \Pi_{\b}(v)\, \dd v \\
& \qquad \qquad -2(\a+\b+3/2) e^{-t(\a+\b+3/2)/2} \sin\frac{\theta}2\sin\frac{\varphi}2 \\
& \qquad \qquad \qquad \times	G_{t/4}^{\a+\b+3/2}\Big(u \sin\frac{\theta}2\sin\frac{\varphi}2 +
	v\cos\frac{\theta}2\cos\frac{\varphi}2,1\Big)
	\, \Pi_{\a}(u)\, \dd u \, \dd\Pi_{-1/2}(v) \\
& \qquad \qquad -2(\a+\b+3/2) e^{-t(\a+\b+3/2)/2} \cos\frac{\theta}2\cos\frac{\varphi}2 \\
& \qquad \qquad \qquad \times	G_{t/4}^{\a+\b+3/2}\Big(u \sin\frac{\theta}2\sin\frac{\varphi}2 +
	v\cos\frac{\theta}2\cos\frac{\varphi}2,1\Big)
	\, \dd\Pi_{-1/2}(u) \, \Pi_{\b}(v)\, \dd v \\
& \qquad \qquad + G_{t/4}^{\a+\b+1/2}\Big(u \sin\frac{\theta}2\sin\frac{\varphi}2 + v\cos\frac{\theta}2\cos\frac{\varphi}2,1\Big)
	\, \dd\Pi_{-1/2}(u)\, \dd\Pi_{-1/2}(v) \Big\}.
\end{align*}
\item[(v)] If $-1 < \a,\b < -1/2$ and $\a+\b \le -3/2$, then
\begin{align*}
& \frac{h_0^{\a,\b}\,\Gamma(\a+\b+2)}{\sqrt{\pi}\,\Gamma(\a+\b+5/2)} G_t^{\a,\b}(\cos\theta,\cos\varphi) \\
& \quad = \iint \Big\{(\a+\b+5/2) e^{-t(\a+\b+2)} \sin\theta\sin\varphi \\
& \qquad \qquad \qquad \times	G_{t/4}^{\a+\b+5/2}\Big(u \sin\frac{\theta}2\sin\frac{\varphi}2 +
	v\cos\frac{\theta}2\cos\frac{\varphi}2,1\Big)
 \, \Pi_{\a}(u)\, \dd u \, \Pi_{\b}(v)\, \dd v \\
& \qquad \qquad -2 e^{-t(\a+\b+3/2)/2} \sin\frac{\theta}2\sin\frac{\varphi}2 \\
& \qquad \qquad \qquad \times	G_{t/4}^{\a+\b+3/2}\Big(u \sin\frac{\theta}2\sin\frac{\varphi}2 +
	v\cos\frac{\theta}2\cos\frac{\varphi}2,1\Big)
	\, \Pi_{\a}(u)\, \dd u \, \dd \Pi_{-1/2}(v) \\
& \qquad \qquad -2 e^{-t(\a+\b+3/2)/2} \cos\frac{\theta}2\cos\frac{\varphi}2 \\
& \qquad \qquad \qquad \times	G_{t/4}^{\a+\b+3/2}\Big(u \sin\frac{\theta}2\sin\frac{\varphi}2 +
	v\cos\frac{\theta}2\cos\frac{\varphi}2,1\Big)
	\, \dd\Pi_{-1/2}(u) \, \Pi_{\b}(v)\, \dd v \\
& \qquad \qquad + H_{t/4}^{\a+\b+1/2}\Big(u \sin\frac{\theta}2\sin\frac{\varphi}2 + v\cos\frac{\theta}2\cos\frac{\varphi}2 \Big)
	\, \dd \Pi_{-1/2}(u)\, \dd \Pi_{-1/2}(v) \Big\}.
\end{align*}
\end{itemize}
\end{lem}

The result below is essentially the \emph{comparison principle} obtained in {\cite[Thm.\,3.5]{NoSj}}.
Here we state a slightly stronger version, which is implicitly contained in the proof in \cite{NoSj}.
\begin{lem} \label{lem:comp}
Let $\a,\b > -1$. Given $\epsilon,\delta \ge 0$ and $\a \ge -\epsilon/2$, $\b \ge -\delta/2$,
$$
\big[(1-x)(1-y)\big]^{\epsilon/2} \big[(1+x)(1+y)\big]^{\delta/2} {G}_t^{\a+\epsilon,\b+\delta}(x,y)
	\le e^{\frac{\epsilon+\delta}2 (\a+\b+1+\frac{\epsilon+\delta}2)t} {G}_t^{\a,\b}(x,y)
$$
for all $x,y \in [-1,1]$ and $t > 0$, with the convention that $(1\pm x)^0 = 1$ for $x = \mp 1$.
Moreover, the assumption $\a \ge -\epsilon/2$ can be deleted if $\epsilon = 0$,
and similarly for the assumption $\b \ge -\delta/2$ in case $\delta = 0$.
\end{lem}

\subsection{Estimates related to the density $\Pi_{\a}(u)$} \label{ssec:Pia} \,

\medskip

The following result is a refinement of \cite[Lem.\,2.2]{NSS0}.
\begin{lem} \label{lem:mes}
Let $\a \in (-1,-1/2)$ be fixed. Then the density $|\Pi_{\a}(u)|$ defines a finite measure on $[-1,1]$ and
\begin{equation*}
|\Pi_\al(u)| \simeq \left\{ K_\al \atop k_{\al}\right\}\, |u|(1-|u|)^{\al+1/2}, \qquad u \in (-1,1),
\end{equation*}
where
\begin{equation*}
	k_\al :=\frac{2^{\al-1/2}|\al+1/2|\Gamma(\al+1) }{\sqrt{\pi}\,\Gamma(\al+3/2)}, \qquad
	K_\al := \frac{\Gamma(\al+1)}{\sqrt{\pi}\,\Gamma(\al+3/2)}.
\end{equation*}
\end{lem}

\begin{proof}
	Fix $\al\in(-1,-1/2)$. Since the relevant expressions are even, we may restrict our attention to $u\in(0,1)$. 
	Observe that
	\begin{align*}
		|\Pi_{\a}(u)| & = \frac{\Gamma(\al+1)}{\sqrt{\pi}\,|\Gamma(\al+1/2)|} \int_0^u (1-w^2)^{\al-1/2}\,\dd w \\
		& \leq \frac{\Gamma(\al+1)}{\sqrt{\pi}\,|\Gamma(\al+1/2)|} \int_0^u (1-w)^{\al-1/2}\,\dd w
			= \frac{\Gamma(\al+1)}{\sqrt{\pi}\,\Gamma(\al+3/2)} \big[(1-u)^{\al+1/2}-1\big]
	\end{align*}
	and similarly
	$$
		|\Pi_{\a}(u)| \geq \frac{2^{\al-1/2} \Gamma(\al+1)}{\sqrt{\pi}\,\Gamma(\al+3/2)} \big[(1-u)^{\al+1/2}-1\big]. 
	$$
	Now let
	$$
		f(u) = \frac{(1-u)^{\al+1/2}-1}{u(1-u)^{\al+1/2}}, \qquad u \in (0,1).
	$$
	Notice that $f$ is increasing (since $f' > 0$, as can be verified by an elementary analysis)
	and $f(0^+) = -(\al+1/2)$, $f(1^{-}) = 1$. It follows that
	$$
		|\al+1/2| u(1-u)^{\al+1/2}\leq (1-u)^{\al+1/2} -1 \leq u (1-u)^{\al+1/2}.
	$$
	This together with the previous estimates allows one to conclude the proof.
\end{proof}

\begin{lem} \label{lem:mes2}
Let $\a \in (-1,-1/2)$ be fixed. Then
$$
\int_{u'}^1 |\Pi_{\a}(u)|\, \dd u \simeq \left\{ K_{\a} \atop k_{\a} \right\} \, \frac{1}{\a+3/2} (1-u')^{\a+3/2}, \qquad
	u' \in [0,1].
$$
\end{lem}

\begin{proof}
Denote the integral in question by $I$. To show the upper bound we use Lemma \ref{lem:mes} getting
$$
I \le K_{\a} \int_{u'}^1 u(1-u)^{\a+1/2}\, \dd u \le K_{\a} \int_{u'}^1 (1-u)^{\a+1/2}\, \dd u = \frac{K_{\a}}{\a+3/2} (1-u')^{\a+3/2}.
$$

For the lower bound, instead of Lemma \ref{lem:mes}, we use the definition of $\Pi_{\al}(u)$ and Fubini's theorem to write
\begin{align*}
\frac{\sqrt{\pi}\,|\Gamma(\a+1/2)| }{\Gamma(\a+1)}\, I & = \int_{u'}^1 \int_0^{u} \big(1-w^2\big)^{\a-1/2}\, \dd w\, \dd u 
 = \int_0^1 (1-w^2)^{\a-1/2} \int_{u' \vee w}^1 \dd u \, \dd w \\
& = \int_0^{u'} \big(1-w^2\big)^{\a-1/2} \int_{u'}^1 \dd u\, \dd w + \int_{u'}^1 \big(1-w^2\big)^{\a-1/2} \int_w^1 \dd u \, \dd w \\
& = (1-u') \int_0^{u'} \big(1-w^2\big)^{\a-1/2}\, \dd w + \int_{u'}^1 \big(1-w^2\big)^{\a-1/2}(1-w) \, \dd w.
\end{align*}
Neglecting the first term in the last sum and estimating the second one gives
$$
\frac{\sqrt{\pi}\,|\Gamma(\a+1/2)|}{\Gamma(\a+1)}\, I \ge 2^{\a-1/2}\int_{u'}^1 (1-w)^{\a+1/2}\, \dd w
	= \frac{2^{\a-1/2}}{\a+3/2} (1-u')^{\a+3/2}.
$$
The conclusion follows.
\end{proof}

\begin{lem} \label{lem:mes3}
Let $\a \in (-1,-1/2)$ be fixed. Then
$$
\int_{u'}^1 |\Pi_{\a}(u)|\, \dd u \simeq 
\left\{ L_{\a} \atop l_{\a} \right\} \, \frac{\dd \Pi_{\a+2}(u')}{\dd u'}, \qquad u' \in [0,1],
$$
where
$$
l_{\a} := \frac{|\a+1/2|}{4(\a+1)(\a+2)}, \qquad L_{\a} := \frac{1}{(\a+1)(\a+2)}.
$$
\end{lem}

\begin{proof}
Let $u' \in [0,1]$. We have
$$
\frac{\dd \Pi_{\a+2}(u')}{\dd u'} = \frac{\Gamma(\a+3)}{\sqrt{\pi}\,\Gamma(\a+5/2)} \big[ (1-u')(1+u')\big]^{\a+3/2},
$$
hence
$$
(1-u')^{\a+3/2} \simeq \left\{ 1 \atop 2^{-\a-3/2}\right\} \, \frac{\sqrt{\pi}\,\Gamma(\a+5/2)}{\Gamma(\a+3)}
\frac{\dd \Pi_{\a+2}(u')}{\dd u'}.
$$
Combining this with Lemma \ref{lem:mes2} we get the desired bounds.
\end{proof}

Finally, for further reference we state a simple observation. Fix $\a > -1/2$. Then
$$
\frac{\dd \Pi_{\a}(u)}{\dd u} = 
\frac{\Gamma(\a+1)}{\sqrt{\pi}\,\Gamma(\a+1/2)} \big(1-u^2\big)^{\a-1/2} \simeq 
	\left\{ 1 \vee 2^{\a-1/2} \atop 1 \wedge 2^{\a-1/2} \right\} \, \frac{\Gamma(\a+1)}{\sqrt{\pi}\,\Gamma(\a+1/2)} (1-u)^{\a-1/2}
$$
for $u \in [0,1)$. Consequently,
\begin{equation} \label{obs1}
(1-u)^{\a-1/2} \simeq  \left\{ N_{\a}\atop n_{\a} \right\} \, \frac{\dd \Pi_{\a}(u)}{\dd u}, \qquad u \in [0,1),
\end{equation}
with
$$
n_{\a} := \big( 1 \wedge 2^{1/2-\a}\big) \frac{\sqrt{\pi}\, \Gamma(\a+1/2)}{\Gamma(\a+1)}, \qquad
N_{\a} := \big( 1 \vee 2^{1/2-\a}\big) \frac{\sqrt{\pi}\, \Gamma(\a+1/2)}{\Gamma(\a+1)}.
$$

\subsection{Estimates related to the function $F(u,v)$} \label{ssec:Fuv}	\,

\medskip

Let $0 \le \theta,\varphi \le \pi$. For $0 \le u,v \le 1$ denote
\begin{equation} \label{def:F}
F(u,v) = F_{\theta,\varphi}(u,v)
	:= \bigg[\arccos\bigg(u\sin\frac{\theta}2\sin\frac{\varphi}2 + v\cos\frac{\theta}2\cos\frac{\varphi}2\bigg)\bigg]^2.
\end{equation}
Then $0 \le F(1,1) = (\theta - \varphi)^2/4 \le F(u,v) \le \pi^2/4 = F(0,0)$.

The following result is a refinement of \cite[Lem.\,5.4]{NSS3}.
\begin{lem} \label{lem:exp}
The bounds
\begin{equation} \label{eq:11}
F(u,v)-F(1,1) \simeq \left\{ \mathfrak{B} \atop \mathfrak{b}\right\} \big[ \te\va (1-u) +(\pi-\te)(\pi-\va)(1-v) \big]
\end{equation}
hold for $\theta,\varphi \in [0,\pi]$ and $u,v \in [0,1]$, with $\mathfrak{b}:= 2/\pi^{2}$ and $\mathfrak{B}:=1/2$.
The constants $\mathfrak{b}$ and $\mathfrak{B}$ are optimal.
\end{lem}

\begin{proof}
	We first show the lower bound. By continuity we may assume that $\te \ne \va$. Then, in particular, $F(u,v) > 0$.
	Notice that by differentiating in $u$ and $v$ the equation 
	\begin{equation*}
		\cos\sqrt{F(u,v)}=u\sin\frac\te2\sin\frac\va2 + v\cos\frac\te2\cos\frac\va2
	\end{equation*}
		we obtain
	\begin{align*}
		\partial_u F(u,v)&=-2\frac{\sqrt{F(u,v)}}{\sin\sqrt{F(u,v)}}\sin\frac\te2\sin\frac\va2,\\
		\partial_v F(u,v)&=-2\frac{\sqrt{F(u,v)}}{\sin\sqrt{F(u,v)}}\cos\frac\te2\cos\frac\va2.
	\end{align*}
	
	Fix $u,v\in[0,1]$. By the Mean Value Theorem there exists $c=(c_1,c_2)\in [0,1]^2$ depending on $u$ and $v$ such that
	\begin{align*}
		F(u,v) - F(1,1) & =  \partial_u F(c) (u-1) +  \partial_v F(c) (v-1) \\
		& =2\frac{\sqrt{F(c)}}{\sin\sqrt{F(c)}}\Big[\sin\frac\te2\sin\frac\va2(1-u)+\cos\frac\te2\cos\frac\va2(1-v)\Big].
	\end{align*}
	Since we have $\sqrt{F(c)} \in (0,\pi/2]$, in view of the simple estimates
	\begin{equation*}
		1\leq \frac{x}{\sin x}\leq \frac{\pi}{2},\qquad \frac{2x}{\pi}\leq \sin x\leq x,
		\qquad \frac{\pi-2x}{\pi}\leq \cos x\leq \frac{\pi-2x}{2}\qquad x\in(0,\pi/2],
	\end{equation*}
	we obtain the lower bound in \eqref{eq:11}.\,\footnote{
	\,Proceeding in a similar way we can also get the upper bound, but with non-optimal constant $\pi/4$.
	}

Next, we justify the upper bound. Observe that this is equivalent to showing that $G(u,v) \ge 0$, $u,v \in [0,1]^2$, where
\begin{align*}
	G(u,v) = G_{\theta,\varphi}(u,v)
	:= \frac{1}{2} \big[ (1-u) \te\va  + (1-v) (\pi-\te)(\pi-\va) \big] + (\te-\va)^2/4 - F_{\theta,\varphi}(u,v).
\end{align*}
Elementary computations lead to
\begin{align*}
\partial_u^2 G(u,v) 
& = - \partial_u^2 F(u,v) = 
	\frac{-2 \big( \sin\frac{\theta}2\sin\frac{\varphi}2 \big)^2 }{1 - (u\sin\frac{\theta}2\sin\frac{\varphi}2
		+ \cos\frac{\theta}2\cos\frac{\varphi}2)^2} \\
	& \qquad \times
	\left[ 1 - \arccos\bigg(u\sin\frac{\theta}2\sin\frac{\varphi}2 + v\cos\frac{\theta}2\cos\frac{\varphi}2\bigg) 
	\frac{u\sin\frac{\theta}2\sin\frac{\varphi}2 + v\cos\frac{\theta}2\cos\frac{\varphi}2}{\sqrt{1 -
		(u\sin\frac{\theta}2\sin\frac{\varphi}2 +v \cos\frac{\theta}2\cos\frac{\varphi}2)^2}}  \right].
\end{align*}
Since we have $x \arccos x \le \sqrt{1-x^2}$, $x \in (0,1)$, which can be easily checked because the function
$(0,1) \ni x \mapsto \arccos x - {\sqrt{1-x^2}}/{x}$ is increasing, we see that $\partial_u^2 G(u,v) \le 0$. Consequently, 
for any fixed $v \in [0,1]$ the function
$[0,1] \ni u \mapsto G(u,v)$ is concave and its minimal value is attained at one of the endpoints.
By the symmetry $G_{\theta,\varphi}(u,v) = G_{\pi-\theta,\pi-\varphi}(v,u)$ we obtain
$G(u,v) \ge \inf_{\theta,\varphi \in [0,\pi]} G_{\theta,\varphi}(0,0) \wedge G_{\theta,\varphi}(0,1) \wedge G_{\theta,\varphi}(1,1)$.
Since $G(1,1)=0$ and $G(0,0) = (\theta+\varphi-\pi)^2/4$, it suffices to check that $G(0,1) \ge 0$.

Notice that showing $G(0,1)\ge 0$ is equivalent to justifying that
\begin{align*}
	 \frac{\sqrt{\theta^2 + \varphi^2}}{2} - \arccos\bigg( \cos\frac{\theta}2\cos\frac{\varphi}2\bigg) \ge 0,
		\qquad \theta,  \varphi \in [0,\pi].
\end{align*}
Let 
\begin{align*}
	X(\theta) = X_{\varphi}(\theta)
	:= \frac{\sqrt{\theta^2 + \varphi^2}}{2} - \arccos\bigg( \cos\frac{\theta}2\cos\frac{\varphi}2\bigg).
\end{align*}
Since $X(0) = 0$ it suffices to prove that $X'(\theta) \ge 0$. Observe that
\begin{align*}
	X'(\theta) =
	\frac{1}{2} \left[ \frac{\theta}{\sqrt{\theta^2 + \varphi^2}} -  \frac{ \sin\frac{\theta}2 \cos\frac{\varphi}2 }{\sqrt{1 - 
			(\cos\frac{\theta}2\cos\frac{\varphi}2)^2}} \right].
\end{align*}
Since $1 - (\cos\frac{\theta}2\cos\frac{\varphi}2)^2 =  \sin^2\frac{\varphi}2 + (\sin\frac{\theta}2\cos\frac{\varphi}2)^2$,
it is not hard to see that $X'(\theta) \ge 0$ is equivalent to showing that
\begin{align*}
	\theta \sin\frac{\varphi}2 \ge 
	\varphi \sin\frac{\theta}2 \cos\frac{\varphi}2, \qquad \theta,  \varphi \in [0,\pi].
\end{align*}
The latter inequality, however, is easily seen to hold true.

Finally, let us look at the optimality of $\mathfrak{B}$ and $\mathfrak{b}$.
Letting $u=v$, $\theta=\varphi=0$ and then $u \to 1^{-}$ we see that $\mathfrak{b}$ is optimal.
On the other hand, taking $u=v=0$, $\va=0$ and then $\te \to \pi^-$ we get the optimality of $\mathfrak{B}$.

The proof of Lemma~\ref{lem:exp} is finished.
\end{proof}

\subsection{Estimates of various integrals} \label{ssec:int}	\,

\medskip

\begin{lem} \label{lem:Gmod2}
	Let $\omega \in [1/2,2]$. Then
	$$
		\int_{\xi}^{\infty} u^{\omega-1} e^{-u}\, \dd u \simeq 
		\left\{ M_{\omega} \atop m_{\omega} \right\}\, (1 + \xi)^{\omega-1} e^{-\xi}, \qquad \xi \ge 0,
	$$
		with the constants
	$$
		m_{\omega} := \frac{1-2^{-\omega}}{e\,\omega}, \qquad M_{\omega} := \frac{2}{(1-1/e)^2 \, \omega}.
	$$
\end{lem}

To justify the above we need the following auxiliary result.
\begin{lem} \label{lem:3042}
	Let $\omega \in [1/2,2]$. Then
	$$
	\int_{a}^{b} u^{\omega-1} \, \dd u \simeq \left\{ 2/\omega \atop (1-2^{-\omega})/\omega \right\}\, b^{\omega-1} (b-a),
		\qquad 0 \le a < b < \infty.
	$$
\end{lem}

\begin{proof}
\noindent \textbf{Case 1: $b \ge 2a$.} Here we have 
\begin{align*}
\int_{a}^{b} u^{\omega-1} \, \dd u \le \int_{0}^{b} u^{\omega-1} \, \dd u = b^{\omega}/\omega \le 
\frac{2}{\omega} (b-a) b^{\omega-1}.
\end{align*}
On the other hand,
$$
	\int_{a}^{b} u^{\omega-1} \, \dd u \ge \int_{b/2}^{b} u^{\omega-1} \, \dd u
	= \frac{1-2^{-\omega}}{\omega} b^{\omega} \ge \frac{1-2^{-\omega}}{\omega} (b-a) b^{\omega-1}.
$$

\noindent \textbf{Case 2: $a \le b < 2a$.} We have 
\begin{align*}
	\int_{a}^{b} u^{\omega-1} \, \dd u \le (b-a)
	\begin{cases}
		b^{\omega-1} &\text{ if }\;\; \omega \ge 1\\
		(b/2)^{\omega-1} &\text{ if }\;\; \omega < 1
	\end{cases},
\end{align*}
and 
\begin{align*}
	\int_{a}^{b} u^{\omega-1} \, \dd u \ge (b-a)
	\begin{cases}
		(b/2)^{\omega-1} &\text{ if }\;\; \omega \ge 1\\
		b^{\omega-1} &\text{ if }\;\; \omega < 1
	\end{cases}.
\end{align*}

This implies Lemma~\ref{lem:3042}, since for $\omega \in [1,2]$ we have
$\frac{1-2^{-\omega}}{\omega} \wedge 2^{1-\omega} = \frac{1-2^{-\omega}}{\omega}$ and $\frac{2}{\omega} \vee 1= \frac{2}{\omega}$,
and for $\omega \in [1/2,1]$ we have $\frac{1-2^{-\omega}}{\omega} \wedge 1 = \frac{1-2^{-\omega}}{\omega}$ and
$\frac{2}{\omega} \vee 2^{1-\omega} = \frac{2}{\omega}$.
\end{proof}

\begin{proof}[Proof of Lemma~\ref{lem:Gmod2}]
Using Lemma~\ref{lem:3042} we get
\begin{align*}
\int_{\xi}^{\infty} u^{\omega-1} e^{-u}\, \dd u \ge \int_{\xi}^{\xi + 1} u^{\omega-1} e^{-u}\, \dd u
\ge e^{-\xi - 1} \int_{\xi}^{\xi + 1} u^{\omega-1} \, \dd u
\ge \frac{1-2^{-\omega}}{e\,\omega} (\xi + 1)^{\omega-1} e^{-\xi}.
\end{align*} 
This is the lower bound in question.

Next, we prove the upper bound. Using again Lemma~\ref{lem:3042} we obtain
\begin{align*}
	\int_{\xi}^{\infty} u^{\omega-1} e^{-u}\, \dd u \le \sum_{k \ge 0} e^{-\xi - k}
	\int_{\xi + k}^{\xi + k +1} u^{\omega-1} \, \dd u \le \frac{2}{\omega} e^{-\xi} \sum_{k \ge 0} e^{-k} (\xi + k + 1)^{\omega-1}.
\end{align*} 
Further, observe that 
\begin{align*}
(\xi + k + 1)^{\omega-1} \le (\xi + 1)^{\omega-1} (1+k).
\end{align*} 
Combining this with the identity $\sum_{k \ge 0} (k+1) e^{-k} = (1-1/e)^{-2}$ we get the desired bound.
\end{proof}

Denote
$$
\mathbb{D}_{\a} := \Gamma(\a+3/2)^{1/{(\a+1/2)}}, \qquad \a > -3/2,
$$
with the value for $\a=-1/2$ understood in the limiting sense,
$$
\mathbb{D}_{-1/2} = \lim_{\a \to -1/2} \mathbb{D}_{\a} = e^{-\gamma},
$$
where $\gamma = - \Gamma'(1) \approx 0.577216$ is the Euler--Mascheroni constant.
Note that $\mathbb{D}_{\a}$ is strictly increasing in $\a$.\,\footnote{
\,For the monotonicity of $\mathbb{D}_{\a}$ on $(-1/2,\infty)$, see e.g.\ \cite[Sec.\,1]{CQ1}.
To show that $\mathbb{D}_{\a}$ is strictly increasing on $(-3/2,-1/2)$ it is enough to verify that
$(\mathbb{D}_{\a})' = (\a+1/2)^{-2} \mathbb{D}_{\a}[(\a+1/2)\psi(\a+3/2)-\log\Gamma(\a+3/2)] > 0$ for $\a \in (-3/2,-1/2)$;
here $\psi$ is the logarithmic derivative of $\Gamma$ (the digamma function). Let $H_{\a}$ be the expression in the last square brackets.
Since $H_{-1/2}=0$, the task boils down to checking that $(H_{\a})' = (\a+1/2)\psi'(\a+3/2) < 0$ for $\a \in (-3/2,-1/2)$.
This, however, follows instantly from known positivity of the trigamma function $\psi'$, see e.g.\
\cite[Chap.\,5, $\S$5.15, Form.\,5.15.1]{DLMF}.
}
Further, one has $\mathbb{D}_{-1} = 1/\pi$, $\mathbb{D}_0=\pi/4$ and $\mathbb{D}_{1/2} = 1$.

The next result is a refinement of \cite[Lem.\,5.5]{NSS3}. 
\begin{lem} \label{lem:sss}
Let $\a \ge -1/2$. Then
$$
\int_{[0,1]} \exp\Big( -\xi (1-s) \Big)\, \dd \Pi_{\a}(s) \simeq \left\{ B_{\a} \atop b_{\a} \right\} \,
	\big( \mathbb{D}_{\a} \vee \xi \big)^{-\a-1/2}, \qquad \xi \ge 0,
$$
with the constants
$$
b_{\a} := \big( 1 \wedge 2^{\a-1/2} \big) \frac{\Gamma(\a+1)}{\sqrt{\pi}} \exp\big(-\mathbb{D}_{\a}\big), \qquad
B_{\a} := \big( 1 \vee 2^{\a-1/2} \big) \frac{\Gamma(\a+1)}{\sqrt{\pi}},
\qquad \alpha > -1/2,
$$
and $B_{-1/2}:= 1/2 =: b_{-1/2}$.
\end{lem}

\begin{rem} \label{rem:D}
It is known (see \cite[Sec.\,1]{CQ1} and references given there) that the function $x \mapsto [\Gamma(x+1)]^{1/x}/(x+1)$ is
strictly decreasing for $x > 0$, and it has limits $e^{-\gamma}$ as $x \to 0^+$ and $1/e$ as $x\to \infty$.
Thus we have the bounds
\begin{equation} \label{DDD}
\frac{\a+3/2}{e} < \mathbb{D}_{\a} < \frac{\a+3/2}{e^{\gamma}}, \qquad \a > -1/2,
\end{equation}
which show that $\mathbb{D}_{\a}$ depends roughly linearly on the parameter.

Let $\mathbb{E}_{\a} = c_1 (\a + c_2)$. The condition $\mathbb{D}_{\a}/\mathbb{E}_{\a} \to 1$ as $\a \to \infty$ is equivalent to
$c_1 = 1/e$. Then the choice of $c_2$ can be made e.g.\ by imposing the condition $\mathbb{D}_{-1/2} = \mathbb{E}_{-1/2}$, which
implies $c_2 = e^{1-\gamma}+1/2 \approx 2.026205$. In this case $\mathbb{E}_{\a} = (\a+1/2)/e + e^{-\gamma}$ and from \eqref{DDD}
it follows that
$$
\mathbb{D}_{\a} \simeq \left\{ e^{1-\gamma} \atop e^{\gamma-1} \right\}\, \mathbb{E}_{\a}, \qquad \a \ge -1/2.
$$
Note that $e^{\gamma-1} \approx 0.655220$ and $e^{1-\gamma} \approx 1.526205$. It can be shown that in fact
$$
\mathbb{D}_{\a} \simeq \left\{ 11/10 \atop 1 \right\}\, \mathbb{E}_{\a}, \qquad \a \ge -1/2,
$$
but we will not provide the proof;
anyhow, these bounds become evident just by plotting the ratio $\mathbb{D}_{\a}/\mathbb{E}_{\a}$
with the aid of any common software tool for mathematics.
\end{rem}

\begin{proof}[{Proof of Lemma \ref{lem:sss}}]
We may assume that $\a > -1/2$, since for $\a=-1/2$ the integral is equal to $1/2$ (in particular, it is independent of $\xi$)
and the lemma clearly holds. Then, in view of the definition of $\dd \Pi_{\a}$ and by a simple change of variable, we have
\begin{equation} \label{bb0}
\int_{[0,1]} \exp\Big( -\xi (1-s) \Big)\, \dd \Pi_{\a}(s) \simeq \left\{ 1 \vee 2^{\a-1/2} \atop 1 \wedge 2^{\a-1/2}\right\}\,
	\frac{\Gamma(\a+1)}{\sqrt{\pi}\,\Gamma(\a+1/2)} \int_0^1 e^{-u \xi} u^{\a-1/2}\, \dd u.
\end{equation}
Denote the last integral by $I$ and let $\mathbb{D}> 0$ be a constant that will be specified later.

For $0 \le \xi \le \mathbb{D}$ we can estimate $I$ either by neglecting the exponential factor or by replacing it by
its minimal value, getting
\begin{equation} \label{bb1}
\frac{1}{\a+1/2} e^{-\mathbb{D}} \le I \le \frac{1}{\a+1/2}, \qquad \xi \le \mathbb{D}.
\end{equation}

Let now $\xi \ge \mathbb{D}$. To get an upper bound we first change the variable and then enlarge the interval of integration
to $(0,\infty)$, obtaining
$$
I = \xi^{-\a-1/2} \int_0^{\xi} e^{-v} v^{\a-1/2}\, \dd v \le \xi^{-\a-1/2}\, \Gamma(\a+1/2).
$$
For a lower bound we write
$$
I \ge \xi^{-\a-1/2} \int_0^{\mathbb{D}} e^{-v} v^{\a-1/2}\, \dd v = \xi^{-\a-1/2}\, \mathbb{D}^{\a+1/2}
	\int_0^1 e^{-\mathbb{D} w} w^{\a-1/2}\, \dd w \ge \xi^{-\a-1/2}\, \mathbb{D}^{\a+1/2} e^{-\mathbb{D}} \frac{1}{\a+1/2}.
$$
Thus
\begin{equation} \label{bb2}
\frac{1}{\a+1/2} \mathbb{D}^{\a+1/2} e^{-\mathbb{D}} \xi^{-\a-1/2} \le I \le \frac{1}{\a+1/2} \Gamma(\a+3/2) \xi^{-\a-1/2},
	\qquad \xi \ge \mathbb{D}.
\end{equation}

Taking into account the structure of \eqref{bb1} and \eqref{bb2}, we see that a natural choice for $\mathbb{D}$ is
such that $\mathbb{D}^{\a+1/2} = \Gamma(\a+3/2)$, that is $\mathbb{D} = \mathbb{D}_{\a}$. Then a combination of
\eqref{bb1} and \eqref{bb2} leads to
$$
\Gamma(\a+1/2) e^{-\mathbb{D}_{\a}} \big( \mathbb{D}_{\a} \vee \xi \big)^{-\a-1/2} \le I \le
\Gamma(\a+1/2) \big( \mathbb{D}_{\a} \vee \xi \big)^{-\a-1/2}, \qquad \xi \ge 0.
$$
This together with \eqref{bb0} proves the lemma.
\end{proof}

Bring in notation
\begin{equation} \label{def:Psi}
\Psi_{\a}^{\kappa}(t,\theta,\varphi) := \big[ \mathbb{D}_{\a}t \vee \kappa \theta \varphi \big]^{-\a-1/2}.
\end{equation}
Note that for $\a > -1$ and $\kappa > 0$ one has 
(vide Theorem \ref{thm:jhk} and the expression $Z^{\a,\b}(t;\theta,\varphi)$)
\begin{equation*}
\Psi_{\a}^{\kappa}(t,\theta,\varphi) \simeq \left\{ \big[{1}/({\mathbb{D}_{\a}\vee\kappa})\big]^{\a+1/2} \vee
	\big[{2}/({\mathbb{D}_{\a}\wedge \kappa})\big]^{\a+1/2} \atop \big[{1}/({\mathbb{D}_{\a}\vee\kappa})\big]^{\a+1/2} \wedge
	\big[{2}/({\mathbb{D}_{\a}\wedge \kappa})\big]^{\a+1/2} \right\} \; \frac{1}{(t+\theta\varphi)^{\a+1/2}},
	\qquad \theta,\varphi \in [0,\pi], \quad t > 0.
\end{equation*}
\begin{lem} \label{lem:intF}
Let $\a,\b \ge -1/2$. Then
$$
\iint_{[0,1]^2} e^{-{F_{\theta,\varphi}(u,v)}/t} \, \dd\Pi_{\a}(u)\, \dd\Pi_{\b}(v) \simeq
\left\{ B_{\a}B_{\b} \Psi_{\a}^{\mathfrak{b}}(t,\theta,\varphi) \Psi_{\b}^{\mathfrak{b}}(t,\pi-\theta,\pi-\varphi) \atop
 b_{\a}b_{\b} \Psi_{\a}^{\mathfrak{B}}(t,\theta,\varphi) \Psi_{\b}^{\mathfrak{B}}(t,\pi-\theta,\pi-\varphi) \right\}\,
t^{\a+\b+1} e^{-{(\theta-\varphi)^2}/{(4t)}}
$$
for $\theta, \varphi \in [0,\pi]$ and $t > 0$.
\end{lem}

\begin{proof}
We prove the upper bound, proving the lower one is analogous. Observe that
$$
\iint_{[0,1]^2} e^{-{F_{\theta,\varphi}(u,v)}/{t}} \, \dd\Pi_{\a}(u)\, \dd\Pi_{\b}(v) =
e^{-{(\theta-\varphi)^2}/{(4t)}} \iint_{[0,1]^2} e^{-{[F_{\theta,\varphi}(u,v)-F_{\theta,\varphi}(1,1)]}/{t}} \,
\dd\Pi_{\a}(u)\, \dd\Pi_{\b}(v).
$$
It is enough to bound suitably the last double integral which we denote by $I$.

Using Lemma \ref{lem:exp} and then Lemma \ref{lem:sss} we get
\begin{align*}
I & \le \int_{[0,1]} e^{-{\mathfrak{b}\theta\varphi}(1-u)/{t}} \, \dd\Pi_{\a}(u) \,
	\int_{[0,1]} e^{-{\mathfrak{b}(\pi-\theta)(\pi-\varphi)}(1-v)/{t}} \, \dd\Pi_{\b}(v) \\
& \le t^{\a+\b+1} \frac{B_{\a}}{[\mathbb{D}_{\a}t \vee \mathfrak{b}\theta\varphi]^{\a+1/2}}
		 \frac{B_{\b}}{[\mathbb{D}_{\b}t \vee \mathfrak{b}(\pi-\theta)(\pi-\varphi)]^{\b+1/2}},
\end{align*}
as desired.
\end{proof}

Assuming that $\a,\b \ge -1/2$, for the case $\varphi = 0$ observe that
$$
\iint_{[0,1]^2} e^{-{F_{\theta,0}(u,v)}/{t}} \, \dd\Pi_{\a}(u)\, \dd\Pi_{\b}(v)
= \frac{1}2 \int_{[0,1]} e^{-{F_{\theta,0}(1,v)}/{t}} \, \dd\Pi_{\b}(v).
$$
Then we have the following.
\begin{lem} \label{lem:intF0}
Let $\b \ge -1/2$. Then
$$
\int_{[0,1]} e^{-{F_{\theta,0}(1,v)}/{t}} \, \dd\Pi_{\b}(v) \simeq
\left\{ B_{\b} \Psi_{\b}^{\mathfrak{b}}(t,\pi-\theta,\pi) \atop b_{\b} \Psi_{\b}^{\mathfrak{B}}(t,\pi-\theta,\pi) \right\}\,
	t^{\b+1/2} e^{-{\theta^2}/{(4t)}}, \qquad \theta \in [0,\pi], \quad t > 0.
$$
\end{lem}

\begin{proof}
The reasoning is the same as in the proof of Lemma \ref{lem:intF}. We leave the details to the reader.
\end{proof}

\subsection{Facts and estimates related to the gamma function} \label{ssec:gamma}	\,

\medskip

The duplication formula, cf.\ e.g.\ \cite[Chap.\,5, $\S$5.5, Form.\,5.5.5]{DLMF},
$$
\Gamma(2x) = \frac{2^{2x-1}}{\sqrt{\pi}} \Gamma(x) \Gamma(x+1/2), \qquad x \notin -\mathbb{N}/2,
$$
is used implicitly across the paper without further mention.

In the sequel we will use the following bounds and a fact:
\begin{equation} \label{Gest}
\sqrt{2\pi} x^{x-1/2} e^{-x} < \Gamma(x) < \sqrt{2\pi} x^{x-1/2} e^{-x} e^{1/(12x)}, \qquad x > 0,
\end{equation}
\begin{equation} \label{GGin}
\frac{\Gamma(x)}{\Gamma(y)} < x^{x-y}, \qquad 0 < y < x,
\end{equation}
\begin{equation} \label{binin}
\binom{x+y}{y} = \frac{\Gamma(x+y+1)}{\Gamma(x+1)\Gamma(y+1)} \quad \textrm{is strictly increasing in $y>0$ for each $x>0$ fixed.}
\end{equation}
For inequalities \eqref{Gest}, see e.g.\ \cite[Chap.\,5, $\S$5.6, Form.\,5.6.1]{DLMF}.
Estimate \eqref{GGin} can be concluded e.g.\ from Wendel's inequality, see \cite[Form.\,(4)]{wendel} and is a consequence
of log-convexity of $\Gamma$. The fact \eqref{binin} can be verified by observing that $\binom{x+y}{y} = \frac{1}{x B(x,y+1)}$
and $z \mapsto B(x,z)$ is strictly decreasing for $z>0$; here $B$ is the beta function.

We will also need the following.
\begin{lem} \label{lem:GAMest}
	For every integer $j\ge 1$,
	$$
		\Gamma(j + 1/2) > \frac{\sqrt{2\pi}}{e^{1/12}}\, j^j e^{-j}.
	$$
\end{lem}

\begin{proof}
Observe that
$$
\Gamma(j+1/2) = (j-1/2)\cdot \ldots\cdot (1/2)\, \Gamma(1/2) = \frac{(2j-1)!!}{2^j}\sqrt{\pi} =
\frac{(2j-1)!!}{2^j} \frac{(2j)!!}{2^j j!}\sqrt{\pi}
= \frac{\sqrt{\pi}}{4^j} \frac{(2j)!}{j!}.
$$
We shall use Stirling-type inequalities. It is known, see \cite{Rob}, that for any integer $n \ge 1$
$$
n! = \sqrt{2\pi}\,n^{n+1/2} e^{-n} e^{r_n},
$$
where $r_n$ satisfies $1/(12 n+1) < r_n < 1/(12 n)$. In particular, we have
\begin{equation} \label{sp15}
\sqrt{2\pi}\,n^{n+1/2} e^{-n} < n! < \sqrt{2\pi}\,n^{n+1/2} e^{-n} e^{1/12}, \qquad n \ge 1.
\end{equation}
Thus,
$$
\frac{(2j)!}{j!} \ge \frac{\sqrt{2}}{e^{1/12}} 4^j j^j e^{-j}, \qquad j \ge 1.
$$
This together with the initial observation concludes the proof.
\end{proof}

\subsection{Estimates related to the Bessel function $K_{\nu}$} \label{ssec:Knu}	\,

\medskip

We denote by $K_{\nu}$ the modified Bessel function of the second kind (Macdonald function) of order $\nu$, see e.g.\ \cite{watson}
for the definition and basic properties. For $\nu > -1/2$ and positive arguments $K_{\nu}$ can be represented as
(cf.\ \cite[Chap.\,10, $\S$10.32(i), Form.\,10.32.8]{DLMF})
$$
K_{\nu}(y) = \frac{\sqrt{\pi}y^{\nu}}{2^{\nu}\Gamma(\nu+1/2)} \int_1^{\infty} e^{-yu}\big( u^2-1\big)^{\nu-1/2}\, \dd u, \qquad y > 0.
$$
The following result is a part of \cite[Thm.\,2]{Pal}.\,\footnote{
\,There seem to be no better upper bounds uniform in $x$ and $\nu$ for $K_{\nu}(x)$ available up to date.
}
\begin{lem}[\cite{Pal}] \label{lem:PalK}
Let $\beta > 1/2$, $\nu \ge 0$ and $y>0$. Then
$$
\sqrt{\frac{2y}{\pi}} e^y K_{\nu}(y) \le \frac{2}{\pi} \bigg(\frac{y^2}{y^2+\nu^2}\bigg)^{1/4}
	\bigg( \frac{\nu+\sqrt{y^2+\nu^2}}{y} \bigg)^{\nu} \exp\bigg( -\sqrt{y^2+\nu^2} + y + \frac{\beta}{\sqrt{y^2+\nu^2}}\bigg)
$$
provided that $\sqrt{y^2+\nu^2} \ge 2\beta/(2\beta-1)$.
\end{lem}

We now extract from the above a slightly less precise and less general bound, but of much simpler form that will be
convenient for our purpose.
\begin{lem} \label{lem:estKnu}
Let $\nu \ge 5/2$ and $0 < \epsilon \le 2/3$. Then
$$
\sqrt{\frac{2y}{\pi}} e^{y} K_{\nu}(y) \le \frac{2}{\pi} e^{1/(\sqrt{2}\nu)} (1+\epsilon)^{\nu}, \qquad y \ge \frac{2\nu}{3\epsilon}.
$$
\end{lem}

\begin{proof}
We use Lemma \ref{lem:PalK} with $\beta = 1$. Notice that in view of the assumption $\nu \ge 5/2$ the condition
$\sqrt{y^2+\nu^2} \ge 2\beta/(2\beta-1)$ is satisfied for all $y > 0$. We have
\begin{align*}
\sqrt{\frac{2y}{\pi}} e^y K_{\nu}(y) & \le \frac{2}{\pi} \bigg(\frac{y^2}{y^2+\nu^2}\bigg)^{1/4}
	\bigg( \frac{\nu+\sqrt{y^2+\nu^2}}{y} \bigg)^{\nu} \exp\bigg( -\sqrt{y^2+\nu^2} + y + \frac{1}{\sqrt{y^2+\nu^2}}\bigg) \\
& < \frac{2}{\pi} \bigg( \frac{\nu+\sqrt{y^2+\nu^2}}{y} \bigg)^{\nu} \exp\bigg( -\sqrt{y^2+\nu^2}
	+ y + \frac{1}{\sqrt{y^2+\nu^2}}\bigg).
\end{align*}
Observe that the assumptions $0 < \epsilon \le 2/3$ and $y \ge 2\nu/(3\epsilon)$ imply $y \ge \nu$, so
\begin{align*}
\exp\bigg( -\sqrt{y^2+\nu^2} + y + \frac{1}{\sqrt{y^2+\nu^2}}\bigg) & =
\exp\bigg( \frac{-\nu^2}{y+\sqrt{y^2+\nu^2}} \bigg)  \exp\bigg(\frac{1}{\sqrt{y^2+\nu^2}}\bigg) \\
& \le \exp\bigg( \frac{-\nu^2}{y+\sqrt{y^2+\nu^2}} \bigg) \exp\bigg( \frac{1}{\sqrt{2}\nu} \bigg).
\end{align*}
It follows that for $\nu \ge 5/2$ and $y \ge 2\nu/(3\epsilon)$
\begin{align*}
\sqrt{\frac{2y}{\pi}} e^y K_{\nu}(y) & < \frac{2}{\pi}e^{1/(\sqrt{2}\nu)}\bigg[ \frac{\nu+\sqrt{y^2+\nu^2}}{y}
	\exp\bigg( \frac{- \nu}{y+\sqrt{y^2+\nu^2}}\bigg) \bigg]^{\nu} \\
& = \frac{2}{\pi}e^{1/(\sqrt{2}\nu)}\bigg[ \Big({\nu}/{y}+\sqrt{1+(\nu/y)^2}\Big)
	\exp\bigg( \frac{- \nu/y}{1+\sqrt{1+(\nu/y)^2}}\bigg) \bigg]^{\nu}.
\end{align*}

It is elementary to check that
$$
\Big( x + \sqrt{1+x^2}\Big) \exp\bigg( - \frac{x}{1+\sqrt{1+x^2}} \bigg) \le 1 + \frac{2}{3} x, \qquad x \in [0,1].
$$
Consequently,
$$
\sqrt{\frac{2y}{\pi}} e^y K_{\nu}(y) < \frac{2}{\pi} e^{1/(\sqrt{2}\nu)} \bigg( 1 + \frac{2\nu}{3y}\bigg)^{\nu} \le
	\frac{2}{\pi} e^{1/(\sqrt{2}\nu)} (1+\epsilon)^{\nu},
$$
which is the desired bound.
\end{proof}

\begin{cor} \label{cor:estKnu}
Let $\nu \in \mathbb{N} + 1/2$ and $0 < \epsilon \le 2/3$. Then
$$
\sqrt{\frac{2y}{\pi}} e^y K_{\nu}(y) \le (1+\epsilon)^{\nu-1/2}, \qquad y \ge \frac{2\nu}{3\epsilon}.
$$
\end{cor}

\begin{proof}
For $\nu \ge 9/2$ this follows from Lemma \ref{lem:estKnu}, since then
$$
(1+\epsilon)^{-1/2} \ge (1+2/3)^{-1/2} > (2/\pi) e^{1/(9\sqrt{2}/2)} \ge (2/\pi)e^{1/(\sqrt{2}\nu)}, \qquad \epsilon \in (0,2/3].
$$
On the other hand, for $\nu = 1/2, 3/2, 5/2, 7/2$ we can verify the bound directly using explicit expressions for $K_{\nu}$,
cf.\ \cite[Chap.\,III, Sec.\,3$\cdot$71, Form.\,(12)]{watson},
$$
\begin{array}{ll}
 \sqrt{{2y}/{\pi}}\, e^y K_{1/2}(y) = 1, 
& \qquad\sqrt{{2y}/{\pi}}\, e^y K_{3/2}(y) = 1 + 1/y, \\
 \sqrt{{2y}/{\pi}}\, e^y K_{5/2}(y) = 1 + {3}/y + {3}/{y^2},
& \qquad \sqrt{{2y}/{\pi}}\, e^y K_{7/2}(y) = 1 + {6}/y + {15}/{y^2} + {15}/{y^3}.
\end{array}
$$
The relevant facts here are that $1/y \le 3\epsilon/(2\nu)$ and that $\epsilon^n$ is decreasing in $n$.
We leave elementary details to the reader.
\end{proof}

\section{Estimates related to odd-dimensional spheres} \label{sec:oddsph}

Let $S^d \subset \mathbb{R}^{d+1}$ be the Euclidean unit sphere of dimension $d \ge 1$ equipped with the standard
non-normalized area measure $\sigma_d$. Let $\mathcal{K}_t^{d}(\xi,\eta)$ be the heat kernel on $(S^d,\sigma_d)$.
By rotational invariance, $\mathcal{K}_t^d(\xi,\eta)$ depends on $\xi$ and $\eta$ only through their geodesic distance
$\dist(\xi,\eta) = \arccos\langle\xi,\eta\rangle$ on $S^d$. Thus it can be written in terms of a one-dimensional
function $K_t^d$ on $[0,\pi]$,
$$
\mathcal{K}_t^d(\xi,\eta) = K_t^d\big(\dist(\xi,\eta)\big), \qquad \xi,\eta \in S^d, \quad t > 0.
$$

Denote by $W_t$ the one-dimensional Gauss-Weierstrass kernel, and observe that $\vartheta_t$ is its $2\pi$-periodization,
$$
W_t(x) = \frac{1}{\sqrt{4\pi t}} \exp\bigg({-\frac{x^2}{4t}}\bigg), \qquad \vartheta_t(x) = \sum_{n \in \mathbb{Z}} W_t(x+2\pi n).
$$
It is known, see e.g.\ \cite[Sec.\,2]{NSS}, that
\begin{equation} \label{sp1}
K_t^{2N+1}(\varphi) = \frac{e^{tN^2}}{(2\pi)^N} (-D)^N \vartheta_t(\varphi), \qquad N \ge 0,
\end{equation}
where
\begin{equation} \label{sp4}
D=D_{\varphi} = \frac{1}{\sin\varphi} \frac{\dd}{\dd \varphi}.
\end{equation}

On the other hand, we have the connection (see \cite[Form.\,(4)]{NSS})
\begin{equation} \label{sp44}
K_t^d(\varphi) = \frac{1}{\sigma_{d-1}(S^{d-1})} G_t^{d/2-1}(\cos\varphi,1)
		= \frac{\Gamma(d/2)}{2\pi^{d/2}} G_t^{d/2-1}(\cos\varphi,1), \qquad d \ge 1;
\end{equation}
here $\sigma_0(S^0) = 2$ in case $d=1$. Thus
\begin{equation} \label{sp2}
G_t^{N-1/2}(\cos\varphi,1) = \frac{\sqrt{4\pi}}{2^N \Gamma(N+1/2)} e^{tN^2} (-D)^N \vartheta_t(\varphi), \qquad N \ge 0.
\end{equation}

The aim of Section \ref{sec:oddsph} is to estimate $(-D)^N\vartheta_t(\varphi)$ in the restricted range $\varphi \in [0,\pi/2]$
and for $t > 0$ reasonably small. In view of \eqref{sp2}, this will instantly imply analogous bounds for $G_t^{\lambda}(x,1)$
when $\lambda \in \mathbb{N}-1/2$ and $x \in [0,1]$.

\subsection{Lower bound for $(-D)^N\vartheta_t(\varphi)$} \label{ssec:eslow} \,

\medskip

To get a lower bound for $(-D)^N\vartheta_t(\varphi)$ we shall appeal to a general theory.
Since $S^d$ is a complete Riemannian manifold with positive Ricci curvature, \cite[Thm.\,5.6.1]{Da}\,\footnote{
\,The result was originally proved by Cheeger and Yau \cite{ChY}, following an earlier work of Debiard, Gaveau
and Mazet \cite{DGM}; they, in fact, gave a more general version applicable to manifolds with Ricci curvature bounded below.}
implies
$$
K_t^{d}(\varphi) \ge \frac{1}{(4\pi t)^{d/2}} \exp\bigg({-\frac{\varphi^2}{4t}}\bigg), \qquad \varphi \in [0,\pi], \quad t > 0.
$$
Note that this estimate is sharp in the sense that (see the proof of \cite[Thm.\,5.6.1]{Da})
\begin{equation} \label{sp3}
\lim_{t \to 0^+} (4\pi t)^{d/2} K_t^d(0) = 1.
\end{equation}

Consequently, by \eqref{sp1}, we get the following result. 
\begin{prop} \label{prop:thlow}
Let $N \ge 0$. Then
$$
(-D)^N \vartheta_t(\varphi) \ge
	 \frac{1}{2^N} e^{-tN^2} t^{-N} W_t(\varphi), \qquad \varphi \in [0,\pi], \quad t > 0.
$$
\end{prop}
The bound of Proposition \ref{prop:thlow} is quite reasonable for, say, $\varphi\in [0,\pi/2]$, and the multiplicative
constant is sharp because of \eqref{sp3}. On the other hand, this bound is far from optimal for $\varphi$ close to $\pi$.
We note that the methods developed below to obtain upper bounds for $(-D)^N\vartheta_t(\varphi)$ can be adapted to
provide also the lower bound, but in general with worse multiplicative constant and smaller range of $t$.

\subsection{Upper bound for $(-D)^N\vartheta_t(\varphi)$ and $t \lesssim 1/N$} \label{ssec:upN}	\,

\medskip

We consider $N \ge 1$ and roughly follow a strategy from \cite{NSS}, but with simplifications
(due to the focus only on the upper bound and the restricted range of $\varphi \in [0,\pi/2]$) and significant refinements.

Let $F$ be an even entire function (e.g.\ the Gaussian or the hyperbolic cosine).
It is elementary to show, see \cite[Lem.\,3.1]{NSS}, that
\begin{equation} \label{sp5}
D^N_{\varphi} \big[ F(v\varphi) \big] = \sum_{j=1}^N v^{2j} L^j F(v\varphi) \Phi_{N,j}(\varphi),
	\qquad \varphi \in [0,\pi/2], \quad v > 0, \quad N \ge 1,
\end{equation}
where
$$
L = \frac{1}{z} \frac{\dd}{\dd z}
$$
and $\Phi_{N,j}$ are bounded functions on $[0,\pi/2]$. In \eqref{sp5} and other places below $D^N_{\varphi}[F(v\varphi)]$
means that the operator $D^N$ is applied to the function $\varphi \mapsto F(v\varphi)$, and $L^j F(v\varphi)$
stands for $L^j F$ evaluated at $v\varphi$.

To make \eqref{sp5} useful for our purpose, we need to estimate $\Phi_{N,j}$. To do that, we now make these factors more explicit.
Denote
\begin{equation} \label{sp55}
\Psi(\varphi) := \frac{\varphi}{\sin \varphi}
\end{equation}
so that $\Phi_{1,1} = \Psi$. As in the proof of \cite[Lem.\,3.1]{NSS}, applying $D_{\varphi}$ to term number $j$ in \eqref{sp5}
and using the Leibniz rule to evaluate $D_{\varphi}[L^jF(v\varphi)]$ we get
$$
D_{\varphi}\big[ v^{2j}L^j F(v\varphi) \Phi_{N,j}(\varphi)\big] = v^{2j+2} L^{j+1}F(v\varphi) \frac{\varphi}{\sin\varphi}
	\Phi_{N,j}(\varphi) + v^{2j} L^j F(v \varphi) \frac{1}{\sin\varphi} \Phi'_{N,j}(\varphi).
$$
From this we infer that
\begin{equation} \label{sp6}
\Phi_{N,j} = \Psi \big[ \Phi_{N-1,j-1} +  L \Phi_{N-1,j}\big], \qquad N \ge 2, \quad 1 \le j \le N,
\end{equation}
with the convention that $\Phi_{N,j} = 0$ whenever $j \notin \{1,\ldots,N\}$. Observe that $\Phi_{N,N} = \Psi^N$.

The recurrence \eqref{sp6} is the same as that satisfied by Comtet's $A$-polynomials investigated in \cite{Com73}.
To be precise, Comtet used the ordinary differentiation operator instead of $L$, but this is irrelevant for the present
considerations since both the operators obey the Leibniz rule. Thus $\Phi_{N,j}$ is identified with the polynomial
$\Apol_{N,j}$ in \cite{Com73}. More precisely, $\Phi_{N,j}$ is of the form
$$
	\Phi_{N,j}(\varphi) = \Apol_{N,j}\big(\Psi(\varphi),L\Psi(\varphi),\ldots,L^{N-j}\Psi(\varphi) \big),
$$
where
\begin{equation} \label{sp7}
\Apol_{N,j}(\lambda_0,\ldots,\lambda_{N-j}) = \sum_{k_0+k_1 + \ldots + k_{N-j} = N \atop k_1 + 2k_2 + \ldots + (N-j)k_{N-j} = N-j}
	a_{N,j}(k_0,\ldots,k_{N-j})\; \lambda_0^{k_0} \cdot \ldots \cdot \lambda_{N-j}^{k_{N-j}}.
\end{equation}
Here the summation\,\footnote{
\,The number of terms in \eqref{sp7} is $p(N-j)$, the number of integer partitions of $N-j$, see e.g.\ \cite[Chap.\,II]{Com}
and \cite[A000041]{oeis}.
No explicit formula for the partition function $p$ is known. A classical result by Hardy and Littlewood is the asymptotic
$p(n) \sim \frac{1}{4\sqrt{3}n}\exp\big(\pi\sqrt{\frac{2n}3}\big)$. Concerning estimates of $p(n)$, see for instance
\cite{Ne} and references given there.
}
is over all integer $k_0,\ldots,k_{N-j} \ge 0$ satisfying the indicated conditions.
All the coefficients $a_{N,j}$ appearing in \eqref{sp7} are strictly positive natural numbers.
Each $\Apol_{N,j}$ is a homogeneous polynomial\,\footnote{
\,Tabulation of the polynomials $\Apol_{N,j}$ up to degree $7$ is given in \cite[Sec.\,5]{Com73}.
The sequence of coefficients $a_{N,j}(k_0,\ldots,k_{N-j})$ ordered according to $N$, $j$ and $k_0,\ldots,k_{N-j}$ 
can be found in \cite[A139605]{oeis} (weights for expansions of iterated derivatives).
}
of degree $N$ and weight $N-j$, namely
\begin{equation} \label{sp8}
\Apol_{N,j}(\ldots, ab^{i}\lambda_{i},\ldots)
	= a^N b^{N-j} \Apol_{N,j}(\ldots,\lambda_{i},\ldots).
\end{equation}

We now establish a recurrence relation for the coefficients $a_{N,j}(k_0,\ldots,k_{N-j})$, which may be of independent
interest.\,\footnote{
\,No closed formula for $a_{N,j}$ seems to be available in the literature.
Nevertheless, some special cases can be computed explicitly. For instance, except
for the immediate case $\Phi_{N,N} = \Psi^N$, one has $\Phi_{N,N-1} = {N \choose 2} \Psi^{N-1} L\Psi$ for $N \ge 2$,
$\Phi_{N,N-2} = {N \choose 3} \Psi^{N-1} L^2\Psi + \frac{3N-5}{4} {N \choose 3 } \Psi^{N-2} \big( L\Psi)^2$ for $ N \geq 3$,
$\Phi^{N,N-3} = {N\choose 4} \Psi^{N-1} L^3\Psi + 2(N-2){N\choose 4 } \Psi^{N-2} (L\Psi) (L^2\Psi)
		+\frac{(N-3)(N-2)}{2} {N\choose 4} \Psi^{N-3} (L\Psi)^3$ for $N\geq 4$, furthermore
$\Phi_{N,N-4} = {N \choose 5} \Psi^{N-1} L^4\Psi + \frac{5N-13}{2} {N\choose 5} \Psi^{N-2} (L\Psi)(L^3\Psi)
		+\frac{15N^2-85N+116}{6}{N\choose 5} \Psi^{N-3} (L\Psi)^2 (L^2\Psi) + \frac{15N^3-150N^2+485N-502}{48} {N\choose 5}
			\Psi^{N-4} (L\Psi)^4$ for  $N\geq 5$.
}
It is convenient to introduce some auxiliary notation. For $N \ge 1$ and $1 \le j \le N$ denote
$$
\mathcal{I}(N,j) := \bigg\{ k = (k_0,\ldots,k_{N-j}) \in \mathbb{N}^{N-j+1}\; : \;
	\sum_{i=0}^{N-j} k_i = N \;\;\textrm{and}\;\; \sum_{i=0}^{N-j} i k_i = N-j \bigg\},
$$
so that the range of summation in \eqref{sp7} is $k \in \mathcal{I}(N,j)$.
Observe that for $k \in \mathcal{I}(N,j)$ one has $k_0 \ge j$. Further, we adopt standard conventions concerning multi-indices
and multiplication, for instance for $k = (k_0,\ldots,k_{N-j})$ and $\lambda = (\lambda_0,\ldots,\lambda_{N-j})$,
the expression $\lambda^{k}$ is understood as $\lambda_0^{k_0}\cdot \ldots \cdot \lambda_{N-j}^{k_{N-j}}$. Thus \eqref{sp7}
can be written as
$$
\Apol_{N,j}(\lambda) = \sum_{k \in \mathcal{I}(N,j)} a_{N,j}(k)\, \lambda^k.
$$
Finally, we denote by $e_0, e_1, \ldots$ successive coordinate vectors and we identify multi-indices that differ only by
zeros in the last entries, e.g.\ $(k_0,\ldots,k_{N-j}) \equiv (k_0,\ldots,k_{N-j},0,\ldots,0)$.

\begin{lem} \label{lem:rec77}
Let $N \ge 1$. Then $a_{N,N}(N e_0) = 1$, and for $1 \le j \le N+1$, $k \in \mathcal{I}(N+1,j)$ one has
\begin{equation} \label{sp9}
	\begin{split}
		a_{N+1,j}(k) & = \chi_{\{j \neq 1\}}\, a_{N,j-1}(k - e_0) \\
		& \quad + \chi_{\{j\neq N+1\}}\sum_{i=0}^{N-j} \chi_{\{k_{i+1} \neq 0\}} (k_i + \chi_{\{i \neq 0\}})\,
			a_{N,j}(k-e_0+e_i-e_{i+1}).
	\end{split}
\end{equation}
\end{lem}

\begin{proof}
Let $N \ge 1$. To begin with, we observe that $a_{N,N}(N e_0)=1$ (notice that $\mathcal{I}(N,N) = \{N e_0\}$).
Indeed, by \eqref{sp6} we have
$$
a_{N+1,N+1}\big((N+1)e_0\big) \lambda_0^{N+1} = \Apol_{N+1,N+1}(\lambda) = \lambda_0 \Apol_{N,N}(\lambda)
	= a_{N,N}(Ne_0) \lambda_0^{N+1}
$$
and hence $a_{N,N}(Ne_0) = a_{1,1}(e_0) = 1$.

Now fix $j\in\{1,\ldots,N\}$. By the Leibniz rule we have
\begin{equation*}
	\Psi L\big((\mathcal{L}\Psi)^k\big) = \sum_{i=0}^{N-j} k_i (\mathcal{L}\Psi)^{k+e_0-e_i+e_{i+1}},
\end{equation*}
where for $k=(k_0,\ldots,k_{N-j})$ we used the notation $(\mathcal{L}\Psi)^k = \prod_{i=0}^{N-j} (L^i \Psi)^{k_i}$.

Hence, by \eqref{sp6} and \eqref{sp7}, we obtain
\begin{equation} \label{sp11}
\Apol_{N+1,j}(\lambda) = \chi_{\{j \neq 1\}} \sum_{k \in \mathcal{I}(N,j-1)} a_{N,j-1}(k)\lambda^{k+e_0}
	+ \sum_{k \in \mathcal{I}(N,j)} a_{N,j}(k) \sum_{i=0}^{N-j} k_i \lambda^{k+e_0-e_i+e_{i+1}}.
\end{equation}
Now the crucial observations are that
\begin{equation*}
\big\{ k+e_0 : k \in \mathcal{I}(N,j-1)\big\}=	\mathcal{I}(N+1,j),
\end{equation*}
and that for $0\le i\le N-j$
$$
\big\{ k+e_0-e_i+e_{i+1} : k \in \mathcal{I}(N,j), k_i \neq 0 \big\} =  \mathcal{I}(N+1,j) \cap\{k_{i+1}\neq 0 \},
$$
which can be easily verified. Therefore,
\begin{align*}
\Apol_{N+1,j}(\lambda) & = \chi_{\{j \neq 1\}} \sum_{k \in \mathcal{I}(N+1,j)} a_{N,j-1}(k-e_0) \lambda^k \\
	& \quad + \sum_{i=0}^{N-j} \sum_{k \in \mathcal{I}(N+1,j)} \chi_{\{k_{i+1}\neq 0\}} (k_i + \chi_{\{i \neq 0\}})
		a_{N,j}(k-e_0+e_i-e_{i+1}) \lambda^k.
\end{align*}
This proves \eqref{sp9}.
\end{proof}

Next, we evaluate $\Apol_{N,j}$ at a particular point.\,\footnote{
\,Some other evaluations of $\Apol_{N,j}$ are already known, see \cite{Com73}. In particular,
$\Apol_{N,j}(1,\ldots,1) = \left[ N \atop j \right]$ (unsigned Stirling numbers of the first kind) and
$\Apol_{N,j}(1,1,0,\ldots,0) = \left\{ N \atop j \right\}$ (Stirling numbers of the second kind).
The numbers appearing on the right-hand side of \eqref{sp10} are coefficients of Bessel polynomials of degree $N-1$
multiplying there monomials of degree $N-j$, see \cite[A001498]{oeis}.
}
\begin{lem} \label{lem:Aspec}
Let $N \ge 1$ and $1 \le j \le N$. Then
\begin{equation} \label{sp10}
\Apol_{N,j}\big(0!,1!,\ldots,(N-j)!\big) = \frac{(2N-j-1)!}{2^{N-j}(N-j)!(j-1)!}.
\end{equation}
\end{lem}

\begin{proof}
For $j=N$, and in particular when $N=1$, the lemma holds since $\Apol_{N,N}(\lambda)=\lambda_0^N$. Otherwise, we use induction on $N$.
For the inductive step assume that $1 \le j \le N$ and observe that by \eqref{sp11} one has
\begin{align*}
	\Apol_{N+1,j}(\lambda_0,\ldots,\lambda_{N+1-j}) 
    & =
    \chi_{\{j\neq 1\}} \lambda_0 \Apol_{N,j-1}(\lambda_0,\ldots,\lambda_{N+1-j}) \\
	& \quad	+ \sum_{k\in\mathcal{I}(N,j)} a_{N,j}(k) \Bigg(\prod_{l=0}^{N-j} \lambda_l^{k_l}\Bigg)
			\sum_{i=0}^{N-j} \frac{k_i \lambda_0\lambda_{i+1}}{\lambda_i}.
	\end{align*}
Thus, by the inductive hypothesis and the identity
	\begin{equation*}
		\sum_{i=0}^{N-j} (i+1)k_i = 2N-j, \qquad k\in\mathcal{I}(N,j),	
	\end{equation*}
we obtain
	\begin{align*}
		\Apol_{N+1,j}\big(0!,\ldots,(N+1-j)!\big)
			& = \frac{\chi_{\{j\neq1\}} (2N-j)!}{2^{N+1-j} (N+1-j)! (j-2)!} + (2N-j) \frac{(2N-j-1)!}{2^{N-j} (N-j)! (j-1)!} \\
			& = \frac{(2N-j)!\big[\chi_{\{j\neq 1\}}(j-1)+2(N-j+1) \big]}{2^{N+1-j}(N+1-j)!(j-1)!} \\
			& = \frac{(2(N+1)-j-1)!}{2^{N+1-j} (N+1-j)! (j-1)!}.
	\end{align*}
The last expression is the right-hand side in \eqref{sp10} with $N$ replaced by $N+1$. This concludes the proof.
\end{proof}

In order to estimate $\Phi_{N,j}(\varphi)$ from above we shall need upper bounds for $L^i\Psi(\varphi)$, which we now establish.
The key observation here is the series expansion (which follows from the well known expansion of the cosecant function,
see e.g.\ \cite[pp.\,88--89]{Com})
\begin{equation}\label{sp12}
	\Psi(\varphi) = \sum_{n=0}^{\infty} \frac{(-1)^{n+1} (2^{2n}-2) B_{2n}}{(2n)!} \varphi^{2n}
	= \sum_{n=0}^{\infty} \frac{4 (2^{2n-1}-1)\zeta(2n)}{(2\pi)^{2n}} \varphi^{2n}, \qquad \varphi \in [0,\pi).
\end{equation}
Here $B_{2n}$ are Bernoulli numbers and $\zeta$ is the Riemann zeta function; note that $\zeta(0)=-1/2$.
The expansion in \eqref{sp12} has strictly positive coefficients, and clearly $L^i\Psi$ is of the same form,
again with strictly positive coefficients. So $L^i\Psi$, $i=0,\ldots,N-1$, are increasing on $[0,\pi)$.
The maxima $(L^i\Psi)(\pi/2)$ over $[0,\pi/2]$ can be quite precisely estimated from the positive series.\,\footnote{
\,The main bound of Lemma \ref{lem:LjPsi} is quite precise except for several initial $i$
(the accuracy improves with $i$ growing). For those $i$, however, it is straightforward to compute exact values of
$L^i \Psi(\pi/2)$. One has $\Psi(\pi/2) = \pi/2 \approx 1.570796 \; [\approx 2.666667]$,
$(L\Psi)(\pi/2) = 2/\pi \approx 0.636620 \; [\approx 0.720506]$,
$(L^2\Psi)(\pi/2) = 2/\pi-8/\pi^3 \approx 0.378607 \; [\approx 0.389347]$ and
$(L^3\Psi)(\pi/2) = 96/\pi^5 \approx 0.313705 \; [\approx 0.315593]$,
where in square brackets the corresponding approximations of the upper bound from Lemma \ref{lem:LjPsi} are given.
}
\begin{lem}\label{lem:LjPsi}
	Let $i \in \mathbb{N}$. Then
	\begin{equation*}
		0 < L^i\Psi(\varphi)\leq L^i \Psi\Big(\frac{\pi}{2}\Big)
		< \frac{8^{i+1} i!}{3^{i+1} \pi^{2i}}, \qquad \varphi \in [0,\pi/2].
	\end{equation*}
\end{lem}

\begin{proof}
We need to prove only the last inequality.
For $i=0$ the inequality is trivial so assume that $i \ge 1$. By \eqref{sp12} and the series representation of the zeta function we have
	\begin{align*}
		L^i\Psi \Big( \frac{\pi}{2}\Big) 
		&= \sum_{n=i}^\infty \frac{4(2^{2n-1}-1) \zeta(2n) 2n(2n-2)\cdot\ldots \cdot (2n-2i+2)}{(2\pi)^{2n}}
			\Big( \frac{\pi}{2}\Big)^{2n-2i}\\
		& =(2\pi)^{-2i} \sum_{n=0}^\infty \frac{4(2^{2n+2i-1}-1)\zeta(2n+2i)}{4^{2n}}\, \frac{(2n+2i)!!}{(2n)!!}\\
		&= \Big(\frac{2}{\pi}\Big)^{2i} \sum_{n=0}^\infty \Big(\frac{2}{2^{2n+2i}} - \frac{4}{4^{2n+2i}}\Big)
			\frac{\zeta(2n+2i)(2n+2i)!!}{(2n)!!} \\
		& = \Big(\frac{2}{\pi}\Big)^{2i} \sum_{k=1}^\infty \sum_{n=0}^\infty \bigg(\frac{2}{(2k)^{2n+2i}}
			- \frac{4}{(4k)^{2n+2i}}\bigg) \frac{(2n+2i)!!}{(2n)!!}\\
		& = 2\Big(\frac{2}{\pi}\Big)^{2i} \sum_{k=1}^\infty (-1)^{k+1} \sum_{n=0}^\infty \frac{(2n+2i)!!}{(2n)!! (2k)^{2n+2i}},
	\end{align*}
	where in the last equality we used the fact that for an absolutely convergent series
	$\sum_{k=1}^{\infty} b_k$ one has $\sum_{k=1}^\infty (b_k -2b_{2k}) = \sum_{k=1}^\infty (-1)^{k+1} b_k$.
	Using induction it is easy to check that
	\begin{equation*}
		\sum_{n=0}^\infty \frac{(2n+2i)!!}{(2n)!!} x^{2n} = \frac{2^i i!}{(1-x^2)^{i+1}}, \qquad x \in (-1,1).
	\end{equation*}
	Thus we arrive at the identity
	$$
		L^i \Psi\Big(\frac{\pi}{2}\Big) =  8i!\Big(\frac{8}{\pi^2}\Big)^i \sum_{k=1}^\infty (-1)^{k+1} \frac{k^2}{(4k^2-1)^{i+1}}.
	$$
    
	Now we finish the proof by observing that the above alternating series is strictly bounded from above by the first term
	which equals $3^{-(i+1)}$. 
\end{proof}

We are now in a position to estimate $\Phi_{N,j}(\varphi)$.
\begin{lem} \label{lem:ePhi}
Let $N \ge 1$ and $1 \le j \le N$. Then
\begin{equation*}
	0 < \Phi_{N,j}(\varphi) < \Big(\frac{8}3\Big)^N \Big( \frac{4}{3\pi^2}\Big)^{N-j} \Big(\frac{3\pi}{16} \Big)^j
		\frac{(2N-j-1)!}{(N-j)!(j-1)!}, \qquad \varphi \in [0,\pi/2].
\end{equation*}
\end{lem}

\begin{proof}
By strict positivity of $a_{N,j}(k)$, $k\in\mathcal{I}(N,j)$, the polynomial $\Apol_{N,j}(\lambda)$ is strictly increasing
in each variable $\lambda_i \ge 0$, $i =0,\ldots,N-j$. Since $(L^i\Psi)(\varphi)$, $i=0,1,2,\ldots$, are strictly positive
for $\varphi \in [0,\pi/2]$, it follows that $\Phi_{N,j}>0$ on $[0,\pi/2]$. Moreover, we can estimate $\Phi_{N,j}$ on
$[0,\pi/2]$ from above just by estimating $\Apol_{N,j}(\lambda)$, where for $\lambda_i$ we insert the upper bounds for
$L^i\Psi$ on $[0,\pi/2]$ established in Lemma \ref{lem:LjPsi} (for $i=0$ we use the better bound $\Psi(\varphi) \le \Psi(\pi/2)=\pi/2$).
This gives
\begin{align*}
	\Phi_{N,j}(\varphi) & \leq \Apol_{N,j}\bigg(\frac{\pi}2, \frac{8^{1+1}\cdot  1!}{3^{1+1}\pi^{2\cdot 1}},\ldots,
		\frac{8^{N-j+1}(N-j)!}{3^{N-j+1}\pi^{2(N-j)}} \bigg) \\
		& \le \Big(\frac{3\pi}{16}\Big)^j \Apol_{N,j}\bigg(\frac{8^{0+1} \cdot 0!}{3^{0+1}\pi^{2\cdot 0}},\ldots,
		\frac{8^{N-j+1}(N-j)!}{3^{N-j+1}\pi^{2(N-j)}} \bigg) \\
	& = \Big(\frac{8}3\Big)^N \Big( \frac{8}{3\pi^2}\Big)^{N-j} \Big(\frac{3\pi}{16}\Big)^j \Apol_{N,j}\big(0!,\ldots,(N-j)!\big);
\end{align*}
here the last inequality follows from the structure of $\Apol_{N,j}$, see \eqref{sp7}, and the observation that the summation
constraints in \eqref{sp7} force $k_0 \ge j$, while the succeeding identity is a consequence of~\eqref{sp8}.

Now an application of Lemma \ref{lem:Aspec} concludes the proof.
\end{proof}

We continue with auxiliary estimates that will allow us to control the right-hand side in \eqref{sp5}.
Our next objective is to bound from above sums of expressions resulting from the estimate of Lemma \ref{lem:ePhi}.
Instead of estimating first each term of the sum, we choose another strategy leading to a more precise final bound.
The lemma below is essentially \cite[Chap.\,III, Sec.\,3$\cdot$71, Form.\,(12)]{watson};
see also e.g.\ \cite[Sec.\,4.1.7, Form.\,(19)]{PBM}.
\begin{lem} \label{lem:Kid}
Let $N \ge 1$. Then
$$
\sum_{j=1}^N \frac{(2N-j-1)!}{(N-j)!(j-1)!}\, x^j = x^N \bigg[ \sqrt{\frac{x}{\pi}} e^{x/2} K_{N-1/2}(x/2)\bigg], \qquad x > 0,
$$
where $K_{\nu}$ is the modified Bessel function of the second kind (Macdonald function) of order $\nu$.
\end{lem}

Standard asymptotics for $K_{\nu}$ (cf.\ \cite[Chap.\,VII, Sec.\,7$\cdot$23, Form.\,(1)]{watson}) yield
$$
K_{\nu}(x/2) = \sqrt{\frac{\pi}{x}} e^{-x/2} \Big( 1 + \mathcal{O}\big(x^{-1}\big) \Big), \qquad x \to \infty,
$$
which shows that the expression in square brackets in Lemma \ref{lem:Kid} tends to $1$ as $x \to \infty$.
To estimate that expression we shall use bounds for $K_{\nu}$ from Corollary \ref{cor:estKnu}.

In Corollary \ref{cor:estKnu} letting $\nu=N-1/2$, $N \ge 1$, we obtain for $0 < \epsilon \le 2/3$
\begin{equation} \label{sp13}
\sqrt{\frac{x}{\pi}} e^{x/2} K_{N-1/2}(x/2) \le (1+\epsilon)^{N-1}, \qquad x \ge \frac{4(N-1/2)}{3\epsilon}.
\end{equation}

After all the preparations we now come back to estimating from above
\begin{equation} \label{sp14}
(-D)^N \vartheta_t(\varphi) = (-D)^N W_t(\varphi) + (-D)^N \sum_{0 \neq n \in \mathbb{Z}} W_t(\varphi + 2\pi n).
\end{equation}
Here $(-D)^N W_t$ can be interpreted as the main term, and the remaining part on the right-hand side as an error term.
We first treat the main term.
\begin{lem} \label{lem:mNth}
Let $N \ge 1$. Then for each $0 < \epsilon \le 2/3$
$$
\big| (-D)^N W_t(\varphi) \big| < \frac{1}{1+\epsilon} \bigg[ \frac{\pi}4 (1+\epsilon)\bigg]^N t^{-N} W_t(\varphi),
	\qquad \varphi \in [0,\pi/2], \quad 0 < t \le \frac{27\pi^3\epsilon}{512(N-1/2)}.
$$
\end{lem}

\begin{proof}
We consider $\varphi \in [0,\pi/2]$. Let $F = W_{1/4}$ and $v=(4t)^{-1/2}$, so that $W_t(\varphi) = v F(v\varphi)$.
Using \eqref{sp5} and observing that $L^j F = (-2)^j F$ we get
$$
\big| (-D)^N W_t(\varphi)\big| = W_t(\varphi) \bigg| \sum_{j=1}^N \Big(\frac{-1}{2t}\Big)^{j} \Phi_{N,j}(\varphi)\bigg|
	\le W_t(\varphi) \sum_{j=1}^N (2t)^{-j} \Phi_{N,j}(\varphi).
$$
Denote the last sum by $\mathcal{S}$ and let $x = 9\pi^3/(128t)$. Applying Lemma \ref{lem:ePhi} and then
Lemma \ref{lem:Kid} we obtain
$$
\mathcal{S} < \Big(\frac{32}{9\pi^2}\Big)^N \sum_{j=1}^N x^j \frac{(2N-j-1)!}{(N-j)!(j-1)!}
	= t^{-N} \Big(\frac{\pi}4\Big)^N \bigg[ \sqrt{\frac{x}{\pi}} e^{x/2} K_{N-1/2}(x/2)\bigg].
$$
The above expression in square brackets can be bounded by means of \eqref{sp13}. This gives
$$
\mathcal{S} < t^{-N} \Big(\frac{\pi}4\Big)^N (1+\epsilon)^{N-1}, \qquad x \ge \frac{4(N-1/2)}{3\epsilon}, \quad 0 < \epsilon \le 2/3.
$$

The conclusion follows.
\end{proof}

To estimate the error term in \eqref{sp14} we will need two more auxiliary technical results.
\begin{lem} \label{lem:cosh}
Let $j \ge 0$. Then
$$
L^j \cosh z \le \frac{\sqrt{\pi}}{2^j \Gamma(j+1/2)} e^z, \qquad z \ge 0.
$$
\end{lem}

\begin{proof}
We have
$$
\cosh z = \sum_{n=0}^{\infty} \frac{z^{2n}}{(2n)!}.
$$
Applying $L^j$ and renumerating the series, we see that
$$
L^j\cosh(z) = \sum_{n=0}^{\infty} \frac{1}{(2n+2j-1)\cdot \ldots \cdot (2n+1)} \frac{z^{2n}}{(2n)!}, \qquad j \ge 1.
$$
It follows that
$$
L^j\cosh z \le \frac{1}{(2j-1)!!} \cosh z, \qquad z \ge 0, \quad j \ge 1.
$$
Since, cf.\ \cite[\S 5.4, Form.\,5.4.2]{DLMF}, $(2j-1)!!= 2^j \Gamma(j+1/2)/\sqrt{\pi}$, this implies the lemma.
\end{proof}
	
\begin{lem} \label{lem:auxsph}
	The bound
	\begin{align*}
		\sum_{n = 1}^{\infty} 
		e^{-\pi^2n^2/t }  e^{\pi n \varphi/t} n^{2j}
		\le 
		e^{-\pi^2/(2t) } + \frac{1}{2}  \Big( \frac{2}{\pi} \Big)^{2j+1}\Gamma(j+1/2) t^{j+1/2} e^{-2\pi^2/t }
	\end{align*}
	holds for $\varphi \in [0,\pi/2]$, $0 < t \le \pi^2/(4N)$ and $1 \le j \le N$.
\end{lem}

\begin{proof}
	Let $\varphi, t$ and $j$ be as in the statement.
	Observe that the term $n=1$ is bounded above by $e^{-\pi^2/(2t)}$.
	On the other hand, we have 
	$$
		e^{-\pi^2n^2/t} e^{\pi n \varphi/t} \le e^{-2\pi^2/t} e^{-\pi^2 n^{2}/(4t)}, \qquad n \ge 2.
	$$
	The function $f_{j,t}\colon x \mapsto e^{-\pi^2 x^2/(4t) } x^{2j}$  is increasing on the interval $(0, 2\sqrt{tj}/\pi)$
	and decreasing on $(2\sqrt{tj}/\pi,\infty)$. The constraints on $j$ and $t$ imply that $f_{j,t}$ is decreasing on $(1,\infty)$.
	This means that $f_{j,t}(n) \le \int_{n-1}^{n} f_{j,t}$ for all $n \ge 2$. Consequently, we get
	$$
		\sum_{n = 2}^{\infty} e^{-\pi^2n^2/(4t) } n^{2j} \le \int_{0}^{\infty } e^{-\pi^2 x^2/(4t) } x^{2j} \, \dd x
		= \frac{1}{2}\Big( \frac{2}{\pi} \Big)^{2j+1} t^{j+1/2} \Gamma(j+1/2),
	$$
	where to get the last identity we changed the variable $\pi^{2} x^{2}/(4t) \mapsto y$.
	
	Altogether, the above gives the desired conclusion.
\end{proof}

We are now ready to estimate the error term in \eqref{sp14}.
\begin{lem} \label{lem:eNth}
Let $N \ge 2$. Then for each $0 < \epsilon \le 1/3$ the bound
$$
\bigg| (-D)^N \sum_{0 \neq n \in \mathbb{Z}} W_t(\varphi + 2\pi n) \bigg|
	\le \frac{21}{500}\, \frac{1}{1+\epsilon} \Big[ \frac{\pi}4 (1+\epsilon)\Big]^N t^{-N} W_t(\varphi)
$$
holds for $\varphi \in [0,\pi/2]$ and $0 < t \le 27\pi^3 \epsilon/ [512(N-1/2)]$.
\end{lem}

\begin{proof}
We have
\begin{align*}
(-D)^N \sum_{0 \neq n \in \mathbb{Z}} W_t(\varphi+2\pi n) & 
	= \sum_{n=1}^{\infty} (-D)^N \big[ W_t(\varphi-2\pi n) + W_t(\varphi+2\pi n)\big] \\
& = \sum_{n=1}^{\infty} (-D)^N \Big[ 2e^{-\pi^2n^2/t} W_t(\varphi) \cosh\frac{\pi n \varphi}{t} \Big] \\
& = (-1)^N 2 \sum_{n=1}^{\infty} e^{-\pi^2 n^2/t} \sum_{k=0}^N \binom{N}{k} D^{N-k}W_t(\varphi)
	D_{\varphi}^k \Big[ \cosh\frac{\pi n \varphi}{t} \Big],
\end{align*}
where in the last step we applied the Leibniz rule for $D$. It is enough we estimate suitably the expression
$$
\mathtt{S} := 2 \sum_{n=1}^{\infty} e^{-\pi^2 n^2/t} \sum_{k=0}^N \binom{N}{k} \big| D^{N-k}W_t(\varphi)\big| \,
	\bigg| D_{\varphi}^k \Big[ \cosh\frac{\pi n \varphi}{t} \Big] \bigg|.
$$
To proceed, we distinguish the part of $\mathtt{S}$ corresponding to $k=0$, call it $\mathtt{S}_0$.
Throughout the proof we assume that $0 < \epsilon \le 1/3$ and $\varphi \in [0,\pi/2]$ and $0 < t \le 27\pi^3 \epsilon/ [512(N-1/2)]$.

Considering $\mathtt{S}_0$, we use the bound $\cosh z \le e^{z}$, $z \ge 0$, to write
\begin{equation} \label{sp23}
\mathtt{S}_0 \le 2 \big| D^N W_t(\varphi) \big| \sum_{n=1}^{\infty} e^{-\pi^2 n(n-1/2)/t}.
\end{equation}
Since the constraint on $t$ implies $t \le \pi/8$, we can estimate the above series as
$$
\sum_{n=1}^{\infty} e^{-\pi^2 n(n-1/2)/t} \le \sum_{n=1}^{\infty} e^{-8\pi n(n-1/2)} < \sum_{n=1}^{\infty} e^{-4\pi n^2}
	= \frac{1}2 \bigg( \frac{\sqrt[4]{2}+1}{\Gamma(3/4)}\sqrt[4]{\frac{\pi}{32}} - 1\bigg) \approx 3.487342\cdot 10^{-6};
$$
here the only equality follows from evaluation of $\sum_{n \in \mathbb{Z}} e^{-4 \pi n^2}$ that can be found e.g.\ in
\cite[Sec.\,5]{Yi1}. On the other hand, to bound $|D^N W_t(\varphi)|$ we use Lemma \ref{lem:mNth}. This gives
\begin{equation} \label{sp16}
\mathtt{S}_0 < 7\cdot 10^{-6}\, \frac{1}{1+\epsilon} \Big[ \frac{\pi}4 (1+\epsilon)\Big]^N t^{-N} W_t(\varphi),
	\qquad t \le \frac{27\pi^3\epsilon}{512(N-1/2)}.
\end{equation}

We pass to estimating $\mathtt{S}-\mathtt{S}_0$. Specifying \eqref{sp5} to $F=\cosh$ and $v=\pi n/t$, we get
$$
D^k_{\varphi} \Big[ \cosh\frac{\pi n \varphi}{t} \Big] = \sum_{j=1}^k \Big(\frac{\pi n}t \Big)^{2j}
	L^j(\cosh)\Big(\frac{\pi n \varphi}{t} \Big) \Phi_{k,j}(\varphi).
$$
Next, using Lemmas \ref{lem:ePhi} and \ref{lem:cosh} we obtain
$$
\bigg| D^k_{\varphi} \Big[ \cosh\frac{\pi n \varphi}{t} \Big] \bigg|
	\le \sum_{j=1}^k \Big( \frac{\pi n}t\Big)^{2j} \frac{\sqrt{\pi}}{2^j \Gamma(j+1/2)} e^{\pi n \varphi/t} \Big(\frac{8}3\Big)^k
		\Big(\frac{4}{3\pi^2}\Big)^{k-j} \Big(\frac{3\pi}{16}\Big)^j \frac{(2k-j-1)!}{(k-j)!(j-1)!}.
$$
Combining this with the expression for $\mathtt{S}$ and then applying Lemma \ref{lem:auxsph} we see that
\begin{align*}
\mathtt{S}-\mathtt{S}_0
& \le 2\sum_{k=1}^{N} \binom{N}{k} \big| D^{N-k}W_t(\varphi)\big| \sum_{n=1}^{\infty} e^{-\pi^2 n^2/t} \\
	& \qquad \times \sum_{j=1}^k \Big( \frac{\pi n}t\Big)^{2j} \frac{\sqrt{\pi} e^{\pi n \varphi/t}}{2^j \Gamma(j+1/2)}
		\Big(\frac{8}3\Big)^k \Big(\frac{4}{3\pi^2}\Big)^{k-j} \Big(\frac{3\pi}{16}\Big)^j \frac{(2k-j-1)!}{(k-j)!(j-1)!} \\
& \le I_1 + I_2,
\end{align*}
where
\begin{align*}
I_1 & := 2\sqrt{\pi} e^{-\pi^2/(2t)} \sum_{k=1}^N \binom{N}{k} \big| D^{N-k} W_t(\varphi) \big| \Big( \frac{32}{9\pi^2} \Big)^k
	\sum_{j=1}^{k} \bigg( \frac{9\pi^5}{128t^2}\bigg)^j \frac{1}{\Gamma(j+1/2)} \frac{(2k-j-1)!}{(k-j)!(j-1)!}, \\
I_2 & := \frac{2}{\sqrt{\pi}} e^{-2\pi^2/t} \sqrt{t} \sum_{k=1}^N \binom{N}{k} \big| D^{N-k} W_t(\varphi) \big|
	\Big( \frac{32}{9\pi^2} \Big)^k \sum_{j=1}^{k} \bigg( \frac{9\pi^3}{32t}\bigg)^j \frac{(2k-j-1)!}{(k-j)!(j-1)!}.
\end{align*}
We will bound $I_1$ and $I_2$ separately, and first consider $I_1$.

Since $x^j e^{-x} \le j^j e^{-j}$, $x > 0$, $j \ge 1$, we have $e^{-\pi^2/(4t)} \le j^j e^{-j}(4t/\pi^2)^j$.
Combining this with Lemma \ref{lem:GAMest} produces
$$
e^{-\pi^2/(4t)} \le \frac{e^{1/12}}{\sqrt{2\pi}} \Big( \frac{4t}{\pi^2} \Big)^j \Gamma(j+1/2), \qquad j \ge 1.
$$
Therefore,
\begin{equation} \label{sp24}
I_1 \le \sqrt{2}\,e^{1/12} e^{-\pi^2/(4t)} \sum_{k=1}^N \binom{N}{k} \big| D^{N-k} W_t(\varphi) \big|
	\Big( \frac{32}{9\pi^2} \Big)^k \sum_{j=1}^{k} \bigg( \frac{9\pi^3}{32t}\bigg)^j \frac{(2k-j-1)!}{(k-j)!(j-1)!}.
\end{equation}
By Lemma \ref{lem:Kid} and \eqref{sp13}
we get for $1 \le k \le N$
\begin{equation} \label{sp25}
\Big( \frac{32}{9\pi^2} \Big)^k \sum_{j=1}^{k} \bigg( \frac{9\pi^3}{32 t}\bigg)^j \frac{(2k-j-1)!}{(k-j)!(j-1)!}
	\le \Big( \frac{\pi}{t} \Big)^k (1+\epsilon)^{k-1}, \qquad t \le \frac{27\pi^3\epsilon}{128(k-1/2)},
\end{equation}
and from Lemma \ref{lem:mNth} we know that for $1 \le k \le N-1$
$$
\big| D^{N-k} W_t({\varphi}) \big| < \Big( \frac{\pi}4 \Big)^{N-k} (1+\epsilon)^{N-k-1} t^{-(N-k)} W_t(\varphi),
	\qquad t \le \frac{27\pi^3 \epsilon}{512(N-k-1/2)}.
$$
Consequently, including the term corresponding to $k=N$,
\begin{align*}
I_1 & \le \sqrt{2}\,e^{1/12} e^{-\pi^2/(4t)} \frac{1}{1+\epsilon} \Big[ \frac{\pi}4 (1+\epsilon) \Big]^N t^{-N} W_t(\varphi)
	\sum_{k=1}^N \binom{N}{k} 4^k \\
& = \Big( \sqrt{2}\,e^{1/12} e^{-\pi^2/(4t)} (5^N-1) \Big) \frac{1}{1+\epsilon}
	\Big[ \frac{\pi}4 (1+\epsilon) \Big]^N t^{-N} W_t(\varphi), \qquad t \le \frac{27\pi^3\epsilon}{512(N-1/2)}.
\end{align*}
Here we can bound the factor in parentheses
\begin{align*}
\sqrt{2}\,e^{1/12} e^{-\pi^2/(4t)} (5^N-1) & \le \sqrt{2}\,e^{1/12} e^{-128(N-1/2)/(27\pi\epsilon)} (5^N-1) \\
	& \le \sqrt{2}\,e^{1/12} e^{-128(N-1/2)/(9\pi)} (5^N-1),
\end{align*}
and the last expression takes the maximal value over $N \ge 2$ for $N=2$,
which is $24 \sqrt{2}\,e^{1/12-64/(3\pi)} \approx 0.041476$. Thus
\begin{equation} \label{sp17}
I_1 < \frac{415}{10000}\, \frac{1}{1+\epsilon}
	\Big[ \frac{\pi}4 (1+\epsilon) \Big]^N t^{-N} W_t(\varphi), \qquad t \le \frac{27\pi^3\epsilon}{512(N-1/2)}.
\end{equation}

It remains to estimate $I_2$. Proceeding as in the case of $I_1$, we get
$$
I_2 \le \bigg( \frac{2}{\sqrt{\pi}} e^{-2\pi^2/t} \sqrt{t} \big(5^N-1\big) \bigg) \frac{1}{1+\epsilon}
	\Big[ \frac{\pi}4 (1+\epsilon) \Big]^N t^{-N} W_t(\varphi), \qquad t \le \frac{27\pi^3\epsilon}{512(N-1/2)}.
$$
Here the factor in parentheses can be estimated by inserting the maximal admissible value of $t$,
$$
\frac{2}{\sqrt{\pi}} e^{-2\pi^2/t} \sqrt{t} \big(5^N-1\big) \le 
	\frac{2}{\sqrt{\pi}} e^{-1024(N-1/2)/(27\pi\epsilon)} \sqrt{\frac{27\pi^3\epsilon}{512(N-1/2)}} \big(5^N-1\big).
$$
The expression on the right-hand side above takes its maximal value over $0 < \epsilon \le 1/3$ and $N \ge 2$ for $\epsilon=1/3$
and $N=2$, which is $3\sqrt{3}\pi e^{-512/(3\pi)} \approx 4.167091 \cdot 10^{-23}$. Thus
\begin{equation} \label{sp18}
I_2 < 10^{-22}\, \frac{1}{1+\epsilon}
	\Big[ \frac{\pi}4 (1+\epsilon) \Big]^N t^{-N} W_t(\varphi), \qquad t \le \frac{27\pi^3\epsilon}{512(N-1/2)}.
\end{equation}

Collecting the bounds \eqref{sp16}, \eqref{sp17} and \eqref{sp18} we conclude the estimate stated in the lemma.
\end{proof}

Summarizing the main results of Section \ref{ssec:upN}, Lemmas \ref{lem:mNth} and \ref{lem:eNth} imply the following.
\begin{thm} \label{thm:Nth}
Let $N \ge 2$. Then for each $0 < \epsilon \le 1/3$
$$
(-D)^N \vartheta_{t}(\varphi) < \Big(1+\frac{21}{500} \Big) \frac{1}{1+\epsilon}
	\Big[ \frac{\pi}4 (1+\epsilon) \Big]^N t^{-N} W_t(\varphi),
	\qquad \varphi \in [0,\pi/2], \quad 0 < t \le \frac{27\pi^3 \epsilon}{512(N-1/2)}.
$$
\end{thm}
Note that $27\pi^3/512 \approx 1.635097$.

\subsection{Improved upper bound for $(-D)^N\vartheta_t(\varphi)$ when $t \lesssim 1/N^2$} \label{ssec:upN2}	\,

\medskip

In this section we present another approach to estimating $(-D)^N W_t(\varphi)$, which results in a better
(in terms of the dependence on $N$) multiplicative constant in the upper bound for $(-D)^N\vartheta_t(\varphi)$.
This improvement, however, requires restricting the range of admissible $t$.

We start with several preparatory technical results. Define
\begin{equation} \label{sp21}
h_1(\varphi) = \varphi, \qquad h_{j+1}(\varphi) = h_j'(\varphi)\sin\varphi -(2j-1)h_j(\varphi)\cos\varphi, \qquad j \ge 1.
\end{equation}
Further, for $\varphi \notin \pi \mathbb{Z}$ let
$$
g_j(\varphi)=\frac{h_j(\varphi)}{(\sin\varphi)^{2j-1}}, \qquad j \ge 1.
$$
Observe that $g_{j+1}(\varphi)=D g_j(\varphi)$, $j \ge 1$, where $D$ is defined in \eqref{sp4}.
Thus $g_j=D^{j-1}\Psi$, $j \ge 1$, with $\Psi$ defined in \eqref{sp55}. We also introduce functions
$$
w_j(\varphi)=\frac{h_j(\varphi)(\pi-\varphi)^{j-1}}{j! \varphi^j (\sin\varphi)^{j-1}}, \qquad j \ge 1.
$$
\begin{lem} \label{lem:h}
Each of the functions $h_j(\varphi)$, $j \ge 1$, is non-negative and increasing for $\varphi \in [0,\pi]$.
Moreover, for $j \ge 1$
\begin{equation} \label{sp20}
	h'_{j+1}(\varphi)=j^2 h_j(\varphi)\sin\varphi.
\end{equation}
\end{lem}

\begin{proof}
	Since $h_j(0)=0$, $j \ge 1$, and $h_1 \ge 0$ on $[0,\pi]$, it suffices to prove \eqref{sp20}. This will be done by induction.
	
	It is straightforward to verify \eqref{sp20} for $j=1$.
	Assume now that \eqref{sp20} holds for some $j \ge 1$. By \eqref{sp21} and the inductive hypothesis
	\begin{align*}
		h'_{j+2} (\varphi) &= \big[ j^2 h_j(\varphi)\sin^2\varphi -(2j+1)h_{j+1}(\varphi)\cos\varphi\big]'\\
		&= 2j^2 h_j(\varphi)\sin\varphi\cos\varphi + j^2h_j'(\varphi)\sin^2\varphi
			-(2j+1)h_{j+1}'(\varphi)\cos\varphi+(2j+1)h_{j+1}(\varphi)\sin\varphi\\
		&= 2j^2 h_j(\varphi)\sin\varphi\cos\varphi + j^2h_j'(\varphi)\sin^2\varphi-(2j+1)j^2 h_{j}(\varphi)\sin\varphi\cos\varphi\\
		&	\quad +(2j+1)h_j'(\varphi)\sin^2\varphi-(4j^2-1)h_j(\varphi)\sin\varphi\cos\varphi\\
		&= -(j+1)^2 (2j-1)h_j(\varphi)\sin\varphi\cos\varphi +(j+1)^2h_j'(\varphi)\sin^2\varphi\\
		&=(j+1)^2h_{j+1}(\varphi)\sin\varphi.
	\end{align*}
    
	The lemma follows.
\end{proof}
	
\begin{rem} \label{rem:hgw}
		Let $j \ge 1$. The following two-step recurrence relations hold:
		\begin{align*}
			h_{j+2}(\varphi) & = -(2j+1)\cos\varphi\, h_{j+1}(\varphi) + j^2 \sin^2\varphi\, h_j(\varphi), \\
			g_{j+2}(\varphi) & = \frac{1}{\sin^2\varphi}\Big[ -(2j+1)\cos\varphi\, g_{j+1}(\varphi)
				+ j^2 g_j(\varphi)\Big], \\
			w_{j+2}(\varphi) & = 
			 - \frac{2j+1}{j+2}  \cot\varphi \,\frac{\pi-\varphi}{\varphi} \, w_{j+1}(\varphi)
					+ \frac{j^2}{(j+1)(j+2)} \bigg( \frac{\pi-\varphi}{\varphi}\bigg)^2 w_j(\varphi).
		\end{align*}	
		The relation for $h_j$ is verified by combining \eqref{sp21} (with $j$ replaced by $j+1$) with \eqref{sp20}.
		The remaining relations follow from that for $h_j$.
	\end{rem}
	
	\begin{lem} \label{lem:g}
	The functions $g_j(\varphi)$, $j \ge 1$, are positive and increasing on $(0,\pi)$. Moreover,
	\begin{align} \label{sp22}
		\begin{split}
		g_j(\varphi) 
		&\leq g_j\Big(\frac{\pi}2\Big)
		=
			\begin{cases}
				\big[ (j-2)!!  \big]^2 &\text{ if $j \ge 2$ is even} \\
			{\pi} \big[ (j-2)!!  \big]^2/2 &\text{ if $j \ge 2$ is odd} \\
			{\pi}/{2} &\text{ if } j = 1
		\end{cases},
		\qquad \varphi\in(0,\pi/2].
	\end{split}
	\end{align} 
\end{lem}

\begin{proof}
	From Lemma~\ref{lem:h} we see that $g_{j}$ is positive. Further, using \eqref{sp21} we get
	\begin{align*}
		g'_{j} (\varphi) =
		\frac{h'_{j} (\varphi) (\sin \varphi)^{2j-1} - (2j-1) h_{j} (\varphi) (\sin \varphi)^{2j-2} \cos \varphi }{(\sin \varphi)^{4j-2}}
		= \frac{h_{j+1} (\varphi) }{(\sin \varphi)^{2j}}.
	\end{align*}
	This, in view of Lemma~\ref{lem:h}, implies that $g_{j}$ is increasing.
	
	It remains to show \eqref{sp22}. Observe that $g_j(\varphi) \le g_j(\pi/2) = h_j(\pi/2)$.
	Using \eqref{sp21} and \eqref{sp20} we get the recurrence
	\begin{align*}
		h_{j+1}(\pi/2) = h_j'(\pi/2) = (j-1)^2 h_{j-1} (\pi/2), \qquad j \ge 2.
	\end{align*}
	Together with $h_1(\pi/2)=\pi/2$ and $h_2(\pi/2)=1$, this leads to the formulas
	$$
		h_{2j}(\pi/2) = \big[ (2j-2)!! \big]^2, \qquad h_{2j+1}(\pi/2) = \frac{\pi}{2} \big[ (2j-1)!! \big]^2, \qquad j \ge 1.
	$$
    
	The proof is complete.
	\end{proof}

Now we can estimate $w_j(\varphi)$, though the bound obtained is not optimal.\,\footnote{
\,We believe that the stronger estimate $w_j(\varphi) \le (\pi/2)^{j-1}$, $\varphi \in (0,\pi/2]$, is true, but proving
this seems to be much more involved. In this connection, the relations from Remark \ref{rem:hgw} are potentially useful.
}
	\begin{cor} \label{cor:w}
		For each $j \ge 1$,
		$$
			w_{j} (\varphi) \le \pi^{j-1}, \qquad \varphi\in(0,\pi/2].
		$$
	\end{cor}
	
	\begin{proof}
		Using Lemma~\ref{lem:g} we infer that $w_{1}(\varphi) = 1$ and for $j \ge 2$ and $\varphi \in (0,\pi/2]$
		$$
			w_j(\varphi) = \frac{g_j(\varphi) (\pi-\varphi)^{j-1} (\sin\varphi)^{j}}{j! \varphi^j }
			\le \frac{ \big[ (j-2)!!  \big]^2 \pi^{j} }{2j!} \le \frac{(j-1)! \pi^{j} }{2 j!} \le \frac{\pi^{j}}{4} \le \pi^{j-1},
		$$
		as desired.
	\end{proof}

	The next result is essentially known. It follows directly by combining an exact expression for partial Bell polynomials
	$\mathbf{B}_{N,n}$ (see \cite[Thm.\,A, p.\,134]{Com}) with known particular values $\mathbf{B}_{N,n}(1!,2!,3!,\ldots)$,
	see \cite[Form.\,{[3h]}, p.\,135]{Com}.\,\footnote{
\,The numbers appearing as coefficients on the right-hand side of the identity in Lemma \ref{lem:Lah}
	are the (unsigned) Lah numbers, cf.\ \cite[p.\,156]{Com}.
}
	\begin{lem}\label{lem:Lah}
		Let $N \ge 1$ and $s>0$. Then
		$$
			\sum_{k_1+\ldots +Nk_N=N} \frac{N!}{k_1!\cdot\ldots\cdot k_N!} s^{k_1+\ldots+k_N}
				=\sum_{n=1}^N \binom{N}{n} \frac{(N-1)!}{(n-1)!} s^n,
		$$
		where the first summation is over all $k_1,\ldots,k_N \ge 0$ satisfying the indicated condition.
	\end{lem}

After the above preparations we are in a position to prove the following.
\begin{lem} \label{lem:W}
Let $N \ge 1$. Then for each $\delta > 0$
$$
\big|(-D)^N W_t(\varphi)\big| \le e^{\delta} \Big( \frac{\varphi}{2\sin\varphi}\Big)^N t^{-N} W_t(\varphi),
	\qquad \varphi \in [0,\pi/2], \quad 0 < t \le \frac{\delta}{4 N^2}.
$$
\end{lem}

\begin{proof}
Throughout the proof we assume that $N,\delta,t,\varphi$ are as in the statement.
By Fa\`a di Bruno's formula, cf.\ \cite[Chap.\,III, Sec.\,3.4]{Com},\,\footnote{
\,The classical Fa\`a di Bruno's formula pertains to successive derivatives of a composition of two functions,
	see e.g.\ \cite{Jo,Cr} for details and historical background. The formula has a combinatorial nature and
	can be put into an abstract framework as in \cite[Sec.\,3.4]{Com}. Here we use a variant of the classical formula
	with the standard differentiation replaced by $-D$, and the relevant fact that makes this replacement legitimate is
	that $-D$ obeys the Leibniz rule.
}
	\begin{equation*}
	(-D)^N W_t(\varphi)= W_t(\varphi)  \sum_{k_1+\ldots +Nk_N=N} \frac{N!}{k_1!\cdot\ldots\cdot k_N!} (2t)^{-(k_1+\ldots+k_N)}
	\prod_{j=1}^N \bigg(\frac{(-1)^{j+1}g_j(\varphi)}{j!}\bigg)^{k_j}.
	\end{equation*}
	Then making use of the functions $h_j$ and $w_j$ we see that
	\begin{align*}
	(-D)^N W_t(\varphi) & =\Big(\frac{\varphi}{2t\sin\varphi}\Big)^N W_t(\varphi) \\
	& \qquad \times \sum_{k_1+\ldots +Nk_N=N} \frac{N!}{k_1 !\cdot\ldots\cdot k_N!}
	\Big(\frac{2t}{\pi-\varphi}\Big)^{N-(k_1+\ldots+k_N)} \prod_{j=1}^N \big[(-1)^{j+1}w_j(\varphi)\big]^{k_j}.
	\end{align*}
	This can be written as
	\begin{equation*}
	(-D)^N W_t(\varphi)=\Big(\frac{\varphi}{2t\sin\varphi}\Big)^N W_t(\varphi)\big[1+E_{t,N}(\varphi)\big],
	\end{equation*}
	where
	\begin{equation*}
	E_{t,N}(\varphi) := \sum_{\substack{k_1+\ldots +Nk_N=N\\ k_1\neq N}} \frac{N!}{k_1!\cdot\ldots\cdot k_N!}
	\Big(\frac{2t}{\pi-\varphi}\Big)^{N-(k_1+\ldots+k_N)} \prod_{j=1}^N \big[(-1)^{j+1} w_j(\varphi)\big]^{k_j}.
	\end{equation*}
	It is enough to prove that $|E_{t,N}(\varphi)|\leq e^{\delta}-1$ and to this end we may assume that $N \ge 2$.
	
	By Corollary~\ref{cor:w} we have
	\begin{align*}
	|E_{t,N}(\varphi)|&\leq \sum_{\substack{k_1+\ldots +Nk_N=N\\ k_1\neq N}} \frac{N!}{k_1!\cdot\ldots\cdot k_N!}
	\Big(\frac{2t}{\pi/2}\Big)^{N-(k_1+\ldots+k_N)} \prod_{j=1}^N \pi^{jk_j-k_j}\\
	&=
	\sum_{\substack{k_1+\ldots +Nk_N=N\\ k_1\neq N}} \frac{N!}{k_1!\cdot\ldots\cdot k_N!} (4t)^{N-(k_1+\ldots+k_N)}\\
	&\leq \sum_{\substack{k_1+\ldots +Nk_N=N\\ k_1\neq N}} \frac{N!}{k_1!\cdot\ldots\cdot k_N!}
		\Big(\frac{\delta}{N^2}\Big)^{N-(k_1+\ldots+k_N)}.
	\end{align*}
	Applying now Lemma \ref{lem:Lah} with $s=N^2/\delta$ we obtain
	$$
	|E_{t,N}(\varphi)|\leq \sum_{n=1}^{N-1} \binom{N}{n} \frac{(N-1)!}{(n-1)!} \Big(\frac{\delta}{N^2}\Big)^{N-n}.
	$$
	Reversing the order of summation we arrive at
	$$
	|E_{t,N}(\varphi)|\leq \sum_{n=1}^{N-1} \binom{N}{n} \frac{(N-1)!}{(N-n-1)!} \Big(\frac{\delta}{N^2}\Big)^{n}
		=\sum_{n=1}^{N-1}  \frac{N!}{N^n (N-n)!} \frac{(N-1)!}{N^n (N-n-1)!} \frac{\delta^n}{n!}.
	$$
	Notice that the first two factors under the last sum are bounded above by $1$. Thus
	$$
	|E_{t,N}(\varphi)| \leq \sum_{n=1}^{N-1} \frac{\delta^n}{n!}\leq e^{\delta}-1,
	$$
	and this finishes the proof.
\end{proof}

\begin{cor} \label{cor:WW}
Let $N \ge 1$. Then for each $\delta > 0$
$$
\big|(-D)^N W_t(\varphi)\big| \le e^{\delta} \Big(\frac{\pi}4\Big)^N \, t^{-N}W_t(\varphi), \qquad \varphi \in [0,\pi/2],
	\quad 0 < t \le \frac{\delta}{4N^2}.
$$
\end{cor}

Next, we estimate the error term in \eqref{sp14}.
\begin{lem} \label{lem:Werr}
Let $N \ge 1$. Then for each $0 < \delta \le N$
$$
\bigg| (-D)^N \sum_{0 \neq n \in \mathbb{Z}} W_t(\varphi + 2\pi n) \bigg| \le \frac{1}{1000} e^{\delta} \Big(\frac{\pi}4\Big)^N
	t^{-N} W_t(\varphi), \qquad \varphi \in [0,\pi/2], \quad 0 < t \le \frac{\delta}{4N^2}.
$$
\end{lem}

\begin{proof}
The reasoning is parallel to the proof of Lemma \ref{lem:eNth}, the main difference is that now we use Corollary \ref{cor:WW}
instead of Lemma~\ref{lem:mNth}. Adopting the notation from the proof of Lemma~\ref{lem:eNth}, we need to bound $\mathtt{S}$,
and we split this task to estimating separately $\mathtt{S}_0$, $I_1$ and $I_2$ (these expressions are now the same as before).
Throughout the proof we assume that $\varphi \in [0,\pi/2]$, $\delta \le N$ and $t \le \delta/(4N^2)$.
Then $t$ under consideration is not bigger than $1/(4N)$. Further, we specify $\epsilon = 128/(27\pi^3)$ so that
$1/(4N) < 27\pi^3\epsilon/[512(N-1/2)]$.

Since the constraint on $t$ implies $t \le 1/4 < \pi/10$, we have (see \eqref{sp23})
$$
\mathtt{S}_0  \le 2 \big|D^N W_t(\varphi)\big| \sum_{n=1}^{\infty} e^{-\pi^2 n(n-1/2)/t}
	\le \big|D^N W_t(\varphi)\big| \sum_{0 \neq n \in\mathbb{Z}} e^{-5\pi n^2}.
$$
Evaluating the last series\,\footnote{
\,We could use more precise bound $1/4 < \pi/12$ and evaluate the series $\sum_{0 \neq n \in \mathbb{Z}}e^{-6\pi n^2}$ with the
	aid of \cite[Thm.\,5.5(vi)]{Yi1}, but the formula would be much more complicated.
}
by means of \cite[Thm.\,5.5(v)]{Yi1} we get
$$
\sum_{0 \neq n \in\mathbb{Z}} e^{-5\pi n^2} = \frac{\pi^{1/4} \sqrt{5+2\sqrt{5}}}{\Gamma(3/4) 5^{3/4}}-1 \approx 3.014035\cdot 10^{-7}.
$$
This, together with Corollary \ref{cor:WW}, gives
$$
\mathtt{S}_0 \le \frac{1}3 10^{-6} e^{\delta} \Big(\frac{\pi}4\Big)^N t^{-N} W_t(\varphi), \qquad t \le \frac{\delta}{4N^2}.
$$

Considering $I_1$, combining \eqref{sp24} with \eqref{sp25} and Corollary \ref{cor:WW} produces
\begin{align*}
I_1 & \le \frac{\sqrt{2}\,e^{1/12}}{1+\epsilon} e^{-\pi^2/(4t)} e^{\delta} \Big(\frac{\pi}4\Big)^N t^{-N} W_t(\varphi)
	\sum_{k=1}^N \binom{N}{k} \Big[ \frac{64}{3\pi}(1+\epsilon)\Big]^k \\
& \le \bigg\{ \frac{\sqrt{2}\,e^{1/12}}{1+\epsilon} \bigg( \Big[\frac{64}{3\pi}(1+\epsilon)+1\Big]^N-1\bigg) e^{-\pi^2 N} \bigg\}
	e^{\delta} \Big(\frac{\pi}4\Big)^N t^{-N} W_t(\varphi).
\end{align*}
Recall that $\epsilon = 128/(27\pi^3)$. The expression in curly brackets above is maximal over $N=1,2,\ldots$ for $N=1$, and the maximal
value is $\approx 0.000540$. Hence
$$
I_1 \le 6\cdot 10^{-4} e^{\delta} \Big(\frac{\pi}4\Big)^N t^{-N} W_t(\varphi), \qquad t \le \frac{\delta}{4N^2}.
$$

Estimation of $I_2$ is quite similar, we get
\begin{align*}
I_2 & \le \frac{2}{\sqrt{\pi}(1+\epsilon)}\sqrt{t} e^{-2\pi^2/t} e^{\delta} \Big(\frac{\pi}4\Big)^N t^{-N} W_t(\varphi)
	\sum_{k=1}^N \binom{N}{k} \Big[ \frac{64}{3\pi}(1+\epsilon)\Big]^k \\
& \le \bigg\{ \frac{2}{\sqrt{\pi}(1+\epsilon)}\sqrt{\frac{1}{4N}}
	\bigg( \Big[\frac{64}{3\pi}(1+\epsilon)+1\Big]^N-1\bigg) e^{-8 \pi^2 N} \bigg\}
	e^{\delta} \Big(\frac{\pi}4\Big)^N t^{-N} W_t(\varphi).
\end{align*}
Again the expression in curly brackets takes its maximal value for $N=1$, which is $\approx 1.962529\cdot 10^{-34}$. Consequently,
$$
I_2 \le 2\cdot10^{-34} e^{\delta} \Big(\frac{\pi}4\Big)^N t^{-N} W_t(\varphi), \qquad t \le \frac{\delta}{4N^2}.
$$

The lemma follows directly from the above bounds for $\mathtt{S}_0$, $I_1$ and $I_2$.
\end{proof}

Summarizing the main results of Section \ref{ssec:upN2}, Corollary \ref{cor:WW} and Lemma \ref{lem:Werr} imply the following.
\begin{thm} \label{thm:opN2}
Let $N \ge 1$. Then for each $0 < \delta \le N$
$$
(-D)^N \vartheta_t(\varphi) < \Big(1 + \frac{1}{1000}\Big) e^{\delta} \Big(\frac{\pi}4\Big)^N t^{-N} W_t(\varphi),
	\qquad \varphi \in [0,\pi/2], \quad 0 < t \le \frac{\delta}{4N^2}.
$$
\end{thm}
Note that $e\pi/4 \approx 2.134934$ (this constant arises as exponentiation base with the choice $\delta = N$).

\subsection{Upper bounds for $\vartheta_t(\varphi)$ and $-D\vartheta_t(\varphi)$} \label{ssec:estN01}	\,

\medskip

We will use simple direct analysis of the series defining $\vartheta_t(\varphi)$ to establish upper bounds for
$(-D)^N \vartheta_t(\varphi)$ when $N=0,1$.
For $N=1$ this will be a refinement of the previously obtained estimates.
\begin{prop} \label{prop:estN0}
One has
$$
\vartheta_t(\varphi) < \Big[ 1 + 2e^{-\pi^2/(2t)}\Big] W_t(\varphi), \qquad \varphi \in [0,\pi/2], \quad 0 < t \le \pi^2/2.
$$
\end{prop}

\begin{proof}
Observe that 
$$
\vartheta_t(\varphi) = W_t(\varphi) \bigg( 1 + \sum_{0 \neq n \in \mathbb{Z}} e^{-\pi n (\varphi + \pi n)/t} \bigg).
$$
Denote the above series by $S_t(\varphi)$. Assuming that $\varphi \in [0,\pi/2]$, we can write
\begin{align*}
S_t(\varphi) & = e^{-\pi(\pi-\varphi)/t} + \sum_{n \in \mathbb{Z},\, n\neq 0,-1} e^{-\pi n(\varphi + \pi n)/t}
 \le e^{-\pi^2/(2t)} + \sum_{n \ge 1} e^{-\pi^2 n^2/t} + \sum_{n \le -2} e^{-\pi^2 n (2n+1)/(2t)} \\
& < e^{-\pi^2/(2t)} + 2 \sum_{n \ge 1} e^{-\pi^2 n^2/t} < e^{-\pi^2/(2t)} + 2 \sum_{n \ge 1} e^{-\pi^2 n/t} 
 = e^{-\pi^2/(2t)} + \frac{2e^{-\pi^2/t}}{1-e^{-\pi^2/t}}.
\end{align*}
Thus
$$
\vartheta_t(\varphi) < W_t(\varphi) \bigg[ 1 + e^{-\pi^2/(2t)} + \frac{2e^{-\pi^2/t}}{1-e^{-\pi^2/t}} \bigg],
	\qquad \varphi \in [0,\pi/2], \quad t > 0.
$$
It is elementary to check that the second term in square brackets above dominates the third one when $t \le \pi^2/2$.
The conclusion follows.
\end{proof}
Note that the expression in square brackets in Proposition \ref{prop:estN0} is increasing in $t$.
Its value at $t=1$ is $\approx 1.014384$ and at $t=\pi/2$ the value is $\approx 1.086428$.

\begin{prop} \label{prop:estN1}
One has
$$
-D\vartheta_t(\varphi) < \frac{\pi}4 t^{-1} W_t(\varphi), \qquad \varphi \in [0,\pi/2], \quad 0 < t \le \pi^2/2.
$$
\end{prop}

\begin{proof}
By a direct computation we see that
$$
-D \vartheta_t(\varphi) = t^{-1} W_t(\varphi) \frac{\varphi}{2\sin\varphi} \bigg( 1 + \sum_{0 \neq n \in \mathbb{Z}}
	\frac{1}{\varphi} (\varphi + 2\pi n) e^{-\pi n (\varphi + \pi n)/t} \bigg).
$$
We need to estimate the above series, which we denote by $\mathcal{S}_t(\varphi)$.
To control the relevant cancellations we pair terms corresponding to $\pm n$. This leads to
$$
\mathcal{S}_t(\varphi) = \sum_{n \ge 1} e^{-\pi^2 n^2/t} \Big[ 2 \cosh\frac{\pi n \varphi}t - \frac{4\pi n}{\varphi}
	\sinh\frac{\pi n \varphi}t \Big],
$$
where the value at $\varphi=0$ of each term is understood in the limiting sense.

Denote the expression in square brackets above by $f_{n,t}(\varphi)$. We claim that $\varphi \mapsto f_{n,t}(\varphi)$
is decreasing on $[0,\pi/2]$ for each fixed $n \ge 1$ and $0 < t \le \pi^2 n^2/2$. To prove the claim, we compute
$$
f'_{n,t}(\varphi) = \frac{2\pi n}{\varphi t} \Big[ \Big( \varphi + \frac{2t}{\varphi}\Big) \sinh\frac{\pi n \varphi}{t}
	- 2\pi n \cosh\frac{\pi n \varphi}{t} \Big]
$$
and observe that it is enough to show
\begin{equation} \label{sp26}
\frac{\varphi}{2\pi n} \le \coth\frac{\pi n \varphi}t - \frac{t}{\pi n \varphi}, \qquad \varphi \in [0,\pi/2], \quad
	t \le \frac{\pi^2 n^2}2, \quad n \ge 1.
\end{equation}
Let $x = \pi n \varphi/t$. It is straightforward to check that
$$
\frac{\varphi}{2\pi n} \le \frac{1 \wedge x}4, \qquad \varphi \in [0,\pi/2], \quad t \le \frac{\pi^2 n^2}2, \quad n \ge 1.
$$
Then \eqref{sp26} follows from the bound $(1\wedge x)/4 \le \coth x - 1/x$, $x > 0$, which is elementary to verify.
The claim follows.

From the claim we infer that
$$
f_{n,t}(\varphi) \le f_{n,t}(0) = 2 - \frac{4\pi^2 n^2}t < 0, \qquad \varphi \in [0,\pi/2], \quad
	t \le \frac{\pi^2 n^2}2, \quad n \ge 1.
$$
This leads to
$$
\mathcal{S}_t(\varphi) \le \sum_{n \ge 1} e^{-\pi^2 n^2/t} \bigg[ 2 - \frac{4\pi^2 n^2}t \bigg] < 0, \qquad
	\varphi \in [0,\pi/2], \quad t \le \frac{\pi^2}2.
$$
Consequently,
$$
- D\vartheta_t(\varphi) < t^{-1} W_t(\varphi) \frac{\varphi}{2\sin\varphi} \le \frac{\pi}4 t^{-1} W_t(\varphi),
	\qquad \varphi \in [0,\pi/2], \quad t \le \frac{\pi^2}2,
$$
as stated in the lemma.
\end{proof}

\subsection{Summary of bounds for $(-D)^N \vartheta_t(\varphi)$} \label{ssec:sumary}	\,

\medskip

The result below is a compilation of the lower bound from Proposition \ref{prop:thlow} restricted to $\varphi \in [0,\pi/2]$
and Theorem \ref{thm:Nth}, Theorem \ref{thm:opN2} and Propositions~\ref{prop:estN0} and \ref{prop:estN1}.
\begin{thm} \label{thm:theta}
The following bounds of $(-D)^N\vartheta_t(\varphi)$ hold for $\varphi \in [0,\pi/2]$.
\begin{itemize}
\item[(a)]
For $N \ge 0$,
$$
(-D)^N\vartheta_t(\varphi) \ge \frac{1}{2^N} e^{-tN^2} t^{-N} W_t(\varphi), \qquad t > 0.
$$
\item[(b)]
For $N \ge 1$ and $\epsilon \in (0,1/3]$,
$$
(-D)^N\vartheta_t(\varphi) < \Big(1+\frac{21}{500}\Big) \frac{1}{1+\epsilon} \Big[ \frac{\pi}4 (1+\epsilon)\Big]^N t^{-N} W_t(\varphi),
	\qquad 0 < t \le \frac{27\pi^3 \epsilon}{512(N-1/2)}.
$$
\item[(c)]
For $N \ge 1$ and $\delta \in (0,N]$,
$$
(-D)^N\vartheta_t(\varphi) < \Big(1+\frac{1}{1000}\Big) e^{\delta} \Big(\frac{\pi}4\Big)^N t^{-N} W_t(\varphi),
	\qquad 0 < t \le \frac{\delta}{4N^2}.
$$
\item[(d)]
For $N=0$ and $N=1$,
$$
\vartheta_t(\varphi) < \Big[ 1+2e^{-\pi^2/(2t)}\Big] W_t(\varphi), \qquad
	-D\vartheta_t(\varphi) < \frac{\pi}4 t^{-1}W_t(\varphi), \qquad 0 < t \le \frac{\pi^2}2.
$$
\end{itemize}
\end{thm}

Items (b)--(d) of Theorem \ref{thm:theta}, with the choices of specific $\epsilon = 256/(27\pi^3) \approx 0.305792$
and $\delta = 1$ in (b) and (c), respectively, imply the following.
\begin{cor} \label{cor:theta}
Let $N \ge 0$. Then
$$
(-D)^N \vartheta_t(\varphi) < t^{-N} W_t(\varphi) \times
	\begin{cases}
		\mathfrak{w}_0\, \mathfrak{w}_1^N, & \;\; \textrm{if}\;\; 0 < t \le \frac{1}{2N+1} \;\; \textrm{and}\;\; N \ge1\\
		\mathfrak{w}'_0 \,({\pi}/4)^N, & \;\; \textrm{if} \;\; 0 < t \le \frac{1}{(2N+1)^2}
	\end{cases},
	\qquad \varphi \in [0,\pi/2],
$$
where $\mathfrak{w}_0$, $\mathfrak{w}_1$ and $\mathfrak{w}'_0$ are numerical constants given by\,\footnote{
\,Note that $\mathfrak{w}_0 \approx 0.797983$, $\mathfrak{w}'_0 \approx 2.721$, $\mathfrak{w}_1 \approx 1.025567$
and $\pi/4 \approx 0.785398$.
}
$$
\mathfrak{w}_0 := \Big( 1+\frac{21}{500} \Big) \frac{1}{1+256/(27\pi^3)}, \qquad
\mathfrak{w}'_0 := \Big(1+\frac{1}{1000}\Big) e, \qquad
\mathfrak{w}_1 := \frac{\pi}4 \Big( 1+ \frac{256}{27\pi^3}\Big).
$$
\end{cor}

Combining Theorem \ref{thm:theta}(a) and Corollary \ref{cor:theta} with \eqref{sp2} we conclude bounds for $G^{\lambda}_t(x,1)$
when $\lambda \in \mathbb{N}-1/2$ and $x \in [0,1]$.
\begin{cor} \label{cor:Gsph}
Let $\lambda \in \mathbb{N}-1/2$. The following bounds of $G_t^{\lambda}(\cos\varphi,1)$ hold for $\varphi \in [0,\pi/2]$.
\begin{itemize}
\item[(a)]
$$
G_t^{\lambda}(\cos\varphi,1) \ge \frac{1}{4^{\lambda+1/2}\Gamma(\lambda+1)} \frac{1}{t^{\lambda+1}}
	e^{-\varphi^2/(4t)}, \qquad  t > 0.
$$
\item[(b)]
\begin{align*}
G_t^{\lambda}(\cos\varphi,1) & < \frac{\mathfrak{w}_0 (\mathfrak{w}_1/2)^{\lambda+1/2}}{\Gamma(\lambda+1)} e^{t(\lambda+1/2)^2}
	\frac{1}{t^{\lambda+1}} e^{-\varphi^2/(4t)} \\
& \le  \frac{\mathfrak{w}_0 (\sqrt{e}\, \mathfrak{w}_1/2)^{\lambda+1/2}}{\Gamma(\lambda+1)}
	\frac{1}{t^{\lambda+1}} e^{-\varphi^2/(4t)}, \qquad 0 < t \le \frac{1}{2\lambda+2}, \quad \lambda \neq -1/2.
\end{align*}
\item[(c)]
\begin{align*}
G_t^{\lambda}(\cos\varphi,1) & < \frac{\mathfrak{w}'_0 (\pi/8)^{\lambda+1/2}}{\Gamma(\lambda+1)} e^{t(\lambda+1/2)^2}
	\frac{1}{t^{\lambda+1}} e^{-\varphi^2/(4t)} \\
& \le  \frac{\sqrt[4]{e}\, \mathfrak{w}'_0 (\pi/8)^{\lambda+1/2}}{\Gamma(\lambda+1)}
	\frac{1}{t^{\lambda+1}} e^{-\varphi^2/(4t)}, \qquad 0 < t \le \frac{1}{(2\lambda+2)^2}.
\end{align*}
\end{itemize}
\end{cor}
The discrepancy ratios in Corollary \ref{cor:Gsph} between the upper and the lower bounds are equal
$\mathfrak{w}_0 (2\sqrt{e}\, \mathfrak{w}_1)^{\lambda+1/2}$ in case $t \le 1/(2\lambda+2)$ and
$\sqrt[4]{e} \, \mathfrak{w}'_0 (\pi/2)^{\lambda+1/2}$ in case $t \le 1/(2\lambda+2)^2$.\,\footnote{
\,Note that $2\sqrt{e}\, \mathfrak{w}_1 \approx 3.381748$, $\sqrt[4]{e}\, \mathfrak{w}'_0 \approx 3.493833$
and $\pi/2 \approx 1.570796$.
}

Finally, we state more precise upper bounds for the special cases $\lambda = \pm 1/2$, see Theorem~\ref{thm:theta}(d).
The lower bounds in Corollary \ref{cor:Gsphspec} are just specifications to $\lambda = \pm 1/2$ of the bound from
Corollary~\ref{cor:Gsph}(a).
\begin{cor} \label{cor:Gsphspec}
The following bounds of $G_t^{\pm 1/2}(\cos\varphi,1)$ hold for $\varphi \in [0,\pi/2]$:
\begin{align*}
\frac{1}{\sqrt{\pi}} \frac{1}{t^{1/2}} e^{-\varphi^2/(4t)} & \le G_t^{-1/2}(\cos\varphi,1) < \frac{1}{\sqrt{\pi}}
	\Big[ 1 + 2 e^{-\pi^2/(2t)}\Big] \frac{1}{t^{1/2}} e^{-\varphi^2/(4t)}, \qquad 0< t \le \frac{\pi^2}2, \\
\frac{1}{\sqrt{4\pi}} \frac{1}{t^{3/2}} e^{-\varphi^2/(4t)} & \le G_t^{1/2}(\cos\varphi,1) < \frac{\sqrt{\pi}}4 e^t
	\frac{1}{t^{3/2}} e^{-\varphi^2/(4t)}, \qquad 0 < t \le \frac{\pi^2}2.
\end{align*}
Moreover, the lower bounds of $G_t^{\pm 1/2}(\cos\varphi,1)$ hold also for $t > \pi^2/2$.
\end{cor}

\section{The unified proof} \label{sec:uproof}

In what follows, in each of the Steps A--F we express the multiplicative constants in terms of the multiplicative
constants from the previous `father' step (the father for Step A is Section \ref{sec:oddsph}),
so that later all the dependence can be traced back to the beginning. Then also each step can be analyzed independently or, perhaps,
in the future independently refined. In this sense Steps A--F together with Section \ref{sec:oddsph} are independent modules
constituting the unified proof.

\subsection{Step A} \,

\medskip

\noindent \underline{Estimates of $G_t^{\lambda}(x,1)$ for $x \in [0,1]$, $t \le T_A$ and $\lambda \in \mathbb{N}-1/2$.}

\medskip

This step is based on the results of Section \ref{sec:oddsph}.
\begin{biglem} \label{lem:A}
Let $\lambda \in \mathbb{N} - 1/2$ and $T_A > 0$. Then
$$
G_t^{\lambda}(\cos\theta,1) \simeq \left\{ C^A_{\lambda,T_A} \atop c^A_{\lambda,T_A} \right\} \,
	\frac{1}{t^{\lambda+1}} \exp\bigg( -\frac{\theta^2}{4t}\bigg)
$$
for $0 \le \theta \le \pi/2$ and $0 < t \le T_A$, with the constant
$$
 c^A_{\lambda,T_A} \equiv c_{\lambda,\infty}^A 
 := \frac{1}{4^{\lambda+1/2}\Gamma(\lambda+1)}
$$
and some upper multiplicative constant $C_{\lambda,T_A}^A$ which can be specified according to choices of $\lambda$ and $T_A$ 
as follows:
$$
C^A_{\lambda,T_A} 
:= 
\frac{1}{\Gamma(\lambda+1)} \times
	\begin{cases}
		\mathfrak{w}_0(\sqrt{e}\, \mathfrak{w}_1/2)^{\lambda+1/2} &
			\;\; \textrm{if} \;\; \lambda \neq  \pm 1/2 \;\; \textrm{and} \;\; T_A = 1/(2\lambda+2) \\
		\sqrt[4]{e} \, \mathfrak{w}'_0 (\pi/8)^{\lambda+1/2} &
			\;\;\textrm{if} \;\; \lambda \neq \pm 1/2 \;\; \textrm{and}  \;\; T_A = 1/(2\lambda+2)^2 \\
		1 + 2 \exp(-\pi^2/(2T_A)) & \;\; \textrm{if} \;\; \lambda = -1/2 \;\; \textrm{and} \;\; T_A \le \pi^2/2 \\
		\pi \exp(T_A)/8 & \;\; \textrm{if} \;\; \lambda = 1/2 \;\; \textrm{and} \;\; T_A \le \pi^2/2
	\end{cases}.
$$
The first line in the above formula gives correct $C^A_{\lambda,T_A}$ also in case $\lambda = 1/2$, whereas
the second line gives correct $C^A_{\lambda,T_A}$ in both cases $\lambda = \pm 1/2$,
but the last two lines provide stronger results for these~$\lambda$.
\end{biglem}

\begin{proof}
The lower bound is just Corollary \ref{cor:Gsph}(a). The upper bounds with specified multiplicative constants follow
from items (b) and (c) in Corollary \ref{cor:Gsph} and Corollary \ref{cor:Gsphspec}.
The upper bound with unspecified multiplicative constant and no constraints on $T_A$ holds\,\footnote{\,A posteriori
the upper bound with no constraints on $T_A$ also follows from the unified proof with constrained $T_A$, see
the discussion in the beginning of Section \ref{sec:lmtime}.
}
because of Theorem \ref{thm:jhk}.
\end{proof}

\begin{rem} \label{rem:A}
To proceed with the remaining steps of the unified proof, in Step A it would be sufficient to estimate the even part
of the kernel, $[G_t^{\lambda}(x,1)]_{\textrm{even}} = [G_t^{\lambda}(x,1)+G_t^{\lambda}(-x,1)]/2$, which could potentially have
better upper bounds, but worse lower bounds.
Note that $G_t^{\lambda}(x,1)/2 \le [G_t^{\lambda}(x,1)]_{\textrm{even}} \le G_t^{\lambda}(x,1)$ for $x \in [0,1]$,
the second inequality since $x \mapsto G_t^{\lambda}(x,1)$ is increasing on $[-1,1]$, see Lemma \ref{lem:diff}.
\end{rem}

\subsection{Step B} \,

\medskip

\noindent \underline{Estimates of $G_t^{\a,\b}(x,1)$
for $x \in [0,1]$, $t \le T_B$ and $\a,\b \ge -1/2$ such that $\a+\b \in \mathbb{N}$.}

\medskip

This step is based on Step A. The relevant technical tools are Lemma \ref{lem:red_ext}(i) and Lemma \ref{lem:intF0}.
\begin{biglem} \label{lem:B}
Let $\a,\b \ge -1/2$ be such that $\a+\b \in \mathbb{N}$ and let $T_B > 0$. Then
$$
G_t^{\a,\b}(\cos\theta,1) \simeq \left\{ C^B_{\a,\b,T_B} \atop c^B_{\a,\b,T_B} \right\} \,
	\frac{1}{t^{\a+1}} \exp\bigg( -\frac{\theta^2}{4t}\bigg)
$$
for $0 \le \theta \le \pi/2$ and $0 < t \le T_B$, with the constants
\begin{align*}
c^B_{\a,\b,T_B} & := {c}^A_{\a+\b+1/2,T_B/4} \times
\frac{4^{\a+\b+3/2} b_{\b}}{\big[\mathbb{D}_{\b}T_B \vee \mathfrak{B}\pi^2 \big]^{\b+1/2}}
	\times \frac{h_0^{\a+\b+1/2}}{h_0^{\a,\b}}, \\
C^B_{\a,\b,T_B} & := {C}^A_{\a+\b+1/2,T_B/4} \times 
	\frac{4^{\a+\b+2} B_{\b}}{ (\mathfrak{b}\pi^2/2)^{\b+1/2}} \times \frac{h_0^{\a+\b+1/2}}{h_0^{\a,\b}}.
\end{align*}
\end{biglem}

When referring to Lemma \ref{lem:B} in the ultraspherical case $\lambda=\a=\b$, notation of the resulting constants
$c^B_{\lambda,\lambda,T_B}$, $C^B_{\lambda,\lambda,T_B}$ will be abbreviated to $c^B_{\lambda,T_B}$, $C^B_{\lambda,T_B}$, respectively.

\begin{proof}[{Proof of Lemma \ref{lem:B}}]
Fix $T_B > 0$. Along the proof we consider $\theta \in [0,\pi/2]$ and $0 < t \le T_B$.
By Lemma \ref{lem:red_ext}(i)
\begin{align*}
\frac{h_0^{\a,\b}}{h_0^{\a+\b+1/2}} G_t^{\a,\b}(\cos\theta,1) & = \int G_{t/4}^{\a+\b+1/2}\Big(v \cos\frac{\theta}2,1\Big)
	\, \dd\Pi_{\b}(v) \\
& \simeq \left\{ 2 \atop 1 \right\} \int_{[0,1]} {G}_{t/4}^{\a+\b+1/2} \Big(v \cos\frac{\theta}2,1\Big) \, \dd\Pi_{\b}(v),
\end{align*}
the last relation by monotonicity of the kernel in the first variable (see Lemma \ref{lem:diff}) and symmetry of the measure.

Observing now that $v \cos\frac{\theta}2 = \cos\sqrt{F_{\theta,0}(1,v)}$ with $F$ defined in \eqref{def:F},
and then using Lemma \ref{lem:A}, we get
$$
\frac{h_0^{\a,\b}}{h_0^{\a+\b+1/2}} G_t^{\a,\b}(\cos\theta,1) \simeq
\left\{ 2 {C}^A_{\a+\b+1/2,T_B/4} \atop {c}^A_{\a+\b+1/2,T_B/4} \right\}\, \Big(\frac{4}{t}\Big)^{\a+\b+3/2}
	\int_{[0,1]} \exp\bigg( - \frac{F_{\theta,0}(1,v)}t\bigg) \, \dd\Pi_{\b}(v).
$$
To estimate the last integral we use Lemma \ref{lem:intF0} obtaining
$$
\int_{[0,1]} \exp\bigg( - \frac{F_{\theta,0}(1,v)}t\bigg) \, \dd\Pi_{\b}(v) \simeq
	\left\{ B_{\beta} \Psi_{\beta}^{\mathfrak{b}}(t,\pi-\theta,\pi) \atop b_{\beta} \Psi_{\beta}^{\mathfrak{B}}(t,\pi-\theta,\pi) \right\}
		t^{\beta+1/2} \exp\bigg( -\frac{\theta^2}{4t}\bigg);
$$
recall, see \eqref{def:Psi}, that $\Psi_{\beta}^{\kappa}(t,\theta,\varphi) = [\mathbb{D}_{\b}t \vee \kappa\theta\varphi]^{-\b-1/2}$.
Since
\begin{align*}
\Psi_{\beta}^{\mathfrak{b}}(t,\pi-\theta,\pi) & \le \Psi_{\beta}^{\mathfrak{b}}(t,\pi/2,\pi)
= [\mathbb{D}_{\b}t \vee \mathfrak{b}\pi^2/2]^{-\b-1/2} \le (\mathfrak{b}\pi^2/2)^{-\b-1/2}, \\
\Psi_{\beta}^{\mathfrak{B}}(t,\pi-\theta,\pi) & \ge\Psi_{\beta}^{\mathfrak{B}}(t,\pi,\pi)
= [\mathbb{D}_{\b}t \vee \mathfrak{B}\pi^2]^{-\b-1/2} \ge [\mathbb{D}_{\b}T_B \vee \mathfrak{B}\pi^2]^{-\b-1/2},
\end{align*}
we arrive at the bounds
$$
\frac{h_0^{\a,\b}}{h_0^{\a+\b+1/2}} G_t^{\a,\b}(\cos\theta,1) \simeq 
\left\{ 2 C^A_{\a+\b+1/2,T_{B/4}} B_{\b} (\mathfrak{b}\pi^2/2)^{-\b-1/2} \atop
	c^A_{\a+\b+1/2,T_{B}/4} b_{\b} [\mathbb{D}_{\b}T_{B}\vee \mathfrak{B}\pi^2]^{-\b-1/2} \right\}
		\frac{4^{\a+\b+3/2}}{t^{\a+1}} \exp\bigg( -\frac{\theta^2}{4t}\bigg).
$$

The lemma follows.
\end{proof}

\begin{rem} \label{rem:B}
The proof of Lemma \ref{lem:B} shows that the result could be strengthened by keeping, instead of estimating, the factors
$\Psi_{\b}^{\mathfrak{b}}(t,\pi/2,\pi)$ and $\psi_{\b}^{\mathfrak{B}}(t,\pi,\pi)$. Then the bounds would have the form
$$
G_t^{\a,\b}(\cos\theta,1) \simeq \left\{ \widetilde{C}^{B}_{\a,\b,T_{B}} \Psi_{\b}^{\mathfrak{b}}(t,\pi/2,\pi) \atop
	\widetilde{c}^{B}_{\a,\b,T_{B}} \Psi_{\b}^{\mathfrak{B}}(t,\pi,\pi) \right\} \,
		\frac{1}{t^{\a+1}} \exp\bigg( -\frac{\theta^2}{4t}\bigg),
$$
with the constants
\begin{align*}
\widetilde{c}^B_{\a,\b,T_{B}} & := {c}^A_{\a+\b+1/2,T_{B}/4} \times
	{4^{\a+\b+3/2} b_{\b}} \times \frac{h_0^{\a+\b+1/2}}{h_0^{\a,\b}}, \\
\widetilde{C}^B_{\a,\b,T_{B}} & := {C}^A_{\a+\b+1/2,T_{B}/4} \times
	{4^{\a+\b+2} B_{\b}} \times \frac{h_0^{\a+\b+1/2}}{h_0^{\a,\b}}.
\end{align*}
\end{rem}

\subsection{Step C} \,

\medskip

\noindent \underline{Estimates of $G_t^{\lambda}(x,1)$ for $x \in [0,1]$, $t \le T_C$ and $\lambda \in (0,\infty)$.}

\medskip

This step is based on Step B and Lemma \ref{lem:comp}.
Note that $\lambda \in \mathbb{N}+1/2$ is already covered by Lemma \ref{lem:A} and $\lambda \in \mathbb{N}$ is covered by
Lemma \ref{lem:B}.

Denote
\begin{align*}
\omega_{\lambda}(t) & :=
2^{\lceil 2\lambda\rceil/2-\lambda} e^{(\lambda-\lceil 2\lambda \rceil/2)(\lambda+\lceil 2\lambda \rceil/2 +1)t}, \\
\Omega_{\lambda}(t) & :=
2^{\lceil 2\lambda\rceil/2-\lambda-1/2} e^{(\lambda-\lceil 2\lambda \rceil/2+1/2)(\lambda+\lceil 2\lambda \rceil/2 +1/2)t}.
\end{align*}
Observe that for $\lambda > 0$ the function $\omega_{\lambda}(t)$ is non-increasing in $t$, while $\Omega_{\lambda}(t)$ is increasing.
Moreover,
\begin{equation} \label{omega_inq}
e^{-(\lambda+3/4)t} \le \omega_{\lambda}(t), \qquad \Omega_{\lambda}(t) \le e^{(\lambda+1/2)t}, \qquad \lambda, t > 0,
\end{equation}
Notice that the relation $\omega_{\lambda}(t) \le \Omega_{\lambda}(t)$ is not true in general.

\begin{biglem} \label{lem:C}
Let $\lambda > 0$ and let $T_C > 0$. Then
$$
G_t^{\lambda}(\cos\theta,1) \simeq \left\{ C^C_{\lambda,T_C} \atop c^C_{\lambda,T_C} \right\} \,
	\frac{1}{t^{\lambda+1}} \exp\bigg( -\frac{\theta^2}{4t}\bigg)
$$
for $0 \le \theta \le \pi/2$ and $0 < t \le T_C$, with the constants
$$
c^C_{\lambda,T_C}  := c^B_{\lambda,\lceil 2\lambda\rceil - \lambda,T_C} \times \omega_{\lambda}(T_C), \qquad
C^C_{\lambda,T_C}  := C^B_{\lambda,\lceil 2\lambda\rceil - \lambda-1,T_C} \times \Omega_{\lambda}(T_C).
$$
\end{biglem}

\begin{proof}
To begin with, we observe that Lemma \ref{lem:comp} specified to $\epsilon = 0$ implies the bound
\begin{equation} \label{cp}
G_t^{\a,\b+\delta}(\cos\theta,1) \le 2^{-\delta/2} e^{\frac{\delta}2 (\a+\b+1+\frac{\delta}2)t} G_t^{\a,\b}(\cos\theta,1),
	\qquad \theta \in [0,\pi/2], \quad t > 0,
\end{equation}
where $\a,\b > -1$, $\delta \ge 0$ and $\b \ge -\delta/2$ (the last assumption only if $\delta > 0$).

Throughout the proof $\lambda > 0$ and we consider $N \in \mathbb{N}$ such that
$$
N < 2\lambda \le N+1,
$$
that is $N = \lceil 2\lambda \rceil -1$.

Take $\a=\b=\lambda$ and $\delta = N+1-2\lambda$.
Clearly, $\delta \ge 0$ and $\beta > -\delta/2$.
Since with the parameters chosen the assumptions for \eqref{cp} are satisfied, we get the bound
(notice that $\delta/2 = \lceil 2\lambda \rceil/2 - \lambda$)
$$
G_t^{\lambda,N+1-\lambda}(\cos\theta,1) \le \frac{1}{\omega_{\lambda}(t)}
		G_t^{\lambda,\lambda}(\cos\theta,1), \qquad \theta \in [0,\pi/2], \quad t > 0.
$$
Observing that $\lambda+(N+1-\lambda) = N+1 \in \mathbb{N}\setminus \{0\}$
and using the lower bound from Lemma \ref{lem:B} leads to the estimate
\begin{equation*}
G_t^{\lambda,\lambda}(\cos\theta,1) \ge
	\omega_{\lambda}(T_C) c^B_{\lambda,\lceil 2\lambda\rceil-\lambda,T_C} \,
		\frac{1}{t^{\lambda+1}} \exp\bigg(-\frac{\theta^2}{4t}\bigg)
\end{equation*}
that holds for $\theta \in [0,\pi/2]$ and $t \in (0,T_C]$.
This is the lower bound of the lemma.

Next, take $\a=\lambda$ and $\b = \lambda - \delta$, where $\delta = 2\lambda-N$.
Obviously, $\delta > 0$. Further, $\b \ge -1/2$ (this is equivalent to $\lambda \le N+1/2$, which is true since $2\lambda \le N+1$;
notice that $\b=-1/2$ only when $\lambda=1/2$, the case effectively covered by Lemma \ref{lem:A})
and $\b \ge -\delta/2$ as can be easily checked.
Thus we can use \eqref{cp} to write (notice that $\delta/2 = \lambda -\lceil 2\lambda\rceil/2 + 1/2$)
$$
G_t^{\lambda,\lambda}(\cos\theta,1) \le \Omega_{\lambda}(t)
	G_t^{\lambda,N-\lambda}(\cos\theta,1), \qquad \theta \in [0,\pi/2], \quad t > 0.
$$
Observing that $\lambda+ (N-\lambda) = N \in \mathbb{N}$ and making use of the upper bound in Lemma \ref{lem:B} leads to
\begin{equation*}
G_t^{\lambda,\lambda}(\cos\theta,1) \le \Omega_{\lambda}(T_C)
	C^B_{\lambda,\lceil 2\lambda\rceil-\lambda-1,T_C}\,
		\frac{1}{t^{\lambda+1}} \exp\bigg(-\frac{\theta^2}{4t}\bigg)
\end{equation*}
for $\theta \in [0,\pi/2]$ and $t \in (0,T_C]$.
This is the upper bound of the lemma.
\end{proof}

\begin{rem} \label{rem:C}
Lemma \ref{lem:C} could be stated in a simpler, though weaker, form with $\omega_{\lambda}(T_C)$ and $\Omega_{\lambda}(T_C)$
replaced by the bounds from \eqref{omega_inq}. On the other hand, the result could be strengthened by keeping factors
in $t$. Further strengthening is possible by combining the proof of Lemma \ref{lem:C} with Remark~\ref{rem:B}, which would result
in the bounds
$$
G_t^{\lambda}(\cos\theta,1) \simeq \left\{ \widetilde{C}^{B}_{\lambda,\lceil 2\lambda\rceil-\lambda-1,T_C}
	\Psi^{\mathfrak{b}}_{\lceil 2\lambda \rceil - \lambda -1}(t,\pi/2,\pi)\, \Omega_{\lambda}(t) \atop
	\widetilde{c}^{B}_{\lambda,\lceil 2\lambda\rceil-\lambda,T_C}
	\Psi^{\mathfrak{B}}_{\lceil 2\lambda \rceil - \lambda}(t,\pi,\pi)\, \omega_{\lambda}(t) \right\} \,
	\frac{1}{t^{\lambda+1}} \exp\bigg({-\frac{\theta^2}{4t}}\bigg)
$$
for $0 \le \theta \le \pi/2$ and $0 < t \le T_C$; here $\lambda > 0$.
\end{rem}

\subsection{Step D} \,

\medskip

\noindent \underline{Estimates of $G_t^{\lambda}(x,1)$ for $x \in [0,1]$, $t \le T_{D}$ and $\lambda \in (-1,0)$.}

\medskip

This step is based on Step C and Lemma \ref{lem:diff}, the heat equation \eqref{he_G}, and Lemma \ref{lem:Gmod2}.
See the proof of \cite[Lem.\,5.1]{NSS3}.
Note that in Lemma \ref{lem:D} below we could get better constants in the case $\lambda=-1/2$ by referring in the proof
to Lemma \ref{lem:A} instead of Lemma \ref{lem:C}, but we leave this improvement aside since Lemma \ref{lem:D} will
not be used for this particular value of $\lambda$.

\begin{biglem}  \label{lem:D}
Let $-1 < \lambda < 0$ and let $T_{D} > 0$. Then
$$
{G}_t^{\lambda}(\cos\theta,1) \simeq \left\{ C^{D}_{\lambda,T_{D}} \atop c^{D}_{\lambda,T_{D}} \right\} \,
	\frac{1}{t^{\lambda+1}} \exp\bigg( -\frac{\theta^2}{4t}\bigg)
$$
for $0 \le \theta \le \pi/2$ and $0 < t \le T_{D}$, with the constants
\begin{align*}
c^{D}_{\lambda,T_{D}} &  := \frac{8(\lambda+1)}{\pi} e^{-(2\lambda+2)T_D} \bigg[ c^C_{\lambda+1,T_D} \wedge \bigg(
	c^C_{\lambda+2,T_D} \times m_{\lambda+2} \times \frac{8(\lambda+2)}{\pi} e^{-(2\lambda+4)T_D}\bigg) \bigg], \\
C^{D}_{\lambda,T_{D}} &  := 4(\lambda+1) \bigg[ C^C_{\lambda+1,T_D} \vee \bigg(C^C_{\lambda+2,T_D} \times M_{\lambda+2}
	\times \frac{16(\lambda+2)}{\pi^2}\times \Big(\frac{16 T_D}{\pi^2} +1 \Big)^{\lambda+1}\bigg)\bigg].
\end{align*}
\end{biglem}

\begin{proof}
We first estimate $G_t^{\lambda}(0,1)$. Recalling that $G_t^{\lambda}(x,1)$ satisfies the heat equation \eqref{he_G} based on the
Jacobi-ultraspherical operator $J^{\lambda}$ and using Lemma \ref{lem:diff} twice we can write
$$
\frac{\dd}{\dd t} G_t^{\lambda}(0,1) = - J^{\lambda}_x G_t^{\lambda}(x,1)\big|_{x=0}
	= \frac{\dd^2}{\dd x^2} G_t^{\lambda}(x,1)\big|_{x=0} = 4(\lambda+1)(\lambda+2) e^{-t(4\lambda+6)} G_t^{\lambda+2}(0,1).
$$
Estimating the last expression with the aid of Lemma \ref{lem:C} we get
$$
\frac{\dd}{\dd t} G_t^{\lambda}(0,1) \simeq \left\{ C^C_{\lambda+2,T_D} \atop c^C_{\lambda+2,T_D} e^{-(4\lambda+6)T_D}\right\}\,
	4(\lambda+1)(\lambda+2) t^{-\lambda-3} e^{-{\pi^2}/{(16 t)}}, \qquad t \le T_D.
$$
Then, since $G_t^{\lambda}(0,1) \to 0$ as $t \to 0^+$ (this follows from the above by an argument given in the proof of
\cite[Lem.\,5.1]{NSS3}; it can also be seen as a consequence of \cite[Lem.\,5.1]{NSS3}), we have
$$
G_t^{\lambda}(0,1) = \int_0^t \frac{\dd}{\dd \tau} G_{\tau}^{\lambda}(0,1)\, \dd \tau
	\simeq \left\{ C^C_{\lambda+2,T_D} \atop c^C_{\lambda+2,T_D} e^{-(4\lambda+6)T_D}\right\}\, 4(\lambda+1)(\lambda+2)
	\int_0^t \tau^{-\lambda-3} e^{-{\pi^2}/{(16\tau)}}\, \dd \tau.
$$
Here in the last integral we change the variable and then apply Lemma \ref{lem:Gmod2} getting
\begin{align} \nonumber
\int_0^t \tau^{-\lambda-3} e^{-{\pi^2}/{(16\tau)}}\, \dd \tau  & = \Big(\frac{16}{\pi^2}\Big)^{\lambda+2}
	\int_{\frac{\pi^2}{16t}}^{\infty} u^{(\lambda+2)-1}e^{-u}\, \dd u \\
	& \simeq \left\{ M_{\lambda+2} \atop m_{\lambda+2} \right\} \frac{16}{\pi^2} \Big( \frac{16}{\pi^2}t + 1\Big)^{\lambda+1}
		t^{-\lambda-1}e^{-{\pi^2}/{(16t)}} \label{in12}
\end{align}
for $t \le T_D$. It follows that
\begin{equation} \label{in10}
\begin{split}
G_t^{\lambda}(0,1) & \simeq \left\{ C^C_{\lambda+2,T_D} M_{\lambda+2} (16 T_D/\pi^2+1)^{\lambda+1}\atop
	c^C_{\lambda+2,T_D} m_{\lambda+2}  e^{-(4\lambda+6)T_D} \right\}\, \\
	& \qquad \times \frac{64(\lambda+1)(\lambda+2)}{\pi^2}\,e^{-{[(\pi/2)^2-\theta^2]}/{(4t)}}\, t^{-\lambda-1}
		e^{-{\theta^2}/{(4t)}},
	\qquad t \le T_D 
	\end{split}
\end{equation}
(here we artificially insert factors containing $\theta$ for the sake of clarity when utilizing these bounds in the sequel).

Next, we estimate $G_t^{\lambda}(x,1)$ for $x \in [0,1]$ and $t \le T_D$.
Using Lemma \ref{lem:diff} we write
\begin{equation} \label{in15}
G_t^{\lambda}(x,1) = G_t^{\lambda}(0,1) + \int_0^x \frac{\dd}{\dd y} G_t^{\lambda}(y,1)\, \dd y
	= G_t^{\lambda}(0,1) + 2(\lambda+1) e^{-t(2\lambda+2)} \int_0^x G_t^{\lambda+1}(y,1)\, \dd y.
\end{equation}
To proceed, we focus on the last integral. For $x = \cos\theta$, $\theta \in [0,\pi/2]$, we have
$$
\int_0^x G_t^{\lambda+1}(y,1) \, \dd y = \int_{\theta}^{\pi/2} G_t^{\lambda+1}(\cos\varphi,1) \sin\varphi \, \dd\varphi
	\simeq \left\{ 1 \atop 2/\pi \right\} \, \int_{\theta}^{\pi/2} G_t^{\lambda+1}(\cos\varphi,1) \, \varphi \, \dd\varphi.
$$
Now Lemma \ref{lem:C} comes into play giving
$$
\int_0^x G_t^{\lambda+1}(y,1)\, \dd y \simeq \left\{ C^C_{\lambda+1,T_D} \atop c^C_{\lambda+1,T_D} {2}/{\pi} \right\}
	t^{-\lambda-2} \int_{\theta}^{\pi/2} e^{-{\varphi^2}/{(4t)}} \varphi \, \dd \varphi, \qquad t \le T_D.
$$
Evaluating the last integral, we obtain
\begin{equation} \label{in20}
\int_0^x G_t^{\lambda+1}(y,1)\, \dd y \simeq \left\{ C^C_{\lambda+1,T_D} \atop c^C_{\lambda+1,T_D} {2}/{\pi} \right\}\,
2 \Big( 1 - e^{-{[(\pi/2)^2-\theta^2]}/{{4t}}}\Big) \, t^{-\lambda-1} e^{-{\theta^2}/{(4t)}}, \qquad t \le T_D.
\end{equation}

Combining \eqref{in15} with \eqref{in10} and \eqref{in20} produces
\begin{align*}
& \bigg[ \Big( c^C_{\lambda+2,T_D} m_{\lambda+2} \frac{64(\lambda+1)(\lambda+2)}{\pi^2}  e^{-(4\lambda+6)T_D} \Big) \wedge
 \bigg( \frac{8(\lambda+1)}{\pi} c^C_{\lambda+1,T_D} e^{-(2\lambda+2)T_D} \bigg) \bigg] \,
	t^{-\lambda-1} e^{-{\theta^2}/{(4t)}} \\
& \le G_t^{\lambda}(\cos\theta,1)\\ & \le \bigg[ \bigg(C^C_{\lambda+2,T_D} M_{\lambda+2}
	\frac{64(\lambda+1)(\lambda+2)}{\pi^2}\Big[\frac{16 T_D}{\pi^2} + 1\Big]^{\lambda+1}\bigg) \vee
 \Big(4(\lambda+1) C^C_{\lambda+1,T_D}\Big) \bigg] \, t^{-\lambda-1} e^{-{\theta^2}/{(4t)}}
\end{align*}
for $\theta \in [0,\pi/2]$ and $t \le T_D$.
The lemma follows.
\end{proof}

\subsection{Step E} \,

\medskip

\noindent \underline{Estimates of $H_t^{\lambda}(x)$ for $x \in [0,1]$, $t \le T_{E}$ and $\lambda \in (-3/2,-1]$.}

\medskip

This step is based essentially on Step C (with Step B being involved only to get better multiplicative constants in a certain case)
and Lemma \ref{lem:diff}, Lemma \ref{lem:diffH}, the heat equation \eqref{he_H}, and Lemma~\ref{lem:Gmod2}.
Step D is also of some relevance. See the proof of \cite[Lem.\,5.3]{NSS3}.

In the statement below, and also in its proof, we use the notation $c^{\blacklozenge}_{\lambda,T}$,
$C^{\blacklozenge}_{\lambda,T}$, where $\blacklozenge = B$ if $\lambda =1$ and $\blacklozenge = C$ if $\lambda \neq 1$.

\begin{biglem}  \label{lem:E}
Let $-3/2 < \lambda \le -1$ and let $T_{E} > 0$. Then
$$
H_t^{\lambda}(\cos\theta) \simeq \left\{ C^{E}_{\lambda,T_{E}} \atop c^{E}_{\lambda,T_{E}} \right\} \,
	\frac{1}{t^{\lambda+1}} \exp\bigg( -\frac{\theta^2}{4t}\bigg)
$$
for $0 \le \theta \le \pi/2$ and $0 < t \le T_{E}$, with the constants
\begin{align*}
c^{E}_{\lambda,T_{E}} &  := \frac{32(1-1/e)}{\pi^2} (\lambda+2) e^{-(4\lambda+6)T_E}\, c^{\blacklozenge}_{\lambda+2,T_E} \times\bigg[
		 2 m_{\lambda+2}\Big(\frac{16 T_E}{\pi^2} + 1 \Big)^{\lambda+1} \wedge 1 \bigg], \\
C^{E}_{\lambda,T_{E}} &  := 16(\lambda+2)\,  C^{\blacklozenge}_{\lambda+2,T_E}
\times \bigg[  \frac{4}{\pi^2} M_{\lambda+2} \vee 1\bigg].
\end{align*}
\end{biglem}

\begin{proof}
We first estimate $H_t^{\lambda}(0)$ similarly as $G_t^{\lambda}(0,1)$ in the proof of Lemma \ref{lem:D}.
Using the fact that $H_t^{\lambda}$ satisfies the heat equation based on the Jacobi-ultraspherical operator
$J^{\lambda}$, see \eqref{he_H}, and then Lemma \ref{lem:diffH}, we write
$$
\frac{\dd}{\dd t} H_t^{\lambda}(0) = - J^{\lambda} H_t^{\lambda}(x)\big|_{x=0} = \frac{\dd^2}{\dd x^2} H_t^{\lambda}(x)\big|_{x=0}
	= 4(\lambda+2) e^{-t(4\lambda+6)} G_t^{\lambda+2}(0,1).
$$
Estimating the last expression with the aid of Lemmas \ref{lem:B} and \ref{lem:C} we get
$$
\frac{\dd}{\dd t} H_t^{\lambda}(0) \simeq \left\{ C^{\blacklozenge}_{\lambda+2,T_E} \atop c^{\blacklozenge}_{\lambda+2,T_E}\right\}\,
	4 (\lambda+2) e^{-t(4\lambda+6)} t^{-\lambda-3} e^{-{\pi^2}/{(16 t)}}, \qquad t \le T_E.
$$
Taking into account that $H_t^{\lambda}(0) \to 0$ as $t \to 0^+$ (this follows by Lemma \ref{lem:red_ext}(v),
Lemma~\ref{lem:diff}, Lemma~\ref{lem:D}, and an argument given in the proof of \cite[Lem.\,5.3]{NSS3}; it is also a consequence
of \cite[Lem.\,5.3]{NSS3}) we infer that
$$
H_t^{\lambda}(0) = \int_0^{t} \frac{\dd}{\dd \tau} H_{\tau}^{\lambda}(0)\, \dd\tau
	\simeq \left\{ C^{\blacklozenge}_{\lambda+2,T_E} \atop c^{\blacklozenge}_{\lambda+2,T_E} e^{-(4\lambda+6)T_E} \right\} 4(\lambda+2)
	\int_0^{t} \tau^{-\lambda-3} e^{-{\pi^2}/{(16\tau)}}\, \dd \tau.
$$
Then, estimating as in \eqref{in12} with the aid of Lemma \ref{lem:Gmod2}, we get
\begin{equation} \label{inH10}
\begin{split}
H_t^{\lambda}(0) & \simeq \left\{ C^{\blacklozenge}_{\lambda+2,T_E} M_{\lambda+2} \atop
	c^{\blacklozenge}_{\lambda+2,T_E} m_{\lambda+2} (16 T_E/\pi^2 + 1)^{\lambda+1}e^{-(4\lambda+6)T_E} \right\} \\
	& \qquad \times \frac{64(\lambda+2)}{\pi^2}
		e^{-{[(\pi/2)^2-\theta^2]}/{(4t)}}\, t^{-\lambda-1} e^{-{\theta^2}/{(4t)}}, \qquad t \le T_E
\end{split}
\end{equation}
(here we artificially write factors containing $\theta$ for the sake of clarity when utilizing these bounds in the sequel).

Now we estimate $H_t^{\lambda}(x)$ for $x \in [0,1]$ and $t \le T_E$.
Using Lemma \ref{lem:diffH} and then Lemma \ref{lem:diff} we get
\begin{align}
H_t^{\lambda}(x) & = H_t^{\lambda}(0) + \int_0^x \frac{\dd}{\dd y} H_t^{\lambda}(y)\, \dd y
	= H_t^{\lambda}(0) + 2 e^{-t(2\lambda+2)} \int_0^x \Big[ G_t^{\lambda+1}(y,1) \Big]_{\textrm{odd}}\, \dd y \nonumber \\
& = H_t^{\lambda}(0) + e^{-t(2\lambda+2)} \int_0^x \int_{-y}^y \frac{\dd}{\dd z} G_t^{\lambda+1}(z,1)\, \dd z\, \dd y \nonumber \\
& = H_t^{\lambda}(0) + 2(\lambda+2) e^{-t(4\lambda+6)} \int_0^x \int_{-y}^y G_t^{\lambda+2}(z,1)\, \dd z \, \dd y \nonumber \\
& \simeq H_t^{\lambda}(0) + \left\{2 \atop 1\right\} 2(\lambda+2)e^{-t(4\lambda+6)}
	\int_0^x \int_{0}^y G_t^{\lambda+2}(z,1)\, \dd z \, \dd y, \label{inH15} 
\end{align}
the last relation by monotonicity of $z \mapsto G_t^{\lambda+2}(z,1)$, see Lemma \ref{lem:diff}.

To proceed, we focus on the last double integral. Letting $x = \cos\theta$, $\theta \in [0,\pi/2]$, we have
\begin{align*}
\int_0^x \int_0^y G_t^{\lambda+2}(z,1)\, \dd z \, \dd y & = \int_{\theta}^{\pi/2}\int_{\varphi}^{\pi/2}
	G_t^{\lambda+2}(\cos\kappa,1) \sin\kappa\, \dd \kappa \, \sin\varphi \, \dd \varphi \\
& \simeq \left\{ 1 \atop (2/\pi)^2 \right\} \int_{\theta}^{\pi/2}\int_{\varphi}^{\pi/2}
	G_t^{\lambda+2}(\cos\kappa,1) \, \kappa\, \dd \kappa \, \varphi \, \dd \varphi.
\end{align*}
Then, by Lemmas \ref{lem:B} and \ref{lem:C},
$$
\int_0^x \int_0^y G_t^{\lambda+2}(z,1)\, \dd z \, \dd y \simeq \left\{ C^{\blacklozenge}_{\lambda+2,T_E}
	\atop c^{\blacklozenge}_{\lambda+2,T_E} (2/\pi)^2\right\}
	t^{-\lambda-3} \int_{\theta}^{\pi/2} \int_{\varphi}^{\pi/2} e^{-{\kappa^2}/{(4t)}}\, \kappa \, \dd \kappa \, \varphi\, \dd\varphi.
$$
Evaluating the last double integral gives
\begin{equation} \label{inH20}
\begin{split}
 & \int_0^x \int_0^y G_t^{\lambda+2}(z,1)\, \dd z \, \dd y
	 \simeq \left\{ C^{\blacklozenge}_{\lambda+2,T_E} \atop c^{\blacklozenge}_{\lambda+2,T_E} (2/\pi)^2\right\} \\
	& \qquad \times 4 \bigg[ 1- \bigg( 1 + \frac{(\pi/2)^2-\theta^2}{4t}\bigg) e^{-{[(\pi/2)^2-\theta^2]}/{(4t)}} 
		\bigg] t^{-\lambda-1} e^{-{\theta^2}/{(4t)}}, \qquad t \le T_E.
\end{split}
\end{equation}

It remains to combine \eqref{inH15} with \eqref{inH10} and \eqref{inH20}, observing that range of the function
$[0,\infty) \ni \eta \mapsto e^{-\eta} + 1 - (1+\eta)e^{-\eta} = 1-\eta e^{-\eta}$ is $[1-1/e,1]$
(think $\eta = [{(\pi/2)^2-\theta^2}]/{(4t)}$). We get
\begin{align*}
	& \frac{32(1-1/e)}{\pi^2} (\lambda+2) e^{-(4\lambda+6)T_E}\, c^{\blacklozenge}_{\lambda+2,T_E} \times\bigg[
		 2 m_{\lambda+2}\Big(\frac{16 T_E}{\pi^2} + 1 \Big)^{\lambda+1} \wedge 1 \bigg] t^{-\lambda-1}e^{-{\theta^2}/{(4t)}} \\
&	
 \le H_t^{\lambda}(\cos\theta) \le 16(\lambda+2)\,  C^{\blacklozenge}_{\lambda+2,T_E}
	\times \bigg[  \frac{4}{\pi^2} M_{\lambda+2} \vee 1\bigg] t^{-\lambda-1} e^{-{\theta^2}/{(4t)}}
\end{align*}
for $\theta \in [0,\pi/2]$ and $t \le T_E$. The conclusion follows.
\end{proof}

\begin{rem} \label{rem:E}
Let $-3/2 < \lambda \le -1$ and let $T_{E} > 0$. 
Then
\begin{align*}
	2 m_{\lambda+2}\Big(\frac{16 T_E}{\pi^2} + 1 \Big)^{\lambda+1} \wedge 1 
	& = 
	2 m_{\lambda+2}\Big(\frac{16 T_E}{\pi^2} + 1 \Big)^{\lambda+1}, \\
 \frac{4}{\pi^2} M_{\lambda+2} \vee 1
 & =
 \frac{4}{\pi^2} M_{\lambda+2},
\end{align*}	
and consequently
\begin{align*}
	c^{E}_{\lambda,T_{E}} &  = \frac{64(1-1/e)}{\pi^2} (\lambda+2) e^{-(4\lambda+6)T_E}\, c^{\blacklozenge}_{\lambda+2,T_E} 
	m_{\lambda+2}\Big(\frac{16 T_E}{\pi^2} + 1 \Big)^{\lambda+1}, \\
	C^{E}_{\lambda,T_{E}} 
	&  = 
	\frac{64(\lambda+2)}{\pi^2}\,  C^{\blacklozenge}_{\lambda+2,T_E} M_{\lambda+2}.
\end{align*}

To justify this, observe that $2m_{\lambda+2} \le 2m_{1/2} < 1$, since $m_{\lambda}$ is decreasing.
On the other hand, using the bound $\pi(1-1/e) \le 2$  we get ${4}{\pi^{-2}} M_{\lambda+2} \ge 1$
in the range of $\lambda$ in question.
\end{rem}

We do not include the simplified constants in Lemma \ref{lem:E} to keep its more abstract form that is convenient for potential future
improvements. More precisely, Remark \ref{rem:E} is based on particular, perhaps non-optimal,
constants delivered by Lemma \ref{lem:Gmod2} and might not hold once they are refined.

\subsection{Step F} \,

\medskip

\noindent \underline{Estimates of $G_t^{\a,\b}(x,y)$ for $x,y \in [-1,1]$, $t \le T$ and $\a,\b > -1$.}

\medskip

This final step is based essentially on Steps C--E, see \cite[Sec.\,6]{NSS3}, however Steps A and B will also be involved to get better
multiplicative constants in certain cases. The technical tools that are relevant for Step~F are the pivotal Lemma \ref{lem:red_ext}
along with Lemmas \ref{lem:diff}, \ref{lem:diffH}, \ref{lem:mes3} and \ref{lem:intF}.

In the statement below, and also in its proof, we use the notation $c^{\bigstar}_{\lambda,T}$, $C^{\bigstar}_{\lambda,T}$, where
$\bigstar = A$ if $\lambda \in \mathbb{N}-1/2$,
$\bigstar = B$ if $\lambda \in \mathbb{N}$,
$\bigstar = C$ if $\lambda \in (0,\infty)$ and $\lambda \notin \mathbb{N}/2$,
$\bigstar = D$ if $-1/2 \neq \lambda \in (-1,0)$, and
$\bigstar = E$ if $\lambda \in (-3/2,-1]$.
We will also use two auxiliary functions,
$$
q_{\lambda}(t) := 2(\lambda+1)(\lambda+2)e^{-t(\lambda+3/2)}, \qquad \widetilde{q}_{\lambda}(t) := 2(\lambda+2)e^{-t(\lambda+3/2)},
$$
and the notation
$
\Lambda = \Lambda_{\a,\b} := \a+\b+1/2.
$
\begin{biglem} \label{lem:F}
Let $\a,\b > -1$ and let $T>0$. Then
$$
G_t^{\a,\b}(\cos\theta,\cos\varphi) \simeq
	\left\{ C_{\a,\b,T} \Psi_{\a}^{\mathfrak{b}}(t,\theta,\varphi) \Psi_{\b}^{\mathfrak{b}}(t,\pi-\theta,\pi-\varphi) \atop
		c_{\a,\b,T} \Psi_{\a}^{\mathfrak{B}}(t,\theta,\varphi) \Psi_{\b}^{\mathfrak{B}}(t,\pi-\theta,\pi-\varphi) \right\} \,
		\frac{1}{\sqrt{t}} \exp\bigg({-\frac{(\theta-\varphi)^2}{4t}}\bigg)
$$
for $0 \le \theta,\varphi \le \pi$ and $0 < t \le T$, with the constants given as follows.
\begin{itemize}
\item[$\blacktriangleright$] In case $\a,\b \ge - 1/2$,
$$
c_{\a,\b,T}  := c^{\bigstar}_{\Lambda,T/4} \times 
	{4^{\Lambda+1} b_{\a} b_{\b}} \times \frac{h_0^{\Lambda}}{h_0^{\a,\b}}, \qquad
C_{\a,\b,T}  := C^{\bigstar}_{\Lambda,T/4} \times  
	{4^{\Lambda+2} B_{\a} B_{\b}} \times \frac{h_0^{\Lambda}}{h_0^{\a,\b}}.
$$
\item[$\blacktriangleright$] In case $-1 < \b < -1/2 \le \a$,
\begin{align*}
c_{\a,\b,T} & := 4^{\Lambda+1/2}b_{\a} \bigg[ c^{\bigstar}_{\Lambda+2,T/4} q_{\Lambda}(T)
\Big(\frac{4\mathbb{D}_{\b}}{\pi^2\mathbb{D}_{\b+2}}\Big)^2 \frac{4}{\mathfrak{B}^2} l_{\b}b_{\b+2} \wedge
		 c^{\bigstar}_{\Lambda,T/4} \mathbb{D}_{\b}^{\b+1/2} \bigg] \frac{h_0^{\Lambda}}{h_0^{\a,\b}},\\
C_{\a,\b,T} & := 4^{\Lambda+3/2}B_{\a} \bigg[ C^{\bigstar}_{\Lambda+2,T/4} q_{\Lambda}(0)\frac{4}{\mathfrak{b}^2}
	   L_{\b}B_{\b+2} + C^{\bigstar}_{\Lambda,T/4} \mathbb{D}_{\b}^{\b+1/2} \bigg] \frac{h_0^{\Lambda}}{h_0^{\a,\b}}.
\end{align*}
\item[$\blacktriangleright$] In case $-1 < \a < -1/2 \le \b$,
\begin{align*}
c_{\a,\b,T} & := 4^{\Lambda+1/2}b_{\b} \bigg[ c^{\bigstar}_{\Lambda+2,T/4} q_{\Lambda}(T)
\Big(\frac{4\mathbb{D}_{\a}}{\pi^2\mathbb{D}_{\a+2}}\Big)^2 \frac{4}{\mathfrak{B}^2} l_{\a}b_{\a+2} \wedge
		 c^{\bigstar}_{\Lambda,T/4} \mathbb{D}_{\a}^{\a+1/2} \bigg] \frac{h_0^{\Lambda}}{h_0^{\a,\b}},\\
C_{\a,\b,T} & := 4^{\Lambda+3/2}B_{\b} \bigg[ C^{\bigstar}_{\Lambda+2,T/4} q_{\Lambda}(0)\frac{4}{\mathfrak{b}^2}
	   L_{\a}B_{\a+2} + C^{\bigstar}_{\Lambda,T/4} \mathbb{D}_{\a}^{\a+1/2} \bigg] \frac{h_0^{\Lambda}}{h_0^{\a,\b}}.
\end{align*}
\item[$\blacktriangleright$] In case $-1 < \a,\b < -1/2$ and $\Lambda_{\a,\b} > -1$,
\begin{align*}
c_{\a,\b,T} & := 4^{\Lambda+1}\bigg[ c^{\bigstar}_{\Lambda+4,T/4} q_{\Lambda+2}(T) q_{\Lambda}(T) \Big(\frac{2}{\pi}\Big)^8
l_{\a}l_{\b} b_{\a+2} b_{\b+2} \Big( \frac{2\mathbb{D}_{\a}\mathbb{D}_{\b}}{\mathbb{D}_{\a+2}\mathbb{D}_{\b+2}\mathfrak{B}^2}\Big)^2 \\
	& \hspace{45pt} \wedge c^{\bigstar}_{\Lambda+2,T/4} q_{\Lambda}(T)\Big(\frac{2}{\pi}\Big)^4 \frac{1}{\mathfrak{B}^2} \bigg\{
		\Big( \frac{\mathbb{D}_{\a}}{\mathbb{D}_{\a+2}}\Big)^2\mathbb{D}_{\b}^{\b+1/2} l_{\a}b_{\a+2} \\ & \hspace{45pt} \wedge
			\Big( \frac{\mathbb{D}_{\b}}{\mathbb{D}_{\b+2}}\Big)^2\mathbb{D}_{\a}^{\a+1/2} l_{\b} b_{\b+2}\bigg\} 
				 \wedge c^{\bigstar}_{\Lambda,T/4} \frac{1}4\mathbb{D}_{\a}^{\a+1/2} \mathbb{D}_{\b}^{\b+1/2}\bigg]
				\frac{h_0^{\Lambda}}{h_0^{\a,\b}}, \\
C_{\a,\b,T} & := 4^{\Lambda+2}\bigg[ C^{\bigstar}_{\Lambda+4,T/4} q_{\Lambda+2}(0) q_{\Lambda}(0) \frac{4}{\mathfrak{b}^4}
	L_{\a}L_{\b} B_{\a+2} B_{\b+2} \\
	& \hspace{45pt} + C^{\bigstar}_{\Lambda+2,T/4} q_{\Lambda}(0) \frac{1}{\mathfrak{b}^2} 
			\Big( L_{\a}B_{\a+2}\mathbb{D}_{\b}^{\b+1/2} + L_{\b} B_{\b+2}\mathbb{D}_{\a}^{\a+1/2} \Big)\\
				& \hspace{45pt} + C^{\bigstar}_{\Lambda,T/4} \frac{1}{4}\mathbb{D}_{\a}^{\a+1/2} \mathbb{D}_{\b}^{\b+1/2}\bigg]
				\frac{h_0^{\Lambda}}{h_0^{\a,\b}}.
\end{align*}
\item[$\blacktriangleright$] In case $-1 < \a,\b < -1/2$ and $\Lambda_{\a,\b} \le -1$,
\begin{align*}
c_{\a,\b,T} & := 4^{\Lambda+1}\bigg[ c^{\bigstar}_{\Lambda+4,T/4} q_{\Lambda+2}(T) \widetilde{q}_{\Lambda}(T) \Big(\frac{2}{\pi}\Big)^8
	l_{\a}l_{\b} b_{\a+2} b_{\b+2} \Big( \frac{2\mathbb{D}_{\a}\mathbb{D}_{\b}}{\mathbb{D}_{\a+2}\mathbb{D}_{\b+2}\mathfrak{B}^2}\Big)^2 \\
	& \hspace{45pt} \wedge c^{\bigstar}_{\Lambda+2,T/4} \widetilde{q}_{\Lambda}(T)\Big(\frac{2}{\pi}\Big)^4
		\frac{1}{\mathfrak{B}^2} \bigg\{
		\Big( \frac{\mathbb{D}_{\a}}{\mathbb{D}_{\a+2}}\Big)^2\mathbb{D}_{\b}^{\b+1/2} l_{\a}b_{\a+2}\\ & \hspace{45pt} \wedge
			\Big( \frac{\mathbb{D}_{\b}}{\mathbb{D}_{\b+2}}\Big)^2\mathbb{D}_{\a}^{\a+1/2} l_{\b}b_{\b+2}\bigg\}
			 \wedge c^{\bigstar}_{\Lambda,T/4} \frac{1}{2}\mathbb{D}_{\a}^{\a+1/2} \mathbb{D}_{\b}^{\b+1/2}\bigg]
				\frac{\sqrt{\pi}\,\Gamma(\Lambda+2)}{h_0^{\a,\b}\Gamma(\Lambda+3/2)}, \\
C_{\a,\b,T} & := 4^{\Lambda+2}\bigg[ C^{\bigstar}_{\Lambda+4,T/4} q_{\Lambda+2}(0) \widetilde{q}_{\Lambda}(0) \frac{4}{\mathfrak{b}^4}
	L_{\a}L_{\b} B_{\a+2} B_{\b+2} \\
	& \hspace{45pt} + C^{\bigstar}_{\Lambda+2,T/4} \widetilde{q}_{\Lambda}(0) \frac{1}{\mathfrak{b}^2} 
			\Big( L_{\a}B_{\a+2}\mathbb{D}_{\b}^{\b+1/2} + L_{\b} B_{\b+2}\mathbb{D}_{\a}^{\a+1/2} \Big)\\
				& \hspace{45pt} + C^{\bigstar}_{\Lambda,T/4} \frac{1}{4}\mathbb{D}_{\a}^{\a+1/2} \mathbb{D}_{\b}^{\b+1/2}\bigg]
				\frac{\sqrt{\pi}\,\Gamma(\Lambda+2)}{h_0^{\a,\b}\Gamma(\Lambda+3/2)}.
\end{align*}
\end{itemize}
\end{biglem}

\begin{proof}[{Proof of Lemma \ref{lem:F}, the case $\a,\b \ge -1/2$.}]
By Lemma \ref{lem:red_ext}(i)
$$
{I}:= \frac{h_0^{\a,\b}}{h_0^{\a+\b+1/2}} G_t^{\a,\b}(\cos\theta,\cos\varphi)
	= \iint G_{t/4}^{\a+\b+1/2}\Big( \cos\sqrt{F_{\theta,\varphi}(u,v)}, 1\Big)\, \dd \Pi_{\a}(u)\, \dd \Pi_{\b}(v).
$$
In view of monotonicity of $z \mapsto G_{t/4}^{\a+\b+1/2}(z,1)$, see Lemma \ref{lem:diff}, and symmetry of the measures
involved, this implies
$$
{I} \simeq \left\{ 4 \atop 1 \right\}
	\iint_{[0,1]^2} G_{t/4}^{\a+\b+1/2}\Big( \cos\sqrt{F_{\theta,\varphi}(u,v)}, 1\Big)\, \dd \Pi_{\a}(u)\, \dd \Pi_{\b}(v).
$$
Then under the last double integral $\sqrt{F_{\theta,\varphi}(u,v)} \in [0,\pi/2]$ and we can use Lemmas
\ref{lem:A}, \ref{lem:B}, \ref{lem:C} and \ref{lem:D} getting
$$
{I} \simeq \left\{ 4 C^{\bigstar}_{\a+\b+1/2,T/4} \atop c^{\bigstar}_{\a+\b+1/2,T/4} \right\} \Big(\frac{4}t\Big)^{\a+\b+3/2} 
	\iint_{[0,1]^2} e^{-{F_{\theta,\varphi}(u,v)}/{t}}\, \dd\Pi_{\a}(u)\, \dd\Pi_{\b}(v)
$$
for $\theta,\varphi \in [0,\pi]$ and $0 < t \le T$.
Now an application of Lemma \ref{lem:intF} shows that
$$
{I} \simeq
	\left\{ C^{\bigstar}_{\a+\b+1/2,T/4} 4^{\a+\b+5/2} B_{\a}B_{\b}
		\Psi_{\a}^{\mathfrak{b}}(t,\theta,\varphi) \Psi_{\b}^{\mathfrak{b}}(t,\pi-\theta,\pi-\varphi) \atop
		c^{\bigstar}_{\a+\b+1/2,T/4} 4^{\a+\b+3/2} b_{\a}b_{\b}
		\Psi_{\a}^{\mathfrak{B}}(t,\theta,\varphi) \Psi_{\b}^{\mathfrak{B}}(t,\pi-\theta,\pi-\varphi) \right\} \,
		\frac{1}{\sqrt{t}} e^{-(\theta-\varphi)^2/(4t)}
$$
for $\theta,\varphi \in [0,\pi]$ and $0 < t \le T$. From this the estimates stated in the lemma follow for
$\a$ and $\b$ under consideration.
\end{proof}

\begin{rem} \label{rem:F}
Considering the case when $\a,\b \ge -1/2$ and $-1/2 < \a+\b \notin \mathbb{N}$, the bounds of Lemma~\ref{lem:F}
could be strengthened by employing in the proof the estimates from Remarks \ref{rem:B} and \ref{rem:C} instead of the bounds
from Lemmas \ref{lem:B} and \ref{lem:C}.
The details are straightforward and left to interested readers.
\end{rem}

\begin{proof}[{Proof of Lemma \ref{lem:F}, the cases when either $-1 < \b < -1/2 \le \a$ or $-1 < \a < -1/2 \le \b$.}]
It is sufficient to prove the result only in the first case, for symmetry reasons. Then in the other case the result follows
by swapping $\a$ with $\b$ and $\theta,\varphi$ with $\pi-\theta, \pi-\varphi$, respectively;
note that $G_t^{\b,\a}(-x,-y) = G_t^{\a,\b}(x,y)$.
Thus from now on we assume that $-1 < \b < -1/2 \le \a$.
All estimates we write in this proof are for $\theta,\varphi \in [0,\pi]$ and $0 < t \le T$.

By Lemma \ref{lem:red_ext}(ii)
$$
\frac{h_0^{\a,\b}}{h_0^{\a+\b+1/2}} G_t^{\a,\b}(\cos\theta,\cos\varphi)
	= 2(\a+\b+3/2)e^{-t(\a+\b+3/2)/2} I_1 + I_2,
$$
where
\begin{align} \label{zx0}
I_1 & := - \cos\frac{\theta}2\cos\frac{\varphi}2 \iint G_{t/4}^{\a+\b+3/2}\Big( \cos\sqrt{F_{\theta,\varphi}(u,v)},1\Big)
	\, \dd\Pi_{\a}(u) \, \Pi_{\b}(v)\, \dd v, \\
I_2 & := \iint G_{t/4}^{\a+\b+1/2}\Big( \cos\sqrt{F_{\theta,\varphi}(u,v)},1\Big) \, \dd\Pi_{\a}(u)\, \dd\Pi_{-1/2}(v). \nonumber
\end{align}
We will estimate $I_1$ and $I_{2}$ separately. Notice that both the quantities are positive, in case of $I_1$ since
$z \mapsto G_{t/4}^{\a+\b+3/2}(z,1)$ is increasing (see Lemma \ref{lem:diff}) and $\Pi_{\b}$ is odd and decreasing.

Concerning $I_2$, by monotonicity of $z \mapsto G_{t/4}^{\a+\b+1/2}(z,1)$ and symmetry of the measures involved,
$$
I_2 \simeq \left\{ 4 \atop 1 \right\}\, \iint_{[0,1]^2} G_{t/4}^{\a+\b+1/2}\Big( \cos\sqrt{F_{\theta,\varphi}(u,v)},1 \Big)\,
	\dd\Pi_{\a}(u)\, \dd\Pi_{-1/2}(v).
$$
Estimating the integrand here by means of Lemmas \ref{lem:A}, \ref{lem:B}, \ref{lem:C} and \ref{lem:D} we get
$$
I_2 \simeq \left\{ 4 C^{\bigstar}_{\a+\b+1/2,T/4} \atop c^{\bigstar}_{\a+\b+1/2,T/4} \right\}\,
	\Big( \frac{4}t \Big)^{\a+\b+3/2} \iint_{[0,1]^2} e^{-F_{\theta,\varphi}(u,v)/t} \, \dd\Pi_{\a}(u)\, \dd\Pi_{-1/2}(v).
$$
Now an application of Lemma \ref{lem:intF} leads to
\begin{equation} \label{zx1}
I_2 \simeq \left\{ C^{\bigstar}_{\a+\b+1/2,T/4} 4^{\a+\b+2} B_{\a} \Psi_{\a}^{\mathfrak{b}}(t,\theta,\varphi) \atop
	c^{\bigstar}_{\a+\b+1/2,T/4} 4^{\a+\b+1} b_{\a} \Psi_{\a}^{\mathfrak{B}}(t,\theta,\varphi) \right\}\,
		\frac{1}{t^{\b+1}} e^{-(\theta-\varphi)^2/(4t)}.
\end{equation}

Passing to $I_1$, recall that $\Pi_{\b}(v)$ is an odd function which is negative for $v > 0$. Thus
\begin{multline*}
I_1 = \cos\frac{\theta}2\cos\frac{\varphi}2 \int_{[-1,1]} \int_{[0,1]} \Big[ 
	G_{t/4}^{\a+\b+3/2}\Big( \cos\sqrt{F_{\theta,\varphi}(u,v)},1\Big) \\
		  - G_{t/4}^{\a+\b+3/2}\Big( \cos\sqrt{F_{\theta,\varphi}(u,-v)},1\Big) \Big]\, |\Pi_{\b}(v)|\, \dd v \, \dd\Pi_{\a}(u).
\end{multline*}
Here we can write the difference as the integral from $-v$ to $v$  of the derivative and then apply Lemma \ref{lem:diff}. This gives
\begin{multline*}
I_1 = 2 (\a+\b+5/2) e^{-t (\a+\b+5/2)/2} \Big[ \cos\frac{\theta}2 \cos\frac{\varphi}2 \Big]^2 \\
	\times \int_{[-1,1]} \int_{[0,1]} \int_{-v}^v 
	G_{t/4}^{\a+\b+5/2}\Big( \cos\sqrt{F_{\theta,\varphi}(u,v')},1\Big)\, \dd v' \, |\Pi_{\b}(v)|\, \dd v \, \dd\Pi_{\a}(u).
\end{multline*}

Denote the above triple integral by $\widetilde{I}_1$. Observe that the integrand there is increasing both in $u$ and $v'$,
see Lemma \ref{lem:diff}. Hence, taking into account symmetries of the measures $\dd \Pi_{\a}$ and $\dd v'$, we may restrict
the integration in these two variables to $[0,1]$ and $[0,v]$, respectively, getting
$$
\widetilde{I}_1 \simeq \left\{ 4 \atop 1 \right\}\, \int_{[0,1]}\int_{[0,1]}\int_0^v
	G_{t/4}^{\a+\b+5/2}\Big( \cos\sqrt{F_{\theta,\varphi}(u,v')},1\Big)\, \dd v' \, |\Pi_{\b}(v)|\, \dd v \, \dd\Pi_{\a}(u).
$$
Now we use Fubini's theorem and integrate first in $v$, which produces a factor that can be estimated with the aid of
Lemma \ref{lem:mes3},
$$
\int_{v'}^1 |\Pi_{\b}(v)|\, \dd v \simeq \left\{ L_{\b} \atop l_{\b} \right\}\, \frac{\dd \Pi_{\b+2}(v')}{\dd v'},
	\qquad v' \in [0,1].
$$
Consequently,
$$
\widetilde{I}_1 \simeq \left\{ 4 L_{\b} \atop l_{\b} \right\}\, \iint_{[0,1]^2}
	G_{t/4}^{\a+\b+5/2}\Big( \cos\sqrt{F_{\theta,\varphi}(u,v')},1\Big)\, \dd\Pi_{\a}(u) \, \dd \Pi_{\b+2}(v').
$$
From this, using Lemmas \ref{lem:A}, \ref{lem:B} and \ref{lem:C}, we obtain
$$
\widetilde{I}_1 \simeq \left\{ C^{\bigstar}_{\a+\b+5/2,T/4} 4 L_{\b} \atop c^{\bigstar}_{\a+\b+5/2,T/4}l_{\b} \right\}\,
	\Big(\frac{4}t\Big)^{\a+\b+7/2} \iint_{[0,1]^2}
	e^{-F_{\theta,\varphi}(u,v')/t}\, \dd\Pi_{\a}(u) \, \dd \Pi_{\b+2}(v').
$$
Finally, applying Lemma \ref{lem:intF} we arrive at
\begin{equation} \label{zx2}
\begin{split}
\widetilde{I}_1 & \simeq \left\{ C^{\bigstar}_{\a+\b+5/2,T/4} 4^{\a+\b+9/2} L_{\b} B_{\a}B_{\b+2}
	\Psi_{\a}^{\mathfrak{b}}(t,\theta,\varphi) \Psi_{\b+2}^{\mathfrak{b}}(t,\pi-\theta,\pi-\varphi) \atop
	c^{\bigstar}_{\a+\b+5/2,T/4} 4^{\a+\b+7/2} l_{\b} b_{\a}b_{\b+2}
	\Psi_{\a}^{\mathfrak{B}}(t,\theta,\varphi) \Psi_{\b+2}^{\mathfrak{B}}(t,\pi-\theta,\pi-\varphi) \right\}\,\\
	& \qquad \times \frac{1}{\sqrt{t}}e^{-(\theta-\varphi)^2/(4t)}.
\end{split}
\end{equation}

Next, we want to combine the bounds \eqref{zx1} and \eqref{zx2} for $I_2$ and $\widetilde{I}_1$ with the formula
\begin{equation} \label{zx3}
\frac{h_0^{\a,\b}}{h_0^{\a+\b+1/2}} G_t^{\a,\b}(\cos\theta,\cos\varphi)
	= 2 q_{\a+\b+1/2}(t) \Big[\cos\frac{\theta}2\cos\frac{\varphi}2\Big]^2 \widetilde{I}_1 + I_2.
\end{equation}
To do that we need to bound suitably two auxiliary expressions.
For $\kappa > 0$, $\lambda \in (-1,-1/2)$, $\theta,\varphi \in [0,\pi]$ and $t>0$ consider
\begin{align} \label{zxO1}
O_1^{\kappa,\lambda} = O_1^{\kappa,\lambda}(t,\theta,\varphi) & := \Big[\sin\frac{\theta}2\sin\frac{\varphi}2\Big]^2 \times
	\frac{\Psi_{\lambda+2}^{\kappa}(t,\theta,\varphi)}{\Psi_{\lambda}^{\kappa}(t,\theta,\varphi)} \\
	& = \frac{[\sin({\theta}/2) \sin({\varphi}/2)]^2}{
		[\mathbb{D}_{\lambda+2}t \vee \kappa \theta \varphi]^{\lambda+5/2}
		[\mathbb{D}_{\lambda}t \vee \kappa \theta \varphi]^{-\lambda-1/2}}, \nonumber \\
O_2^{\kappa,\lambda} = O_2^{\kappa,\lambda}(t,\theta,\varphi) & := \frac{1}{\Psi_{\lambda}^{\kappa}(t,\theta,\varphi) t^{\lambda+1/2}}
	= \frac{1}{[\mathbb{D}_{\lambda}t \vee \kappa \theta \varphi]^{-\lambda-1/2} t^{\lambda+1/2}}. \label{zxO2}
\end{align}
Using the bounds $\theta/\pi \le \sin(\theta/2) \le \theta/2$ for $\theta \in [0,\pi]$
and the fact that $\mathbb{D}_{\lambda}$ is increasing with respect to the parameter, we see that
$$
\frac{1}{\pi^4} \frac{[\theta\varphi]^2}{[\mathbb{D}_{\lambda+2}t \vee \kappa \theta \varphi]^2} \le O_1^{\kappa,\lambda}
	\le \frac{1}{16} \frac{[\theta \varphi]^2}{[\mathbb{D}_{\lambda}t \vee \kappa \theta\varphi]^2}.
$$
Notice that $O_1^{\kappa,\lambda} \le 1/(16 \kappa^2)$ and $O_2^{\kappa,\lambda} \le \mathbb{D}_{\lambda}^{\lambda+1/2}$. Moreover,
if $\mathbb{D}_{\lambda} t \ge \kappa \theta\varphi$, then $O_2^{\kappa,\lambda} = \mathbb{D}_{\lambda}^{\lambda+1/2}$, and if
$\mathbb{D}_{\lambda} t < \kappa \theta\varphi$, then $O_1^{\kappa,\lambda}
	\ge (\kappa \mathbb{D}_{\lambda+2}/\mathbb{D}_{\lambda})^{-2}/\pi^4$.
Therefore, for any $c_1,c_2 >0$,
\begin{equation} \label{zx4}
\frac{c_1}{\pi^4} \Big(\frac{\mathbb{D}_{\b}}{\mathbb{D}_{\b+2}\kappa}\Big)^2 \wedge c_2 \mathbb{D}_{\b}^{\b+1/2} \le
c_1 O_1^{\kappa,\b}(t,\pi-\theta,\pi-\varphi) + c_2 O_2^{\kappa,\b}(t,\pi-\theta,\pi-\varphi)
	\le \frac{c_1}{16 \kappa^2} + c_2 \mathbb{D}_{\b}^{\b+1/2}.
\end{equation}

Combining \eqref{zx3} with \eqref{zx1} and \eqref{zx2} and using \eqref{zx4} with $\kappa = \mathfrak{b}$ and
$\kappa = \mathfrak{B}$ we conclude that
$$
\frac{h_0^{\a,\b}}{h_0^{\a+\b+1/2}} G_t^{\a,\b}(\cos\theta,\cos\varphi) \simeq \left\{
	\check{C}_{\a,\b,T} \Psi_{\a}^{\mathfrak{b}}(t,\theta,\varphi) \Psi_{\b}^{\mathfrak{b}}(t,\pi-\theta,\pi-\varphi) \atop
	\check{c}_{\a,\b,T} \Psi_{\a}^{\mathfrak{B}}(t,\theta,\varphi) \Psi_{\b}^{\mathfrak{B}}(t,\pi-\theta,\pi-\varphi) \right\} \,
	\frac{1}{\sqrt{t}} e^{-(\theta-\varphi)^2/(4t)},
$$
where
\begin{align*}
\check{c}_{\a,\b,T} & = \frac{q_{\a+\b+1/2}(T)}{\pi^4} \Big( \frac{\mathbb{D}_{\b}}{\mathbb{D}_{\b+2}\mathfrak{B}}\Big)^2
	c^{\bigstar}_{\a+\b+5/2,T/4} 4^{\a+\b+4} l_{\b} b_{\a} b_{\b+2} \wedge \mathbb{D}_{\b}^{\b+1/2} c^{\bigstar}_{\a+\b+1/2,T/4}
		4^{\a+\b+1} b_{\a}, \\
\check{C}_{\a,\b,T} & = \frac{q_{\a+\b+1/2}(0)}{\mathfrak{b}^2} C^{\bigstar}_{\a+\b+5/2,T/4} 4^{\a+\b+3} L_{\b} B_{\a} B_{\b+2}
	+ \mathbb{D}_{\b}^{\b+1/2} C^{\bigstar}_{\a+\b+1/2,T/4} 4^{\a+\b+2} B_{\a}.
\end{align*}
This implies the estimates stated in the lemma for $\a$ and $\b$ under consideration.
\end{proof}

\begin{proof}[{Proof of Lemma \ref{lem:F}, the case when $-1 < \a,\b < -1/2$.}]
All estimates we write in this part of the proof are for $\theta,\varphi \in [0,\pi]$ and $0 < t \le T$,
and the standing assumption is $\a,\b \in (-1,-1/2)$.
We first estimate three quantities which are involved in the representations of $G_t^{\a,\b}(\cos\theta,\cos\varphi)$ from
items (iv) and (v) of Lemma \ref{lem:red_ext}. Let
\begin{align*}
J_1 & := \sin\theta \sin\varphi \iint G_{t/4}^{\a+\b+5/2}\Big( \cos\sqrt{F_{\theta,\varphi}(u,v)},1\Big)
	\, \Pi_{\a}(u)\, \dd u \, \Pi_{\b}(v)\, \dd v, \\
J_2 & := - \sin\frac{\theta}2 \sin\frac{\varphi}2 \iint G_{t/4}^{\a+\b+3/2}\Big( \cos\sqrt{F_{\theta,\varphi}(u,v)},1\Big)
	\, \Pi_{\a}(u)\, \dd u \, \dd \Pi_{-1/2}(v), \\
J_3 & := - \cos\frac{\theta}2 \cos\frac{\varphi}2 \iint G_{t/4}^{\a+\b+3/2}\Big( \cos\sqrt{F_{\theta,\varphi}(u,v)},1\Big)
	\, \dd \Pi_{-1/2}(u) \, \Pi_{\b}(v) \, \dd v.
\end{align*}
Each of $J_1$, $J_2$ and $J_3$ is positive, as can be seen from the estimates that follow.

To treat $J_1$, recall that $\Pi_{\a}(u)$ and $\Pi_{\b}(v)$ are odd functions of $u$ and $v$, which are negative for
$u>0$ and $v>0$, respectively. Thus
$$
J_1 = \sin\theta \sin\varphi \iint_{[0,1]^2} \Delta(u,v) \, |\Pi_{\a}(u)|\, \dd u \, |\Pi_{\b}(v)|\, \dd v,
$$
where $\Delta(u,v)$ is the second difference given by
\begin{align*}
\Delta(u,v) & := G_{t/4}^{\a+\b+5/2}\Big( \cos\sqrt{F_{\theta,\varphi}(u,v)},1\Big)
	- G_{t/4}^{\a+\b+5/2}\Big( \cos\sqrt{F_{\theta,\varphi}(-u,v)},1\Big) \\
	& \qquad - G_{t/4}^{\a+\b+5/2}\Big( \cos\sqrt{F_{\theta,\varphi}(u,-v)},1\Big)
	+ G_{t/4}^{\a+\b+5/2}\Big( \cos\sqrt{F_{\theta,\varphi}(-u,-v)},1\Big).
\end{align*}
Now the key observation is that
$$
\Delta(u,v) = \iint_{|u'|<u,\, |v'|<v} \partial_{u'}\partial_{v'} \bigg[
	G_{t/4}^{\a+\b+5/2}\Big( \cos\sqrt{F_{\theta,\varphi}(u',v')},1\Big)\bigg] \, \dd u'\, \dd v'.
$$
Applying Lemma \ref{lem:diff} twice we see that the function under the double integral here is equal
$$
\frac{1}2 q_{\a+\b+5/2}(t) \sin\theta \sin\varphi \; G_{t/4}^{\a+\b+9/2}\Big( \cos\sqrt{F_{\theta,\varphi}(u',v')},1\Big).
$$
We conclude that
\begin{align*}
J_{1} & = \frac{1}2 q_{\a+\b+5/2}(t) \big( \sin\theta \sin\varphi )^2 \\ & \qquad \times \iint_{[0,1]^2} \iint_{|u'|<u,\, |v'|<v}
	G_{t/4}^{\a+\b+9/2}\Big( \cos\sqrt{F_{\theta,\varphi}(u',v')},1\Big) \, \dd u'\, \dd v'
		\, |\Pi_{\a}(u)|\, \dd u \, |\Pi_{\b}(v)|\, \dd v.
\end{align*}
To proceed, denote the last quadruple integral by $\widetilde{J}_1$, so that
\begin{equation} \label{zx5}
J_{1} = 8 q_{\a+\b+5/2}(t) \Big[\sin\frac{\theta}2 \sin\frac{\varphi}2\Big]^2  \Big[\cos\frac{\theta}2 \cos\frac{\varphi}2\Big]^2
	\widetilde{J}_1.
\end{equation}

The integrand in $\widetilde{J}_1$ is increasing in both $u'$ and $v'$ (see Lemma \ref{lem:diff}), hence we can restrict
the inner integration to $0 < u' < u$, $0 < v' < v$ getting
$$
\widetilde{J}_1 \simeq \left\{ 4 \atop 1\right\} \, \iint_{[0,1]^2} \int_0^u \int_0^v
	G_{t/4}^{\a+\b+9/2}\Big( \cos\sqrt{F_{\theta,\varphi}(u',v')},1\Big) \, \dd u'\, \dd v'
		\, |\Pi_{\a}(u)|\, \dd u \, |\Pi_{\b}(v)|\, \dd v.
$$
Using now Fubini's theorem to integrate first in $u$ and $v$ we obtain
$$
\widetilde{J}_1 \simeq \left\{ 4 \atop 1\right\} \, \iint_{[0,1]^2} \int_{u'}^1 |\Pi_{\a}(u)|\, \dd u \,
	\int_{v'}^1 |\Pi_{\b}(v)|\, \dd v\; G_{t/4}^{\a+\b+9/2}\Big( \cos\sqrt{F_{\theta,\varphi}(u',v')},1\Big) \, \dd u'\, \dd v'.
$$
Then applying Lemma \ref{lem:mes3} twice we get
$$
\widetilde{J}_1 \simeq \left\{ 4 L_{\a}L_{\b} \atop l_{\a} l_{\b} \right\} \, \iint_{[0,1]^2}
	G_{t/4}^{\a+\b+9/2}\Big( \cos\sqrt{F_{\theta,\varphi}(u',v')},1\Big) \, \dd\Pi_{\a+2}(u') \, \dd\Pi_{\b+2}(v').
$$
Further, making use of Lemmas \ref{lem:B} and \ref{lem:C} (notice that $5/2<\a+\b+9/2<7/2$) leads to
$$
\widetilde{J}_1 \simeq \left\{ C^{\bigstar}_{\a+\b+9/2,T/4} 4 L_{\a}L_{\b} \atop c^{\bigstar}_{\a+\b+9/2,T/4} l_{\a}l_{\b}\right\}\,
	\Big(\frac{4}t\Big)^{\a+\b+11/2} \iint_{[0,1]^2} e^{-F_{\theta,\varphi}(u',v')/t}\, \dd\Pi_{\a+2}(u') \, \dd\Pi_{\b+2}(v').
$$
Eventually, applying Lemma \ref{lem:intF} we arrive at
\begin{equation} \label{zx6}
\begin{split}
\widetilde{J}_1 & \simeq \left\{ C^{\bigstar}_{\a+\b+9/2,T/4} 4^{\a+\b+13/2} L_{\a}L_{\b} B_{\a+2}B_{\b+2}
	\Psi^{\mathfrak{b}}_{\a+2}(t,\theta,\varphi) \Psi^{\mathfrak{b}}_{\b+2}(t,\pi-\theta,\pi-\varphi) \atop
	c^{\bigstar}_{\a+\b+9/2,T/4} 4^{\a+\b+11/2} l_{\a}l_{\b} b_{\a+2}b_{\b+2}
	\Psi^{\mathfrak{B}}_{\a+2}(t,\theta,\varphi) \Psi^{\mathfrak{B}}_{\b+2}(t,\pi-\theta,\pi-\varphi) \right\} \\
		& \qquad \times \frac{1}{\sqrt{t}} e^{-(\theta-\varphi)^2/(4t)}.
\end{split}
\end{equation}

Next, we focus on $J_3$. This quantity is like $I_1$ from \eqref{zx0}. The only difference is that in $J_3$ we have
$\dd \Pi_{-1/2}(u)$ instead of $\dd \Pi_{\a}(u)$ (with some $\a \ge -1/2$), but the analysis is exactly the same.
Thus
\begin{equation} \label{zx7}
J_3 = 2(\a+\b+5/2) e^{-t(\a+\b+5/2)/2} \Big[ \cos\frac{\theta}2 \cos\frac{\varphi}2 \Big]^2 \widetilde{J}_3,
\end{equation}
where
\begin{equation*}
\widetilde{J}_3 \simeq \left\{ C^{\bigstar}_{\a+\b+5/2,T/4} 4 L_{\b} \atop c^{\bigstar}_{\a+\b+5/2,T/4} l_{\b} \right\} \,
	\Big( \frac{4}t \Big)^{\a+\b+7/2} \iint_{[0,1]^2} e^{-F_{\theta,\varphi}(u,v')/t}\, \dd\Pi_{-1/2}(u)\, \dd\Pi_{\b+2}(v').
\end{equation*}
Then, by Lemma \ref{lem:intF},
\begin{equation} \label{zx8}
\widetilde{J}_3 \simeq \left\{ C^{\bigstar}_{\a+\b+5/2,T/4} 4^{\a+\b+4} L_{\b} B_{\b+2}
	\Psi^{\mathfrak{b}}_{\b+2}(t,\pi-\theta,\pi-\varphi) \atop c^{\bigstar}_{\a+\b+5/2,T/4} 4^{\a+\b+3} l_{\b} b_{\b+2}
	\Psi^{\mathfrak{B}}_{\b+2}(t,\pi-\theta,\pi-\varphi) \right\} \, \frac{1}{t^{\a+1/2}} \frac{1}{\sqrt{t}} e^{-(\theta-\varphi)^2/(4t)}.
\end{equation}

Finally, bounds for $J_2$ follow from those for $J_3$ by swapping $\a$ with $\b$ and $\theta,\varphi$ with $\pi-\theta,\pi-\varphi$,
respectively. We get
\begin{equation} \label{zx9}
J_2 = 2(\a+\b+5/2) e^{-t(\a+\b+5/2)/2} \Big[ \sin\frac{\theta}2 \sin\frac{\varphi}2 \Big]^2 \widetilde{J}_2,
\end{equation}
where
\begin{equation} \label{zx10}
\widetilde{J}_2 \simeq \left\{ C^{\bigstar}_{\a+\b+5/2,T/4} 4^{\a+\b+4} L_{\a} B_{\a+2}
	\Psi^{\mathfrak{b}}_{\a+2}(t,\theta,\varphi) \atop c^{\bigstar}_{\a+\b+5/2,T/4} 4^{\a+\b+3} l_{\a} b_{\a+2}
	\Psi^{\mathfrak{B}}_{\a+2}(t,\theta,\varphi) \right\} \, \frac{1}{t^{\b+1/2}} \frac{1}{\sqrt{t}} e^{-(\theta-\varphi)^2/(4t)}.
\end{equation}

It is now convenient to consider two cases.\\
\noindent \textbf{Case 1: $\a+\b+1/2 > -1$.} Here we also need to bound the quantity
$$
J_4:= \iint G_{t/4}^{\a+\b+1/2}\Big( \cos\sqrt{F_{\theta,\varphi}(u,v)},1\Big) \, \dd\Pi_{-1/2}(u)\, \dd\Pi_{-1/2}(v)
$$
appearing in Lemma \ref{lem:red_ext}(iv). By monotonicity of $z \mapsto G_{t/4}^{\a+\b+1/2}(z,1)$, see Lemma \ref{lem:diff},
we have
$$
J_4 \simeq \left\{ 1 \atop 1/4\right\} \, G_{t/4}^{\a+\b+1/2}\Big( \cos\sqrt{F_{\theta,\varphi}(1,1)},1\Big).
$$
Then Lemma \ref{lem:D} implies
\begin{equation} \label{zx11}
J_4 \simeq \left\{ C^{\bigstar}_{\a+\b+1/2,T/4} 4^{\a+\b+3/2} \atop c^{\bigstar}_{\a+\b+1/2,T/4} 4^{\a+\b+1/2}\right\}\,
	\frac{1}{t^{\a+1/2}} \frac{1}{t^{\b+1/2}} \frac{1}{\sqrt{t}} e^{-(\theta-\varphi)^2/(4t)}.
\end{equation}

In view of Lemma \ref{lem:red_ext}(iv), see also \eqref{zx5}, \eqref{zx7} and \eqref{zx9},
\begin{align} \nonumber
\frac{h_0^{\a,\b}}{h_0^{\a+\b+1/2}} G_t^{\a,\b}(\cos\theta,\cos\varphi) & =
	\frac{1}2 q_{\a+\b+1/2}(t) J_1 + 2(\a+\b+3/2) e^{-t(\a+\b+3/2)/2}(J_2+J_3) + J_4 \\ \label{zx14}
& = 4 q_{\a+\b+1/2}(t) q_{\a+\b+5/2}(t) \Big[ \sin\frac{\theta}2\sin\frac{\varphi}2\Big]^2
	\Big[ \cos\frac{\theta}2\cos\frac{\varphi}2\Big]^2 \widetilde{J}_1 \\
& \quad + 2 q_{\a+\b+1/2}(t) \bigg( \Big[ \sin\frac{\theta}2\sin\frac{\varphi}2\Big]^2 \widetilde{J}_2
	+ \Big[ \cos\frac{\theta}2\cos\frac{\varphi}2\Big]^2 \widetilde{J}_3 \bigg) + J_4. \nonumber
\end{align}
Taking into account the above and \eqref{zx6}, \eqref{zx8}, \eqref{zx10} and \eqref{zx11},
we see that in order to conclude the desired bounds for the Jacobi heat kernel it is necessary to estimate suitably an auxiliary
expression
\begin{align*}
\Upsilon & := c_{1,1} O_1^{\kappa,\a}(t,\theta,\varphi) O_1^{\kappa,\b}(t,\pi-\theta,\pi-\varphi)
	+ c_{1,2} O_1^{\kappa,\a}(t,\theta,\varphi) O_2^{\kappa,\b}(t,\pi-\theta,\pi-\varphi) \\
	& \quad + c_{2,1} O_2^{\kappa,\a}(t,\theta,\varphi) O_1^{\kappa,\b}(t,\pi-\theta,\pi-\varphi)
	+ c_{2,2} O_2^{\kappa,\a}(t,\theta,\varphi) O_2^{\kappa,\b}(t,\pi-\theta,\pi-\varphi),
\end{align*}
where $c_{1,1},c_{1,2},c_{2,1},c_{2,2}$ and $\kappa$ are certain positive constants and $O_1^{\kappa,\a}$, $O_2^{\kappa,\a}$
are defined in \eqref{zxO1} and \eqref{zxO2}. The observations that led us to \eqref{zx4} now imply the bounds
\begin{align*}
\Upsilon & \ge \frac{c_{1,1}}{\pi^8\kappa^4} \Big(\frac{\mathbb{D}_{\a}}{\mathbb{D}_{\a+2}}\Big)^2
	\Big(\frac{\mathbb{D}_{\b}}{\mathbb{D}_{\b+2}}\Big)^2
	\wedge \frac{c_{1,2}}{\pi^4\kappa^2} \Big(\frac{\mathbb{D}_{\a}}{\mathbb{D}_{\a+2}}\Big)^2 \mathbb{D}_{\b}^{\b+1/2}
	\wedge \frac{c_{2,1}}{\pi^4\kappa^2} \Big(\frac{\mathbb{D}_{\b}}{\mathbb{D}_{\b+2}}\Big)^2 \mathbb{D}_{\a}^{\a+1/2}
	\wedge c_{2,2} \mathbb{D}_{\a}^{\a+1/2} \mathbb{D}_{\b}^{\b+1/2}, \\
\Upsilon & \le \frac{c_{1,1}}{256 \kappa^4} + \frac{c_{1,2}}{16\kappa^2} \mathbb{D}_{\b}^{\b+1/2}
	+ \frac{c_{2,1}}{16\kappa^2}\mathbb{D}_{\a}^{\a+1/2} + {c_{2,2}} \mathbb{D}_{\a}^{\a+1/2} \mathbb{D}_{\b}^{\b+1/2}.
\end{align*}
It is now enough to notice, see \eqref{zx14} and \eqref{zx6}, \eqref{zx8}, \eqref{zx10}, \eqref{zx11},
that the lower bound of the lemma follows by taking
\begin{align*}
c_{1,1} & = 4 q_{\a+\b+1/2}(T) q_{\a+\b+5/2}(T) c^{\bigstar}_{\a+\b+9/2,T/4} 4^{\a+\b+11/2} l_{\a}l_{\b} b_{\a+2} b_{\b+2},\\
c_{1,2} & = 2 q_{\a+\b+1/2}(T) c^{\bigstar}_{\a+\b+5/2,T/4} 4^{\a+\b+3} l_{\a} b_{\a+2}, \\
c_{2,1} & = 2 q_{\a+\b+1/2}(T) c^{\bigstar}_{\a+\b+5/2,T/4} 4^{\a+\b+3} l_{\b} b_{\b+2}, \\
c_{2,2} & = c^{\bigstar}_{\a+\b+1/2,T/4} 4^{\a+\b+1/2},
\end{align*}
while to get the upper bound one takes
\begin{align*}
c_{1,1} & = 4 q_{\a+\b+1/2}(0) q_{\a+\b+5/2}(0) C^{\bigstar}_{\a+\b+9/2,T/4} 4^{\a+\b+13/2} L_{\a}L_{\b} B_{\a+2} B_{\b+2},\\
c_{1,2} & = 2 q_{\a+\b+1/2}(0) C^{\bigstar}_{\a+\b+5/2,T/4} 4^{\a+\b+4} L_{\a} B_{\a+2}, \\
c_{2,1} & = 2 q_{\a+\b+1/2}(0) C^{\bigstar}_{\a+\b+5/2,T/4} 4^{\a+\b+4} L_{\b} B_{\b+2}, \\
c_{2,2} & = C^{\bigstar}_{\a+\b+1/2,T/4} 4^{\a+\b+3/2}.
\end{align*}

The conclusion follows.

\noindent \textbf{Case 2: $\a+\b+1/2 \le -1$.} Here we need to bound the additional quantity
$$
J_5:= \iint H_{t/4}^{\a+\b+1/2}\Big( \cos\sqrt{F_{\theta,\varphi}(u,v)}\Big) \, \dd\Pi_{-1/2}(u)\, \dd\Pi_{-1/2}(v)
$$
appearing in Lemma \ref{lem:red_ext}(v). Since $z\mapsto H_{t/4}^{\a+\b+1/2}(z)$ is even and increasing in $[0,1]$
(see Lemma~\ref{lem:diffH}), we have
$$
J_5 = \frac{1}2 \bigg[ H_{t/4}^{\a+\b+1/2}\Big(\cos\frac{\theta-\varphi}2\Big) + 
	 H_{t/4}^{\a+\b+1/2}\Big(\cos\frac{\theta+\varphi}2\Big) \bigg]
	\simeq \left\{ 2 \atop 1 \right\} \, \frac{1}2 H_{t/4}^{\a+\b+1/2}\Big(\cos\frac{\theta-\varphi}2\Big).
$$
Then by Lemma \ref{lem:E}
\begin{equation} \label{zx12}
J_5 \simeq \left\{ C^{\bigstar}_{\a+\b+1/2,T/4} 4^{\a+\b+3/2} \atop c^{\bigstar}_{\a+\b+1/2,T/4} 4^{\a+\b+1} \right\} \,
	\frac{1}{t^{\a+1/2}}\frac{1}{t^{\b+1/2}} \frac{1}{\sqrt{t}} e^{-(\theta-\varphi)^2/(4t)}.
\end{equation}

In view of Lemma \ref{lem:red_ext}(v),
\begin{align*}
\frac{h_0^{\a,\b}\,\Gamma(\a+\b+2)}{\sqrt{\pi}\,\Gamma(\a+\b+5/2)} G_t^{\a,\b}(\cos\theta,\cos\varphi) & =
	\frac{1}2 \widetilde{q}_{\a+\b+1/2}(t) J_1 + 2 e^{-t(\a+\b+3/2)/2}(J_2+J_3) + J_5 \\
& = 4 \widetilde{q}_{\a+\b+1/2}(t) q_{\a+\b+5/2}(t) \Big[ \sin\frac{\theta}2\sin\frac{\varphi}2\Big]^2
	\Big[ \cos\frac{\theta}2\cos\frac{\varphi}2\Big]^2 \widetilde{J}_1 \\
& \quad + 2\widetilde{q}_{\a+\b+1/2}(t) \bigg( \Big[ \sin\frac{\theta}2\sin\frac{\varphi}2\Big]^2 \widetilde{J}_2
	\!+ \Big[ \cos\frac{\theta}2\cos\frac{\varphi}2\Big]^2 \widetilde{J}_3 \bigg) + J_5.
\end{align*}
Proceeding as in Case 1, but using now \eqref{zx12} instead of \eqref{zx11}, we conclude for $\a$ and $\b$ under consideration
the estimates for $G_t^{\a,\b}(\cos\theta,\cos\varphi)$ stated in Lemma \ref{lem:F}.

The proof of Lemma \ref{lem:F} is complete.
\end{proof}

\subsection{Refinements in special cases} \, \label{sec:ref}

\medskip

In this section we provide refinements of the bounds from Lemma \ref{lem:F} for $\varphi=0$ in the cases when
$\a,\b \ge -1/2$ and $\a+\b \in \mathbb{N}-1$ or $\a \in \mathbb{N}$ and $\b=-1/2$.
These special cases are of interest because of their relevance for
estimates of the spherical heat kernel and, more generally, heat kernels on compact rank one symmetric spaces.

\begin{biglem} \label{lem:G}
Let $\a,\b \ge -1/2$ be such that $\a+\b \in \mathbb{N}-1$ and let $T > 0$. Then
$$
G_t^{\a,\b}(\cos\theta,1) \simeq
	\left\{ C^{\textrm{ref}}_{\a,\b,T} \Psi_{\b}^{\mathfrak{b}}(t,\pi-\theta,\pi)  \atop
		c^{\textrm{ref}}_{\a,\b,T} \Psi_{\b}^{\mathfrak{B}}(t,\pi-\theta,\pi) \right\} \,
		\frac{1}{{t}^{\a+1}} \exp\bigg({-\frac{\theta^2}{4t}}\bigg)
$$
for $0 \le \theta \le \pi$ and $0 < t \le T$, with the constants
\begin{align*}
c^{\textrm{ref}}_{\a,\b,T} & := c^A_{\a+\b+1/2,T/4} \times 4^{\a+\b+3/2}b_{\b} \times \frac{h_0^{\a+\b+1/2}}{h_0^{\a,\b}},\\
C^{\textrm{ref}}_{\a,\b,T} & := C^A_{\a+\b+1/2,T/4} \times 4^{\a+\b+2}B_{\b} \times \frac{h_0^{\a+\b+1/2}}{h_0^{\a,\b}}.
\end{align*}
\end{biglem}

\begin{proof}
As in the proof of Lemma \ref{lem:B}, by Lemma \ref{lem:red_ext}(i) we have
\begin{align*}
\frac{h_0^{\a,\b}}{h_0^{\a+\b+1/2}} G_t^{\a,\b}(\cos\theta,1)
	& = \int G_{t/4}^{\a+\b+1/2}\Big(\cos\sqrt{F_{\theta,0}(1,v)},1\Big)
	\, \dd\Pi_{\b}(v) \\
& \simeq \left\{ 2 \atop 1 \right\} \int_{[0,1]} {G}_{t/4}^{\a+\b+1/2} \Big(\cos\sqrt{F_{\theta,0}(1,v)},1\Big) \, \dd\Pi_{\b}(v).
\end{align*}
Applying now Lemma \ref{lem:A}, which is legitimate due to the assumptions on $\a$ and $\b$, we see that
$$
\frac{h_0^{\a,\b}}{h_0^{\a+\b+1/2}} G_t^{\a,\b}(\cos\theta,1) \simeq
	\left\{ 2 C^A_{\a+\b+1/2,T/4} \atop c^A_{\a+\b+1/2,T/4} \right\} \Big(\frac{4}t\Big)^{\a+\b+3/2}
		\int_{[0,1]} e^{-F_{\theta,0}(1,v)/t} \, \dd\Pi_{\b}(v).
$$
The last integral can be estimated by means of Lemma \ref{lem:intF0}. This gives
$$
\frac{h_0^{\a,\b}}{h_0^{\a+\b+1/2}} G_t^{\a,\b}(\cos\theta,1) \simeq
	\left\{ 2 C^A_{\a+\b+1/2,T/4} B_{\b} \Psi_{\b}^{\mathfrak{b}}(t,\pi-\theta,\pi) \atop
		c^A_{\a+\b+1/2,T/4} b_{\b} \Psi_{\b}^{\mathfrak{B}}(t,\pi-\theta,\pi)\right\} 
		\frac{4^{\a+\b+3/2}}{t^{\a+1}} \, e^{-\theta^2/(4t)}
$$
for $\theta \in [0,\pi]$ and $0 < t \le T$. The lemma follows.
\end{proof}

\begin{biglem} \label{lem:H}
Let $\a \in \mathbb{N}$ and let $T > 0$. Then
$$
G_t^{\a,-1/2}(\cos\theta,1) \simeq \left\{ C^{\textrm{ref}}_{\a,T} \atop c^{\textrm{ref}}_{\a,T} \right\} \,
		\frac{1}{{t}^{\a+1}} \exp\bigg({-\frac{\theta^2}{4t}}\bigg)
$$
for $0 \le \theta \le \pi$ and $0 < t \le T$, with the constants
\begin{align*}
c^{\textrm{ref}}_{\a,T} & := c^A_{2\a+1/2,T/16} \times 2\Big(\frac{32}{\mathbb{D}_{\a}T/4\vee\mathfrak{B}\pi^2}\Big)^{\a+1/2} b_{\a}
	\times \frac{h_0^{2\a+1/2}}{h_0^{\a}},\\
C^{\textrm{ref}}_{\a,T} & := C^A_{2\a+1/2,T/16} \times 8 \Big(\frac{64}{\mathfrak{b}\pi^2}\Big)^{\a+1/2} B_{\a}
	\times \frac{h_0^{2\a+1/2}}{h_0^{\a}}.
\end{align*}
\end{biglem}

\begin{proof}
From \cite[Lem.\,3.4]{NSS2} we know that for $x,y \in [-1,1]$ and $t > 0$
$$
G_t^{\a,-1/2}\big( 2x^2-1, 2y^2-1 \big) = 2^{-\a-3/2} \big[ G_{t/4}^{\a}(x,y) + G_{t/4}^{\a}(-x,y) \big].
$$
This implies
$$
G_t^{\a,-1/2}(\cos\theta,1) = \frac{1}{2^{\a+3/2}} \Big[ G_{t/4}^{\a}\Big(\cos\frac{\theta}2,1\Big)
	+ G_{t/4}^{\a}\Big(-\cos\frac{\theta}2,1\Big) \Big], \qquad \theta \in [0,\pi].
$$
Using monotonicity of $z \mapsto G_{t/4}^{\a}(z,1)$, see Lemma \ref{lem:diff}, we infer that
$$
G_t^{\a,-1/2}(\cos\theta,1) \simeq \left\{ 2 \atop 1 \right\} \, \frac{1}{2^{\a+3/2}} G_{t/4}^{\a}\Big(\cos\frac{\theta}2,1\Big).
$$
Now Lemma \ref{lem:G} can be applied and we get
$$
G_t^{\a,-1/2}(\cos\theta,1) \simeq \left\{ 2 C^{\textrm{ref}}_{\a,\a,T/4} \Psi_{\a}^{\mathfrak{b}}(t/4,\pi-\theta/2,\pi)
	\atop c^{\textrm{ref}}_{\a,\a,T/4} \Psi_{\a}^{\mathfrak{B}}(t/4,\pi-\theta/2,\pi) \right\} \,
		\frac{1}{2^{\a+3/2}} \Big( \frac{4}{t}\Big)^{\a+1} e^{-\theta^2/(4t)}
$$
for $\theta \in [0,\pi]$ and $0 < t \le T$. Estimating
$$
 \Psi_{\a}^{\mathfrak{b}}(t/4,\pi-\theta/2,\pi) \le \big( \mathfrak{b}\pi^2/2\big)^{-\a-1/2}, \qquad
 \Psi_{\a}^{\mathfrak{B}}(t/4,\pi-\theta/2,\pi) \ge \big[ \mathbb{D}_{\a} T/4 \vee \mathfrak{B}\pi^2 \big]^{-\a-1/2},
$$
and inserting the formulas for $C^{\textrm{ref}}_{\a,\a,T/4}$ and $c^{\textrm{ref}}_{\a,\a,T/4}$ leads to the desired conclusion.
\end{proof}

\section{Large and medium time bounds} \label{sec:lmtime}

The unified proof from Section \ref{sec:uproof} provides sharp small time estimates for the Jacobi heat kernel,
with complete trace of multiplicative constants. In this section we complement that result by discussing medium and
large time bounds.

Let $\a,\b > -1$ and suppose that the sharp small time bounds from Lemma \ref{lem:F} hold for some fixed $t=t_0=T > 0$.
Then it is straightforward to see that
\begin{equation} \label{Lee}
G_{t_0}^{\a,\b}(x,y) \simeq \left\{ \texttt{C}_{\a,\b,t_0} \atop \texttt{c}_{\a,\b,t_0} \right\}, \qquad x,y \in [-1,1], 
\end{equation}
where the positive constants $\texttt{C}_{\a,\b,t_0}$ and $\texttt{c}_{\a,\b,t_0}$ express explicitly in terms of the
constants and other quantities appearing in the bounds of Lemma \ref{lem:F}.

On the other hand, if \eqref{Lee} holds for a fixed $t_0 > 0$ ($t_0$ presumably depending on $\a$ and $\b$), then also
\begin{equation*}
G_t^{\a,\b}(x,y) \simeq \left\{ \texttt{C}_{\a,\b,t_0} \atop \texttt{c}_{\a,\b,t_0} \right\},
	\qquad x,y \in [-1,1], \quad \textrm{all}\;\; t \ge t_0,
\end{equation*}
with the same constants. This is justified in the following well-known way.

Let $t > t_0$. By the semigroup property
$$
G_t^{\a,\b}(x,y) = \int_{-1}^1 G_{t-t_0}^{\a,\b}(x,z)\, G_{t_0}^{\a,\b}(z,y) \, \dd\varrho_{\a,\b}(z). 
$$
Therefore, in view of \eqref{Lee},
$$
G_t^{\a,\b}(x,y) \simeq \left\{ \texttt{C}_{\a,\b,t_0} \atop \texttt{c}_{\a,\b,t_0} \right\}
	\int_{-1}^1 G_{t-t_0}^{\a,\b}(x,z)\, \dd\varrho_{\a,\b}(z).
$$
Now it remains to observe that the last integral is equal 1 since the Jacobi semigroup, being Markovian,
is in particular conservative.

Thus sharp medium and large time bounds for $G_t^{\a,\b}(x,y)$ are in the above sense covered by the sharp small time bounds.
The following statement is an instant consequence of the above considerations.
\begin{prop} \label{prop:large}
Let $\a,\b > -1$. Assume that the bounds of Lemma \ref{lem:F} hold with some $t_0=T>0$, that is
$$
G_{t_0}^{\a,\b}(\cos\theta,\cos\varphi) \simeq
	\left\{ C_{\a,\b,t_0} \Psi_{\a}^{\mathfrak{b}}(t_0,\theta,\varphi) \Psi_{\b}^{\mathfrak{b}}(t_0,\pi-\theta,\pi-\varphi) \atop
		c_{\a,\b,t_0} \Psi_{\a}^{\mathfrak{B}}(t_0,\theta,\varphi) \Psi_{\b}^{\mathfrak{B}}(t_0,\pi-\theta,\pi-\varphi) \right\} \,
		\frac{1}{\sqrt{t_0}} \exp\bigg({-\frac{(\theta-\varphi)^2}{4t_0}}\bigg)
$$
for $\theta,\varphi \in [0,\pi]$. Then, with the same constants $C_{\a,\b,t_0}$ and $c_{\a,\b,t_0}$,
\begin{align*}
G_t^{\a,\b}(x,y) & \simeq \left\{ C_{\a,\b,t_0} \big(1 \vee [1\vee 2/(\mathbb{D}_{\a} t_0)]^{-\a-1/2}\big)
	\big(1 \vee [1\vee 2/(\mathbb{D}_{\b} t_0)]^{-\b-1/2}\big) \atop
		c_{\a,\b,t_0} \big(1 \wedge [1\vee \pi^2/(2\mathbb{D}_{\a} t_0)]^{-\a-1/2}\big)
	\big(1 \wedge [1\vee \pi^2/(2\mathbb{D}_{\b} t_0)]^{-\b-1/2}\big) e^{-\pi^2/(4t_0)} \right\} \\
& \qquad \times \frac{1}{t_0^{\a+\b+3/2} \Gamma(\a+3/2) \Gamma(\b+3/2)}
\end{align*}
for $x,y \in [-1,1]$ and all $t \ge t_0$.
\end{prop}
The bounds of Proposition \ref{prop:large} can be simplified by distinguishing cases when $\a$ and/or $\b$
are below $-1/2$ or not. Then estimates with the aid of \eqref{DDD} lead to further simplifications.
The details are straightforward and left to interested readers.

Note that with $t_0 \sim 1/(\a+\b+3)$ the ratio of the upper and lower bounds from Proposition \ref{prop:large} grows
no faster than exponentially in $\a$ and $\b$, assuming that the ratio $C_{\a,\b,t_0}/c_{\a,\b,t_0}$ satisfies the same property.
As for the latter, this is indeed the case, at least when $\a,\b \ge -1/2$; see Section~\ref{ssec:Jacobi}.

\medskip

The above deliberations, however, may not provide satisfactorily precise estimates for large times.
Fortunately, for large $t$ the bounds resulting from Proposition \ref{prop:large}
can substantially be refined by exploiting the (oscillating) series representation of $G_t^{\a,\b}(x,y)$.
This will be our objective for the remaining part of Section~\ref{sec:lmtime}.

In what follows we always assume $\a,\b > -1$ and $n \ge 0$; these assumptions can locally be strengthened.
Recall that
\begin{equation} \label{skr}
G_t^{\a,\b}(x,y) = \sum_{n=0}^{\infty} e^{-t n(n+\a+\b+1)} \frac{P_n^{\a,\b}(x) P_n^{\a,\b}(y)}{h_n^{\a,\b}},
\end{equation}
where
$$
h_n^{\a,\b} = \frac{2^{\a+\b+1}\Gamma(n+\a+1)\Gamma(n+\b+1)}{(2n+\a+\b+1) \Gamma(n+\a+\b+1)\Gamma(n+1)},
$$
with appropriate convention when $n=0$ and $\a+\b+1=0$. In particular,
$$
h_0^{\a,\b} = \frac{2^{\a+\b+1}\Gamma(\a+1)\Gamma(\b+1)}{\Gamma(\a+\b+2)}.
$$
By combining \eqref{skr} with standard bounds for Jacobi polynomials it can be easily concluded that
$$
G_t^{\a,\b}(x,y) \to \frac{1}{h_0^{\a,\b}}, \qquad t \to \infty,
$$
and the convergence is uniform in $x,y \in [-1,1]$.
Our aim is to make this convergence quantitative in terms of suitable estimates.

From \eqref{skr} and the identity $n(n+\a+\b+1) - (\a+\b+2) = (n-1)(n+\a+\b+2)$ it follows that
$$
G_t^{\a,\b}(x,y) = \frac{1}{h_0^{\a,\b}}\Big[1+E^{\al,\be}_t(x,y)\Big],
$$
where
$$
E^{\al,\be}_t(x,y) :=
e^{-t(\a+\b+2)} \Bigg( \frac{h_0^{\al,\be}}{h_1^{\a,\b}} P_1^{\a,\b}(x) P_1^{\a,\b}(y)
	+ \sum_{n=2}^{\infty} e^{-t(n-1)(n+\a+\b+2)} \frac{h_0^{\al,\be}}{h_n^{\a,\b}} P_n^{\a,\b}(x) P_n^{\a,\b}(y) \Bigg).
$$
We now give an upper bound on $|E^{\al,\be}_t(x,y)|$ that is independent of $x,y \in [-1,1]$,
holds for $t\geq T$ (where $T = 2 \log 2 \approx 1.386294$ is a preselected threshold), and vanishes exponentially as $t \to \infty$.
This result of course implies precise sharp bounds for the Jacobi heat kernel provided that $t$ is sufficiently large.
It also implies Theorem \ref{thm:mainL}.
\begin{thm} \label{thm:large}
	Let $\a,\b > -1$ and $x,y \in [-1,1]$. Assume that $t\geq 2\log 2$. Then
	$$
		\big|E^{\al,\be}_t(x,y)\big|\leq \frac{ 2}{\al\wedge\be+1} e^{-(t-1)(\a+\b+2)} \quad \textrm{when} \quad \al\vee\be\ge-1/2
	$$
	and
	$$
		\big|E^{\al,\be}_t(x,y)\big|\leq \frac{27/10}{(\al+1)(\be+1)} e^{-(t-1)(\a+\b+2)} \quad \textrm{when} \quad \al\vee\be < -1/2.
	$$
\end{thm}

We point out the following simple consequence of Theorem \ref{thm:large}.
\begin{cor} \label{cor:large}
	Let $\a,\b > -1$ and $x,y \in [-1,1]$. Then
	\begin{equation} \label{bd75}
		G_t^{\al,\be}(x,y)\simeq \left\{ {3/2 \atop 1/2} \right\} \frac{1}{h_0^{\al,\be}},\qquad t \ge T_{\al,\be} \vee 2\log 2,
	\end{equation}
	where
	$$
	 T_{\al,\be}= 1+ \frac{1}{\al+\be+2} \log \bigg(\frac{4}{\al\wedge\be+1}\bigg)  \quad \textrm{when} \quad \al\vee\be\geq -1/2
	$$
	and
	$$
	T_{\al,\be}= 1+ \frac{1}{\al+\be+2} \log \bigg(\frac{27/5}{(\al+1)(\be+1)}\bigg) \quad \textrm{when} \quad \al\vee\be < -1/2.
	$$
\end{cor}

Notice that if $\al,\be\geq -1/2$ then $T_{\a,\b} \le 1+\log 8$, so in this case the bounds \eqref{bd75} hold in particular
	for $t \ge 1 + \log 8 \approx 3.079442$.

Proving Theorem \ref{thm:large} requires some preparations.
To begin with, note that (cf.\ \cite[Chap.\,IV, Sec.\,4.5, (4.5.1)]{Sz})
$$
P_1^{\a,\b}(x) = \frac{1}2 (\a+\b+2)x + \frac{1}{2} (\a-\b),
$$
so
$$
\max_{x \in [-1,1]} |P_1^{\a,\b}(x)| = \max\big\{ P_1^{\a,\b}(1), - P_1^{\a,\b}(-1) \big\} = \a \vee \b + 1.
$$
Thus, for $x,y \in [-1,1]$,
\begin{equation}\label{eq:P1}
	\bigg| \frac{h_0^{\al,\be}}{h_1^{\a,\b}} P_1^{\a,\b}(x) P_1^{\a,\b}(y)\bigg| 
	\leq \frac{(\al+\be+3)(\a \vee \b +1)^2}{(\a+1)(\b+1)} = \frac{(\al+\be+3) (\al\vee\be+1)}{\al\wedge\be+1}. 
\end{equation}

For general $n\geq 0$ we have (cf.\ \cite[Chap.\,VII, Sec.\,7.32, Thm.\,7.32.1]{Sz})
\begin{equation} \label{ePbab}
\max_{x \in [-1,1]} |P_n^{\a,\b}(x)| = \max\big\{P_n^{\a,\b}(1),|P_n^{\a,\b}(-1)|\big\} = \binom{n+ \a \vee \b}{n} \quad
	\textrm{when} \;\; \a \vee \b \ge -1/2.
\end{equation}
In case $\a \vee \b < -1/2$, there seems to be no explicit formula known for the maximum, but it is known that its order
of magnitude is $\sim n^{-1/2}$.
The corresponding quantitative bound in terms of $\a$ and $\b$ does not seem to be available in the literature.
Since the range $\a,\b < -1/2$ is not crucial for applications, we afford non-optimal (in the sense of the dependence on $n$)
but relatively simple estimates, which are stated in the following lemma.

\begin{lm} \label{lem:ePsab}
Let $\a,\b \in (-1,-1/2)$. Then
$$
\max_{x\in [-1,1]} |P_n^{\a,\b}(x)| \le \frac{2n + \a \wedge \b +1}{\a \wedge \b + 1} \binom{n+\a \wedge \b}{n}, \qquad n \ge 0.
$$
\end{lm}

\begin{proof}
The case $n = 0$ is trivial, so assume $n \ge 1$.
Using the identity\,\footnote{
\,This identity can be justified, for instance, by combining \cite[Chap.\,18, $\S$18.9, Form.\,18.9.5]{DLMF}
with the well-known formula $P_n^{\b,\a} (-x) = (-1)^{n} P_n^{\a,\b} (x)$, see \cite[Chap.\,18, $\S$18.6(i)]{DLMF}.
}
$$
(2n + \a + \b + 1) P_n^{\a,\b} = (n+\a+\b+1) P_n^{\a+1,\b} - (n + \b)P_{n-1}^{\a+1,\b}
$$
we can write
$$
|P_n^{\a,\b}(x)| \le \frac{n+\a+\b+1}{2n+\a+\b+1} |P_n^{\a+1,\b}(x)| + \frac{n+\b}{2n + \a+\b+1}|P_{n-1}^{\a+1,\b}(x)|
\le |P_n^{\a+1,\b}(x)| + |P_{n-1}^{\a+1,\b}(x)|.
$$
Combining this with \eqref{ePbab} gives
$$
\max_{x\in [-1,1]} |P_n^{\a,\b}(x)| \le \binom{n+\a+1}{n} + \binom{n-1 + \a+1}{n-1}.
$$
For symmetry reasons, we also have
$$
\max_{x\in [-1,1]} |P_n^{\a,\b}(x)| \le \binom{n+\b+1}{n} + \binom{n-1 + \b+1}{n-1}.
$$

Now we can choose the tighter of these two bounds, see \eqref{binin}, which leads to
$$
\max_{x\in [-1,1]} |P_n^{\a,\b}(x)| \le \binom{n+\a\wedge \b +1}{n} + \binom{n-1 + \a\wedge \b+1}{n-1}
	= \frac{2n+\a\wedge \b+1}{\a \wedge \b + 1} \binom{n+\a \wedge \b}{n}.
$$
This is the bound claimed in the lemma.
\end{proof}

Next, we estimate for $n \ge 2$ the quantity
\begin{equation*}
	K_n^{\a,\b} := \frac{h_0^{\al,\be}}{h_n^{\a,\b}} \bigg( \max_{x\in [-1,1]} |P_n^{\a,\b}(x)|\bigg)^2.
\end{equation*}

\begin{lm}\label{lm:1}
	Let $\a,\b > -1$ and $n \ge 2$. Then
$$
	K^{\al,\be}_n \leq \frac{7\sqrt{6}\, e^{1/30}\, \Gamma(\al\wedge\be+1) }{3\sqrt{5}\,\Gamma(\al+\be+2)\Gamma(\a \vee \b +1)}
		(n+\a+\b+2)^{2(\a \vee \b)+1} \quad \textrm{when} \quad \a \vee \b \ge -1/2
$$
and
$$
K^{\al,\be}_n \leq \frac{18 \Gamma(\al+1) \Gamma(\be+1)}{\Gamma(\al+\be+2)}  (n+\a+\b+2)^{2(\a \wedge \b) + 3}
	\quad \textrm{when} \quad \a \vee \b < -1/2.
$$
\end{lm}

In the proof we shall apply \eqref{Gest}, which together with the bound $1+A \le e^{A}$, $A \in \RR$, implies
\begin{equation}\label{Gest2}
	\frac{\Gamma(x)}{\Gamma(y)} < e^{1/(12x)}   \sqrt{\frac{y}{x}}\Big(1+\frac{x-y}{y}\Big)^y e^{-(x-y)} x^{x-y}
	\le  e^{1/(12x)}   \sqrt{\frac{y}{x}} x^{x-y},\qquad x,y>0.
\end{equation}
For $x>y$ we shall sometimes use \eqref{GGin} instead.

\begin{proof}[{Proof of Lemma \ref{lm:1}}] Throughout the proof we always assume $n \ge 2$.

	Consider first the case $\a \vee \b \ge -1/2$. By \eqref{ePbab} we obtain
	\begin{align*}
		\frac{K_n^{\a,\b}}{h_0^{\al,\be}} & = \frac{1}{h_n^{\a,\b}} \binom{n+\a \vee \b}{n}^2 \\
		& = \frac{2^{-\al-\be-1}(2n+\a+\b+1)\Gamma(n+\a+\b+1)\Gamma(n+1)}{\Gamma(n+\a+1)\Gamma(n+\b+1)}
		\bigg( \frac{\Gamma(n+ \a \vee \b + 1)}{\Gamma(n+1)\Gamma(\a \vee \b +1)}\bigg)^2 \\
		& \le \frac{7\cdot 2^{-\al-\be-1}}{3\Gamma(\a \vee \b +1)^2} \frac{\Gamma(n+\a\vee\b +1)}{\Gamma(n + \a\wedge \b+1)}
		\frac{\Gamma(n+\a+\b+2)}{\Gamma(n+1)},
	\end{align*}
	where the inequality emerges from estimating the factor $(2n + \a+\b+1)$ by $7(n+\a+\b+1)/3$.
	Using \eqref{GGin} to the second factor above and \eqref{Gest2} to the third one, we further estimate
	\begin{align*}
		\frac{K_n^{\a,\b}}{h_0^{\al,\be}} & \le \frac{7\cdot 2^{-\al-\be-1}}{3\Gamma(\a \vee \b +1)^2}
		(n+\a \vee \b +1)^{|\a-\b|}\, \frac{\sqrt{n+1}\, e^{1/[12(n+\al+\be+2)]}}{\sqrt{n+\al+\be+2}} (n+\a+\b+2)^{\a+\b+1} \\
		& \le \frac{7\sqrt{6}\,  e^{1/30}\, 2^{-\al-\be-1}}{3\sqrt{5}\,\Gamma(\a \vee \b +1)^2} (n+\a+\b+2)^{2(\a \vee \b)+1},
	\end{align*}
	where we applied the relations $|\a-\b| = \a \vee \b - \a \wedge \b$, $\a+\b + |\a-\b|=2(\a \vee \b)$ and
	$n + \a \vee \b +1 \le n + \a + \b +2$. Thus,
	\begin{equation*}
		K_n^{\al,\be} \leq \frac{7\sqrt{6}\, e^{1/30} \,
		\Gamma(\al\wedge\be+1) }{3\sqrt{5}\,\Gamma(\al+\be+2)\Gamma(\a \vee \b +1)} (n+\a+\b+2)^{2(\a \vee \b)+1}.
	\end{equation*}
		
	Next, we move to the case $\a \vee \b < -1/2$, i.e.\ $\a,\b \in (-1,-1/2)$. By Lemma \ref{lem:ePsab}
	\begin{multline*}
		\frac{K_n^{\a,\b}}{h_0^{\al,\be}}
		 \le \frac{1}{h_n^{\a,\b}} \bigg( \frac{2n+\a \wedge \b +1}{\a \wedge \b +1}\bigg)^2
		\binom{n+ \a \wedge \b}{n}^2 \\
		 = \frac{(2n+\a+\b+1) \Gamma(n+\a + \b +1) \Gamma(n+1)}{2^{\a+\b+1}\Gamma(n+\a+1)\Gamma(n+\b+1)}
		\bigg( \frac{2n+\a \wedge \b +1}{\a \wedge \b +1}\bigg)^2
		\bigg( \frac{\Gamma(n +\a \wedge \b +1)}{\Gamma(n+1)\Gamma(\a \wedge \b +1)} \bigg)^2.
	\end{multline*}
	Here we use the bound $2n+\a+\b+1 \le 3(n+\a+\b+1)$ getting
	\begin{align*}
		\frac{K_n^{\a,\b}}{h_0^{\al,\be}}
		 & \le \frac{3\cdot 2^{-\a-\b-1} }{\Gamma(\a \wedge \b +2)^2} \frac{\Gamma(n+\a+\b+2)}{\Gamma(n+1)}
		\frac{\Gamma(n+\a \wedge \b +1)^2}{\Gamma(n+\a+1)\Gamma(n+\b+1)} (2n + \a \wedge \b +1)^2 \\
		& = \frac{3\cdot 2^{-\a-\b-1}}{\Gamma(\a \wedge \b +2)^2} \frac{\Gamma(n+\a+\b+2)}{\Gamma(n+\a \vee \b +1)}
		\frac{\Gamma(n+\a \wedge \b +1)}{\Gamma(n)} \frac{2n + \a \wedge \b +1}{n} (2n+ \a \wedge \b +1).
	\end{align*}
	In the above expression we bound the last two factors by $9/4$ and $2(n + \a \wedge \b +1)$, respectively.
	For two other factors, we apply \eqref{GGin}. This produces
	\begin{align*}
		\frac{K_n^{\a,\b}}{h_0^{\al,\be}}
		 & \le \frac{27\cdot 2^{-\a-\b-1} }{2 \Gamma(\a \wedge \b + 2)^2} (n+\a+\b+2)^{\a+\b-\a \vee \b +1}
		(n + \a \wedge \b +1)^{\a \wedge \b +1} (n + \a \wedge \b +1) \\
		& \le \frac{ 27\cdot 2^{-\a-\b-1} }{ 2\Gamma(\a \wedge \b + 2)^2} (n+\a+\b+2)^{2(\a \wedge \b) + 3}.
	\end{align*}
	Since for $\a,\b \in (-1,-1/2)$ one has $\Gamma(\a \wedge \b + 2) > 0.88$ (see \cite[Chap.\,5, \S 5.4(iii)]{DLMF}), we have
	$27/[2\Gamma(\a \wedge \b +1)^2] < 18$. Therefore,
	$$
	K_n^{\a,\b} \le \frac{18\Gamma(\al+1) \Gamma(\be+1)}{\Gamma(\al+\be+2)} (n+\a+\b+2)^{2(\a \wedge \b) + 3},
	$$
	as desired.
\end{proof}

Another ingredient needed to prove Theorem \ref{thm:large} is the following technical result.
\begin{lm}\label{lm:2}
	Let $\delta\geq0$, $\zeta\geq\zeta_0\geq 0$ and $t\geq T>0$ be such that $\zeta\le 2t(\delta+1)$. Then
	\begin{equation*}
		\sum_{n=2}^\infty e^{-tn(n+\delta)} (n+\delta)^\zeta\leq \frac{1}{2T} \Big( \sqrt{2\pi} \sqrt{\zeta_0+1}\,
		e^{-\zeta_0+1/[12(\zeta_0+1)]}+1\Big)  e^{-2t(1+\delta)} 2^\zeta (1+\delta)^\zeta.
	\end{equation*}
\end{lm}
\begin{proof}
	Denote the above series by $S$ and let $\delta,\zeta,t$ be as in the statement of the lemma.
	Notice that the function $x\mapsto e^{-2t(x+\delta)}(x+\delta)^\zeta$ is decreasing on $[1,\infty)$. Indeed, its derivative equals
	\begin{equation*}
		-2t e^{-2t(x+\delta)}(x+\delta)^{\zeta} \bigg(1-\frac{\zeta}{2t(x+\delta)}\bigg),
	\end{equation*}
	which is negative by the assumption $\zeta\le 2t(\delta+1)$. Hence,
\begin{align*}
		S\leq \int_1^\infty e^{-2t(x+\delta)} (x+\delta)^{\zeta}\,\dd x 
		& = 
		(2t)^{-\zeta-1} \int_{2t(\delta+1)}^{\infty} e^{-x} x^\zeta\,\dd x\\
		& = 
		(2t)^{-\zeta-1} e^{-2t(\delta+1)} \int_0^\infty e^{-x} \big(x+2t(\delta+1)\big)^\zeta\,\dd x.
\end{align*}

	Using the bound $(A+B)^\zeta \leq 2^\zeta(A^\zeta+B^\zeta)$, $A,B>0$, and \eqref{Gest}, we further obtain
	\begin{align*}
		S&\le  (2t)^{-\zeta-1} e^{-2t(\delta+1)} 2^\zeta \int_0^\infty e^{-x}\Big(x^{\zeta} + \big[2t(\delta+1)\big]^\zeta\Big)\, \dd x\\
		& =  (2t)^{-\zeta-1} e^{-2t(\delta+1)} 2^{\zeta}  \Big( \Gamma(\zeta+1) + \big[2t(\delta+1)\big]^\zeta\Big)\\
		&\leq \frac{1}{2t} e^{-2t(\delta+1)}  (\delta+1)^\zeta \bigg[ \sqrt{2\pi} 
			\bigg(\frac{\zeta+1}{t(\delta+1)}\bigg)^{\zeta} \sqrt{\zeta+1} e^{-\zeta-1+1/[12(\zeta_0+1)]}  + 2^\zeta\bigg].
	\end{align*}
	
	Now observe that the function $[\zeta_0,\infty)\ni\zeta\mapsto \sqrt{\zeta+1} e^{-\zeta}$
	attains its maximum at $\zeta=\zeta_0$. Thus,
	\begin{equation*}
		S \leq \frac{1}{2T} e^{-2t(\delta+1)} 2^\zeta (\delta+1)^\zeta \bigg[ \sqrt{2\pi} 
		\Big(1+\frac{1}{\zeta}\Big)^{\zeta} \sqrt{\zeta_0+1}e^{-\zeta_0-1} e^{1/[12(\zeta_0+1)]}  + 1\bigg].
	\end{equation*}
	Since $(1+1/\zeta)^\zeta\leq e$, the proof is finished.
\end{proof}

We shall use Lemma \ref{lm:2} in two cases, which we highlight in the following.

\begin{cor}\label{cor:1}
	Let $\al,\be > -1$ and let $t\ge 2\log 2$.
	\begin{itemize}
	\item[(i)]
	If $\al\vee\be\geq -1/2$, then
	\begin{align*}
	& \sum_{n=1}^\infty e^{-tn(n+\al+\be+3)} (n+\al+\be+3)^{2(\al\vee\be)+1} \\
	& \qquad \qquad \leq \bigg(1+\frac{\sqrt{2\pi} e^{1/12} +1}{128\log2}\bigg) e^{-t(\al+\be+4)} (\al+\be+4)^{2(\al\vee\be)+1}.
	\end{align*}
	\item[(ii)]
	If $\al\vee\be < -1/2$, then
	\begin{align*}
	& \sum_{n=1}^\infty e^{-tn(n+\al+\be+3)} (n+\al+\be+3)^{2(\al\wedge\be)+3} \\
	& \qquad \qquad \leq \bigg(1+\frac{2\sqrt{\pi} e^{-23/24} +1}{32\log2}\bigg) e^{-t(\al+\be+4)} (\al+\be+4)^{2(\al\wedge\be)+3}.
	\end{align*}
	\end{itemize}
\end{cor}
\begin{proof}
	In the case $\al\vee\be\ge -1/2$ we apply Lemma \ref{lm:2} with $\delta=\al+\be+3$, $\zeta=2(\al\vee\be)+1$, $\zeta_0=0$
	and $T= 2\log 2$. Notice that the condition $\zeta\leq 2T (\delta+1)$ is fulfilled. Thus, 
	\begin{multline*}
	\sum_{n=1}^\infty e^{-tn(n+\al+\be+3)} (n+\al+\be+3)^{2(\al\vee\be)+1}\\
	 \leq e^{-t(\al+\be+4)} (\al+\be+4)^{2(\al\vee\be)+1}\bigg(1	+ \frac{1}{4\log 2}
	\big( \sqrt{2\pi}  e^{1/12}+1\big)  e^{-t(\al+\be+4)} 2^{2(\al\vee\be)+1} \bigg).
	\end{multline*}
	Since $t\geq 2\log 2$, we have
	\begin{equation*}
	e^{-t(\al+\be+4)} 2^{2(\al\vee\be)+1} \leq 2^{-2(\al\wedge\be)-7}\leq 2^{-5},
	\end{equation*}
	and this together with the previous bound finishes justification of the estimate in (i). 
	
	In the remaining case when $\al,\be\in(-1,-1/2)$ we apply Lemma \ref{lm:2} with
	$\delta=\al+\be+3$, $\zeta=2(\al\wedge\be)+3$, $\zeta_0=1$ and $T = 2\log 2$. Again, $\zeta\leq 2 T (\delta+1)$ holds. Hence,
	\begin{multline*}
	\sum_{n=1}^\infty e^{-tn(n+\al+\be+3)} (n+\al+\be+3)^{2(\al\wedge\be)+3}\\
	\leq e^{-t(\al+\be+4)} (\al+\be+4)^{2(\al\wedge\be)+3}
	\bigg(1+\frac{2\sqrt{\pi} e^{-23/24} +1}{4\log2} e^{-t(\al+\be+4)} 2^{2(\al\wedge\be)+3}\bigg).
	\end{multline*}
	Much as above, due to the constraint on $t$,
	\begin{equation*}
	e^{-t(\al+\be+4)} 2^{2(\al\wedge\be)+3} \leq 2^{-2(\al\vee\be)-5}\leq 2^{-3}.
	\end{equation*}
	This, in view of the preceding estimate, shows (ii) and concludes the proof.
\end{proof}

We are now in a position to prove the main result of Section \ref{sec:lmtime}.
\begin{proof}[{Proof of Theorem \ref{thm:large}}]
	Consider first the case $\al\vee \be\geq -1/2$. By \eqref{eq:P1} and Lemma \ref{lm:1} we obtain
	\begin{multline*}
		\big|E^{\al,\be}_t(x,y)\big|\leq 
		e^{-t(\a+\b+2)} \bigg[ \frac{(\al+\be+3)(\a \vee \b +1)}{\al\wedge\be+1}\\
		+\frac{7\sqrt{6}\, e^{1/30}\, \Gamma(\al\wedge\be+1) }{3\sqrt{5}\,\Gamma(\al+\be+2)\Gamma(\a \vee \b +1)}
		\sum_{n=2}^{\infty} e^{-t(n-1)(n+\a+\b+2)} (n+\a+\b+2)^{2(\a \vee \b)+1}  \bigg].
	\end{multline*}
	Applying now Corollary \ref{cor:1}(i) we get
	\begin{align} 
	\begin{split} \label{inqi}
		\big|E^{\al,\be}_t(x,y)\big| & \leq 
		e^{-t(\a+\b+2)} \frac{\al\vee\be+1}{\al\wedge\be+1} \\
		& \quad \times\bigg(\al+\be+3 +  \frac{\mathfrak{C}\, \Gamma(\al\wedge\be+2) }{\Gamma(\al+\be+2)\Gamma(\a \vee \b +2)}
		e^{-t(\a+\b+4)} (\a+\b+4)^{2(\a \vee \b)+1}  \bigg),
		\end{split}
	\end{align}
	where
	\begin{equation*}
		\mathfrak{C}:= \frac{7\sqrt{6}\, e^{1/30}}{3\sqrt{5}} \Big(1+\frac{\sqrt{2\pi} e^{1/12} +1}{128\log2}\Big)\approx 2.753612.
	\end{equation*}
	
	To proceed, we estimate first a part of the expression in \eqref{inqi}.
	Using \eqref{Gest} we see that
	\begin{multline*}
		\frac{ \Gamma(\al\wedge\be+2)\, (\al+\be+4)^{2(\al\vee\be)+1}}{\Gamma(\al+\be+2)\Gamma(\a \vee \b +2)} 
		\leq \frac{e^{1/12}}{\sqrt{2\pi}} e^{2(\al\vee\be)+2}\, \frac{(\al\wedge\be+2)^{\al\wedge\be+3/2}
		(\al+\be+4)^{2(\al\vee\be)+1} }{(\al+\be+2)^{\al+\be+3/2} (\al\vee\be+2)^{\al\vee\be+3/2} }\\
		= \frac{e^{1/12}}{\sqrt{2\pi}} e^{2(\al\vee\be)+2} \bigg(1+\frac{2}{\al+\be+2}\bigg)^{\al+\be+2}
		\frac{\big(1+\frac{\al\wedge\be+2}{\al\vee\be+2}\big)^{\al\vee\be+2} }{\big(1+\frac{\al\vee\be+2}{\al\wedge\be+2}
		\big)^{\al\wedge\be+2} } \frac{\sqrt{(\al+\be+2)(\al\vee\be+2)}}{(\al+\be+4) \sqrt{\al\wedge\be+2}}.
	\end{multline*}
	The last factor here is bounded by $1$. Further,
notice that since the function $[1,\infty)\ni s\mapsto \frac{(1+1/s)^s}{1+s}$ is decreasing,
we have the bound
	\begin{equation*}
		\frac{\big(1+\frac{A}{B}\big)^B}{\big( 1+\frac{B}{A}\big)^A}
		= \Bigg(\frac{\big(1+\frac{A}{B}\big)^{B/A}}{1+\frac{B}{A}}\Bigg)^A \leq 1, \qquad 0< A\le B<\infty.
	\end{equation*}
	Therefore,
	\begin{equation*}
	\frac{ \Gamma(\al\wedge\be+2) (\al+\be+4)^{2(\al\vee\be)+1}}{\Gamma(\al+\be+2)\Gamma(\a \vee \b +2)}
	\leq \frac{e^{1/12}}{\sqrt{2\pi}}\, e^{2(\al\vee\be)+4}.
	\end{equation*}

	Combining the last estimate with \eqref{inqi} we arrive at
	\begin{multline*}
		\big|E_t^{\al,\be}(x,y)\big|
		 \leq \frac{e^{-(t-1)(\al+\be+2)}}{\al\wedge\be+1}\bigg[ (\al+\be+3)^2 e^{-(\al+\be+2)} \\
		+\frac{\mathfrak{C} e^{1/12}}{\sqrt{2\pi}}
		(\al\vee\be+1)e^{-(t-1)(\al\vee\be+1)} e^{-(t+1)(\al\wedge\be +1)} e^{-2t+2}\bigg].
	\end{multline*}
Applying the simple bounds $s e^{- s}\leq e^{-1}$ and $s^2 e^{- s} \leq 4e^{-2}$ for $s>0$, we further obtain
	\begin{align*}
	\big|E_t^{\al,\be}(x,y)\big| &
	\leq \frac{e^{-(t-1)(\al+\be+2)}}{\al\wedge\be+1}
	\bigg( \frac{4}{e} +\frac{\mathfrak{C}\, e^{1/12}}{e(t -1)\sqrt{2\pi}}\, e^{-2t+2}\bigg) \\ &
	\leq \frac{e^{-(t-1)(\al+\be+2)}}{\al\wedge\be+1}
	\bigg( \frac{4}{e} +\frac{\mathfrak{C}\, e^{13/12}}{16\sqrt{2\pi} (2\log2 -1)}\bigg),
	\end{align*}
	where we used the assumption $t\geq 2\log 2$.
	The constant in the last parentheses is approximately equal to $1.996640$, which is smaller than $2$.
	Consequently, the first bound of the theorem follows.
		
	Now we move to the case $\al,\be\in(-1,-1/2)$. Much as above, by \eqref{eq:P1} and Lemma \ref{lm:1} we get
	\begin{multline*}
		\big|E^{\al,\be}_t(x,y)\big|\leq 
		e^{-t(\a+\b+2)}\frac{1}{(\al+1)(\be+1)}
		\bigg[(\al+\be+3)(\a \vee \b +1)^2 \\ +  18\frac{\Gamma(\al+2) \Gamma(\be+2)}{\Gamma(\al+\be+2)}
		\sum_{n=2}^{\infty} e^{-t(n-1)(n+\a+\b+2)} (n+\al+\be+2)^{2(\al\wedge\be)+3}  \bigg].
	\end{multline*}
	Here the ratio involving gamma functions satisfies\,\footnote{
			\,To justify this bound, let $f_r(s) = \Gamma(s+1)/\Gamma(s+r)$, where $r \in (0,1)$.
			Then $f_r'(s) = f_r(s) [\psi(s+1)-\psi(s+r)] > 0$, the last inequality since the digamma function
			$\psi(z) = \Gamma'(z)/\Gamma(z)$ is increasing for $z > 0$; see e.g.\ \cite[Chap.\,5, $\S$5.15, Form.\,5.15.1]{DLMF}.
			Thus $f_r(s)$ is increasing for $s > 0$. Using this property first with $r=\b+1$, and then with
			$r=1/2$, one gets the desired bound.
		}
	$$
		\frac{\Gamma(\a+2)\Gamma(\b+2)}{\Gamma(\a+\b+2)} \le \frac{\Gamma(-1/2+2) \Gamma(-1/2+2)}{\Gamma(-1/2-1/2+2)} = \frac{\pi}4,
			\qquad \a,\b \in (-1,-1/2).
	$$
	Using this and Corollary \ref{cor:1}(ii) leads to
	\begin{equation*}
	\big|E^{\al,\be}_t(x,y)\big|\leq \frac{
	e^{-t(\a+\b+2)}}{(\al+1)(\be+1)}\bigg[\frac{1}{2}+ \frac{9\pi}2 \Big(1+\frac{2\sqrt{\pi}\, e^{-23/24} +1}{32\log2}\Big)\,
		e^{-t(\a+\b+4)} (\al+\be+4)^{2(\al\wedge\be)+3}  \bigg].
	\end{equation*}
	
	Notice that
	\begin{equation*}
	e^{-t(\a+\b+4)} (\al+\be+4)^{2(\al\wedge\be)+3} \leq 2^{-2(\al+\be+4)} (\al+\be+4)^{\al+\be+3}\leq \frac{9}{64},
	\end{equation*} 
	because the maximum of the function $s\mapsto 2^{-2s} s^{s-1}$ on $[2,3]$ is attained at $s=3$ and equals $9/64$.
	Since
	\begin{equation*}
	\frac{1}{2}+ \frac{81\pi}{128} \Big(1+\frac{2\sqrt{\pi} e^{-23/24} +1}{32\log2}\Big)\approx 2.699527< \frac{27}{10},
	\end{equation*}
	we conclude the second bound of the theorem.
	
	The proof of Theorem \ref{thm:large} is complete.
\end{proof}

\section{Applications: totally explicit sharp heat kernel bounds} \label{sec:app}

In this section we apply the results of Section \ref{sec:uproof} to derive fully explicit sharp small time
bounds for the Jacobi heat kernel and some other related heat kernels that are of interest.
We put emphasis on simplicity of the bounds, hence the results we present are slightly weaker than e.g.\ Lemma \ref{lem:F},
but have uncomplicated form.

Analogous fully explicit sharp medium and large time bounds can then be derived in a straightforward manner, taking into account
the discussion and results presented in Section \ref{sec:lmtime}. We leave the details apart for interested readers.

\subsection{Spherical heat kernel} \label{ssec:spherical} \,

\medskip

We use the notation introduced in the beginning of Section \ref{sec:oddsph}.
Recall, cf.\ \eqref{sp44}, that the heat kernel on the unit sphere $S^d$ of dimension $d \ge 1$ is given directly via the function
$$
K_t^d(\phi) = \frac{\Gamma(d/2)}{2\pi^{d/2}} G_t^{d/2-1}(\cos\phi,1), \qquad \phi \in [0,\pi], \quad t > 0.
$$

Let $T>0$. By Lemma \ref{lem:G}, for $0 < t \le T$ we have
$$
G_t^{d/2-1}(\cos\phi,1) \simeq \left\{ C_{d/2-1,T}^{\textrm{ref}} \Psi_{d/2-1}^{\mathfrak{b}}(t,\pi-\phi,\pi) \atop
	c_{d/2-1,T}^{\textrm{ref}} \Psi_{d/2-1}^{\mathfrak{B}}(t,\pi-\phi,\pi) \right\} \, \frac{1}{t^{d/2}} e^{-\phi^2/(4t)},
$$
where
\begin{align*}
c_{d/2-1,T}^{\textrm{ref}} & = c_{d-3/2,T/4}^{A} 4^{d-1/2} b_{d/2-1} \frac{h_0^{d-3/2}}{h_0^{d/2-1}} \\
	& = c_{d-3/2,T/4}^A \Big(1 \wedge 2^{(d-3)/2}\Big) 2^d \exp\big(-\chi_{\{d\neq 1\}}\mathbb{D}_{d/2-1}\big)
	\frac{\Gamma(d-1/2)}{\Gamma(d/2)}, \\
C_{d/2-1,T}^{\textrm{ref}} & = C_{d-3/2,T/4}^A 4^d B_{d/2-1} \frac{h_0^{d-3/2}}{h_0^{d/2-1}}
	= C_{d-3/2,T/4}^A \Big( 1 \vee 2^{(d-3)/2}\Big) 2^{d+\chi_{\{d \ge 2\}}} \frac{\Gamma(d-1/2)}{\Gamma(d/2)}.
\end{align*}
This implies
\begin{equation} \label{h4h}
\begin{split}
K_t^d(\phi) & \simeq \left\{ C_{d-3/2,T/4}^A 2^{2d-\chi_{\{d=1\}}}
(1\vee 2^{(d-3)/2}) \Gamma(d-1/2) \Psi_{d/2-1}^{\mathfrak{b}}(t,\pi-\phi,\pi) \atop
	c_{d-3/2,T/4}^A 4^{d-1/2} (1 \wedge 2^{(d-3)/2}) \exp\big(-\chi_{\{d\neq 1\}}\mathbb{D}_{d/2-1}\big)
		\Gamma(d-1/2) \Psi_{d/2-1}^{\mathfrak{B}}(t,\pi-\phi,\pi) \right\} \\ & \qquad \times \frac{1}{(4\pi t)^{d/2}} e^{-\phi^2/(4t)}.
\end{split}
\end{equation}

Concerning the lower bound for $K_t^d(\phi)$, we get, see Lemma \ref{lem:A},
\begin{equation} \label{zz1}
K_t^d(\phi) \ge 2 \Big( 1 \wedge 2^{(d-3)/2}\Big) \exp\big(-\chi_{\{d\neq 1\}}\mathbb{D}_{d/2-1}\big)
	\Psi_{d/2-1}^{\mathfrak{B}}(t,\pi-\phi,\pi) \frac{1}{(4\pi t)^{d/2}} e^{-\phi^2/(4t)}.
\end{equation}
This estimate holds for all $t > 0$.
To write the upper bounds, it will be convenient to split into cases due to the structure of the upper
multiplicative constant in Lemma \ref{lem:A}.

When $d=1$, the sphere is just the one-dimensional torus. The function $K_t^1(\phi)$ coincides with the periodized
Gauss-Weierstrass kernel,
$$
K_t^1(\phi) = \vartheta_t(\phi) = \sum_{n \in \mathbb{Z}} W_t(\phi + 2\pi n),
$$
which is relatively straightforward to estimate directly from the series.\,\footnote{
\,For large $t$ one uses the oscillating series, see Section \ref{sec:lmtime},
$K_t^{1}(\phi) = \frac{1}{2\pi} \sum_{n \in \mathbb{Z}} e^{-t n^2} \cos(n\phi)$.
}
Thus precise explicit bounds in this case are known for a long time, at least as folklore.
Nevertheless, just for the sake of illustration we note that in this case our results (cf.\ \eqref{h4h} and Lemma \ref{lem:A}) imply
$$
K_t^1(\phi) \le C^{A}_{-1/2,T/4} 2 \sqrt{\pi} \frac{1}{\sqrt{4\pi t}} e^{-\phi^2/(4t)} = 2\Big(1+2e^{-2\pi^2/T}\Big)
	\frac{1}{\sqrt{4\pi t}} e^{-\phi^2/(4t)}, \qquad t \le T \le 2\pi^2.
$$
Combining this with \eqref{zz1} we obtain somewhat rough estimates
$$
K_t^1(\phi) \simeq \left\{ 2\big(1+2e^{-2\pi^2/T}\big) \atop 1 \right\} \, \frac{1}{\sqrt{4\pi t}} e^{-\phi^2/(4t)},
	\qquad \phi \in [0,\pi], \quad 0 < t \le T \le 2\pi^2;
$$
here the lower bound holds in fact for all $t>0$.

Let now $d=2$. By \eqref{h4h} and Lemma \ref{lem:A},
\begin{align*}
K_t^2(\phi) & \le C^A_{1/2,T/4} 4^2  \Gamma(3/2) \Psi_0^{\mathfrak{b}}(t,\pi-\phi,\pi) \frac{1}{4\pi t} e^{-\phi^2/(4t)} \\
	& = 2 \pi e^{T/4} \Psi_0^{\mathfrak{b}}(t,\pi-\phi,\pi) \frac{1}{4\pi t} e^{-\phi^2/(4t)}, \qquad t \le T \le 2\pi^2.
\end{align*}
This, together with \eqref{zz1}, gives
\begin{equation} \label{kr88}
K_t^2(\phi) \simeq \left\{ 2\pi e^{T/4} \Psi_0^{\mathfrak{b}}(t,\pi-\phi,\pi) \atop
	\sqrt{2} e^{-\pi/4} \Psi_0^{\mathfrak{B}}(t,\pi-\phi,\pi) \right\} \, \frac{1}{4\pi t} e^{-\phi^2/(4t)}, \qquad
		\phi \in [0,\pi], \quad 0 < t \le T \le 2\pi^2.
\end{equation}
Note that for e.g.\ $T=1$ the ratio $2\pi e^{T/4}/(\sqrt{2}e^{-\pi/4}) \approx 12.512168$.

Finally, let $d \ge 3$. Using \eqref{h4h} and Lemma \ref{lem:A} we obtain
$$
K_t^d(\phi) \le
	\begin{cases}
		2\,\mathfrak{w}_0 (2\sqrt{2e}\,\mathfrak{w}_1)^{d-1} \Psi_{d/2-1}^{\mathfrak{b}}(t,\pi-\phi,\pi) (4\pi t)^{-d/2}e^{-\phi^2/(4t)}
			& \textrm{if} \;\; 0 < t \le 2/(d-1/2) \\
		2\sqrt[4]{e}\, \mathfrak{w}'_0 (\pi/\sqrt{2})^{d-1} \Psi_{d/2-1}^{\mathfrak{b}}(t,\pi-\phi,\pi) (4\pi t)^{-d/2}e^{-\phi^2/(4t)}
			& \textrm{if} \;\; 0 < t \le 4/(2d-1)^2
	\end{cases}.
$$
Together with \eqref{zz1}, this gives for $d \ge 3$ and $\theta \in [0,\pi]$
$$
K_t^d(\phi) \simeq \left\{ 2\,\mathfrak{w}_0 (2\sqrt{2e}\,\mathfrak{w}_1)^{d-1} \Psi_{d/2-1}^{\mathfrak{b}}(t,\pi-\phi,\pi)
	\atop 2 \exp(-\mathbb{D}_{d/2-1}) \Psi_{d/2-1}^{\mathfrak{B}}(t,\pi-\phi,\pi) \right\} \,
	\frac{1}{(4\pi t)^{d/2}} e^{-\phi^2/(4t)}, \qquad 0 < t \le \frac{2}{d-1/2},
$$
and
$$
K_t^d(\phi) \simeq \left\{ 2\sqrt[4]{e}\,\mathfrak{w}'_0 (\pi/\sqrt{2})^{d-1} \Psi_{d/2-1}^{\mathfrak{b}}(t,\pi-\phi,\pi)
	\atop 2 \exp(-\mathbb{D}_{d/2-1}) \Psi_{d/2-1}^{\mathfrak{B}}(t,\pi-\phi,\pi) \right\} \,
	\frac{1}{(4\pi t)^{d/2}} e^{-\phi^2/(4t)}, \qquad 0 < t \le \frac{1}{(d-1/2)^2}.
$$
In view of \eqref{DDD} one has $\exp(-\mathbb{D}_{d/2-1}) >  \exp(-e^{-\gamma}) \exp(-e^{-\gamma}/2)^{d-1}$.
Using this and then approximating the relevant numerical quantities\,\footnote{
\, One has $2\,\mathfrak{w}_0 \approx 1.595966$, $2\sqrt{2e}\,\mathfrak{w}_1 \approx 4.782514$,
$2\exp(-e^{-\gamma}) \approx 1.140752$, $\exp(-e^{-\gamma}/2) \approx 0.755232$,
$2\sqrt[4]{e}\,\mathfrak{w}'_0 \approx 6.987667$, $\pi/\sqrt{2} \approx 2.221441$.
}
, one can bound the constants
$$
2\,\mathfrak{w}_0 < \frac{8}5, \quad 2\sqrt{2e}\,\mathfrak{w}_1 < \frac{24}5, \quad
2 \exp(-\mathbb{D}_{d/2-1}) > \frac{11}{10} \cdot \Big(\frac{3}4\Big)^{d-1}, \quad
2\sqrt[4]{e}\, \mathfrak{w}'_0 < 7, \quad \frac{\pi}{\sqrt{2}} < \frac{9}4,
$$
for $d \ge 3$, and independently the constants' ratios as
$$
\frac{2\,\mathfrak{w}_0 (2\sqrt{2e}\,\mathfrak{w}_1)^{d-1}}{2\exp(-\mathbb{D}_{d/2-1})} < \frac{7}5\cdot\Big(\frac{19}3 \Big)^{d-1}, \qquad
\frac{2\sqrt[4]{e}\,\mathfrak{w}'_0 (\pi/\sqrt{2})^{d-1}}{2\exp(-\mathbb{D}_{d/2-1})} < \frac{31}5\cdot 3^{d-1}, \qquad d \ge 3.
$$

We believe that the ratios of optimal upper and lower multiplicative constants in our estimates grow
no slower than exponentially in $d$, due to a certain incompatibility between the structure of these bounds and the actual behavior
of the heat kernel.
An analogous remark pertains to the other heat kernels considered in Sections \ref{ssec:CROSS} and \ref{ssec:Jacobi} below.

\subsection{Heat kernels on compact rank one symmetric spaces} \label{ssec:CROSS} \,

\medskip

Let $\mathbb{M}$ be a compact rank one symmetric space.
Such spaces are completely classified (cf.\ \cite{He1,Wa}), precisely
they are Euclidean spheres, real, complex and quaternionic projective spaces,
and the Cayley projective plane over octonions. Thus $\mathbb{M}$ is a complete connected Riemannian manifold, with no
boundary and strictly positive sectional (and Ricci) curvature.

We denote by $d_{\mathbb{M}}$ the Riemannian geodesic distance on $\mathbb{M}$, and by $\vol_{\mathbb{M}}$ the
Riemannian volume measure on $\mathbb{M}$. The dimension $d$ of $\mathbb{M}$ and the antipodal dimension\,\footnote{
\,Given a fixed point in $\mathbb{M}$, the set of points in $\mathbb{M}$ whose geodesic distance from it is equal to
the diameter of $\mathbb{M}$ forms a submanifold of $\mathbb{M}$, which is called the antipodal manifold.
In case of spheres it consists of only one point, while in the other cases the antipodal manifolds are identified with
compact rank one symmetric spaces of lower dimension; this identification is independent of the initially fixed point.
The antipodal dimension related to $\mathbb{M}$ is the dimension of the antipodal manifold.
For references and more details, see e.g.\ \cite[Sec.\,4]{NSS2}.
}
$\tilde{d}$
related to $\mathbb{M}$ (both over reals) will be indicated by superscripts, i.e.\ $\mathbb{M}^{d,\tilde{d}}$ means
that $\mathbb{M}$ has real dimension $d$ and the associated antipodal real dimension $\tilde{d}$.

Let $K_t^{\mathbb{M}}(x,y)$, $x,y \in \mathbb{M}$, $t>0$, be the heat kernel on $\mathbb{M}$,
that is the integral kernel, with respect to $\vol_{\mathbb{M}}$, of the semigroup generated by the Laplace-Beltrami
operator on $\mathbb{M}$.
It is known, see \cite[Sec.\,4]{NSS2}, that for $\mathbb{M} = \mathbb{M}^{d,\tilde{d}}$
\begin{equation} \label{hkcross}
K_t^{\mathbb{M}}(x,y) = \frac{h_0^{\a,\b}}{\vol_{\mathbb{M}}(\mathbb{M})} \,
	G^{\a,\b}_{\kappa^2 t}\Big( \cos\big[ \kappa\, d_{\mathbb{M}}(x,y)\big] , 1 \Big), \qquad x,y \in \mathbb{M}, \quad t > 0,
\end{equation}
where
$$
\kappa = \frac{\pi}{\diam \mathbb{M}}, \qquad \a = \frac{d}2-1, \qquad \b = \frac{d-\tilde{d}}2 -1.
$$

Totally explicit heat kernel bounds for unit Euclidean spheres are given in Section \ref{ssec:spherical}
(a simple scaling leads to analogous results for spheres of arbitrary radii).
We now treat the remaining compact rank one symmetric spaces.

In what follows, for technical convenience we represent the upper multiplicative constant of Lemma \ref{lem:A} as
$C^A_{\lambda,T_A} = r_1 \cdot r_2^{\lambda+1/2}/\Gamma(\lambda+1)$, where $r_1$ and $r_2$ have appropriate values depending on
the case under consideration.

\subsubsection*{\textbf{\emph{Real projective spaces}}} 
Let $\mathbb{M} = \mathbb{M}^{d,\tilde{d}}$ be a real projective space.
Here we consider dimensions $d=2,3,4,\ldots$, since for $d=1$ the real projective space is identified with a sphere of
dimension $1$ i.e.\ with a torus. The antipodal dimension is $\tilde{d}=d-1$, hence
the parameters $\a$ and $\b$ in \eqref{hkcross} take the form
$$
\a=\frac{d}2-1, \qquad \b=-\frac{1}2.
$$
Further, see e.g.\ \cite[Chap.\,1]{Sh},
$$
{\vol}_{\mathbb{M}}(\mathbb{M}) = \frac{\sqrt{\pi} (4\pi)^{d/2}}{\kappa^d \, \Gamma(d/2+1/2)}.
$$

From a general theory, see \cite[Thm.\,5.6.1]{Da}\,\footnote{
\,This result applies in fact to all compact rank one symmetric spaces, but it does not provide a sharp global lower bound
unless $\mathbb{M}$ is a real projective space (or a one-dimensional torus).
}
, it follows that
$$
K_t^{\mathbb{M}}(x,y) \ge \frac{1}{(4\pi t)^{d/2}} \, e^{-d_{\mathbb{M}}^2(x,y)/{(4t)}}, \qquad t > 0. 
$$
Since our approach does not provide a better lower bound\,\footnote{
\,Our approach provides the same lower bound when the dimension is odd, but it gives a worse lower bound if the dimension is even.
},
we now restrict our attention to upper estimates of $K_t^{\mathbb{M}}(x,y)$.

Consider first the case of odd $d$, i.e.\ $d=3,5,7,\ldots$.
In the context of Lemma \ref{lem:G}, see also Lemma~\ref{lem:A}, after simplifications we have
\begin{equation*}
\frac{h_0^{\a,\b}}{{\vol}_{\mathbb{M}}(\mathbb{M})} \, C_{d/2-1,-1/2,T}^{\textrm{ref}}
	 = 2 r_1 (4r_2)^{(d-1)/2} \, \frac{\kappa^{d}}{(4\pi)^{d/2}}.
\end{equation*}

Then Lemma \ref{lem:G} together with \eqref{hkcross} gives the following bounds valid for odd $d \ge 3$ and $x,y \in \mathbb{M}$:
\begin{align*}
K_t^{\mathbb{M}}(x,y) &  \le 2\, \mathfrak{w}_0 (2\sqrt{e}\,\mathfrak{w}_1)^{(d-1)/2} \,
	\frac{1}{(4\pi t)^{d/2}}\, e^{-d_{\mathbb{M}}^2(x,y)/{(4t)}}, \qquad 0 < t \le \frac{4}{d \kappa^2}, \\
K_t^{\mathbb{M}}(x,y)
& \le 2 \sqrt[4]{e}\, \mathfrak{w}'_0 (\pi/2)^{(d-1)/2} \,
	\frac{1}{(4\pi t)^{d/2}}\, e^{-d_{\mathbb{M}}^2(x,y)/{(4t)}}, \qquad  0 < t \le \frac{4}{(d \kappa)^2}.
\end{align*}

Approximating the relevant numerical quantities\,\footnote{
\,One has $2\,\mathfrak{w}_0 \approx 1.595966$, $2\sqrt{e}\,\mathfrak{w}_1 \approx 3.381748$,
$2\sqrt[4]{e}\,\mathfrak{w}'_0 \approx 6.987667$.
}
, one can bound the constants
$$
2\,\mathfrak{w}_0 < \frac{8}5, \quad 2\sqrt{e}\,\mathfrak{w}_1 < \frac{17}5, \quad 2\sqrt[4]{e}\, \mathfrak{w}'_0 < 7,
$$
for odd $d \ge 3$, and the upper and lower constants' ratios as
$$
\frac{2\, \mathfrak{w}_0 (2\sqrt{e}\,\mathfrak{w}_1)^{(d-1)/2}}{1} < 
\frac{8}5 \cdot \Big(\frac{17}5 \Big)^{(d-1)/2}, \qquad
\frac{2 \sqrt[4]{e}\, \mathfrak{w}'_0 (\pi/2)^{(d-1)/2}}{1} < 7\cdot \Big(\frac{\pi}2\Big)^{(d-1)/2}.
$$

Moreover, for $d=3$ the above bounds for $K_t^{\mathbb{M}}(x,y)$
can be improved with the aid of the finer bounds of Lemma \ref{lem:A} when $\lambda=1/2$.
More precisely, when $d=3$, for $x,y \in \mathbb{M}$ we get
$$
K_t^{\mathbb{M}}(x,y) \le {\pi} \exp(T/4) \frac{1}{(4\pi t)^{3/2}}\, e^{-d_{\mathbb{M}}^2(x,y)/{(4t)}},
	\qquad  0 < t \le  T/\kappa^2,
$$
where $0 < T \le 2\pi^2$ is arbitrary and fixed.

Next, consider the case of even $d$, i.e.\ $d=2,4,6,\ldots$. To proceed, we make an additional assumption $d \ge 4$;
the dimension $d=2$ will be treated separately later.
In the context of Lemma~\ref{lem:H}, see also Lemma \ref{lem:A}, after simplifications we have
\begin{equation*}
\frac{h_0^{\a,\b}}{{\vol}_{\mathbb{M}}(\mathbb{M})} \, C_{d/2-1,T}^{\textrm{ref}}
	= 4 r_1 \big(32 r_2^2\big)^{(d-1)/2} \, \frac{\kappa^{d}}{(4\pi)^{d/2}}.
\end{equation*}

Then Lemma \ref{lem:H} together with \eqref{hkcross}
gives the following bounds valid for even $d \ge 4$ and $x,y \in \mathbb{M}$:
\begin{align*}
K_t^{\mathbb{M}}(x,y) &  \le 4\, \mathfrak{w}_0 (8e\,\mathfrak{w}^2_1)^{(d-1)/2} \,
	\frac{1}{(4\pi t)^{d/2}}\, e^{-d_{\mathbb{M}}^2(x,y)/{(4t)}}, \qquad 0 < t \le \frac{8}{(d-1/2) \kappa^2}, \\
K_t^{\mathbb{M}}(x,y)
& \le 4 \sqrt[4]{e}\, \mathfrak{w}'_0 (\pi^2/2)^{(d-1)/2} \,
	\frac{1}{(4\pi t)^{d/2}}\, e^{-d_{\mathbb{M}}^2(x,y)/{(4t)}}, \qquad  0 < t \le \frac{4}{(d-1/2)^2 \kappa^2}.
\end{align*}

Approximating the relevant numerical quantities\,\footnote{
\,One has
$4\,\mathfrak{w}_0 \approx 3.191932$, $8e\,\mathfrak{w}_1^2 \approx 22.872438$,
$4\sqrt[4]{e}\, \mathfrak{w}'_0 \approx 13.975333$, $\pi^2/2 \approx 4.934802$.
}
, one can bound the constants
$$
4\,\mathfrak{w}_0 < {16}/5, \quad 8e\,\mathfrak{w}_1^2 < 23, \quad 4\sqrt[4]{e}\, \mathfrak{w}'_0 < 14, \quad {\pi^2}/{2} < 5,
$$
and, for even $d \ge 4$, the constants' ratios as
\begin{equation*}
\frac{4\, \mathfrak{w}_0 (8e\,\mathfrak{w}^2_1)^{(d-1)/2}}{1} < \frac{16}5 \cdot 23^{(d-1)/2}, \qquad
\frac{4 \sqrt[4]{e}\, \mathfrak{w}'_0 (\pi^2/2)^{(d-1)/2}}{1} < 14 \cdot 5^{(d-1)/2}.
\end{equation*}

Finally, when $d=2$ we have
\begin{equation*}
\frac{h_0^{\a,\b}}{{\vol}_{\mathbb{M}}(\mathbb{M})} \, C_{0,T}^{\textrm{ref}}
	 = 8 r_1 (4 r_2) \, \frac{\kappa^{2}}{4\pi},
\end{equation*}
hence, using Lemmas \ref{lem:H} and \ref{lem:A} (the latter with the finer upper bound in case $\lambda=1/2$),
we obtain the following bound for $d=2$ and $x,y \in \mathbb{M}$:
\begin{equation*}
K_t^{\mathbb{M}}(x,y)
 \le 4\pi \exp(T/16)
	\frac{1}{4\pi t}\, e^{-d_{\mathbb{M}}^2(x,y)/{(4t)}}, \qquad  0 < t \le T/ \kappa^2, 
\end{equation*}
where $0 < T \le 8\pi^2$ is arbitrary and fixed.

\subsubsection*{\textbf{\emph{Complex projective spaces}}}
Let $\mathbb{M} = \mathbb{M}^{d,\tilde{d}}$ be a complex projective space.
Here we consider dimensions $d=4,6,8,\ldots$, since for $d=2$ the
complex projective space is identified with a sphere of dimension $2$. The antipodal dimension is $\tilde{d}=d-2$, hence
the parameters $\a$ and $\b$ in \eqref{hkcross} take the form
$$
\a=\frac{d}2-1, \qquad \b=0.
$$
Further, see e.g.\ \cite[Chap.\,1]{Sh},
$$
{\vol}_{\mathbb{M}}(\mathbb{M}) = \frac{(4\pi)^{d/2}}{\kappa^d \, \Gamma(d/2+1)}.
$$

In the context of Lemma \ref{lem:G}, see also Lemma \ref{lem:A}, after simplifications we have
\begin{align*}
\frac{h_0^{\a,\b}}{{\vol}_{\mathbb{M}}(\mathbb{M})} \, c_{d/2-1,0,T}^{\textrm{ref}}
	& = \sqrt{2} e^{-\pi/4} \, \frac{\kappa^{d}}{(4\pi)^{d/2}}, \\
\frac{h_0^{\a,\b}}{{\vol}_{\mathbb{M}}(\mathbb{M})} \, C_{d/2-1,0,T}^{\textrm{ref}}
	& = 4 r_1 (4r_2)^{d/2} \, \frac{\kappa^{d}}{(4\pi)^{d/2}}.
\end{align*}

Then Lemma \ref{lem:G} together with \eqref{hkcross} gives the following bounds valid for $x,y \in \mathbb{M}$:
\begin{align*}
K_t^{\mathbb{M}}(x,y)
& \ge \sqrt{2} e^{-\pi/4} \, \Psi_0^{\mathfrak{B}}\big( \kappa^2 t, \pi-\kappa\, d_{\mathbb{M}}(x,y),\pi\big) \,
	\frac{1}{(4\pi t)^{d/2}}\, e^{-d_{\mathbb{M}}^2(x,y)/{(4t)}}, \qquad t > 0, \\
K_t^{\mathbb{M}}(x,y) &  \le 4\, \mathfrak{w}_0 (2\sqrt{e}\,\mathfrak{w}_1)^{d/2} \\
	& \qquad \times \Psi_0^{\mathfrak{b}}\big( \kappa^2 t, \pi-\kappa\, d_{\mathbb{M}}(x,y),\pi\big) \,
	\frac{1}{(4\pi t)^{d/2}}\, e^{-d_{\mathbb{M}}^2(x,y)/{(4t)}}, \qquad 0 < t \le \frac{4}{(d+1) \kappa^2}, \\
K_t^{\mathbb{M}}(x,y)
& \le 4 \sqrt[4]{e}\, \mathfrak{w}'_0 (\pi/2)^{d/2} \\
	& \qquad \times \Psi_0^{\mathfrak{b}}\big( \kappa^2 t, \pi-\kappa\, d_{\mathbb{M}}(x,y),\pi\big) \,
	\frac{1}{(4\pi t)^{d/2}}\, e^{-d_{\mathbb{M}}^2(x,y)/{(4t)}}, \qquad  0 < t \le \frac{4}{(d+1)^2 \kappa^2}.
\end{align*}

Approximating the relevant numerical quantities\,\footnote{
\,One has $\sqrt{2} e^{-\pi/4} \approx 0.644794$,
$4\,\mathfrak{w}_0 \approx 3.191932$, $2\sqrt{e}\,\mathfrak{w}_1 \approx 3.381748$,
$4\sqrt[4]{e}\, \mathfrak{w}'_0 \approx 13.975333$.
}
, one can bound the constants
$$
\sqrt{2} e^{-\pi/4} > \frac{3}5, \quad 4\,\mathfrak{w}_0 < \frac{16}5, \quad 2\sqrt{e}\,\mathfrak{w}_1 < \frac{17}5,
\quad 4 \sqrt[4]{e}\, \mathfrak{w}'_0 < 14,
$$
and the upper and lower constants' ratios as
$$
\frac{4\,\mathfrak{w}_0 (2\sqrt{e}\,\mathfrak{w}_1)^{d/2}}{\sqrt{2} e^{-\pi/4}} < 5 \cdot \Big(\frac{17}5 \Big)^{d/2}, \qquad
\frac{4 \sqrt[4]{e}\, \mathfrak{w}'_0 (\pi/2)^{d/2}}{\sqrt{2} e^{-\pi/4}} < 22\cdot \Big(\frac{\pi}2\Big)^{d/2}.
$$

\subsubsection*{\textbf{\emph{Quaternionic projective spaces}}}
Let $\mathbb{M} = \mathbb{M}^{d,\tilde{d}}$ be a quaternionic projective space.
Here we consider dimensions $d=8,12,16,\ldots$, since for $d=4$ the
quaternionic projective space is identified with a sphere of dimension $4$. The antipodal dimension is $\tilde{d}=d-4$, hence
the parameters $\a$ and $\b$ in \eqref{hkcross} take the form
$$
\a=\frac{d}2-1, \qquad \b=1.
$$
Further, see e.g.\ \cite[Chap.\,1]{Sh},
$$
{\vol}_{\mathbb{M}}(\mathbb{M}) = \frac{(4\pi)^{d/2}}{\kappa^d \, \Gamma(d/2+2)}.
$$

In the context of Lemma \ref{lem:G}, see also Lemma \ref{lem:A}, after simplifications we have
\begin{align*}
\frac{h_0^{\a,\b}}{{\vol}_{\mathbb{M}}(\mathbb{M})} \, c_{d/2-1,1,T}^{\textrm{ref}}
	& = 2 \exp(-\mathbb{D}_1) \, \frac{\kappa^{d}}{(4\pi)^{d/2}}, \\
\frac{h_0^{\a,\b}}{{\vol}_{\mathbb{M}}(\mathbb{M})} \, C_{d/2-1,1,T}^{\textrm{ref}}
	& = 4\sqrt{2} r_1 (4r_2)^{d/2+1} \, \frac{\kappa^{d}}{(4\pi)^{d/2}}.
\end{align*}
Note that $\mathbb{D}_1 = (3\sqrt{\pi}/4)^{2/3} \approx 1.208994$.

Then Lemma \ref{lem:G} together with \eqref{hkcross} gives the following bounds valid for $x,y \in \mathbb{M}$:
\begin{align*}
K_t^{\mathbb{M}}(x,y)
& \ge 2 \exp(-\mathbb{D}_1) \, \Psi_1^{\mathfrak{B}}\big( \kappa^2 t, \pi-\kappa\, d_{\mathbb{M}}(x,y),\pi\big) \,
	\frac{1}{(4\pi t)^{d/2}}\, e^{-d_{\mathbb{M}}^2(x,y)/{(4t)}}, \qquad t > 0, \\
K_t^{\mathbb{M}}(x,y) &  \le 4\sqrt{2}\, \mathfrak{w}_0 (2\sqrt{e}\,\mathfrak{w}_1)^{d/2 + 1} \\
	& \qquad \times \Psi_1^{\mathfrak{b}}\big( \kappa^2 t, \pi-\kappa\, d_{\mathbb{M}}(x,y),\pi\big) \,
	\frac{1}{(4\pi t)^{d/2}}\, e^{-d_{\mathbb{M}}^2(x,y)/{(4t)}}, \qquad 0 < t \le \frac{4}{(d+3) \kappa^2}, \\
K_t^{\mathbb{M}}(x,y)
& \le 4\sqrt{2} \sqrt[4]{e}\, \mathfrak{w}'_0 (\pi/2)^{d/2 + 1} \\
	& \qquad \times \Psi_1^{\mathfrak{b}}\big( \kappa^2 t, \pi-\kappa\, d_{\mathbb{M}}(x,y),\pi\big) \,
	\frac{1}{(4\pi t)^{d/2}}\, e^{-d_{\mathbb{M}}^2(x,y)/{(4t)}}, \qquad  0 < t \le \frac{4}{(d+3)^2 \kappa^2}.
\end{align*}

Approximating the relevant numerical quantities\,\footnote{
\,One has $2\exp(-\mathbb{D}_1) \approx 0.596995$, $4\sqrt{2}\,\mathfrak{w}_0 \approx 4.514073$,
$2\sqrt{e}\,\mathfrak{w}_1 \approx 3.381748$, $4\sqrt{2} \sqrt[4]{e}\, \mathfrak{w}'_0 \approx 19.764106$.
}
, one can bound the constants
$$
2\exp(-\mathbb{D}_1) > \frac{1}2, \quad 4\sqrt{2}\,\mathfrak{w}_0 < \frac{23}5, 
\quad 2\sqrt{e}\,\mathfrak{w}_1 < \frac{17}5, \quad 4\sqrt{2} \sqrt[4]{e}\, \mathfrak{w}'_0 < 20,
$$
and the upper and lower constants' ratios as
$$
\frac{4\sqrt{2}\, \mathfrak{w}_0 (2\sqrt{e}\,\mathfrak{w}_1)^{d/2 + 1}}{2 \exp(-\mathbb{D}_1)}
	< \frac{38}5 \cdot \Big(\frac{17}5\Big)^{d/2+1},
\qquad
\frac{4\sqrt{2} \sqrt[4]{e}\, \mathfrak{w}'_0 (\pi/2)^{d/2 + 1}}{2 \exp(-\mathbb{D}_1)} < 34\cdot \Big(\frac{\pi}2\Big)^{d/2+1}.
$$

\subsubsection*{\textbf{\emph{Cayley projective plane over octonions}}}
Let $\mathbb{M} = \mathbb{M}^{16,8}$ be the Cayley projective plane. 
Here $d=16$ and $\tilde{d}= 8$, hence the parameters $\a$ and $\b$ in \eqref{hkcross} take the form
$$
\a=7, \qquad \b=3.
$$
Further, see e.g.\ \cite[Chap.\,1]{Sh},
$$
{\vol}_{\mathbb{M}}(\mathbb{M}) = \frac{6 (4\pi)^{16/2}}{\kappa^{16} \, \Gamma(16/2+4)}.
$$

In the context of Lemma \ref{lem:G}, see also Lemma \ref{lem:A}, after simplifications we have
\begin{align*}
\frac{h_0^{\a,\b}}{{\vol}_{\mathbb{M}}(\mathbb{M})} \, c_{7,3,T}^{\textrm{ref}}
	& = 2 \exp(-\mathbb{D}_3) \, \frac{\kappa^{16}}{(4\pi)^{8}}, \\
\frac{h_0^{\a,\b}}{{\vol}_{\mathbb{M}}(\mathbb{M})} \, C_{7,3,T}^{\textrm{ref}}
	& = 16\sqrt{2} r_1 (4r_2)^{11} \, \frac{\kappa^{16}}{(4\pi)^{8}}.
\end{align*}
Note that $\mathbb{D}_3 = (105\sqrt{\pi}/16)^{2/7} \approx 2.015904$.

Then Lemma \ref{lem:G} together with \eqref{hkcross} gives the following bounds valid for $x,y \in \mathbb{M}$:
\begin{align*}
K_t^{\mathbb{M}}(x,y)
& \ge 2 \exp(-\mathbb{D}_3) \, \Psi_3^{\mathfrak{B}}\big( \kappa^2 t, \pi-\kappa\, d_{\mathbb{M}}(x,y),\pi\big) \,
	\frac{1}{(4\pi t)^{16/2}}\, e^{-d_{\mathbb{M}}^2(x,y)/{(4t)}}, \qquad t > 0, \\
K_t^{\mathbb{M}}(x,y)
& \le 16\sqrt{2}\, \mathfrak{w}_0 (2\sqrt{e}\,\mathfrak{w}_1)^{11} \\
	& \qquad \times \Psi_3^{\mathfrak{b}}\big( \kappa^2 t, \pi-\kappa\, d_{\mathbb{M}}(x,y),\pi\big) \,
	\frac{1}{(4\pi t)^{16/2}}\, e^{-d_{\mathbb{M}}^2(x,y)/{(4t)}}, \qquad  0 < t \le \frac{4}{23 \kappa^2}, \\
K_t^{\mathbb{M}}(x,y)
& \le 16\sqrt{2} \sqrt[4]{e}\, \mathfrak{w}'_0 (\pi/2)^{11} \\
	& \qquad \times \Psi_3^{\mathfrak{b}}\big( \kappa^2 t, \pi-\kappa\, d_{\mathbb{M}}(x,y),\pi\big) \,
	\frac{1}{(4\pi t)^{16/2}}\, e^{-d_{\mathbb{M}}^2(x,y)/{(4t)}}, \qquad  0 < t \le \frac{4}{(23 \kappa)^2}.
\end{align*}

Approximating the relevant numerical quantities\,\footnote{
\,One has $2\exp(-\mathbb{D}_3) \approx 0.266400$, $16\sqrt{2}\,\mathfrak{w}_0 \approx 18.056293$,
$2\sqrt{e}\,\mathfrak{w}_1 \approx 3.381748$, $16\sqrt{2} \sqrt[4]{e}\, \mathfrak{w}'_0 \approx 79.056423$.
}
, one can bound the constants
$$
2\exp(-\mathbb{D}_3) > \frac{1}4, \quad 16\sqrt{2}\,\mathfrak{w}_0 < 19, 
\quad 2\sqrt{e}\,\mathfrak{w}_1 < \frac{17}5, \quad 16\sqrt{2} \sqrt[4]{e}\, \mathfrak{w}'_0 < 80,
$$
and the upper and lower constants' ratios as
$$
\frac{16\sqrt{2}\, \mathfrak{w}_0 (2\sqrt{e}\,\mathfrak{w}_1)^{11}}{2 \exp(-\mathbb{D}_3)} < 68 \cdot \Big(\frac{17}5\Big)^{11}, \qquad
\frac{16\sqrt{2} \sqrt[4]{e}\, \mathfrak{w}'_0 (\pi/2)^{11}}{2 \exp(-\mathbb{D}_3)} < 297 \cdot \Big(\frac{\pi}2\Big)^{11}.
$$

\subsection{Jacobi heat kernel} \label{ssec:Jacobi} \,

\medskip

In this section we obtain reasonably precise estimates for the constants appearing in the bounds for
$G_t^{\a,\b}(\cos\theta,\cos\varphi)$ in Lemma~\ref{lem:F}.
We carry out the procedure in the most important case $\alpha,\beta \ge -1/2$.
A similar analysis in the remaining cases requires somewhat more effort, but it is merely a matter of lengthy technical elaboration
which is left to interested readers.

The general scheme of tracing constants via Lemmas \ref{lem:A}--\ref{lem:E} in the unified proof is the following.
\begin{itemize}
	\item[$\blacktriangleright$] In case $\a,\b \ge -1/2$,
	\begin{align*}
		\Lambda \in \mathbb{N}-1/2 & \quad \rightsquigarrow \quad \textrm{Lem.\,\ref{lem:A}}, \\
		\Lambda \in \mathbb{N} & \quad \rightsquigarrow \quad \textrm{Lem.\,\ref{lem:B}} \longrightarrow \textrm{Lem.\,\ref{lem:A}}, \\
		\Lambda \in (0,\infty) \setminus (\mathbb{N}/2) & \quad \rightsquigarrow \quad \textrm{Lem.\,\ref{lem:C}} \longrightarrow
		\textrm{Lem.\,\ref{lem:B}} \longrightarrow \textrm{Lem.\,\ref{lem:A}}, \\
		\Lambda \in (-1/2,0) & \quad \rightsquigarrow \quad \textrm{Lem.\,\ref{lem:D}} \longrightarrow
		\textrm{Lem.\,\ref{lem:C}} \longrightarrow \textrm{Lem.\,\ref{lem:B}} \longrightarrow \textrm{Lem.\,\ref{lem:A}}.
	\end{align*}
	\item[$\blacktriangleright$] In case $-1 < \b < -1/2 \le \a$ or $-1 < \a < -1/2 \le \b$,
	\begin{align*}
		\Lambda \in \mathbb{N}-1/2 & \quad \rightsquigarrow \quad 2\times \textrm{Lem.\,\ref{lem:A}}, \\
		\Lambda \in \mathbb{N} & \quad \rightsquigarrow \quad 2 \times \big[\textrm{Lem.\,\ref{lem:B}}
		\longrightarrow \textrm{Lem.\,\ref{lem:A}}\big], \\
		\Lambda \in (0,\infty) \setminus (\mathbb{N}/2) & \quad \rightsquigarrow \quad 2\times
			\big[\textrm{Lem.\,\ref{lem:C}} \longrightarrow
		\textrm{Lem.\,\ref{lem:B}} \longrightarrow \textrm{Lem.\,\ref{lem:A}}\big], \\
		\Lambda \in (-1,0)\setminus\{-1/2\} & \quad \rightsquigarrow \quad \left\langle
		\textrm{Lem.\,\ref{lem:D}} \longrightarrow
		\textrm{Lem.\,\ref{lem:C}} \longrightarrow \textrm{Lem.\,\ref{lem:B}} \longrightarrow \textrm{Lem.\,\ref{lem:A}} \atop
		\textrm{Lem.\,\ref{lem:C}} \longrightarrow \textrm{Lem.\,\ref{lem:B}} \longrightarrow \textrm{Lem.\,\ref{lem:A}} \right\rangle.
	\end{align*}
	\item[$\blacktriangleright$] In case $-1 < \a,\b < -1/2$,
	\begin{align*}
		\Lambda \in (-1,-1/2) & \quad \rightsquigarrow \quad \left\langle
		\textrm{Lem.\,\ref{lem:D}} \longrightarrow
		\textrm{Lem.\,\ref{lem:C}} \longrightarrow \textrm{Lem.\,\ref{lem:B}} \longrightarrow \textrm{Lem.\,\ref{lem:A}} \atop
		2 \times \big[\textrm{Lem.\,\ref{lem:C}} \longrightarrow \textrm{Lem.\,\ref{lem:B}} \longrightarrow \textrm{Lem.\,\ref{lem:A}} \big]
		\right\rangle, \\
		\Lambda = -1 & \quad \rightsquigarrow \quad \left\langle
		\textrm{Lem.\,\ref{lem:E}} \longrightarrow \textrm{Lem.\,\ref{lem:B}} \longrightarrow \textrm{Lem.\,\ref{lem:A}} \atop
		2 \times \big[\textrm{Lem.\,\ref{lem:B}} \longrightarrow \textrm{Lem.\,\ref{lem:A}} \big]
		\right\rangle, \\
		\Lambda \in (-3/2,-1) & \quad \rightsquigarrow \quad \left\langle
		\textrm{Lem.\,\ref{lem:E}} \longrightarrow
		\textrm{Lem.\,\ref{lem:C}} \longrightarrow \textrm{Lem.\,\ref{lem:B}} \longrightarrow \textrm{Lem.\,\ref{lem:A}} \atop
		2 \times \big[\textrm{Lem.\,\ref{lem:C}} \longrightarrow \textrm{Lem.\,\ref{lem:B}} \longrightarrow \textrm{Lem.\,\ref{lem:A}} \big]
		\right\rangle. \\
	\end{align*}
\end{itemize} 

From now on we focus on $\a,\b \ge -1/2$.

\subsubsection*{\textbf{\emph{Technical preliminaries}}}
Recall that $\Lambda = \Lambda_{\a,\b} = \a+\b+1/2$.
The following identities are instant consequences of the relevant formulas.
\begin{equation} \label{11.1}
	\Gamma(\alpha+1) \Gamma(\beta+1) \frac{h_0^{\Lambda}}{h_0^{\a,\b}}
	=
	\frac{\sqrt{\pi}}{2^{\Lambda +1/2}} \Gamma(\Lambda+1), \qquad \alpha, \beta > -1,
\end{equation}
\begin{equation} \label{id:80}
	\frac{h_{0}^{\alpha + 2}}{h_{0}^{\alpha}}
	=
	\frac{(\alpha + 2)(\alpha + 1)}{(\alpha + 5/2)(\alpha + 3/2)},
	\qquad
	\frac{h_{0}^{\alpha,\beta}}{h_{0}^{\alpha+1,\beta+1}}
	=
	\frac{(\alpha + \beta + 3)(\alpha + \beta + 2)}{4(\alpha + 1)(\beta + 1)},
	\qquad \alpha, \beta > -1,
\end{equation}
\begin{equation} \label{id:omeg}
	\frac{\omega_{\lambda+1}(T)}{\omega_\lambda(T)}
	=
	e^{(2\lambda - \lceil 2 \lambda \rceil)T}, \qquad
	\frac{\Omega_{\lambda+1}(T)}{\Omega_\lambda(T)}
	=
	e^{(2\lambda-\lceil 2\lambda \rceil+1)T}, \qquad T, \lambda > 0,
\end{equation}
\begin{equation} \label{id:bB}
	\frac{b_{\lambda+1}}{b_{\lambda}}
	=
	\frac{\lambda + 1}{(1 \wedge 2^{\lambda-1/2})}
	\exp(\mathbb{D}_{\lambda} - \mathbb{D}_{\lambda+1}), \qquad
	\frac{B_{\lambda+1}}{B_{\lambda}}
	=
	\frac{2^{\lambda+1/2} (\lambda + 1)}{(1 \vee 2^{\lambda-1/2}) },
	\qquad \lambda > -1/2.
\end{equation}

Next, we observe that for $\lambda \in \mathbb{N} - 1/2$
\begin{align} \label{id:Aquo}
	\begin{split}
		\frac{c^A_{\lambda+1,T}}{c^A_{\lambda,T}} 
		& =
		\frac{1}{4(\lambda + 1)}, \\
		\frac{C^A_{\lambda+1,T}}{C^A_{\lambda,T}}
		& =
		\frac{1}{\lambda + 1}
		\begin{cases}
			\sqrt{e}\, \mathfrak{w}_1/2 &
			\;\; \textrm{if} \;\;  T = 1/(2\lambda+4) \;\; \textrm{and} \;\; \lambda \neq -1/2\\
			\pi/8 &
			\;\;\textrm{if} \;\;  T = 1/(2\lambda+4)^2 
		\end{cases}
		=:
		\frac{1}{\lambda + 1} \mathcal{C}.
	\end{split}
\end{align}
In general, here and in what follows in Section \ref{ssec:Jacobi},
for the sake of clarity and unified treatment, we do not take into account finer values of $C^A_{\lambda,T}$
and wider ranges of $t$ provided by Lemma \ref{lem:A} in the cases $\lambda = \pm 1/2$. More precise analysis related to those cases
meets no additional difficulties and is left to interested readers.

We now compute analogous quotients for the constants of Lemma \ref{lem:C}.
For any $\lambda \in \mathbb{R}$, denote $$\breve{\lambda} := \lceil 2\lambda \rceil - \lambda$$ and notice that
$\lambda \le \breve{\lambda} < \lambda+1$.
By Lemmas \ref{lem:C} and \ref{lem:B}
\begin{align} \label{id:formC}
\begin{split}
	c^C_{\lambda,T} 
	& = 
	{c}^A_{\lambda + \breve{\lambda} +1/2,T/4} \,
	\frac{4^{\lambda + \breve{\lambda} +3/2}
	b_{\breve{\lambda}}}{\big(\mathbb{D}_{\breve{\lambda}}T \vee \mathfrak{B}\pi^2 \big)^{\breve{\lambda}+1/2}}
	\, \frac{h_0^{\lambda + \breve{\lambda}+1/2}}{h_0^{\lambda,\breve{\lambda}}}\, \omega_{\lambda}(T), \\
	C^C_{\lambda,T} 
	& = 
	{C}^A_{\lambda + \breve{\lambda} -1/2,T/4} \, 
	4^{\lambda + \breve{\lambda} +1}
	B_{\breve{\lambda} - 1}
	\frac{h_0^{\lambda + \breve{\lambda} - 1/2}}{h_0^{\lambda,\breve{\lambda} -1}}\, \Omega_{\lambda}(T).
\end{split}
\end{align}
It is straightforward to see that for any $T,\lambda > 0$
\begin{align*}
	\frac{c^C_{\lambda+1,T}}{c^C_{\lambda,T}} 
	& =
	\frac{c^A_{\lambda + \breve{\lambda} + 5/2,\ T/4}}{c^A_{\lambda + \breve{\lambda} + 1/2,\ T/4}} \,
	4^2 \, \frac{b_{\breve{\lambda}+1}}{b_{\breve{\lambda}}} \,
	\frac{ \big(\mathbb{D}_{\breve{\lambda}} T \vee \mathfrak{B} \pi^2\big)^{\breve{\lambda}+1/2} }
	{ \big(\mathbb{D}_{\breve{\lambda}+1} T \vee \mathfrak{B} \pi^2\big)^{\breve{\lambda}+3/2} }
	\,
	\frac{h_0^{\lambda + \breve{\lambda} + 5/2}}{h_0^{\lambda + \breve{\lambda} + 1/2}} \,
	\frac{h_0^{\lambda,\breve{\lambda}}}{h_0^{\lambda+1,\breve{\lambda}+1}} \,
	\frac{\omega_{\lambda+1}(T)}{\omega_\lambda(T)}, \\
	\frac{C^C_{\lambda+1,T}}{C^C_{\lambda,T}} 
	& =
	\frac{C^A_{\lambda + \breve{\lambda} + 3/2,\ T/4}}{C^A_{\lambda + \breve{\lambda} - 1/2,\ T/4}} 
	\, 4^{2}  \, \frac{B_{\breve{\lambda}}}{B_{\breve{\lambda}-1}} \,
	\frac{h_0^{\lambda + \breve{\lambda} + 3/2}}{h_0^{\lambda + \breve{\lambda} - 1/2}} \,
	\frac{h_0^{\lambda,\, \breve{\lambda}-1}}{h_0^{\lambda+1,\, \breve{\lambda}}} \,
	\frac{\Omega_{\lambda+1}(T)}{\Omega_\lambda(T)}.
\end{align*}
Then using \eqref{id:80}--\eqref{id:Aquo} together with the fact that $\lambda > 0$ implies $\breve{\lambda} \ge 1/2$,
we obtain for $\lambda > 0$
\begin{align} \label{id:Cquo1}
	\begin{split}
		\frac{c^C_{\lambda+1,T}}{c^C_{\lambda,T}} 
		& =
		\frac{1}{4(\lambda + 1)} 
		\,
		\exp(\mathbb{D}_{\breve{\lambda}} - \mathbb{D}_{\breve{\lambda}+1}) 
		\,
		\frac{ \big(\mathbb{D}_{\breve{\lambda}} T \vee \mathfrak{B} \pi^2\big)^{\breve{\lambda}+1/2} }
		{ \big(\mathbb{D}_{\breve{\lambda}+1} T \vee \mathfrak{B} \pi^2\big)^{\breve{\lambda}+3/2} }
		\,
		e^{(2\lambda - \lceil 2 \lambda \rceil)T}, \\
		\frac{C^C_{\lambda+1,T}}{C^C_{\lambda,T}} 
		& =
		\frac{4\mathcal{C}^{2}}{(\lambda + 1) }
		\, \frac{2^{\breve{\lambda}-1/2} }{(1 \vee 2^{\breve{\lambda}-3/2}) } \,
		e^{(2\lambda-\lceil 2\lambda \rceil+1)T}.
	\end{split}
\end{align}

We will also need the following specification of constants in Lemma~\ref{lem:D}.
\begin{prop} \label{prop:D}
	Let $-1 < \lambda < 0$ and let $T > 0$ be such that $T/4 \le 1/(2\lambda+ 2\breve{\lambda}+ 7)$.
	Then, with $c^A_{\lambda,T}$ and $C^A_{\lambda,T}$ as in \eqref{id:Aquo},
	\begin{align*}
		c^C_{\lambda+1,T} \wedge \bigg(
		c^C_{\lambda+2,T} \, m_{\lambda+2} \, \frac{8(\lambda+2)}{\pi} e^{-(2\lambda+4)T}\bigg)
		& =
		c^C_{\lambda+2,T} \, m_{\lambda+2} \, \frac{8(\lambda+2)}{\pi} e^{-(2\lambda+4)T}, \\
		C^C_{\lambda+1,T} \vee \bigg(C^C_{\lambda+2,T} \, M_{\lambda+2}
		\, \frac{16(\lambda+2)}{\pi^2}\, \Big(\frac{16 T}{\pi^2} +1 \Big)^{\lambda+1}\bigg)
		& =
		C^C_{\lambda+2,T} \, M_{\lambda+2}
		\, \frac{16(\lambda+2)}{\pi^2}\, \Big(\frac{16 T}{\pi^2} +1 \Big)^{\lambda+1}.
	\end{align*}
	Consequently, under the assumed conditions
	\begin{align*}
		c^{D}_{\lambda,T} &  = 
		\frac{64(\lambda+1)(\lambda+2)}{\pi^2}\, e^{-(4\lambda+6)T} 
		c^C_{\lambda+2,T} \, m_{\lambda+2}, \\
		C^{D}_{\lambda,T} &  = 
		\frac{64(\lambda+1)(\lambda+2)}{\pi^2}\,
		C^C_{\lambda+2,T} \, M_{\lambda+2}
		\, \Big(\frac{16 T}{\pi^2} +1 \Big)^{\lambda+1}.
	\end{align*}
\end{prop}

\begin{proof}
	From \eqref{DDD} we see that for $\lambda > -1$ and $T/4 \le 1/(2\lambda+ 2\breve{\lambda}+ 7)$
	\begin{align*}
		\mathbb{D}_{\breve{\lambda}+1} T
		\le
		\mathbb{D}_{\breve{\lambda}+2} T
		\le 
		\mathfrak{B} \pi^2.
	\end{align*}	
	Combining this with \eqref{id:Cquo1} we obtain
	\begin{align*} 
		\begin{split}
			\frac{c^C_{\lambda+2,T}}{c^C_{\lambda+1,T}} \,
			m_{\lambda+2} \, \frac{8(\lambda+2)}{\pi} e^{-(2\lambda+4)T}
			& =
			\frac{2 m_{\lambda+2}}{\mathfrak{B} \pi^3} 
			\,
			\exp(\mathbb{D}_{\breve{\lambda} +1} - \mathbb{D}_{\breve{\lambda}+2}) 
			\,
			e^{( - \lceil 2 \lambda \rceil - 4)T}, \\
			\frac{C^C_{\lambda+2,T}}{C^C_{\lambda+1,T}} \,
			M_{\lambda+2}
			\, \frac{16(\lambda+2)}{\pi^2}\, \Big(\frac{16 T}{\pi^2} +1 \Big)^{\lambda+1}
			& \ge
			\frac{64\mathcal{C}^{2} M_{\lambda+2}}{\pi^{2} }
			\, \frac{2^{\breve{\lambda}+1/2}}{(1 \vee 2^{\breve{\lambda}-1/2}) } \,
			e^{(2\lambda-\lceil 2\lambda \rceil+1)T}.
		\end{split}
	\end{align*}
	Using now the estimates ${2 m_{\lambda+2}}/({\mathfrak{B} \pi^3}) \le 1$ and
	${64\mathcal{C}^{2} M_{\lambda+2}}/{\pi^{2} } \ge 1$ we get the desired conclusion.
\end{proof}
Finally, note that in the circumstances of Proposition \ref{prop:D} we have, see \eqref{id:formC},
\begin{align} \label{id:formD}
\begin{split}
	c^D_{\lambda,T} 
	&= 
	\frac{64(\lambda+1) (\lambda+2)}{\pi^2} \, e^{-(4\lambda+6) T}  
	\,
	c^A_{\lambda + \breve{\lambda} + 9/2, T/4} \\
	 & \qquad \times
	\frac{4^{\lambda + \breve{\lambda} + 11/2} \, b_{\breve{\lambda}+2}}{\big( \mathbb{D}_{\breve{\lambda}+2} T 
		\vee \mathfrak{B} \pi^2 \big)^{\breve{\lambda} + 5/2}} \,
	\frac{h_0^{\lambda + \breve{\lambda} + 9/2}}{h_0^{\lambda+2,\breve{\lambda}+2}} \, \omega_{\lambda+2}(T) \, m_{\lambda+2} \\
	&= 
	\frac{2^{\lambda + \breve{\lambda} + 12}(\lambda+1)}{\pi^2} \, e^{-(4\lambda+6) T}  
	\,
	c^A_{\lambda + \breve{\lambda} + 9/2, T/4}\\
	&\qquad  \times
	\frac{1}{\big( \mathbb{D}_{\breve{\lambda}+2}  T \vee \mathfrak{B} \pi^2 \big)^{\breve{\lambda} + 5/2}} \,
	\frac{\Gamma(\lambda + \breve{\lambda} + 11/2)}{\Gamma(\lambda + 3)}
	 \, \omega_{\lambda+2}(T) \, \frac{1-2^{-\lambda-2}}{e}
	 \exp\big(- \mathbb{D}_{\breve{\lambda} +2}\big), \\
	C^D_{\lambda,T} 
	&= 
	\frac{64(\lambda+1) (\lambda+2)}{\pi^2} 
	\,
	C^A_{\lambda + \breve{\lambda} + 7/2, T/4} \\
	&\qquad  \times
	4^{\lambda + \breve{\lambda} + 5} \, B_{\breve{\lambda}+1} \,
	\frac{h_0^{\lambda + \breve{\lambda} + 7/2}}{h_0^{\lambda+2,\breve{\lambda}+1}} \, \Omega_{\lambda+2}(T) \, M_{\lambda+2} \, 
	\Big( \frac{16T}{\pi^2} + 1 \Big)^{\lambda+1} \\
&= 
\frac{2^{\lambda + 2\breve{\lambda} + 27/2}(\lambda+1) }{\pi^2} 
\,
C^A_{\lambda + \breve{\lambda} + 7/2, T/4} \\
&\qquad  \times
\frac{\Gamma(\lambda + \breve{\lambda} + 9/2)}{\Gamma(\lambda + 3)} \, \Omega_{\lambda+2}(T) \, \frac{1}{(1-1/e)^2} \, 
	\Big( \frac{16T}{\pi^2} + 1 \Big)^{\lambda+1},
\end{split}
\end{align}
provided that $T/4 \le 1/(2\lambda+ 2\breve{\lambda}+ 7)$;
here we used the fact that $\breve{\lambda} \ge -1/2$ if $\lambda > -1$.

We are now in a position to estimate the constants $c_{\a,\b,T}$ and $C_{\a,\b,T}$ as well as their quotient
$C_{\a,\b,T}/c_{\a,\b,T}$. We will consider four subcases of the case $\a,\b \ge -1/2$ according to the scheme presented
at the beginning of Section~\ref{ssec:Jacobi}.

\subsubsection*{\textbf{\emph{The case when}  $\mathbf{\Lambda \in \mathbb{N}-1/2}$.}}
To begin with, we focus on the simple case $\Lambda=-1/2$ which, given the standing assumption $\a,\b \ge -1/2$,
occurs only when $\a=\b=-1/2$. But then we have
$$
G_t^{-1/2}(\theta,\varphi) = \vartheta_t(\theta-\varphi) + \vartheta_t(\theta+\varphi) \le 2 \vartheta_t(\theta-\varphi),
	\qquad \theta,\varphi \in [0,\pi], \quad t > 0,
$$
and in view of the upper bound for $K_t^1(\phi) = \vartheta_t(\phi)$ obtained in Section \ref{ssec:spherical} we conclude that
\begin{equation} \label{sc111}
G_t^{-1/2}(\theta,\varphi) \le \frac{2}{\sqrt{\pi}} \big(1+2e^{-2\pi^2/T}\big)
	\frac{1}{\sqrt{t}} e^{-(\theta-\varphi)^2/(4t)}, \qquad \theta,\varphi \in [0,\pi], \quad 0 < t \le T \le 2\pi^2.
\end{equation}
Thus, in what follows we may assume that $\Lambda \neq -1/2$, i.e.\ $\Lambda \in \mathbb{N}+1/2$.

In view of Lemma \ref{lem:F}, we have
\begin{align*}
	c_{\a,\b,T}  
	= 
	c^{A}_{\Lambda,T/4} \, 
	{4^{\Lambda+1} b_{\a} b_{\b}} \, \frac{h_0^{\Lambda}}{h_0^{\a,\b}}, \qquad
	C_{\a,\b,T}  = C^{A}_{\Lambda,T/4} \,  
	{4^{\Lambda+2} B_{\a} B_{\b}} \, \frac{h_0^{\Lambda}}{h_0^{\a,\b}}.
\end{align*}
Using \eqref{11.1} and \eqref{DDD} we see that
\begin{align*}
	c_{\a,\b,T}  & = 
	c^{A}_{\Lambda,T/4} \, 
	4^{\Lambda+1} (1 \wedge 2^{\a-1/2}) (1 \wedge 2^{\b-1/2})
	\frac{\Gamma(\a+1) \Gamma(\b+1)}{\pi} \,
		e^{- \chi_{\{\a\neq -1/2\}}\mathbb{D}_{\a} - \chi_{\{\b\neq -1/2\}}\mathbb{D}_{\b}}\,
		\frac{h_0^{\Lambda}}{h_0^{\a,\b}} \\
		& =
		c^{A}_{\Lambda,T/4} \, 
		4^{\Lambda+1} (1 \wedge 2^{\a-1/2}) (1 \wedge 2^{\b-1/2})
		\frac{\Gamma(\Lambda+1) }{\sqrt{\pi}2^{\Lambda +1/2}} \,
		e^{- \chi_{\{\a\neq -1/2\}}\mathbb{D}_{\a} - \chi_{\{\b\neq -1/2\}}\mathbb{D}_{\b}} \\
		& \ge 
		c^{A}_{\Lambda,T/4} \, 
		2^{\Lambda-1/2} \,
		\frac{\Gamma(\Lambda+1) }{\sqrt{\pi}}\,
		\exp\Big(- \frac{\Lambda+5/2}{e^{\gamma}} \Big),  \\
		C_{\a,\b,T}  
		& = 
		C^{A}_{\Lambda,T/4} \,  
		4^{\Lambda+2} 
			(1 \vee 2^{\a-1/2})
			(1 \vee 2^{\b-1/2}) \, 2^{-\chi_{\{\a = -1/2\}} - \chi_{\{\b = -1/2\}}}\, \frac{\Gamma(\a+1) \Gamma(\b+1)}{\pi}\,
			 \frac{h_0^{\Lambda}}{h_0^{\a,\b}} \\
	& = 
	C^{A}_{\Lambda,T/4} \,  
	4^{\Lambda+2} 
	(1 \vee 2^{\a-1/2})	(1 \vee 2^{\b-1/2}) \,
	2^{-\chi_{\{\a = -1/2\}} - \chi_{\{\b = -1/2\}}} \,
	\frac{\Gamma(\Lambda+1) }{\sqrt{\pi}2^{\Lambda +1/2}}\\
	& \le
	C^{A}_{\Lambda,T/4} \,  
		2^{2\Lambda+4} \,
		\frac{\Gamma(\Lambda+1) }{\sqrt{\pi}} 
		\end{align*}
and
$$
	\frac{C_{\a,\b,T}}{c_{\a,\b,T}}  
	\le 
\frac{C^{A}_{\Lambda,T/4}}{c^{A}_{\Lambda,T/4}}  
\,
2^{\Lambda +9/2}\, \exp\Big(\frac{\Lambda + 5/2}{e^{\gamma}} \Big).
$$

Summarizing, we obtain
\begin{align*}
	c_{\a,\b,T}  
	& \ge 
	\frac{1}{2^{\Lambda+3/2} \sqrt{\pi}} \,
	\exp\Big(- \frac{\Lambda+5/2}{e^{\gamma}} \Big) \\
	& = \frac{1}{2\sqrt{\pi} \exp(2e^{-\gamma}) \, [2\exp(e^{-\gamma})]^{\Lambda+1/2}}
	\qquad \textrm{for} \;\; T >0,\\
	C_{\a,\b,T}  
	& \le
	\frac{2^{2\Lambda+4}}{\sqrt{\pi}}
	 \times
	\begin{cases}
		\mathfrak{w}_0(\sqrt{e}\, \mathfrak{w}_1/2)^{\Lambda+1/2} &
		\;\; \textrm{if} \;\; T = 2/(\Lambda+1) \\
		\sqrt[4]{e} \, \mathfrak{w}'_0 (\pi/8)^{\Lambda+1/2} &
		\;\;\textrm{if}  \;\; T = 1/(\Lambda+1)^2
	\end{cases} \\
	& =
	\begin{cases}
		8 \pi^{-1/2}\,\mathfrak{w}_0\, [2\sqrt{e}\,\mathfrak{w}_1]^{\Lambda+1/2} &
		\;\; \textrm{if} \;\; T = 2/(\Lambda+1) \\
		8\pi^{-1/2}\sqrt[4]{e}\,\mathfrak{w}'_0 \,[\pi/2]^{\Lambda+1/2} &
		\;\;\textrm{if}  \;\; T = 1/(\Lambda+1)^2
	\end{cases}, \\
	\frac{C_{\a,\b,T}}{c_{\a,\b,T}}  
& \le 
2^{3\Lambda +11/2}\, \exp\Big(\frac{\Lambda + 5/2}{e^{\gamma}} \Big)
 \times
\begin{cases}
	\mathfrak{w}_0(\sqrt{e}\, \mathfrak{w}_1/2)^{\Lambda+1/2} &
	\;\; \textrm{if} \;\; T = 2/(\Lambda+1) \\
	\sqrt[4]{e} \, \mathfrak{w}'_0 (\pi/8)^{\Lambda+1/2} &
	\;\;\textrm{if}  \;\; T = 1/(\Lambda+1)^2
\end{cases} \\
& =
\begin{cases}
	16\, \mathfrak{w}_0 \exp(2e^{-\gamma}) \, [4\sqrt{e}\,\mathfrak{w}_1 \exp(e^{-\gamma})]^{\Lambda+1/2} &
	\;\; \textrm{if} \;\; T = 2/(\Lambda+1) \\
	16 \sqrt[4]{e}\,\mathfrak{w}'_0 \exp(2e^{-\gamma}) \, [\pi\exp(e^{-\gamma})]^{\Lambda+1/2} &
	\;\;\textrm{if}  \;\; T = 1/(\Lambda+1)^2
\end{cases}.
\end{align*}

Approximating above the relevant expressions\,\footnote{
\,One has
$2\sqrt{\pi} \exp(2e^{-\gamma}) \approx 10.896385$,
$2\exp(e^{-\gamma}) \approx3.506459$,
$8 \pi^{-1/2}\,\mathfrak{w}_0 \approx 3.601709$,
$2\sqrt{e}\,\mathfrak{w}_1 \approx 3.381748$,
$8\pi^{-1/2}\sqrt[4]{e}\,\mathfrak{w}'_0 \approx 15.769475$,
$\pi/2 \approx 1.570796$,
$16\, \mathfrak{w}_0 \exp(2e^{-\gamma}) \approx 39.245612$,
$4\sqrt{e}\,\mathfrak{w}_1 \exp(e^{-\gamma}) \approx 11.857960$,
$16 \sqrt[4]{e}\,\mathfrak{w}'_0 \exp(2e^{-\gamma}) \approx 171.830270$,
$\pi\exp(e^{-\gamma}) \approx 5.507933$.
},
we obtain the following explicit numerical bounds.
\begin{prop} \label{prop:c1}
Let $\a,\b \ge -1/2$ be such that $\a+\b+1 \in \mathbb{N}\setminus\{0\}$. Then
\begin{align*}
c_{\a,\b,T} & \ge \frac{1}{11 \cdot (18/5)^{\a+\b+1}} \qquad \textrm{for} \;\; T > 0,\\
C_{\a,\b,T} & \le
\begin{cases}
		(11/3) \cdot (17/5)^{\a+\b+1} &
		\;\; \textrm{if} \;\; T = 2/(\a+\b+3/2) \\
		16 \cdot (8/5)^{\a+\b+1} &
		\;\;\textrm{if}  \;\; T = 1/(\a+\b+3/2)^2
	\end{cases}, \\
\frac{C_{\a,\b,T}}{c_{\a,\b,T}} & \le
\begin{cases}
	40 \cdot 12^{\a+\b+1} &
	\;\; \textrm{if} \;\; T = 2/(\a+\b+3/2) \\
	172 \cdot (28/5)^{\a+\b+1} &
	\;\;\textrm{if}  \;\; T = 1/(\a+\b+3/2)^2
\end{cases}.
\end{align*}
\end{prop}

\subsubsection*{\textbf{\emph{The case when} $\mathbf{\Lambda \in \mathbb{N}}$.}} By Lemmas \ref{lem:F} and \ref{lem:B} we have
\begin{align*}
	c_{\alpha,\beta,T} &= 
	c^{B}_{\Lambda,T/4} \, 
	{4^{\Lambda+1} b_{\a} b_{\b}} \, \frac{h_0^{\Lambda}}{h_0^{\a,\b}}
	=
	c^A_{2\Lambda + 1/2, T/16} \,
	\frac{4^{3\Lambda + 5/2} \,
	b_{\Lambda} b_{\alpha} b_{\beta}}{\big( \mathbb{D}_{\Lambda} \, T/4 \vee \mathfrak{B} \pi^2 \big)^{\Lambda + 1/2}} \,
	\frac{h_0^{2\Lambda + 1/2}}{h_0^{\alpha,\beta}}, \\
	C_{\alpha,\beta,T} &= 
	C^{B}_{\Lambda,T/4} \,  
	{4^{\Lambda+2} B_{\a} B_{\b}} \, \frac{h_0^{\Lambda}}{h_0^{\a,\b}}
	=
	C^A_{2\Lambda + 1/2, T/16} \,
	4^{3\Lambda + 4} \, B_{\Lambda} B_{\alpha} B_{\beta} \,
	\frac{h_0^{2\Lambda + 1/2}}{h_0^{\alpha,\beta}}.
\end{align*}
Since $\mathbb{D}_{\Lambda}T/4 \le \mathfrak{B} \pi^2$ for $T/16 \le 1/(4\Lambda + 3)$, we have (see also \eqref{11.1} and \eqref{DDD})
\begin{align*}
	c_{\alpha,\beta,T} 
	&= 
	c^A_{2\Lambda + 1/2, T/16} \,
	\frac{4^{3\Lambda + 5/2} \, (1 \wedge 2^{\Lambda-1/2}) (1 \wedge 2^{\a-1/2}) (1 \wedge 2^{\b-1/2})
	\Gamma(\Lambda+1) \Gamma(\a+1) \Gamma(\b+1)}{\pi^{3/2}(\mathfrak{B} \pi^2)^{\Lambda + 1/2}} \\
	& \qquad \times 
	e^{-  \mathbb{D}_{\Lambda} - \chi_{\{\a\neq -1/2\}} \mathbb{D}_{\a} - \chi_{\{\b\neq -1/2\}} \mathbb{D}_{\b} } \,
	\frac{h_0^{2\Lambda + 1/2}}{h_0^{\alpha,\beta}} \\
	&= 
	c^A_{2\Lambda + 1/2, T/16} \,
	\frac{4^{3\Lambda + 5/2} \, (1 \wedge 2^{\Lambda-1/2}) (1 \wedge 2^{\a-1/2}) (1 \wedge 2^{\b-1/2})
	\Gamma(2\Lambda+3/2) 2^{-3 \Lambda -3/2} }{\sqrt{\pi}(\mathfrak{B} \pi^2)^{\Lambda + 1/2}} \\
	& \qquad \times
	e^{-  \mathbb{D}_{\Lambda} - \chi_{\{\a\neq -1/2\}} \mathbb{D}_{\a} - \chi_{\{\b\neq -1/2\}} \mathbb{D}_{\b} } \\
	&\ge 
	c^A_{2\Lambda + 1/2, T/16} \,
	\frac{2^{4\Lambda + 3/2}   \Gamma(2\Lambda+3/2)  }{ \pi^{2\Lambda + 3/2}} \,
	\exp\Big(- \frac{2\Lambda+4}{e^{\gamma}} \Big), \\
	C_{\alpha,\beta,T} 
	&= 
	C^A_{2\Lambda + 1/2, T/16} \,
	4^{3\Lambda + 4} \, (1 \vee 2^{\Lambda-1/2}) (1 \vee 2^{\a-1/2}) (1 \vee 2^{\b-1/2}) \,
	2^{-\chi_{\{\a = -1/2\}} - \chi_{\{\b = -1/2\}}}  \\
	& \qquad \times
	\frac{\Gamma(\Lambda+1) \Gamma(\a+1) \Gamma(\b+1)}{\pi^{3/2}}\,
	\frac{h_0^{2\Lambda + 1/2}}{h_0^{\alpha,\beta}} \\
	&= 
	C^A_{2\Lambda + 1/2, T/16} \,
	4^{3\Lambda + 4} \, (1 \vee 2^{\Lambda-1/2}) (1 \vee 2^{\a-1/2}) (1 \vee 2^{\b-1/2})\,
	2^{-\chi_{\{\a = -1/2\}} - \chi_{\{\b = -1/2\}}} \\
	& \qquad \times
	 \Gamma(2\Lambda+3/2) \,2^{-3 \Lambda -3/2}/ \sqrt{\pi} \\
	& \le 
		C^A_{2\Lambda + 1/2, T/16} \,
	2^{5\Lambda + 7} \,\Gamma(2\Lambda+3/2) / \sqrt{\pi},
\end{align*}
provided that $T/16 \le 1/(4\Lambda + 3)$. Further, we have
\begin{align*}
	\frac{C_{\alpha,\beta,T}}{c_{\alpha,\beta,T}}
	\le 
	\frac{C^A_{2\Lambda + 1/2, T/16}}{c^A_{2\Lambda + 1/2, T/16}}
	 \,
	2^{\Lambda + 11/2} \pi^{2\Lambda + 1} \,
	\exp\Big( \frac{2\Lambda+4}{e^{\gamma}} \Big),
\end{align*}
provided that $T/16 \le 1/(4\Lambda + 3)$.

Summarizing, we obtain
\begin{align*}
	c_{\alpha,\beta,T} 
	&\ge 
	\frac{1}{\sqrt{2} \pi^{2\Lambda + 3/2}} \,
	\exp\Big(- \frac{2\Lambda+4}{e^{\gamma}} \Big) \\
	& = \frac{1}{\sqrt{2\pi}\exp(3e^{-\gamma}) \, [\pi^2 \exp(2e^{-\gamma})]^{\Lambda+1/2}}
	\qquad \textrm{if} \;\; T = 16/(4\Lambda + 3), \\
	C_{\alpha,\beta,T} 
	& \le 
	\frac{2^{5\Lambda + 7}}{\sqrt{\pi}}
	 \times
	\begin{cases}
		\mathfrak{w}_0(\sqrt{e}\, \mathfrak{w}_1/2)^{2\Lambda + 1} &
		\;\; \textrm{if} \;\; T = 16/(4\Lambda + 3) \\
		\sqrt[4]{e} \, \mathfrak{w}'_0 (\pi/8)^{2\Lambda + 1} &
		\;\;\textrm{if}  \;\; T = 16/(4\Lambda + 3)^{2} 
	\end{cases} \\
	& =
	\begin{cases}
		16 \sqrt{2} \, \pi^{-1/2}\, \mathfrak{w}_0 \, [8 e \, \mathfrak{w}_1^2]^{\Lambda+1/2} &
		\;\; \textrm{if} \;\; T = 16/(4\Lambda + 3) \\
		16 \sqrt{2} \, \pi^{-1/2}\,  \sqrt[4]{e}\, \mathfrak{w}'_0 \, [\pi^2/2]^{\Lambda+1/2} &
		\;\;\textrm{if}  \;\; T = 16/(4\Lambda + 3)^{2} 
	\end{cases}, \\
	\frac{C_{\alpha,\beta,T}}{c_{\alpha,\beta,T}}
& \le 
2^{5\Lambda + 15/2} \pi^{2\Lambda + 1}\,
\exp\Big( \frac{2\Lambda+4}{e^{\gamma}} \Big)
\times
\begin{cases}
	\mathfrak{w}_0(\sqrt{e}\, \mathfrak{w}_1/2)^{2\Lambda + 1} &
	\;\; \textrm{if} \;\; T = 16/(4\Lambda + 3) \\
	\sqrt[4]{e} \, \mathfrak{w}'_0 (\pi/8)^{2\Lambda + 1} &
	\;\;\textrm{if}  \;\; T = 16/(4\Lambda + 3)^{2}
\end{cases} \\
& =
\begin{cases}
	32\, \mathfrak{w}_0 \exp(3e^{-\gamma})\, [8\pi^2 \mathfrak{w}_1^2 \exp(2e^{-\gamma}+1)]^{\Lambda+1/2}  &
	\;\; \textrm{if} \;\; T = 16/(4\Lambda + 3) \\
	32\, \mathfrak{w}'_0 \exp(3e^{-\gamma} + 1/4)\, [2^{-1} \pi^4 \exp(2e^{-\gamma})]^{\Lambda+1/2}  &
	\;\;\textrm{if}  \;\; T = 16/(4\Lambda + 3)^{2}
\end{cases}.
\end{align*}

Approximating above the relevant expressions\,\footnote{
\,One has
$\sqrt{2\pi}\exp(3e^{-\gamma}) \approx 13.508471$,
$\pi^2 \exp(2e^{-\gamma}) \approx 30.337323$,
$16 \sqrt{2} \, \pi^{-1/2}\, \mathfrak{w}_0 \approx 10.187172$,
$8 e \, \mathfrak{w}_1^2 \approx 22.872438$,
$16 \sqrt{2} \, \pi^{-1/2}\,  \sqrt[4]{e}\, \mathfrak{w}'_0 \approx 44.602810$,
$\pi^2/2 \approx 4.934802$,
$32\, \mathfrak{w}_0 \exp(3e^{-\gamma}) \approx 137.613123$,
$8\pi^2 \mathfrak{w}_1^2 \exp(2e^{-\gamma}+1) \approx 693.888536$,
$32\, \mathfrak{w}'_0 \exp(3e^{-\gamma} + 1/4) \approx 602.515779$,
$2^{-1} \pi^4 \exp(2e^{-\gamma}) \approx 149.708689$.
},
we get the following numerical bounds.
\begin{prop} \label{prop:c2}
Let $\a,\b \ge -1/2$ be such that $\a+\b+1/2 \in \mathbb{N}$. Then
\begin{align*}
c_{\a,\b,T} & \ge \frac{1}{14 \cdot 31^{\a+\b+1}} \qquad \textrm{if} \;\; T = 4/(\a+\b+5/4),\\
C_{\a,\b,T} & \le
\begin{cases}
		11 \cdot 23^{\a+\b+1} &
		\;\; \textrm{if} \;\; T = 4/(\a+\b+5/4) \\
		45 \cdot 5^{\a+\b+1} &
		\;\;\textrm{if}  \;\; T = 1/(\a+\b+5/4)^2
	\end{cases}, \\
\frac{C_{\a,\b,T}}{c_{\a,\b,T}} & \le
\begin{cases}
	138 \cdot 694^{\a+\b+1} &
	\;\; \textrm{if} \;\; T = 4/(\a+\b+5/4) \\
	603 \cdot 150^{\a+\b+1} &
	\;\;\textrm{if}  \;\; T = 1/(\a+\b+5/4)^2
\end{cases}.
\end{align*}
\end{prop}

\subsubsection*{\textbf{\emph{The case when} $\mathbf{\Lambda \in (0,\infty) \setminus (\mathbb{N}/2)}$.}}
We have, see Lemmas \ref{lem:F}, \ref{lem:C}, \ref{lem:B},
\begin{align*}
	c_{\alpha,\beta,T}  
	&= 
	c^{C}_{\Lambda,T/4} \, 
	{4^{\Lambda+1} b_{\a} b_{\b}} \, \frac{h_0^{\Lambda}}{h_0^{\a,\b}} \\
	& =
	c^A_{\Lambda + \breve{\Lambda} + 1/2, T/16} \,
	\frac{4^{2 \Lambda + \breve{\Lambda} + 5/2} \, b_{\breve{\Lambda}} \, b_{\alpha} b_{\beta}}
	{\big( \mathbb{D}_{\breve{\Lambda}} \, T/4 \vee \mathfrak{B} \pi^2 \big)^{\breve{\Lambda} + 1/2}} \,
	\frac{h_0^{\Lambda + \breve{\Lambda} + 1/2} \, h_0^{\Lambda}}{h_0^{\Lambda,\breve{\Lambda}} \, h_0^{\alpha,\beta}}
	\, \omega_{\Lambda}(T/4), \\
	C_{\alpha,\beta,T}  
	&= 
	C^{C}_{\Lambda,T/4} \,  
	{4^{\Lambda+2} B_{\a} B_{\b}} \, \frac{h_0^{\Lambda}}{h_0^{\a,\b}} \\
	&=
	C^A_{\Lambda + \breve{\Lambda} - 1/2, T/16} \,
	4^{2 \Lambda + \breve{\Lambda} + 3} \, B_{\breve{\Lambda} - 1} \, B_{\alpha} B_{\beta}
 \,
	\frac{h_0^{\Lambda + \breve{\Lambda} - 1/2} \, h_0^{\Lambda}}{h_0^{\Lambda,\breve{\Lambda} - 1}
	\, h_0^{\alpha,\beta}} \, \Omega_{\Lambda}(T/4).
\end{align*}
Notice that $\breve{\Lambda} - 1 \ne -1/2$.
Taking into account that $\mathbb{D}_{\breve{\Lambda}}T/4 \le \mathfrak{B} \pi^2$ for
$T/16 \le 1/(2\Lambda + 2\breve{\Lambda} + 3)$,
and using the bounds $(1 \wedge 2^{\a-1/2}) (1 \wedge 2^{\b-1/2}) \ge 2^{-3/2}$,
$(1 \vee 2^{\a-1/2}) (1 \vee 2^{\b-1/2}) \le 2^{\Lambda}$ and also \eqref{11.1} and \eqref{DDD}, we obtain
\begin{align*}
	c_{\alpha,\beta,T}  
&= 
c^A_{\Lambda + \breve{\Lambda} + 1/2, T/16} \,
\frac{4^{2\Lambda + \breve{\Lambda} + 5/2} \, b_{\breve{\Lambda}} \, b_{\alpha} b_{\beta}}
{(\mathfrak{B} \pi^2)^{\breve{\Lambda} + 1/2}} \,
\frac{h_0^{\Lambda + \breve{\Lambda} + 1/2} \, h_0^{\Lambda}}{h_0^{\Lambda,\breve{\Lambda}} \, h_0^{\alpha,\beta}}
\, \omega_{\Lambda}(T/4) \\
& =
c^A_{\Lambda + \breve{\Lambda} + 1/2, T/16} \,
\frac{4^{2\Lambda + \breve{\Lambda} + 5/2} \, (1 \wedge 2^{\breve{\Lambda}-1/2}) (1 \wedge 2^{\a-1/2}) (1 \wedge 2^{\b-1/2})
\Gamma(\breve{\Lambda}+1) \Gamma(\a+1) \Gamma(\b+1) }
{\pi^{3/2}(\mathfrak{B} \pi^2)^{\breve{\Lambda} + 1/2}} \\
& \qquad \times
e^{- \mathbb{D}_{\breve{\Lambda}} - \chi_{\{\a\neq -1/2\}} \mathbb{D}_{\a} - \chi_{\{\b\neq -1/2\}} \mathbb{D}_{\b} }
 \,
\frac{h_0^{\Lambda + \breve{\Lambda} + 1/2} \, h_0^{\Lambda}}{h_0^{\Lambda,\breve{\Lambda}}
\, h_0^{\alpha,\beta}} \, \omega_{\Lambda}(T/4) \\
& =
c^A_{\Lambda + \breve{\Lambda} + 1/2, T/16} \,
\frac{2^{2\Lambda  + 2\breve{\Lambda}  + 4} \, (1 \wedge 2^{\breve{\Lambda}-1/2}) (1 \wedge 2^{\a-1/2}) (1 \wedge 2^{\b-1/2})  }
{\pi^{2\breve{\Lambda} + 3/2}} \\
& \qquad \times
e^{- \mathbb{D}_{\breve{\Lambda}} - \chi_{\{\a\neq -1/2\}} \mathbb{D}_{\a} - \chi_{\{\b\neq -1/2\}} \mathbb{D}_{\b} } 
\,
\Gamma(\Lambda+ \breve{\Lambda}  + 3/2) 
 \, \omega_{\Lambda}(T/4) \\
 & \ge 
c^A_{\Lambda + \breve{\Lambda} + 1/2, T/16} \,
\frac{2^{2\Lambda  + 2\breve{\Lambda}  + 5/2}  \Gamma(\Lambda+ \breve{\Lambda}  + 3/2)  }
{\pi^{2\breve{\Lambda} + 3/2}}\,
\exp\Big(- \frac{\Lambda + \breve{\Lambda} + 4}{e^{\gamma}} \Big)\,
e^{-(\Lambda+3/4)T/4}, \\
	C_{\alpha,\beta,T}  
	& =
	C^A_{\Lambda + \breve{\Lambda} - 1/2, T/16} \,
4^{2\Lambda + \breve{\Lambda} + 3} \, B_{\breve{\Lambda} - 1} \, B_{\alpha} B_{\beta}
\,
\frac{h_0^{\Lambda + \breve{\Lambda} - 1/2} \, h_0^{\Lambda}}{h_0^{\Lambda,\breve{\Lambda} - 1}
\, h_0^{\alpha,\beta}} \, \Omega_{\Lambda}(T/4) \\
&=	
C^A_{\Lambda + \breve{\Lambda} - 1/2, T/16} \,
4^{2\Lambda + \breve{\Lambda} + 3} \, 
\frac{ (1 \vee 2^{\breve{\Lambda}-3/2}) (1 \vee 2^{\a-1/2}) (1 \vee 2^{\b-1/2})\,
2^{-\chi_{\{\a = -1/2\}} - \chi_{\{\b = -1/2\}}} }
{\pi^{3/2}} \\
& \qquad \times
\Gamma(\breve{\Lambda}) \Gamma(\a+1) \Gamma(\b+1) \,
\frac{h_0^{\Lambda + \breve{\Lambda} - 1/2} \, h_0^{\Lambda}}{h_0^{\Lambda,\breve{\Lambda} - 1}
\, h_0^{\alpha,\beta}} \, \Omega_{\Lambda}(T/4) \\
&=	
C^A_{\Lambda + \breve{\Lambda} - 1/2, T/16} \,
4^{2\Lambda + \breve{\Lambda} + 3} \, 
\frac{ (1 \vee 2^{\breve{\Lambda}-3/2}) (1 \vee 2^{\a-1/2}) (1 \vee 2^{\b-1/2}) \,
	2^{-\chi_{\{\a = -1/2\}} - \chi_{\{\b = -1/2\}}} }
{\sqrt{\pi}} \\
& \qquad \times
	\frac{ \Gamma(\Lambda + \breve{\Lambda}  + 1/2) }{2^{2\Lambda + \breve{\Lambda} +1/2}}
 \, \Omega_{\Lambda}(T/4) \\
 & \le 
 C^A_{\Lambda + \breve{\Lambda} - 1/2, T/16} \,
 2^{3 \Lambda  + 2 \breve{\Lambda} + 5} 
\pi^{-1/2}  \, \Gamma(\Lambda+ \breve{\Lambda}  + 1/2) 
 \, e^{(\Lambda+1/2)T/4},
\end{align*}
provided that $T/16 \le 1/(2\Lambda + 2\breve{\Lambda} + 3)$.
Further, we have
\begin{align*}
	\frac{C_{\a,\b,T}}{c_{\a,\b,T}}  
	& \le
	2^{ \Lambda  + 5/2} 
	\pi^{2\breve{\Lambda} + 1}\,
\frac{C^A_{\Lambda+ \breve{\Lambda} - 1/2, T/16}  }
{c^A_{\Lambda+ \breve{\Lambda} + 1/2, T/16} \,
	 (\Lambda+ \breve{\Lambda}  + 1/2)  }\,
	e^{(2\Lambda+5/4)T/4} \,
	\exp\Big( \frac{\Lambda + \breve{\Lambda} + 4}{e^{\gamma}} \Big),
\end{align*}
provided that $T/16 \le 1/(2\Lambda + 2\breve{\Lambda} + 3)$.

Summarizing, since $\breve{\Lambda} \in [\Lambda, \Lambda +1)$, we obtain
\begin{align*}
	c_{\alpha,\beta,T}  
	& \ge 
	\frac{\sqrt{2}}{\pi^{2\Lambda + 7/2}}
	\exp\Big(- \frac{2\Lambda + 5}{e^{\gamma}} \Big)\,
	e^{-1} \\
	& = \frac{1}{2^{-1/2} \pi^{5/2}\exp(4e^{-\gamma}+1) \, [\pi^2 \exp(2e^{-\gamma})]^{\Lambda+1/2}}
	\qquad \textrm{if} \;\; T = 16/(4\Lambda + 5),\\
	C_{\alpha,\beta,T}  
	& \le 
	\frac{2^{5 \Lambda + 7}}{\sqrt{\pi}}
	\times 
	\begin{cases}
	e \,	\mathfrak{w}_0(\sqrt{e}\, \mathfrak{w}_1/2)^{2\Lambda + 1} &
		\;\; \textrm{if} \;\; T = 16/(4\Lambda+5) \\
	e^{1/3}  \, \mathfrak{w}'_0 (\pi/8)^{2\Lambda} &
		\;\;\textrm{if} \;\;  T = 16/(4\Lambda+5)^{2}
	\end{cases} \\
	& =
	\begin{cases}
	16 \sqrt{2} \pi^{-1/2} e \, \mathfrak{w}_0 \, [8 e\, \mathfrak{w}_1^2]^{\Lambda+1/2} &
		\;\; \textrm{if} \;\; T = 16/(4\Lambda+5) \\
	128 \sqrt{2} \pi^{-3/2} e^{1/3}\, \mathfrak{w}'_0\, [\pi^2/2]^{\Lambda+1/2} &
		\;\;\textrm{if} \;\;  T = 16/(4\Lambda+5)^{2}
	\end{cases},\\
	\frac{C_{\a,\b,T}}{c_{\a,\b,T}}  
& \le
2^{ 5\Lambda + 13/2} 
\pi^{2\Lambda + 3} \,
\exp\Big( \frac{2\Lambda + 5}{e^{\gamma}} \Big)
\times 
\begin{cases}
e^{2} \,	\mathfrak{w}_0(\sqrt{e}\, \mathfrak{w}_1/2)^{2\Lambda +1} &
	\;\; \textrm{if} \;\; T = 16/(4\Lambda+5) \\
e^{1/5} \,	\sqrt[4]{e} \, \mathfrak{w}'_0 (\pi/8)^{2\Lambda} &
	\;\;\textrm{if} \;\;  T = 16/(4\Lambda+5)^{2}
\end{cases}\\
& =
\begin{cases}
16 \pi^2 \exp(4e^{-\gamma}+2)\, \mathfrak{w}_0\, [8\pi^2 \exp(2e^{-\gamma}+1)\, \mathfrak{w}_1^2]^{\Lambda+1/2} &
	\;\; \textrm{if} \;\; T = 16/(4\Lambda+5) \\
128 \pi \, \mathfrak{w}'_0 \, \exp(4e^{-\gamma} + 9/20)\, [2^{-1}\pi^4 \exp(2e^{-\gamma})]^{\Lambda+1/2} &
	\;\;\textrm{if} \;\;  T = 16/(4\Lambda+5)^{2}
\end{cases}.
\end{align*}

Approximating above the relevant expressions\,\footnote{
\,One has
$2^{-1/2} \pi^{5/2}\exp(4e^{-\gamma}+1) \approx 317.694128$,
$\pi^2 \exp(2e^{-\gamma}) \approx 30.337323$,
$16\sqrt{2} \pi^{-1/2} e \, \mathfrak{w}_0 \approx 27.691605$,
$8 e\, \mathfrak{w}_1^2 \approx 22.872438$,
$128\sqrt{2} \pi^{-3/2} e^{1/3}\, \mathfrak{w}'_0 \approx 123.450697$,
$\pi^2/2 \approx 4.934802$,
$16 \pi^2 \exp(4e^{-\gamma}+2)\, \mathfrak{w}_0\ \approx 8797.460403$,
$8\pi^2 \exp(2e^{-\gamma}+1)\, \mathfrak{w}_1^2 \approx 693.888536$,
$128 \pi\,\mathfrak{w}'_0\,\exp(4e^{-\gamma} + 9/20) \approx 16213.468812$,
$2^{-1}\pi^4 \exp(2e^{-\gamma}) \approx 149.708689$.
},
we obtain the following explicit numerical bounds.
\begin{prop} \label{prop:c3}
Let $\a,\b \ge -1/2$ be such that $\a+\b+1/2 > 0$ and $2\a+2\b + 1 \notin \mathbb{N}$. Then
\begin{align*}
c_{\a,\b,T} & \ge \frac{1}{318 \cdot 31^{\a+\b+1}} \qquad \textrm{if} \;\; T = 4/(\a+\b+7/4),\\
C_{\a,\b,T} & \le
\begin{cases}
		28 \cdot 23^{\a+\b+1} &
		\;\; \textrm{if} \;\; T = 4/(\a+\b+7/4) \\
		124 \cdot 5^{\a+\b+1} &
		\;\;\textrm{if}  \;\; T = 1/(\a+\b+7/4)^2
	\end{cases}, \\
\frac{C_{\a,\b,T}}{c_{\a,\b,T}} & \le
\begin{cases}
	8798 \cdot 694^{\a+\b+1} &
	\;\; \textrm{if} \;\; T = 4/(\a+\b+7/4) \\
	16214 \cdot 150^{\a+\b+1} &
	\;\;\textrm{if}  \;\; T = 1/(\a+\b+7/4)^2
\end{cases}.
\end{align*}
\end{prop}

\subsubsection*{\textbf{\emph{The case when} $\mathbf{\Lambda \in (-1/2,0)}$.}} From Lemma \ref{lem:F} we have
$$
	c_{\a,\b,T}  
	= 
	c^{D}_{\Lambda,T/4} \, 
	{4^{\Lambda+1} b_{\a} b_{\b}} \, \frac{h_0^{\Lambda}}{h_0^{\a,\b}}, \qquad
	C_{\a,\b,T}  = C^{D}_{\Lambda,T/4} \,  
	{4^{\Lambda+2} B_{\a} B_{\b}} \, \frac{h_0^{\Lambda}}{h_0^{\a,\b}}.
$$
Using \eqref{id:formD}, \eqref{11.1} and the fact that 
\begin{align*}
	\mathbb{D}_{\breve{\lambda}+2} T/4
	\le 
	\mathfrak{B} \pi^2 \qquad \text{when \; $\lambda > -1$ \; and \; $T/16 \le 1/(2\Lambda+ 2\breve{\Lambda}+ 11)$,}
\end{align*}
we get (see also \eqref{DDD} and \eqref{omega_inq})
\begin{align*}
		c_{\a,\b,T}
		&= 
		\frac{2^{3\Lambda + 2\breve{\Lambda} + 29/2} }{\pi^2} \, e^{-(4\Lambda+6) T/4}  
		\,
		c^A_{\Lambda + \breve{\Lambda} + 9/2, T/16}\,
		\pi^{-2\breve{\Lambda} - 11/2} \,
		\frac{\Gamma(\Lambda + \breve{\Lambda} + 11/2)}{\Lambda + 2} \\
		&\qquad  \times
			\omega_{\Lambda+2}(T/4) \, \frac{1-2^{-\Lambda-2}}{e}\,
		 e^{ - \mathbb{D}_{\breve{\Lambda} +2}- \chi_{\{\a\neq -1/2\}}\mathbb{D}_{\a} - \chi_{\{\b\neq -1/2\}}\mathbb{D}_{\b} }  \\
		 & \ge 
		 2^{3\Lambda + 2\breve{\Lambda} + 29/2} \,
		 \pi^{-2\breve{\Lambda} - 15/2}
		  \, e^{-(5\Lambda+35/4) T/4}  
		 \,
		 c^A_{\Lambda + \breve{\Lambda} + 9/2, T/16}\\
		 &\qquad  \times
		  \frac{\Gamma(\Lambda + \breve{\Lambda} + 11/2)}{\Lambda + 2}
		 \,  \frac{1-2^{-\Lambda-2}}{e} \,
		 \exp\Big(- \frac{\Lambda + \breve{\Lambda} + 6}{e^{\gamma}} \Big), \\
			C_{\a,\b,T}
		&= 
		\frac{2^{2\Lambda + 2\breve{\Lambda} + 17}}{\pi^{5/2}}
		\,
		C^A_{\Lambda + \breve{\Lambda} + 7/2, T/16}\,
		\frac{\Gamma(\Lambda + \breve{\Lambda} + 9/2)}{\Lambda + 2} \, \Omega_{\Lambda+2}(T/4) \\
		&\qquad  \times \frac{1}{(1-1/e)^2} \,  \Big( \frac{4T}{\pi^2} + 1 \Big)^{\Lambda+1} \,
		(1 \vee 2^{\a-1/2}) (1 \vee 2^{\b-1/2}) \,
		2^{-\chi_{\{\a = -1/2\}} - \chi_{\{\b = -1/2\}}}  \\
		& \le
		\frac{2^{2\Lambda + 2\breve{\Lambda} + 17}}{\pi^{5/2}}
		\,
		C^A_{\Lambda + \breve{\Lambda} + 7/2, T/16} \,
		\frac{\Gamma(\Lambda + \breve{\Lambda} + 9/2)}{\Lambda + 2} \, e^{(\Lambda+5/2)T/4} \,
	  \frac{1}{(1-1/e)^2} \,  \Big( \frac{4T}{\pi^2} + 1 \Big)^{\Lambda+1},
\end{align*}
provided that $T/16 \le 1/(2\Lambda+ 2\breve{\Lambda}+ 11)$.
Since $\Lambda \in (-1/2,0)$, one has $\breve{\Lambda} = - \Lambda$ and consequently 
$$
		c_{\a,\b,T}
 \ge 
\frac{2^{ 13}}{\pi^{ 17/2}}
\, e^{-35T/16}  
\,
c^A_{9/2, T/16} \,
\Gamma(11/2) 
\,  
{\frac{1}{2e}}
\exp\Big(- \frac{6}{e^{\gamma}} \Big) 
$$
and 
$$
	C_{\a,\b,T}
 \le
\frac{2^{18}}{3\pi^{5/2}} 
\,
C^A_{ 7/2, T/16} \,
\Gamma( 9/2)  \, e^{5T/8} \, \frac{1}{(1-1/e)^2} \,  \Big( \frac{4T}{\pi^2} + 1 \Big)^{\Lambda+1},
$$
provided that $T/16 \le 1/11$.
Furthermore, using once again the above identities for $c_{\a,\b,T}$ and $C_{\a,\b,T}$, and also \eqref{DDD}, we get
\begin{align*}
	\frac{C_{\a,\b,T}}{c_{\a,\b,T}}  
	& \le
\frac{2^{4}  \pi^{6}}{9}
\,
\frac{C^A_{ 7/2, T/16} }{c^A_{ 9/2, T/16}}
 \, e^{45T/16} \, \frac{1}{(1-1/e)^2} \,  \Big( \frac{4T}{\pi^2} + 1 \Big)^{\Lambda+1} \frac{e}{1-2^{-\Lambda-2}}
\exp\Big( \frac{6}{e^{\gamma}} \Big),
\end{align*}
provided that $T/16 \le 1/11$.

Summarizing,  we obtain
\begin{align*}
	c_{\a,\b,T}
	& \ge 
	4\,
	\pi^{ - 17/2}
	\, e^{-46/11}  \,
	\exp\Big(- \frac{6}{e^{\gamma}} \Big) \qquad \textrm{if} \;\; T = 16/11, \\
	C_{\a,\b,T}
	& \le
	\frac{2^{18}}{3\pi^{5/2}} 
	  \, \frac{1}{(1-1/e)^2} \,  \Big( \frac{4T}{\pi^2} + 1 \Big)
	 \times
	\begin{cases}
e^{10/11} \,		\mathfrak{w}_0(\sqrt{e}\, \mathfrak{w}_1/2)^{4} &
		\;\; \textrm{if} \;\;  T = 16/11 \\
e^{161/484} \, \mathfrak{w}'_0 (\pi/8)^{4} &
		\;\; \textrm{if} \;\; T = 16/121
	\end{cases}, \\
		\frac{C_{\a,\b,T}}{c_{\a,\b,T}}  
	& \le
	2^{13} \, \pi^{6}
	\, \frac{(1-2^{-3/2})^{-1}}{(1-1/e)^2} \,  \Big( \frac{4T}{\pi^2} + 1 \Big)
	\exp\Big( \frac{6}{e^{\gamma}} \Big)
	\times
	\begin{cases}
e^{56/11} \,		\mathfrak{w}_0(\sqrt{e}\, \mathfrak{w}_1/2)^{4} &
		\;\; \textrm{if} \;\;  T = 16/11 \\
e^{785/484} \, \mathfrak{w}'_0 (\pi/8)^{4} &
		\;\; \textrm{if} \;\; T = 16/121
	\end{cases}.
\end{align*}

Approximating above the relevant expressions\,\footnote{
\,One has
$2^{-2}\pi^{17/2} \exp(6e^{-\gamma} + 46/11) \approx 7.996262\cdot 10^6$,
$\frac{2^{14}}{3\pi^{5/2}}\big( \frac{64}{11\pi^2} + 1 \big) \frac{e^{32/11}}{(1-1/e)^2} \mathfrak{w}_0 \mathfrak{w}_1^4 \approx
20106.561157$,
$\frac{64 \pi^{3/2}}{3} \big( \frac{64}{121\pi^2} + 1 \big) \frac{e^{161/484}}{(1-1/e)^2} \mathfrak{w}'_0  \approx
1188.639581$,
$2^9 \pi^6 \big(\frac{1}{1-2^{-3/2}}\big)
\big(\frac{e}{1-1/e}\big)^2 \big(\frac{64}{11\pi^2} + 1\big) \exp(6e^{-\gamma}+56/11)
\mathfrak{w}_0 \mathfrak{w}_1^4 \approx 9.326602\cdot 10^{10}$,
$2 \pi^{10} \big(\frac{1}{1-2^{-3/2}}\big) \frac{1}{(1-1/e)^2} \big(\frac{64}{121\pi^2} + 1\big)
\exp(6e^{-\gamma} + 785/484) \mathfrak{w}'_0 \approx 3.056413\cdot 10^8$.
}
we get the following explicit numerical bounds.
\begin{prop} \label{prop:c4}
Let $\a,\b \ge -1/2$ be such that $\a+\b+1/2 \in (-1/2,0)$. Then
\begin{align*}
c_{\a,\b,T} & \ge \frac{1}{8\cdot 10^6} \qquad \textrm{if} \;\; T = 16/11,\\
C_{\a,\b,T} & \le
\begin{cases}
		20107 &
		\;\; \textrm{if} \;\; T = 16/11 \\
		1189 &
		\;\;\textrm{if}  \;\; T = 16/121
	\end{cases}, \\
\frac{C_{\a,\b,T}}{c_{\a,\b,T}} & \le
\begin{cases}
	(28/3) \cdot 10^{10} &
	\;\; \textrm{if} \;\; T = 16/11 \\
	31 \cdot 10^7 &
	\;\;\textrm{if}  \;\; T = 16/121
\end{cases}.
\end{align*}
\end{prop}

\subsection{Comments on other related heat kernels} \label{ssec:other} \,

\medskip

The results of Sections \ref{sec:uproof} and \ref{sec:lmtime} can also be used to obtain totally explicit sharp bounds
for further heat kernels that are related to the Jacobi heat kernel. We will not pursue this matter here, however,
we indicate for interested readers some of the heat kernels in question.

\subsubsection*{\textbf{\emph{Fourier--Bessel and Fourier--Dini heat kernels}}}
Sharp bounds for the Fourier--Bessel heat kernel were proved in \cite{MSZ} by a combination of probabilistic and analytic
methods. Earlier, qualitatively sharp estimates were obtained in \cite{NR0,NR}.
A relation with the Jacobi heat kernel\,\footnote{\,The Jacobi heat kernel used in \cite{NR}
is not exactly $G_t^{\a,\b}(x,y)$, nevertheless both kernels are interrelated by a simple identity.
}
established in \cite{NR}, see \cite[Rem.\,3.3]{NR}, delivers a fully analytic proof of the sharp heat kernel bounds and,
together with the results of this paper, explicit multiplicative constants.

Sharp estimates for a certain Fourier--Dini heat kernel were proved recently in \cite{LN}. Similarly as in the Fourier--Bessel case,
there exists a connection with the Jacobi heat kernel, see \cite[Prop.\,3.6]{LN}, which together with our present results
leads to explicit multiplicative constants in the Fourier--Dini heat kernel bounds.

\subsubsection*{\textbf{\emph{Heat kernels related to balls and simplices}}}
Sharp estimates for heat kernels corresponding to certain orthogonal expansions in multi-dimensional balls and simplices were
obtained in \cite{NSS2}. Those kernels are related to the Jacobi heat kernel via certain integral formulas, see
\cite[Sec.\,5,6]{NSS2}, which, together with the present results, can be used to find explicit multiplicative constants
in the bounds in question.
Prior to these results, qualitatively sharp estimates for the discussed heat kernels had been proved independently
in \cite{SjSz} and \cite{KPX, KPX2}.

\subsubsection*{\textbf{\emph{Heat kernels related to cones and conical surfaces}}}
Sharp bounds for heat kernels associated with  orthogonal expansions in multi-dimensional cones and conical surfaces
were derived in \cite{HK}. Again, there is a connection with the Jacobi heat kernel via integral formulas
(see \cite[Lem.\,2.1, Lem.\,2.2]{HK}) which enables
one to find, with the aid of the present paper, explicit multiplicative constants in the heat kernel estimates.

\subsubsection*{\textbf{\emph{Heat kernels related to other solids and surfaces of revolution}}}
There are other domains admitting orthogonal expansions for which the related heat kernels are potentially
related with the Jacobi heat kernel, e.g.\ multi-dimensional paraboloids and hyperboloids, and their surfaces, cf.\ \cite{Xu2,Xu4},
but this issue and possible consequences in the context of the present paper remain to be investigated.

\newpage

\section*{Appendix: tabulation of constants and auxiliary functions}

\begin{table}[h]
\caption{Numerical constants} \label{tab:num}
\begin{tabular}{|c|c|c|}
\hline
constant & value & approximation \\ \hline\hline 
$\gamma$ & $-\Gamma'(1)$ & $0.577216$ \\[2pt] 
$\mathfrak{b}$ \rule[-2pt]{0pt}{15pt} & $2/\pi^2$ & $0.202642$ \\[2pt] 
$\mathfrak{B}$ & $1/2$ & $0.5$ \\[2pt]
$\mathfrak{w}_0$ & $\big(1+\frac{21}{500}\big)\frac{1}{1+256/(27\pi^3)}$ & $0.797983$ \\[2pt] 
$\mathfrak{w}'_0$ & $\big(1+\frac{1}{1000}\big) e$ & $2.721$ \\[2pt]
$\mathfrak{w}_1$ & $\frac{\pi}4\big(1+\frac{256}{27\pi^3}\big)$ & $1.025567$ \\[2pt] \hline
\end{tabular}
\end{table}

\begin{table}[h]
\caption{Direct parametric constants} \label{tab:par}
\begin{tabular}{|c|c|c|}
\hline
constant & formula & parameter range \\ \hline\hline
$h_n^{\ab}$ \rule[-0.5pt]{0pt}{15pt} & $\frac{2^{\alpha+\beta+1}\Gamma(n+\alpha+1)\Gamma(n+\beta+1)}
{(2n+\alpha+\beta+1)\Gamma(n+\alpha+\beta+1)\Gamma(n+1)}$ & $\ab > -1$, $n \ge 1$ \\[3pt] 
$h_0^{\ab}$ & $\frac{2^{\a+\b+1} \Gamma(\a+1)\Gamma(\b+1)}{\Gamma(\a+\b+2)}$ & $\ab > -1$, $n=0$ \\[3pt] 
$h_0^{\a+\b+1/2}$ & $\sqrt{\pi}\, \frac{\Gamma(\a+\b+3/2)}{\Gamma(\a+\b+2)}$ & $\a+\b > -3/2$ \\[3pt] 
$k_{\a}$ & $\frac{2^{\a-1/2}|\a+1/2|\Gamma(\a+1)}{\sqrt{\pi}\,\Gamma(\a+3/2)}$ & $\a > -1$ \\[3pt] 
$K_{\a}$ & $\frac{\Gamma(\a+1)}{\sqrt{\pi}\,\Gamma(\a+3/2)}$ & $\a > -1$ \\[3pt] 
$l_{\a}$ & $\frac{|\a+1/2|}{4(\a+1)(\a+2)}$ & $\a > -1$ \\[3pt]
$L_{\a}$ & $\frac{1}{(\a+1)(\a+2)}$ & $\a > -1$ \\[3pt]
$n_{\a}$ & $(1\wedge 2^{1/2-\a})\frac{\sqrt{\pi}\,\Gamma(\a+1/2)}{\Gamma(\a+1)}$ & $\a > -1/2$ \\[3pt]
$N_{\a}$ & $(1\vee 2^{1/2-\a})\frac{\sqrt{\pi}\,\Gamma(\a+1/2)}{\Gamma(\a+1)}$ & $\a > -1/2$ \\[3pt]
$m_{\omega}$ & $\frac{1-2^{-\omega}}{e\,\omega}$ & $\omega > 0$ \\[3pt]
$M_{\omega}$ & $\frac{2}{(1-1/e)^2\,\omega}$ & $\omega > 0$ \\[3pt]
$\mathbb{D}_{\a}$ & $\Gamma(\a+3/2)^{1/(\a+1/2)}$ & $-1/2 \neq \alpha > -3/2$ \\[3pt]
$\mathbb{D}_{-1/2}$ & $\exp(-\gamma)$ & $\alpha=-1/2$ \\[3pt]
$\mathbb{E}_{\a}$ & $(\a+1/2)/e + \exp(-\gamma)$ & $\alpha \ge -1/2$ \\[3pt]
$b_{\alpha}$ & $(1 \wedge 2^{\a-1/2})\frac{\Gamma(\a+1)}{\sqrt{\pi}}\exp(-\mathbb{D}_{\a})$ & $\alpha > -1/2$ \\[3pt]
$B_{\alpha}$ & $(1 \vee 2^{\a-1/2})\frac{\Gamma(\a+1)}{\sqrt{\pi}}$ & $\alpha > -1/2$ \\[3pt] 
$b_{-1/2},\,B_{-1/2}$ & 1/2 & $\alpha = -1/2$ \\[3pt]
$\Lambda\equiv\Lambda_{\a,\b}$ & $\a+\b+1/2$ & $\a,\b > -1$ \\[3pt]
$\breve{\lambda}$ & $\lceil 2\lambda \rceil - \lambda$ & $\lambda \in \mathbb{R}$ \\[3pt]
\hline
\end{tabular}
\end{table}

\begin{table} 
\caption{Selected auxiliary functions} \label{tab:fun}
\begin{tabular}{|c|c|c|}
\hline
function & formula & parameter/argument range \\ \hline\hline
$F(u,v) \equiv F_{\theta,\varphi}(u,v)$ \rule[-0.5pt]{0pt}{15pt} & $\big[\arccos\big(u\sin\frac{\theta}2\sin\frac{\varphi}2
+ v\cos\frac{\theta}2\cos\frac{\varphi}2\big)\big]^2$ & $u,v \in [-1,1]$, $\theta,\varphi \in [0,\pi] $ \\[3pt]
$\Psi_{\a}^{\kappa}(t,\theta,\varphi)$ & $\big[ \mathbb{D}_{\a}t \vee \kappa \theta \varphi \big]^{-\a-1/2}$
& $\theta,\varphi \in [0,\pi]$, $t, \kappa > 0$ \\[3pt]
$\Psi(\varphi)$ & $\varphi/\sin\varphi$ & $\varphi \in [0,\pi]$ \\[3pt]
$\omega_{\lambda}(t)$ & $2^{\lceil 2\lambda\rceil/2-\lambda}
	e^{(\lambda-\lceil 2\lambda \rceil/2)(\lambda+\lceil 2\lambda \rceil/2 +1)t}$ & $t,\lambda > 0$ \\[3pt]
$\Omega_{\lambda}(t)$ & $2^{\lceil 2\lambda\rceil/2-\lambda-1/2}
	e^{(\lambda-\lceil 2\lambda \rceil/2+1/2)(\lambda+\lceil 2\lambda \rceil/2 +1/2)t}$ & $t,\lambda >0$ \\[3pt]
$q_{\lambda}(t)$ & $2(\lambda+1)(\lambda+2)e^{-t(\lambda+3/2)}$ & $t>0$, $\lambda > -1$ \\[3pt]
$\widetilde{q}_{\lambda}(t)$ & $2(\lambda+2)e^{-t(\lambda+3/2)}$ & $t>0$, $\lambda > -3/2$ \\[3pt]
\hline
\end{tabular}
\end{table}

\newpage


\end{document}